\documentclass{amsart}
\usepackage{enumerate}
\usepackage{color}
\usepackage{verbatim}
\usepackage{graphicx,amssymb}
\makeindex

\newtheorem{thm}{Theorem}[section]
\newtheorem{cor}[thm]{Corollary}
\newtheorem{lem}[thm]{Lemma}
\newtheorem{prop}[thm]{Proposition}
\theoremstyle{definition}
\newtheorem{defn}[thm]{Definition}
\newtheorem{prob}[thm]{Problem}
\newtheorem{cryprob}[thm]{Cryptographic problem}
\newtheorem{comp}[thm]{Computation}
\newtheorem{algo}[thm]{Algorithm}
\newtheorem{crypt}[thm]{Cryptosystem}
\newtheorem{conj}[thm]{Conjecture}
\newtheorem{exam}[thm]{Example}

\theoremstyle{remark}
\newtheorem{rem}[thm]{Remark}

\begin{document}

\title[Generalizations of Laver tables]{Generalizations of Laver tables}
\author{Joseph Van Name}
\email{jvanname@mail.usf.edu}%


\begin{abstract}
We shall generalize the notion of a Laver table to algebras which may have many generators, several fundamental operations, fundamental operations of arity higher than 2, and to algebras where only some of the operations are self-distributive or where the operations satisfy a generalized version of self-distributivity. These algebras mimic the algebras of rank-into-rank embeddings $\mathcal{E}_{\lambda}/\equiv^{\gamma}$ in the sense that composition and the notion of a critical point make sense for these sorts of algebras.
\end{abstract}
\maketitle

\tableofcontents

\section{Introduction}
In this monograph, we shall generalize the notion of a Laver table to several broad classes of self-distributive structures including multigenic Laver tables (which can have arbitrary number of generators), permutative LD-systems, Laver-like LD-systems, partially endomorphic Laver tables (which can have arity and non-distributive operations), twistedly endomorphic Laver tables, Laver-like partially endomorphic algebras and other classes of structures. These structures mimic the self-distributive algebraic structures that one obtains from the quotient algebras $\mathcal{E}_{\lambda}/\equiv^{\gamma}$ of elementary embeddings in the following ways:

\begin{enumerate}
\item Each algebra $\mathcal{E}_{\lambda}/\equiv^{\gamma}$ is a Laver-like LD-system.

\item The notion of a critical point arises naturally in the theory of generalizations of Laver tables. Furthermore, the notion of a critical point is essential to the theory of these generalizations of Laver tables. From the notion of a critical point, one can algebraize the notion of the congruence $\equiv^{\gamma}$ on algebras of elementary embeddings whenever $\gamma$ is a critical point.

\item These algebras may be endowed with a composition operation that mimics the composition of elementary embeddings.

\item The entire algebra $\mathcal{E}_{\lambda}$ may be generalized to an algebraic context as well since $\mathcal{E}_{\lambda}$ is isomorphic to a dense $G_{\delta}$ subset of the inverse limit of topological spaces $\varprojlim_{\gamma}\mathcal{E}_{\lambda}/\equiv^{\gamma}$ where each $\mathcal{E}_{\lambda}/\equiv^{\gamma}$ is given the discrete topology. In this paper, we have used the algebraization of $\mathcal{E}_{\lambda}$ to show that the inverse limits of multigenic Laver tables contain many free sub-LD-systems.
\end{enumerate}

The multigenic Laver tables are self-distributive algebras which extend the double inductive construction of the classical Laver tables to algebras with an arbitrary number of generators. Similarly, the (partially) endomorphic Laver tables, and twistedly endomorphic Laver tables extend the double inductive construction of the classical Laver table to structures of arbitrary arity which satisfy various notions of distributivity. The Laver-like LD-systems are precisely the self-distributive algebras which satisfy the Laver-Steel Theorem (theorems \ref{j3rti90} and \ref{235wfe}) for $\mathcal{E}_{\lambda}/\equiv^{\gamma}$ while the permutative LD-systems are only required to satisfy a consequence of the Laver-Steel Theorem.

Most of the main results about the classical Laver tables extend to the generalizations of Laver tables. For example, under strong large cardinal hypotheses, the free LD-systems on an arbitrary number of generators can be embedded into inverse limits of multigenic Laver tables of the form $(A^{\leq 2^{n}})^{+}$. Furthermore, the algebras $(A^{\leq 2^{n}})^{+}$, although easy to define, have a great amount of combinatorial complexity, and all of the
combinatorial complexity of the algebras $(A^{\leq 2^{n}})^{+}$ is contained in the classical Laver tables and the final matrix which is a large portion of the Sierpinski triangle. Many patterns appear in the algebras $(A^{\leq 2^{n}})^{+}$, but most of these patterns have not rigorously proven or otherwise explained. As the 2-cocycle and 3-cocycle groups have been computed in \cite{DL}, the 2-cocycle groups of the finite reduced Laver-like LD-systems have also been calculated, and we expect for other cocycle groups of generalizations of Laver tables to be calculable as well.

We hope and expect for our investigations into the generalizations of Laver tables to pave the way towards a much greater understanding of the algebras of elementary embeddings and towards many results about finite and countable structures which can only be proven using strong large cardinal hypotheses. In section \ref{t4i80t4qwkao2}, we have began to use large cardinals to prove results about finite algebras which do not have any known ZFC proof.

$\textbf{A guide to reading this monograph:}$ This monograph is the first work that studies the generalizations of the notion of a Laver tables. Furthermore, we intend for this work to be the main reference for the study of generalizations of Laver tables and algebras of rank-into-rank embeddings for at least several years from publication. We shall make this work self-contained and accessible to all mathematicians except for Section \ref{t4i80t4qwkao2} (and to a lesser extent Section \ref{4th2u1xczma9t4gswa}) which requires some previous knowledge of set theory. While this work is mostly accessible to non-set theorists, the intuition behind the salient notions such as that of a critical point and the composition operation in reduced permutative LD-systems arise from the rank-into-rank embeddings. For an introduction to the algebras of rank-into-rank elementary embeddings accessible to people without much of a set-theoretic background, one is referred to Chapter 12 of \cite{D}. For a similarly developed introduction to the algebras of elementary embeddings for set theorists, one is referred to Chapter 11 of \cite{FK}. For a general treatment of large cardinals, see \cite{K}.

The study of the algebras of elementary embeddings is currently mostly reducible to algebra with little pure set theory involved. Therefore, the study of the algebras of elementary embeddings is currently accessible to set-theorists and non-set theorists alike. On the other hand, the theory of the algebras of elementary embeddings is still at a rather primitive stage of development. We expect that in the future, as the theory of the algebras of elementary embeddings grows more advanced, mathematicians would largely investigate the algebras of elementary embeddings in forcing extensions.

Many results in this monograph are proven using a descending-ascending double induction. Furthermore, in order to develop endomorphic Laver tables to sufficient generality, the proofs of the results on endomorphic Laver tables are by a transfinite induction. Since the notion of a Laver table shall be generalized to a universal algebraic setting, we will need to use the language of universal algebra in this monograph even though we shall not use the deep results of universal algebra in this paper. In this monograph, the theory of the generalizations of the Laver tables where there is a single binary operation shall be developed to a much deeper extent than the generalizations of the Laver tables where there are operations of higher arity. 
Chapters 6 and above may be read without having read chapters 4 and 5 (which cover endomorphic algebras).

\subsection{Self-distributivity from elementary embeddings}
\label{4th2u1xczma9t4gswa}
An LD-system is an algebra $(X,*)$ that satisfies the identity $x*(y*z)=(x*y)*(x*z)$. Most interesting LD-systems are not commutative nor are they associative. LD-systems appear in topology in the study of knots and braids \cite{J} and in set theory among the very large cardinals \cite{L95}. LD-systems have also been studied out of a purely algebraic interest \cite{TK}.
The notion of an LD-system is algebraically motivated by the fact that LD-systems have many inner endomorphisms. If $*$ is a binary operation on a set $X$ and $a\in X$, then define $L_{*,a}:X\rightarrow X$ by letting $L_{*,a}(x)=a*x$. Then clearly $(X,*)$ is an LD-system if and only if each function $L_{*,a}$ is an endomorphism (the mapping $L_{*,a}$ shall be called an inner endomorphism). If each function $L_{*,a}$ is an automorphism, then we shall call the algebra $(X,*)$ a \index{rack}\emph{rack}, and a rack that satisfies the identity $x*x=x$ is called a \index{quandle}\emph{quandle}. The quandles and also racks are prominent in knot theory since tame knots can be completely characterized by their knot quandles. Furthermore, braid groups act on tuples from racks, so racks and quandles give braid invariants as well. Racks never contain free LD-subsystems since racks satisfy the identity $(x*x)*y=x*y$ which is not satisfied by free LD-systems even on one generator. The infinite strand braid group $B_{\infty}$ may be endowed with a self-distributive operation $*$ (known as shifted conjugacy) so that every element $x\in B_{\infty}$ freely generates a sub-LD-system of $(B_{\infty},*)$  \cite{D94}. The only generalizations of Laver tables which are also racks are the trivial racks. While the inner endomorphisms on racks are always bijective, the inner endomorphisms on generalizations on Laver tables are either non-surjective or the identity function. There is currently no known application of generalizations of Laver tables to low dimensional topology.

In set theory, the classical Laver tables form a sequence of finite LD-systems which were discovered by Richard Laver in the late 1980's in his investigation of algebras of rank-into-rank embeddings \cite{L95}. The existence of a rank-into-rank embedding is among the strongest large cardinal axioms, and the classical Laver tables and their generalizations are the only known finite structures which arise from large cardinals and whose properties are investigated using large cardinals. The classical Laver tables are also of a purely algebraic interest since they are used to classify all LD-systems generated by a single element \cite{AD97},\cite{AD97+}.

An uncountable cardinal $\kappa$ is said to be \index{inaccessible cardinal}\emph{inaccessible} if $2^{\lambda}<\kappa$ whenever $\lambda<\kappa$ and $\sum_{i\in I}\lambda_{i}<\kappa$ whenever $|I|<\kappa$ and $\lambda<\kappa$. If $\kappa$ is inaccessible cardinal, then $V_{\kappa}$ is a model of ZFC with no inaccessible cardinals. Therefore, since there is a model of ZFC with no inaccessible cardinals, we conclude that the existence of an inaccessible cardinal cannot be proven in ZFC. The most noticeable feature of inaccessible cardinals is their immense size. The notion of inaccessibility is said to be a \index{large cardinal}\emph{large cardinal} notion and is among the weakest of all large cardinal notions.

Let $V$ be the class of all sets, and suppose that $M$ is a subclass of $V$. Then we say that $M$ is transitive if $y\in x,x\in M$ implies that $y\in M$ as well. If $j$ is an elementary embedding between two transitive classes, then $j(\alpha)\geq\alpha$ for each ordinal $\alpha$. If $j$ is an elementary embedding from one transitive class to another transitive class, then define $\mathrm{crit}(j)$ to be the least ordinal $\kappa$ with $j(\kappa)\neq\kappa$. A cardinal $\kappa$ is said to be \emph{measurable}\index{measurable} if there exists some transitive class $M$ and elementary embedding $j:(V,\in)\rightarrow(M,\in)$ with $\mathrm{crit}(j)=\kappa$. Every measurable cardinal is inaccessible, and there are many inaccessible cardinals below the first measurable cardinal. There are even many named large cardinal notions between inaccessibility and measurability, but there are also many large cardinal notions stronger than measurability as well. By imposing conditions on the elementary embedding $j:V\rightarrow M$, we obtain the following large cardinal axioms which are stronger than measurability.

Suppose that $j:V\rightarrow M$ is an elementary embedding where $M$ is a transitive class and $\kappa=\mathrm{crit}(j)$. Then
\begin{enumerate}
\item if $\lambda>\kappa$ and $V_{\lambda}\subseteq M$, then $j$ is said to be a $\lambda$-strong embedding and
$j$ is said to be a \index{$\lambda$-strong} $\lambda$-strong cardinal,

\item if $V_{j(\kappa)}\subseteq M$, then $j$ is said to be a \index{superstrong}\emph{superstrong embedding} and $\kappa$ is said to be a \emph{superstrong cardinal}, and

\item if $\lambda>\kappa$ and $M^{\lambda}\subseteq M$, then $j$ is said to be a \index{$\lambda$-supercompact}\emph{$\lambda$-supercompact embedding} and
$\lambda$ is said to be a \emph{$\lambda$-supercompact cardinal}.
\end{enumerate}

Take note that strongness, superstrongness, and supercompactness impose conditions on the model $M$ that state that the model $M$ is in some sense close to $V$. Given an elementary embedding $j:V\rightarrow M$, as $M$ becomes closer to $V$, the cardinal $\mathrm{crit}(j)$ satisfies stronger large cardinal axioms. Reindhart proposed the existence of an elementary embedding $j:V\rightarrow V$ as a nearly ultimate large cardinal axiom since $M=V$ is the strongest condition one can impose on $M$. However, this nearly ultimate large cardinal axiom is inconsistent with ZFC.

\begin{thm}
\index{Kunen inconsistency}(Kunen inconsistency) There is no non-trivial elementary embedding $j:V\rightarrow V$.
Furthermore, if $j:V_{\alpha}\rightarrow V_{\alpha}$ is a non-trivial elementary embedding, $\mathrm{crit}(j)=\kappa$, and $\lambda=\sup(j^{n}(\kappa))$, then $\alpha=\lambda$ or $\alpha=\lambda+1$.
\end{thm}

We shall call a non-trivial elementary embedding $j:V_{\lambda}\rightarrow V_{\lambda}$ a \index{rank-into-rank}\emph{rank-into-rank embedding} or
an \index{I3}I3 embedding while a non-trivial elementary embedding $j:V_{\lambda+1}\rightarrow V_{\lambda+1}$ is known as an \index{I1}I1 embedding. The critical point of a rank-into-rank embedding shall be called a \emph{rank-into-rank cardinal}. The notion of a rank-into-rank embedding is a slight weakening of the notion of an elementary embedding $j:V\rightarrow V$ since the models $V_{\lambda},V_{\lambda+1}$ are very similar to $V$; the model $V_{\lambda}$ is always a model of ZFC with most kinds of large cardinals including supercompact, strong, and superstrong cardinals. The rank-into-rank cardinals are nearly as close as one can get to the Kunen inconsistency without being inconsistent, and the rank-into-rank cardinals are near the very top of the large cardinal hierarchy in terms of consistency strength.

If $j:V_{\lambda+1}\rightarrow V_{\lambda+1}$ is an I1-embedding, then $j|_{V_{\lambda}}:V_{\lambda}\rightarrow V_{\lambda}$ is an I3 embedding. If there is an I1-embedding $j:V_{\lambda+1}\rightarrow V_{\lambda+1}$, then there is some I3-embedding
$k:V_{\mu}\rightarrow V_{\mu}$ where $\mu<\mathrm{crit}(j)$. The large cardinals axioms between I3 and I1 have been studied in \cite{L97}.

We shall write \index{$\mathcal{E}_{\lambda}$}$\mathcal{E}_{\lambda}$ for the set of all elementary embeddings from $V_{\lambda}$ to $V_{\lambda}$ and we shall write
\index{$\mathcal{E}_{\lambda}^{+}$}$\mathcal{E}_{\lambda}^{+}$ for the set $\mathcal{E}_{\lambda}\setminus\{1_{V_{\lambda}}\}$ of all non-trivial elementary embeddings from
$V_{\lambda}$ to $V_{\lambda}$. Suppose that $j,k\in\mathcal{E}_{\lambda}$. Then for all $\alpha<\lambda$, we have
$k|_{V_{\alpha}}\in V_{\lambda}$, so one can apply $j$ to the truncated elementary embedding $k|_{V_{\alpha}}\in V_{\lambda}$ to obtain $j(k|_{V_{\alpha}})$. By elementarity, $\mathrm{Dom}(j(k|_{V_{\alpha}}))=V_{j(\alpha)}$ and if $\alpha<\beta$, then
$j(k|_{V_{\alpha}})\subseteq j(k|_{V_{\beta}})$. Therefore, $\bigcup_{\alpha<\lambda}j(k|_{V_{\alpha}})$ is a total function from
$V_{\lambda}$ to $V_{\lambda}$.

\begin{prop}
Suppose that $j,k\in\mathcal{E}_{\lambda}$. Then the function $\bigcup_{\alpha<\lambda}j(k|_{V_{\alpha}}):V_{\lambda}\rightarrow V_{\lambda}$ is also an elementary embedding.
\end{prop}

We therefore define an operation $*$ which we shall call application on $\mathcal{E}_{\lambda}$ by letting $j*k=\bigcup_{\alpha<\lambda}j(k|_{V_{\alpha}})$. The algebraic structure $(\mathcal{E}_{\lambda},*,\circ)$ satisfies the following identities:
\begin{enumerate}
\item $j*(k*l)=(j*k)*(j*l)$, \label{4t2hugd3fbhvo}

\item $j\circ k=(j*k)\circ j,j*(k\circ l)=(j*k)\circ (j*l),$ \label{JBNajo34g6ft}

$(j\circ k)*l=j*(k*l),(j\circ k)\circ l=j\circ(k\circ l).$
\end{enumerate}
Laver has shown \cite{D}\cite{L92}\cite{L95}
that if $j\in\mathcal{E}_{\lambda}$ is a non-identity elementary embedding, then the subalgebra of $(\mathcal{E}_{\lambda},*)$ generated by $j$ is freely generated by $j$ with respect to the identity \ref{4t2hugd3fbhvo}, and the subalgebra of $(\mathcal{E}_{\lambda},*,\circ)$ generated by $j$ is freely generated by $j$ with respect to the identities in \ref{4t2hugd3fbhvo} and \ref{JBNajo34g6ft}. Using the algebras of elementary embeddings, Laver around 1989 has proven the following result \cite{L92}.

\begin{thm}
Suppose that there exists a rank-into-rank cardinal. Then the word problem for the free LD-system on one generator is decidable.

\label{t342hu234r}
\end{thm}

While Theorem \ref{t342hu234r} was originally proven using strong large cardinal hypotheses, the large cardinal hypotheses from
Theorem \ref{t342hu234r} were soon removed by Dehornoy \cite{D89}.
\begin{thm}
(ZFC) The word problem for the free LD-systems on an arbitrary number of generators is decidable.
\end{thm}

\subsection{The classical Laver tables from elementary embeddings}

The classical Laver tables are up-to-isomorphism the monogenerated subalgebras of the quotients of $\mathcal{E}_{\lambda}$ modulo some rank. While the large cardinal hypotheses have been removed from Theorem \ref{t342hu234r}, the large cardinal hypotheses have not yet been removed from the result stating that the free LD-system can be embedded into an inverse limit of finite LD-systems.

Suppose now that $\lambda$ is a cardinal and $\gamma<\lambda$ is a limit ordinal. Then define an equivalence relation
$\equiv^{\gamma}$ on $\mathcal{E}_{\lambda}$ by letting $j\equiv^{\gamma}k$ if and only if
$j(x)\cap V_{\gamma}=k(x)\cap V_{\gamma}$ for each $x\in V_{\gamma}$. It is easy to show that
$j\equiv^{\gamma}k$ if and only if $j(x)\cap V_{\gamma}=k(x)\cap V_{\gamma}$ whenever $x\in V_{\lambda}$. The equivalence relation
$\equiv^{\gamma}$ is a congruence on $(\mathcal{E}_{\lambda},*,\circ)$, so one can take the quotient algebra
$\mathcal{E}_{\lambda}/\equiv^{\gamma}$.

If $j\in\mathcal{E}_{\lambda}$ is a non-identity elementary embedding, then define $\mathrm{Iter}(j)$ to be the smallest subset
of $\mathcal{E}_{\lambda}$ closed under application and that contains $j$. Then $\{\mathrm{crit}(k)\mid k\in\mathrm{Iter}(j)\}$ is cofinal in
$\lambda$ of order type $\omega$. Let $\mathrm{crit}_{n}(j)$ denote the $n+1$-th element of the sequence $\{\mathrm{crit}(k)\mid k\in\mathrm{Iter}(j)\}$. The algebras $\mathrm{Iter}(j)/\equiv^{\mathrm{crit}_{n}(j)}$ are finite algebras of cardinality $2^{n}$ and they can be completely described algebraically.

\begin{thm}
Suppose that $n$ is a natural number. Then there exists a unique LD-system $(\{1,\ldots,2^{n}\},*_{n})$ such that
$2^{n}*_{n}1=1$ and $x*_{n}1=x+1$ for $x<2^{n}$.
\label{tw4hu}
\end{thm}

The LD-system mentioned in the above theorem shall be denoted by $A_{n}$ and shall be called the \index{classical Laver table}\emph{$n$-th classical Laver table} or the $2^{n}\times 2^{n}$-classical Laver table. 

Define \index{$j_{[n]}$} $j_{[n]}$ for all natural numbers $n$ by letting $j_{[1]}=j$ and $j_{[n+1]}=j_{[n]}*j$.

\begin{thm}
\begin{enumerate}
\item For all limit ordinals $\gamma<\lambda$ and $j\in\mathcal{E}_{\lambda}^{+}$, there is a unique $m$ such that
for all $k,l\in\mathrm{Iter}(j)$, we have $k\equiv^{\gamma}l$ if and only if $k\equiv^{\mathrm{crit}_{m}(j)}l.$

\item For all $m$, there is a unique isomorphism from $A_{m}$ to $\mathrm{Iter}(j)/\equiv^{\mathrm{crit}_{m}(j)}$ defined by
\[x\mapsto j_{[x]}/\equiv^{\mathrm{crit}_{m}(j)}.\]
\end{enumerate}
\label{Pwehjit}
\end{thm}

The following charts are the multiplication tables for the classical Laver tables $(A_{n},*_{n})$ for $n\leq 4$.

\begin{math}\begin{array}{r|r}
*_{0} &1 \\
\hline
1   & 1 \\
\end{array}\end{math}\,\,
\begin{math}\begin{array}{r|rr}
*_{1} &1 & 2\\
\hline
1   & 2 & 2\\

2   & 1 & 2\\
\end{array}\end{math}\,\,
\begin{math}\begin{array}{r|rrrr}
*_{2} &1 & 2 & 3 & 4 \\
\hline
1   & 2 & 4 & 2 & 4 \\

2   & 3 & 4 & 3 & 4 \\

3   & 4 & 4 & 4 & 4 \\

4   & 1 & 2 & 3 & 4 \\
\end{array}\end{math}\,\,
\begin{math}\begin{array}{r|rrrrrrrr}

*_{3} &1&2&3 &4 &5 &6 &7 &8 \\
\hline

1  &2&4&6&8&2&4&6&8 \\
2  &3&4&7&8&3&4&7&8 \\
3  &4&8&4&8&4&8&4&8 \\
4  &5&6&7&8&5&6&7&8 \\
5  &6&8&6&8&6&8&6&8 \\
6  &7&8&7&8&7&8&7&8 \\
7  &8&8&8&8&8&8&8&8 \\
8  &1&2&3&4&5&6&7&8
 
\end{array}\end{math}

\begin{math}\begin{array}{r|rrrrrrrrrrrrrrrr}

*_{4} & 1 & 2 &  3 &  4 &  5 & 6 &  7 & 8 &  9 & 10 &  11 & 12 &  13 & 14 &  15 &  16\\

\hline

1   & 2 & 12 & 14 & 16 & 2 & 12 & 14 & 16 & 2 & 12 & 14 & 16 & 2 & 12 & 14 & 16 \\
2   & 3 & 12 & 15 & 16 & 3 & 12 & 15 & 16 & 3 & 12 & 15 & 16 & 3 & 12 & 15 & 16 \\
3   & 4 & 8 & 12 & 16 & 4 & 8 & 12 & 16 & 4 & 8 & 12 & 16 & 4 & 8 & 12 & 16 \\
4   & 5 & 6 & 7 & 8 & 13 & 14 & 15 & 16 & 5 & 6 & 7 & 8 & 13 & 14 & 15 & 16 \\
5   & 6 & 8 & 14 & 16 & 6 & 8 & 14 & 16 & 6 & 8 & 14 & 16 & 6 & 8 & 14 & 16 \\
6   & 7 & 8 & 15 & 16 & 7 & 8 & 15 & 16 & 7 & 8 & 15 & 16 & 7 & 8 & 15 & 16 \\
7   & 8 & 16 & 8 & 16 & 8 & 16 & 8 & 16 & 8 & 16 & 8 & 16 & 8 & 16 & 8 & 16 \\
8   & 9 & 10 & 11 & 12 & 13 & 14 & 15 & 16 & 9 & 10 & 11 & 12 & 13 & 14 & 15 & 16 \\
9   & 10 & 12 & 14 & 16 & 10 & 12 & 14 & 16 & 10 & 12 & 14 & 16 & 10 & 12 & 14 & 16 \\
10   & 11 & 12 & 15 & 16 & 11 & 12 & 15 & 16 & 11 & 12 & 15 & 16 & 11 & 12 & 15 & 16 \\
11   & 12 & 16 & 12 & 16 & 12 & 16 & 12 & 16 & 12 & 16 & 12 & 16 & 12 & 16 & 12 & 16 \\
12   & 13 & 14 & 15 & 16 & 13 & 14 & 15 & 16 & 13 & 14 & 15 & 16 & 13 & 14 & 15 & 16 \\
13   & 14 & 16 & 14 & 16 & 14 & 16 & 14 & 16 & 14 & 16 & 14 & 16 & 14 & 16 & 14 & 16 \\
14   & 15 & 16 & 15 & 16 & 15 & 16 & 15 & 16 & 15 & 16 & 15 & 16 & 15 & 16 & 15 & 16 \\
15   & 16 & 16 & 16 & 16 & 16 & 16 & 16 & 16 & 16 & 16 & 16 & 16 & 16 & 16 & 16 & 16 \\
16   & 1 & 2 & 3 & 4 & 5 & 6 & 7 & 8 & 9 & 10 & 11 & 12 & 13 & 14 & 15 & 16 \\

\end{array}
\end{math}

One could compute the classical Laver table $A_{n}$ by hand by computing $x*_{n}y$ by a double induction which is descending on $x$ and where for all $x$ one proceeds by ascending induction on $y$.

Let us now outline some basic facts about the periodicity in the classical Laver tables.
If $x$ is a positive integer, then define $(x)_{r}$ to be the unique positive integer such that
$(x)_{r}=x\mod r$ and $1\leq x\leq r$.
\begin{prop}
The mapping $\pi:A_{n+1}\rightarrow A_{n}$ defined by $\pi(x)=(x)_{2^{n}}$ is a homomorphism.
\end{prop}
\begin{prop}
The mapping $L:A_{n}\rightarrow A_{n+1}$ defined by $L(x)=x+2^{n}$ is a homomorphism.
\end{prop}
\begin{prop}
For all $n\in\omega$ and $1\leq x\leq 2^{n}$, there exists a unique natural number $o_{n}(x)\leq n$ such that
$x*_{n}y=x*_{n}z$ if and only if $y=z\mod 2^{o_{n}(x)}$.
\end{prop}
\begin{prop}
\begin{enumerate}
\item For all $n\in\omega,1\leq x\leq 2^{n},1\leq y<z\leq 2^{o_{n}(x)}$, we have 
\[x*_{n}y<x*_{n}z.\]

\item If $x<2^{n}$, then $x<x*_{n}y$.

\item $2^{n}*_{n}y=y$.
\end{enumerate}
\end{prop}
\begin{prop}
\begin{enumerate}
\item $o_{n+1}(x)=o_{n}(x-2^{n})$ whenever $2^{n}<x<2^{n+1}$.

\item $o_{n+1}(2^{n+1})=n+1$.

\item $o_{n}(x)\leq o_{n+1}(x)\leq o_{n}(x)+1$ whenever $1\leq x\leq 2^{n}$.
\end{enumerate}
\end{prop}
\begin{prop}
\begin{enumerate}
\item $(x+2^{n})*_{n+1}y=2^{n}+(x*_{n}(y)_{2^{n}})$ whenever $1\leq x<2^{n}$.

\item $2^{n}*_{n+1}y=2^{n}+(y)_{2^{n}}$.

\item If $1<x<2^{n}$ and $o_{n+1}(x)=o_{n}(x)+1$, then $x*_{n+1}y=x*_{n}y$ and
$x*_{n+1}(O_{n}(x)+y)=2^{n}+(x*_{n}y)$ for $1\leq y\leq O_{n}(x)$

\item If $1<x<2^{n}$ and $o_{n+1}(x)=o_{n}(x)$, then there is some $c\in\{1,\ldots ,O_{n}(x)\}$ where
\begin{enumerate}
\item $x*_{n+1}y=x*_{n}y$ for $1\leq y\leq c$ and

\item $x*_{n+1}y=2^{n}+(x*_{n}y)$ for $c<y\leq O_{n}(x)$.
\end{enumerate}
\end{enumerate}
\end{prop}

In table \ref{j4ti0tg4q0}, the entries of the form $(i,i*_{4}k)$ are filled with the number $i*_{4}k$ while all the other entries are left blank. We observe that for all $n$, the set $\{(i,i*_{n}j)\mid i,j\in A_{n}\}$ is always a proper subset of the union of a Sierpinkski gasket and the line segment $\{2^{n}\}\times[1,2^{n}]$. We also observe that all the information in the classical Laver table $A_{4}$ is contained in table \ref{j4ti0tg4q0}.

\begin{math}\begin{array}{r|rrrrrrrrrrrrrrrr}   
\label{j4ti0tg4q0}
*_{4} &\! 1 &\! 2 &\! 3  &\!  4 &\!  5 &\! 6 &\! 7 &\! 8 &\! 9  &\! 10 &\! 11  &\! 12 &\! 13  &\! 14 &\! 15  &\! 16 \\

\hline

1    &\!  &\! 2&\!  &\!  &\!  &\!  &\!  &\!  &\!  &\!  &\!  &\!12&\!  &\!14&\!  &\!16\\
2    &\!  &\!  &\! 3&\!  &\!  &\!  &\!  &\!  &\!  &\!  &\!  &\!12&\!  &\!  &\!15&\!16\\
3    &\!  &\!  &\!  &\! 4&\!  &\!  &\!  &\! 8&\!  &\!  &\!  &\!12&\!  &\!  &\!  &\!16\\
4    &\!  &\!  &\!  &\!  &\! 5&\! 6&\! 7&\! 8&\!  &\!  &\!  &\!  &\!13&\!14&\!15&\!16\\
5    &\!  &\!  &\!  &\!  &\!  &\! 6&\!  &\! 8&\!  &\!  &\!  &\!  &\!  &\!14&\!  &\!16\\
6    &\!  &\!  &\!  &\!  &\!  &\!  &\! 7&\! 8&\!  &\!  &\!  &\!  &\!  &\!  &\!15&\!16\\
7    &\!  &\!  &\!  &\!  &\!  &\!  &\!  &\! 8&\!  &\!  &\!  &\!  &\!  &\!  &\!  &\!16\\
8    &\!  &\!  &\!  &\!  &\!  &\!  &\!  &\!  &\! 9&\!10&\!11&\!12&\!13&\!14&\!15&\!16\\
9    &\!  &\!  &\!  &\!  &\!  &\!  &\!  &\!  &\!  &\!10&\!  &\!12&\!  &\!14&\!  &\!16\\
10   &\!  &\!  &\!  &\!  &\!  &\!  &\!  &\!  &\!  &\!  &\!11&\!12&\!  &\!  &\!15&\!16\\
11   &\!  &\!  &\!  &\!  &\!  &\!  &\!  &\!  &\!  &\!  &\!  &\!12&\!  &\!  &\!  &\!16\\
12   &\!  &\!  &\!  &\!  &\!  &\!  &\!  &\!  &\!  &\!  &\!  &\!  &\!13&\!14&\!15&\!16\\
13   &\!  &\!  &\!  &\!  &\!  &\!  &\!  &\!  &\!  &\!  &\!  &\!  &\!  &\!14&\!  &\!16\\
14   &\!  &\!  &\!  &\!  &\!  &\!  &\!  &\!  &\!  &\!  &\!  &\!  &\!  &\!  &\!15&\!16\\
15   &\!  &\!  &\!  &\!  &\!  &\!  &\!  &\!  &\!  &\!  &\!  &\!  &\!  &\!  &\!  &\!16\\ 
16   &\! 1&\! 2&\! 3&\! 4&\! 5&\! 6&\! 7&\! 8&\! 9&\!10&\!11&\!12&\!13&\!14&\!15&\!16\\
\end{array}
\end{math}

Like the chart \ref{j4ti0tg4q0}, in the following images, for $n\in\{5,6,7,8\}$, the pixels of the form
$(i,i*_{n}j)$ are colored black while all the other pixels are colored white. One can easily recover the multiplication table for
$A_{n}$ from its image of the form \ref{4t2jip},\ldots,\ref{t42huiotg42}.
\begin{center}
\label{4t2jip}
\includegraphics{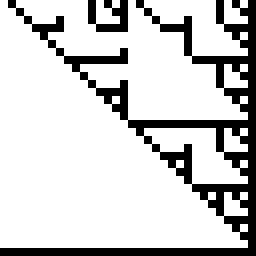}
\end{center}
\begin{center}
\includegraphics{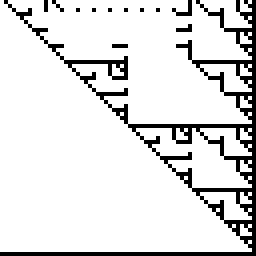}
\end{center}
\begin{center}
\includegraphics{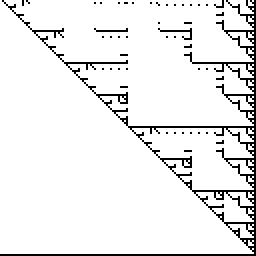}
\end{center}
\begin{center}
\label{t42huiotg42}
\includegraphics{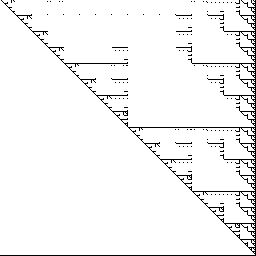}
\end{center}

By reordering the elements of $A_{n}$, one obtains more insightful images that represent the classical Laver tables.
Let \index{$L_{n}:\{0,\ldots,2^{n}-1\}\rightarrow\{0,\ldots,2^{n}-1\}$}
\[L_{n}:\{0,\ldots,2^{n}-1\}\rightarrow\{0,\ldots,2^{n}-1\}\]
be the mapping where $L_{n}(x)$ is obtained by reversing the ordering of the digits in the binary expansion of $x$. More formally, if 
\[x=\sum_{k=0}^{n-1}a_{k}2^{k}\]
where $a_{0},\ldots,a_{n-1}\in\{0,1\}$, then
\[L_{n}(x)=\sum_{k=0}^{n-1}a_{k}2^{n-1-k}.\]
Let 
\index{$L_{n}^{\sharp}:\{1,\ldots,2^{n}\}\rightarrow\{1,\ldots,2^{n}\}$}
\[L_{n}^{\sharp}:\{1,\ldots,2^{n}\}\rightarrow\{1,\ldots,2^{n}\}\]
be the mapping defined by $L_{n}^{\sharp}(x)=L_{n}(x-1)+1$. Define a binary operation \index{$\#_{n}$}$\#_{n}$ on $\{1,\ldots,2^{n}\}$ by
\[x\#_{n}y=L_{n}^{\sharp}(L_{n}^{\sharp}(x)*_{n}L_{n}^{\sharp}(y)).\]
The following images give the multiplication tables for $(\{1,\ldots ,2^{n}\},\#_{n})$ when $n\in\{1,\ldots ,4\}$.

\begin{math}\begin{array}{r|rr}

\#_{1}& 1&2 \\
\hline
1 &  2&  2 \\
2 &  1&  2 \\
\end{array}
\end{math}
\begin{math}\begin{array}{r|rrrr}
\#_{2} &1&2&3&4\\
\hline
1 &  3&  3&  4&  4 \\
2 &  4&  4&  4&  4 \\
3 &  2&  2&  4&  4 \\
4 &  1&  2&  3&  4 \\
\end{array}
\end{math}
\begin{math}\begin{array}{r|rrrrrrrr}
\#_{3}	&	1&2&3&4&5&6&7&8\\
\hline
1 & 5& 5& 6& 6& 7& 7& 8& 8 \\
2 & 6& 6& 6& 6& 8& 8& 8& 8 \\
3 & 7& 7& 7& 7& 8& 8& 8& 8 \\
4 & 8& 8& 8& 8& 8& 8& 8& 8 \\
5 & 3& 3& 4& 4& 7& 7& 8& 8 \\
6 & 4& 4& 4& 4& 8& 8& 8& 8 \\
7 & 2& 2& 4& 4& 6& 6& 8& 8 \\
8 & 1& 2& 3& 4& 5& 6& 7& 8 \\
\end{array}
\end{math}

\begin{math}\begin{array}{r|rrrrrrrrrrrrrrrr}
\#_{4} &1&2&3&4&5&6&7&8&9&10&11&12&13&14&15&16\\
\hline
1  &   9&   9&   9&   9&  12&  12&  12&  12&  14&  14&  14&  14&  16&  16&  16&  16 \\
2  &  10&  10&  10&  10&  12&  12&  12&  12&  14&  14&  14&  14&  16&  16&  16&  16 \\
3  &  11&  11&  11&  11&  12&  12&  12&  12&  15&  15&  15&  15&  16&  16&  16&  16 \\
4  &  12&  12&  12&  12&  12&  12&  12&  12&  16&  16&  16&  16&  16&  16&  16&  16 \\
5  &  13&  13&  13&  13&  14&  14&  14&  14&  15&  15&  15&  15&  16&  16&  16&  16 \\
6  &  14&  14&  14&  14&  14&  14&  14&  14&  16&  16&  16&  16&  16&  16&  16&  16 \\
7  &  15&  15&  15&  15&  15&  15&  15&  15&  16&  16&  16&  16&  16&  16&  16&  16 \\
8  &  16&  16&  16&  16&  16&  16&  16&  16&  16&  16&  16&  16&  16&  16&  16&  16 \\
9  &   5&   5&   5&   5&   8&   8&   8&   8&  14&  14&  14&  14&  16&  16&  16&  16 \\
10 &   6&   6&   6&   6&   8&   8&   8&   8&  14&  14&  14&  14&  16&  16&  16&  16 \\
11 &   7&   7&   7&   7&   8&   8&   8&   8&  15&  15&  15&  15&  16&  16&  16&  16 \\
12 &   8&   8&   8&   8&   8&   8&   8&   8&  16&  16&  16&  16&  16&  16&  16&  16 \\
13 &   3&   3&   4&   4&   7&   7&   8&   8&  11&  11&  12&  12&  15&  15&  16&  16 \\
14 &   4&   4&   4&   4&   8&   8&   8&   8&  12&  12&  12&  12&  16&  16&  16&  16 \\
15 &   2&   2&   4&   4&   6&   6&   8&   8&  10&  10&  12&  12&  14&  14&  16&  16 \\
16 &   1&   2&   3&   4&   5&   6&   7&   8&   9&  10&  11&  12&  13&  14&  15&  16 \\
\end{array}
\end{math}

In image \ref{cvxnjg4bnk94h}, the entries of the form $(x,x\#_{n}y)$ are filled with the number
$x\#_{n}y$ while all the other entries are left blank.

\begin{math}\begin{array}{r|rrrrrrrrrrrrrrrr}
\label{cvxnjg4bnk94h}
\bullet&1&2&3&4&5&6&7&8&9&10&11&12&13&14&15&16\\
\hline
1& &&&&&&&& 9&&& 12&& 14&& 16 \\
2& &&&&&&&&& 10&& 12&& 14&& 16 \\
3& &&&&&&&&&& 11& 12&&& 15& 16 \\
4  & &&&&&&&&&&& 12&&&& 16 \\
5	& &&&&&&&&&&&& 13& 14& 15& 16 \\
6	& &&&&&&&&&&&&& 14&& 16 \\
7	& &&&&&&&&&&&&&& 15& 16 \\
8  & &&&&&&&&&&&&&&& 16 \\
9	& &&&& 5&&& 8&&&&&& 14&& 16 \\
10	& &&&&& 6&& 8&&&&&& 14&& 16 \\
11	& &&&&&& 7& 8&&&&&&& 15& 16 \\
12  & &&&&&&& 8&&&&&&&& 16 \\
13	& && 3& 4&&& 7& 8&&& 11& 12&&& 15& 16 \\
14	& &&& 4&&&& 8&&&& 12&&&& 16 \\
15  & & 2&& 4&& 6&& 8&& 10&& 12&& 14&& 16 \\
16	& 1& 2& 3& 4& 5& 6& 7& 8& 9& 10& 11& 12& 13& 14& 15& 16 \\
\end{array}
\end{math}

In images \ref{32jhti924t} through \ref{3rhu9i3r2}, the pixels of the from $(x,x\#_{n}y)$ are colored black while all of the other pixels are colored white.

\begin{center}
\label{32jhti924t}
\includegraphics{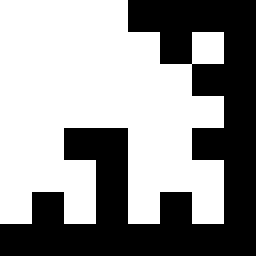}
\end{center}
\begin{center}
\includegraphics{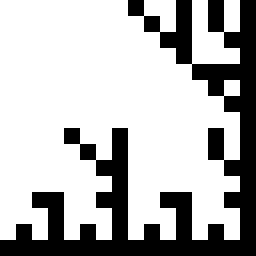}
\end{center}
\begin{center}
\includegraphics{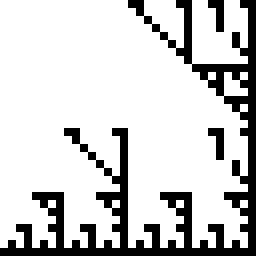}
\end{center}
\begin{center}
\includegraphics{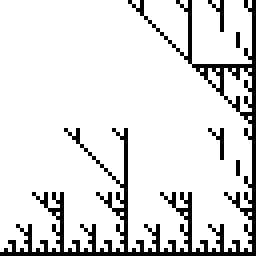}
\label{482itj4}
\end{center}
\begin{center}
\includegraphics{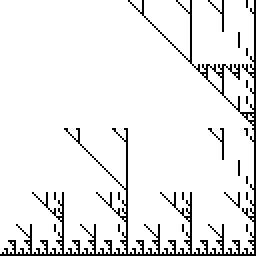}
\end{center}
\begin{center}
\label{3rhu9i3r2}
\includegraphics{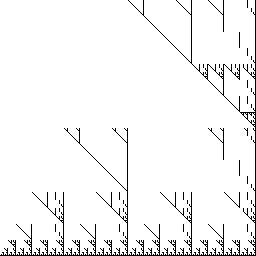}
\end{center}

The multiplication table for $(A_{n},*)$ can be recovered from its image of the same form as $\ref{32jhti924t},\ldots,\ref{3rhu9i3r2}$.
Observe also that the images of the form $\ref{32jhti924t},\ldots,\ref{3rhu9i3r2}$ are essentially the same image where as
$n$ grows larger the corresponding image gives finer detail about the classical Laver tables. If there exists a rank-into-rank cardinal, then the limit as $n\rightarrow\infty$ of the images of the form $\ref{32jhti924t},\ldots,\ref{3rhu9i3r2}$ are eventually fractal-like around every black point. Images $\ref{32jhti924t}$ through $\ref{3rhu9i3r2}$ are as well subsets of Sierpinski Triangle like fractals.
In particular, \ref{482itj4} is a subset of the fractal \ref{full sierpinski} given below.

\begin{center}
\label{full sierpinski}
\includegraphics{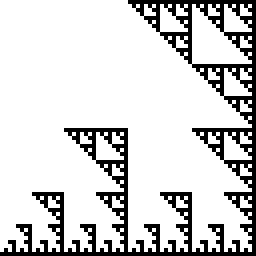}
\end{center}

\begin{center}
\includegraphics{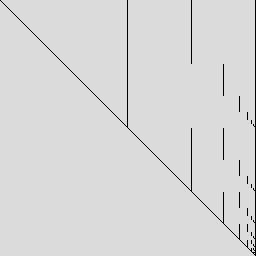}
\end{center}
The above image is a $256\times 256$-region in the $2048\times 2048$-image where the pixels of the form $(x,x\#_{11}y)$ are colored black and all other pixels are colored gray.

There are results about classical Laver tables and their generalizations which can be proven with strong large cardinal hypotheses but where no known proof in ZFC is known.
\begin{thm}
If there exists a rank-into-rank cardinal, then 
$(1)_{n\in\omega}\in\varprojlim_{n\in\omega}A_{n}$ freely generates a sub-LD-system of $\varprojlim_{n\in\omega}A_{n}$.
\label{42t0iieghjo}
\end{thm}
\begin{thm}
Suppose that there exists a rank-into-rank embedding. Then for $x,y\in A_{n}$, if $(x*_{n}x)*_{n}y=2^{n}$, then $x*_{n}y=2^{n}$ as well.
\label{fy1DWY2v4g8uah}
\end{thm}
Theorems \ref{42t0iieghjo} and \ref{fy1DWY2v4g8uah} currently do not have a proof that does not require large cardinal hypotheses,
but theorems \ref{42t0iieghjo} and \ref{fy1DWY2v4g8uah} can be proven using large cardinal hypotheses slightly weaker than the notion of a rank-into-rank cardinal, namely the hypothesis that for all $n$, there exists an $n$-huge cardinal. Theorem \ref{42t0iieghjo} is known to be independent of the axiomatic system Primitive Recursive Arithmetic \cite{RD92}.

An elementary embedding $j:V\rightarrow M$ where $M$ is a transitive class is said to be $n$-huge if
$M^{j^{n}(\kappa)}\subseteq M$ where $\kappa$ is the critical point of $j$. The critical point of an $n$-huge elementary embedding
is said to be an $n$-huge cardinal. The existence of an $n$-huge cardinal is a large cardinal axiom slightly weaker than the existence of a rank-into-rank cardinal. The wholeness axiom is a large cardinal axiom stronger than $n$-hugeness but weaker than the notion of a rank-into-rank cardinal \cite{PK}. There are also several studied large cardinal notions with consistency strength between the notion of an $n$-huge cardinal and the notion of an $n+1$-huge cardinal \cite{KS}.

\begin{thm}
Suppose that there exists a rank-into-rank embedding $j:V_{\lambda}\rightarrow V_{\lambda}$. Then for all $n$ there are
$\lambda$ many $n$-huge cardinals below $\lambda$.
\end{thm}
\begin{thm}
Suppose that for all $n\in\omega$ there exists an $n$-huge cardinal. Then the element
$(1)_{n\in\omega}\in\varprojlim_{n\in\omega}A_{n}$ freely generates a sub-LD-system of $\varprojlim_{n\in\omega}A_{n}$.
\end{thm}
\begin{thm}
Suppose that for all $n\in\omega$ there exists an $n$-huge cardinal. Then for $x,y\in A_{n}$, if $(x*_{n}x)*_{n}y=2^{n}$, then $x*_{n}y=2^{n}$ as well.
\end{thm}

Let $f:\omega\rightarrow\omega+1$ be the function where $f(n)$ is the least natural number such that
$1*_{f(n)}2^{n}<2^{f(n)}$ or $f(n)=\infty$ is no such natural number exists.

\begin{prop}
$f(n)<\omega$ for all natural numbers $n$ if and only if $(1)_{n\in\omega}$ freely generates a subalgebra of $\varprojlim_{n\in\omega}A_{n}$.
\end{prop}
\begin{thm}
If $f(n)<\omega$ for all $n\in\omega$, then the function $f$ is not primitive recursive.
\end{thm}
\begin{prop}
$f(5)>\mathrm{Ack}(9,\mathrm{Ack}(8,\mathrm{Ack}(8,254)))$.
\end{prop}
No upper bound for the growth rate of the function $f$ is known.

There is currently no known closed form expression for computing $x*_{n}y$, and due to the fast growth of the function $f$, we do not
expect for there to exist a closed form expression for $x*_{n}y$.  Randall Dougherty has constructed an algorithm that has be used to calculate $x*_{n}y$ as long as $n\leq 48$ \cite{RD95}. The limiting factor in Dougherty's algorithm is computing and storing certain values of $x*_{n}y$ in memory, and so far no one has written a program that computes even $A_{49}$ yet. 

The notion of a critical point and of composition of elementary embeddings can be generalized from set theory to purely algebraic notions about the classical Laver tables. Define an operation \index{$\circ_{n}$}$\circ_{n}$ on $A_{n}$ by

$x\circ_{n}y=\begin{cases}x &\mbox{if } y=2^{n} \\ 
x*_{n}(y+1)-1 & \mbox{whenever } y<2^{n}.\end{cases}$

\begin{prop}
Suppose $n$ is a natural number and $x,y\in\{1,\ldots ,2^{n}\}$.
\begin{enumerate}
\item $j_{[x\circ_{n}y]}\equiv^{\mathrm{crit}_{n}(j)}j_{[x]}\circ j_{[y]}$ whenever $x,y\in A_{n}$ and $j\in\mathcal{E}_{\lambda}^{+}$.

\item Let $x,y\in A_{n}$ and $j\in\mathcal{E}_{\lambda}^{+}$. Then $\mathrm{crit}(j_{[x]})<\mathrm{crit}(j_{[y]})$ if and only if $\gcd(x,2^{n})<\gcd(y,2^{n})$.
\end{enumerate}
\label{Rhu4t0HI42h2}
\end{prop}

If there exists a rank-into-rank cardinal, then Proposition \ref{t482u9onjkrw} follows from the algebraization of the notion of a critical point and composition found in Proposition \ref{Rhu4t0HI42h2}. On the other hand, Proposition \ref{t482u9onjkrw} can also be proven directly without any reference to large cardinals.
\begin{prop}
Suppose $n$ is a natural number and $x,y,z\in\{1,\ldots ,2^{n}\}$. Then
\label{t482u9onjkrw}
\begin{enumerate}
\item $(x\circ_{n}y)*z=x*_{n}(y*_{n}z)$,

\item $x*_{n}(y\circ_{n}z)=(x*_{n}y)\circ_{n}(x*_{n}z)$,

\item $x\circ_{n}y=(x*_{n}y)\circ_{n}x$,

\item if $\gcd(y,2^{n})\leq\gcd(z,2^{n})$, then $\gcd(x*_{n}y,2^{n})\leq\gcd(x*_{n}z,2^{n})$,

\item $\gcd(x\circ_{n}y,2^{n})=\gcd(x,y,2^{n})$,

\item $\gcd(x*_{n}y,2^{n})\geq\gcd(y,2^{n})$,

\item $\gcd(x*_{n}y,2^{n})>\gcd(y,2^{n})$ whenever $\gcd(x,2^{n})\leq\gcd(y,2^{n})<2^{n}$, and

\item if $\gcd(y,2^{n})<\gcd(z,2^{n})$ and $\gcd(x*_{n}y,2^{n})<2^{n}$, then
\[\gcd(x*_{n}y,2^{n})<\gcd(x*_{n}z,2^{n}).\]
\end{enumerate}
\end{prop}

\section{Permutative LD-systems and critical points}
In this section, we shall investigate the permutative LD-systems which closely resemble the algebras of elementary
embeddings modulo some rank. The notions of the composition and the critical points generalize seamlessly from set theory to permutative LD-systems and are fundamental to the theory of permutative LD-systems.

The notions of a critical point and composition have been algebraized from set-theory to the classical Laver tables
$A_{n}$ (see Proposition \ref{Rhu4t0HI42h2}).  Furthermore, in \cite{DJ}, Dougherty and Jech introduce the notion of a two-sorted embedding algebra, and the axioms of the two-sorted embedding algebras attempt to algebraize the system
$(\mathcal{E}_{\lambda},*,\circ,\mathrm{crit},(\equiv^{\gamma})_{\gamma<\lambda})$. The notion of a two-sorted embedding algebra though is insufficient for algebraizing $\mathcal{E}_{\lambda}$ since two-sorted embedding algebras are defined in terms of 11 axioms.
We believe that the notion of a permutative LD-system is a satisfactory algebraization of the algebras of elementary embeddings.
The permutative LD-systems have a very simple definition, but from the notion of a permutative LD-system, one can extrapolate the notions of the composition operation, a critical point, and $\equiv^{\gamma}$ from the self-distributive operation and show that permutative LD-systems behave as the algebras of elementary embeddings.

In this monograph, we shall hold to the convention that the omitted parentheses shall be grouped on the left. For example,
$w*x*y*z=((w*x)*y)*z$.
\begin{defn}
An algebra $(X,*)$ where $*$ satisfies the self-distributivity identity $x*(y*z)=(x*y)*(x*z)$ shall be called an \emph{LD-system}\index{LD-system}.
\end{defn}
The acronym LD stands for left-distributive. Other authors call LD-systems by other names including left-distributive algebras, self-distributive algebras, shelves, and LD-groupoids. Set theorists and algebraists usually use the left self-distributivity convention $x*(y*z)=(x*y)*(x*z)$ while knot theorists tend to use right self-distributivity $(x*z)*(y*z)=(x*y)*z$ when investigating self-distributive algebras.

\begin{exam}
Every meet-semilattice is an LD-system.
\end{exam}
\begin{exam}
Let $G$ be a group. Define an operation $*$ on $G$ by letting $x*y=xyx^{-1}$. Then $(G,*)$ is an LD-system.
\end{exam}

Take note that if $\mathrm{crit}(j)\leq \mathrm{crit}(k)$, then 
\[\mathrm{crit}(j^{n}*k)=j^{n}(\mathrm{crit}(k)),\]
so
\[\lim_{n\in\omega}\mathrm{crit}(j^{n}*k)=\lambda.\]
On the other hand, if $\mathrm{crit}(j)>\mathrm{crit}(k)$, then
\[\mathrm{crit}(j^{n}*k)=j^{n}(\mathrm{crit}(k))=\mathrm{crit}(k).\]

In particular, if $j,k\in\mathcal{E}_{\lambda}$ and $\gamma<\lambda$ is a limit ordinal, and
\[j,k\in\mathcal{E}_{\lambda},\mathrm{crit}(j)<\gamma,\]
then $\mathrm{crit}(j)\leq \mathrm{crit}(k)$ if and only if there is some
$n\in\omega$ where $j^{n}*k\equiv^{\gamma}1_{V_{\lambda}}$. Furthermore, $j^{n}*k$ can be defined solely in terms of
$*$ by letting $j^{0}*k=k$ and $j^{n+1}*k=j*(j^{n}*k)$. Therefore, one may algebraize the notion of a critical point to 
LD-systems similar to $\mathcal{E}_{\lambda}/\equiv^{\gamma}$ by defining $\mathrm{crit}(x)\leq \mathrm{crit}(y)$
precisely when $x^{n}*y=1$ for some $n$. The permutative LD-systems are precisely the LD-systems where this algebraization of the
notion of a critical point is satisfactory.

\begin{defn}
Define binary term functions $t_{n}$ for all $n\geq 1$ by letting
\[t_{1}(x,y)=y,t_{2}(x,y)=x,\]
and
\[t_{n+2}(x,y)=t_{n+1}(x,y)*t_{n}(x,y)\]
for $n\geq 1$. The terms $t_{n}$ are known as the \emph{Fibonacci terms}\index{Fibonacci terms}
because the number of variables present in $t_{n}(x,y)$\index{$t_{n}(x,y)$} are precisely the Fibonacci numbers and the words obtained from the terms
$t_{n}$ by removing all symbols which are not variables are precisely the \index{Fibonacci words}Fibonacci words \cite{FIB}. 
\end{defn}

The motivation behind the terms $t_{n}$ is that
\[x*(y*z)=t_{n+1}(x,y)*(t_{n}(x,y)*z)\]
for all $n\geq 1$ in any LD-system.
Furthermore, if $(X,*)$ is an LD-system and $\circ$ is a binary operation on $X$ such that $x\circ y=(x*y)\circ x$ for all
$x,y$, then
\[x\circ y=t_{n+1}(x,y)\circ t_{n}(x,y)\]
for all $n\geq 1$.
\begin{defn}
If $X$ is an LD-system and $x\in X$, then we shall say that the element $x$ is a \emph{left-identity}\index{left-identity} if
$x*y=y$ for each $y\in X$. Let \index{$\mathrm{Li}(X)$}$\mathrm{Li}(X)$ denote the set of all left-identities of the LD-system $X$.
\end{defn}
\begin{defn}
We shall say that a subset $L\subseteq X$ of an LD-system $(X,*)$ is a \index{left-ideal}\emph{left-ideal} if $y\in L$ implies that $x*y\in L$ as well.
\end{defn}
\begin{defn}
We shall call an LD-system $(X,*)$ \index{permutative}\emph{permutative} if $\mathrm{Li}(X)$ is a left-ideal and for all $x,y$ there is some $n\geq 1$ such that $t_{n}(x,y)\in \mathrm{Li}(X)$.
\end{defn}
We chose the term permutative since for braids which are sufficiently high with respect to the left-factor ordering,
the Hurwitz action of the positive braid monoid on tuples from a permutative LD-system is determined by its corresponding permutation.

\begin{thm}
(Laver-Steel) Suppose that $j_{n}\in\mathcal{E}_{\lambda}^{+}$ for each $n\in\omega$. Then
\[\sup_{n\in\omega}\mathrm{crit}(j_{0}*\ldots*j_{n})=\lambda.\]
\label{j3rti90}
\end{thm}
\begin{proof}
See \cite{S} for the original proof. A simple proof that does not reference extenders can be found in \cite{RD95}.
The references \cite{L95},\cite{FK},\cite{D} also contain proofs of this result.
\end{proof}

Since $\mathrm{crit}(j)\geq\gamma$ if and only if $j\equiv^{\gamma}1_{V_{\lambda}}$, the Laver-Steel theorem can be reformulated to the following more algebraic statement.
\begin{thm}
(Laver-Steel) Suppose that $j_{n}\in\mathcal{E}_{\lambda}^{+}$ for each $n\in\omega$ and $\gamma<\lambda$ is a limit ordinal. Then there exists an $n\in\omega$ such that
\[j_{0}*\ldots*j_{n}\equiv^{\gamma}1_{V_{\lambda}}.\]
\label{235wfe}
\end{thm}

If $j\in\mathcal{E}_{\lambda}$, then we shall write \index{$[j]_{\gamma}$}$[j]_{\gamma}$ for the equivalence class of $j$ in
$\mathcal{E}_{\lambda}/\equiv^{\gamma}$.
\begin{prop}
The algebras $\mathcal{E}_{\lambda}/\equiv^{\gamma}$ are permutative and
$\mathrm{Li}(\mathcal{E}_{\lambda}/\equiv^{\gamma})=\{[1_{V_{\lambda}}]_{\gamma}\}$.
\label{t4r4ghtj0qhyqfha}
\end{prop}
\begin{proof}
Clearly $1_{V_{\lambda}}*j=j$ for all $j\in\mathcal{E}_{\lambda}$, so $[1_{V_{\lambda}}]_{\gamma}\in\mathrm{Li}(\mathcal{E}_{\lambda}/\equiv^{\gamma})$. If $k\in\mathcal{E}_{\lambda}$ and
$[k]_{\gamma}\neq[1_{V_{\lambda}}]_{\gamma}$, then by elementarity $\mathrm{crit}(k*k)=k(\mathrm{crit}(k))>\mathrm{crit}(k)$.
Since $\mathrm{crit}(k*k)>\mathrm{crit}(k)$ and $\mathrm{crit}(k)<\gamma$, we conclude that
$[k*k]_{\gamma}\neq[k]_{\gamma}$, so $[k]_{\gamma}\not\in\mathrm{Li}(\mathcal{E}_{\lambda}/\equiv^{\gamma})$. Therefore,
$\mathrm{Li}(\mathcal{E}_{\lambda}/\equiv^{\gamma})=\{[1_{V_{\lambda}}]_{\gamma}\}$.

Suppose now that $j,k\in\mathcal{E}_{\lambda}.$ By induction, one has \[t_{n}([j]_{\gamma},[k]_{\gamma})\]
\[=t_{2}([j]_{\gamma},[k]_{\gamma})*t_{1}([j]_{\gamma},[k]_{\gamma})*t_{2}([j]_{\gamma},[k]_{\gamma})*\ldots*
t_{n-2}([j]_{\gamma},[k]_{\gamma}).\]

Therefore, the Laver-Steel theorem proves that there is some $n$ where
\[t_{n}([j]_{\gamma},[k]_{\gamma})=[1_{V_{\lambda}}]_{\gamma}.\]
\end{proof}
\begin{exam}
There are permutative LD-systems $(Z,*)$ where $\mathrm{Li}(Z)$ has more than one element.
Let $(X,*)$ be a permutative LD-system. Let $A$ be the LD-system defined by $a*b=b$ for all $a,b\in A$. Then
the cartesian product $(X\times A,*)$ is also a permutative LD-system where $\mathrm{Li}(X\times A)=\mathrm{Li}(X)\times A$.
\end{exam}
\begin{defn}
Suppose that $X$ is a permutative LD-system, then define $x^{n}*y$ for all $n\geq 0$ by letting $x^{0}*y=y$ and
$x^{n+1}*y=x*(x^{n}*y)$ whenever $n\geq 0$. The statement \index{$\mathrm{crit}(x)$}$\mathrm{crit}(x)\leq \mathrm{crit}(y)$ shall mean $x^{n}*y\in \mathrm{Li}(X)$ for some $n$.
\end{defn}
Take note that if $\mathrm{Li}(X)$ is permutative, then if $x^{n}*y\in \mathrm{Li}(X)$ and $m>n$, then $x^{m}*y\in \mathrm{Li}(X)$ as well.
\begin{defn}
Define the \index{right powers}\emph{right powers} \index{$x^{[n]}$}$x^{[n]}$ for all $n\geq 1$ by letting $x^{[n]}=x^{n-1}*x$. In other words, $x^{[n]}$ is characterized by the property $x^{[1]}=x$ and $x^{[n]}=x*x^{[n-1]}$ for $n>1$.
\end{defn}
The following two lemmas are needed in order to show that the ordering on the critical points in a permutative LD-system is reflexive.
\begin{lem}
Then for each term $t$ on one variable, there is some $N$ such that $t(x)*x^{[n]}=x^{[n+1]}$ for each $n\geq N$ in the variety of LD-systems.
\end{lem}
\begin{proof}

We shall prove this fact by induction on the complexity of the term $t$. If $t(x)=x$, then
$t(x)*x^{[n]}=x^{[n+1]}$ for all $n$. Now assume that $t(x)=t_{1}(x)*t_{2}(x)$. Then there is some $N$ such that if $n\geq N$ then $t_{1}(x)*x^{[n]}=x^{[n+1]}$ and $t_{2}(x)*x^{[n]}=x^{[n+1]}$. Now if $n\geq N+1$, then
\[t(x)*x^{[n]}=t(x)*(t_{1}(x)*x^{[n-1]})=(t_{1}(x)*t_{2}(x))*(t_{1}(x)*x^{[n-1]})\]
\[=t_{1}(x)*(t_{2}(x)*x^{[n-1]})=t_{1}(x)*x^{[n]}=x^{[n+1]}.\]
\end{proof}
\begin{lem}
For each term $t$ of one variable, there are $N,c$ such that
$t(x)^{[n]}=x^{[n+c]}$ for each $n\geq N$ in the variety of LD-systems.
\label{43hifsd}
\end{lem}
\begin{proof}
We shall prove this result by induction on the complexity of $t$. If $t(x)=x$, then $t(x)^{[n]}=x^{[n]}$ for all $n\geq 1$, so the result holds for $t(x)=x$.

Now assume $t(x)=t_{1}(x)*t_{2}(x)$. Then there are $N,c$ such that $t_{2}(x)^{[n]}=x^{[c+n]}$ and $t_{1}(x)*x^{[c+n]}=x^{[c+n+1]}$ for each
$n\geq N$. Therefore, for each $n\geq N$, we have
\[t(x)^{[n]}=(t_{1}(x)*t_{2}(x))^{[n]}=t_{1}(x)*t_{2}(x)^{[n]}=t_{1}(x)*x^{[n+c]}=x^{[n+c+1]}.\]
\end{proof}
While Proposition \ref{t2h8onkfia} consists of basic set-theoretic statements, we shall generalize
Proposition \ref{t2h8onkfia} to less trivial purely algebraic statements about the critical points of permutative LD-systems.
\begin{prop}
\label{t2h8onkfia}

Suppose that $j,k\in\mathcal{E}_{\lambda}^{+}$ and $\alpha,\beta<\lambda$. Then
\begin{enumerate}
\item The mapping $j$ is injective,

\item If $\alpha<\beta$, then $j(\alpha)<j(\beta)$,

\item $j(\alpha)\geq\alpha$ and $j(\alpha)>\alpha$ precisely when $\alpha\geq\mathrm{crit}(j),$

\item $\mathrm{crit}(j*k)=j(\mathrm{crit}(k))$, and

\item $\mathrm{crit}(j\circ k)=\min(\mathrm{crit}(j),\mathrm{crit}(k))$.
\end{enumerate}
\end{prop}
\begin{prop}
Let $(X,*)$ be a permutative LD-system, and let $x,y,r\in X$.
\begin{enumerate}
\item $\mathrm{crit}(x)\leq \mathrm{crit}(x)$.

\item If $\mathrm{crit}(x)\leq \mathrm{crit}(y)$ and $\mathrm{crit}(y)\leq \mathrm{crit}(z)$, then $\mathrm{crit}(x)\leq \mathrm{crit}(z)$.

\item If $\mathrm{crit}(x)\leq \mathrm{crit}(y)$, then $\mathrm{crit}(r*x)\leq \mathrm{crit}(r*y)$.
\end{enumerate}
\label{sfenjuw4}
\end{prop}
\begin{proof}
\begin{enumerate}
\item Since $X$ is permutative, there is some $n$ where $t_{n}(x,x)\in \mathrm{Li}(X)$.
However, by Lemma \ref{43hifsd} there are $r,s$ where
\[x^{s-1}*x=x^{[s]}=t_{n}(x,x)^{[r]}\in\mathrm{Li}(X)\]
since $\mathrm{Li}(X)$ is a left-ideal.

\item Since $\mathrm{crit}(x)\leq \mathrm{crit}(y)$ and $\mathrm{crit}(y)\leq \mathrm{crit}(z)$, there is some $n$ where $x^{n}*y,y^{n}*z\in \mathrm{Li}(X)$. Therefore, we have
\[x^{n}*z=(x^{n}*y)^{n}*(x^{n}*z)=x^{n}*(y^{n}*z)\in \mathrm{Li}(X)\]
because $\mathrm{Li}(X)$ is a left-ideal.

\item Suppose that $\mathrm{crit}(x)\leq \mathrm{crit}(y)$. Then $x^{n}*y\in \mathrm{Li}(X)$ for some $n$, so
\[(r*x)^{n}*(r*y)=r*(x^{n}*y)\in \mathrm{Li}(X)\]
as well. Therefore, $\mathrm{crit}(r*x)\leq \mathrm{crit}(r*y)$.
\end{enumerate}
\end{proof}
\begin{defn}
Let $\simeq$ be the equivalence relation on $X$ where $x\simeq y$ iff $\mathrm{crit}(x)\leq \mathrm{crit}(y)$ and $\mathrm{crit}(y)\leq \mathrm{crit}(x)$. Then we can formally define \index{$\mathrm{crit}(x)$}$\mathrm{crit}(x)$ to be the equivalence class of $x$ with respect to $\simeq$. Define \index{$\mathrm{crit}[X]$}$\mathrm{crit}[X]=\{\mathrm{crit}(x)\mid x\in X\}$. Then $\mathrm{crit}[X]$ is a partial ordering with the ordering $\leq$.
\end{defn}
We shall use Greek lower-case letters to denote the critical points in a permutative LD-system since one typically uses Greek letters to denote ordinals. 
\begin{defn}
For each $x\in X$, define a mapping \index{$x^{\sharp}$}$x^{\sharp}:\mathrm{crit}[X]\rightarrow \mathrm{crit}[X]$ by letting $x^{\sharp}(\mathrm{crit}(r))=\mathrm{crit}(x*r)$ for each $r\in X$. By Proposition \ref{sfenjuw4}, the mapping $x^{\sharp}$ is well-defined and order preserving.
\end{defn}
\begin{lem}
Suppose that $X$ is a permutative LD-system. If $t_{n}(x,y),t_{n+1}(x,y)\in \mathrm{Li}(X)$, then
$x,y\in \mathrm{Li}(X)$.
\label{fewn2}
\end{lem}
\begin{proof}
We shall prove this result by contrapositive and by induction. Suppose that $x\not\in \mathrm{Li}(X)$ or $y\not\in \mathrm{Li}(X)$. Then I claim that for all $n$ we have $t_{n}(x,y)\not\in \mathrm{Li}(X)$ or $t_{n+1}(x,y)\not\in \mathrm{Li}(X)$.
The case when $n=1$ is trivial. Now assume that $n>1$. Then by the induction hypothesis, 
$t_{n}(x,y)\not\in \mathrm{Li}(X)$ or $t_{n-1}(x,y)\not\in \mathrm{Li}(X)$. If $t_{n}(x,y)\not\in \mathrm{Li}(X)$, then the induction step has been proven. If $t_{n}(x,y)\in \mathrm{Li}(X)$ and $t_{n-1}(x,y)\not\in \mathrm{Li}(X)$, then $t_{n+1}(x,y)=t_{n}(x,y)*t_{n-1}(x,y)=t_{n-1}(x,y)\not\in\mathrm{Li}(X)$.
\end{proof}
\begin{thm}
\label{2j1298ur28tye}

Suppose that $X$ is a permutative LD-system.
\begin{enumerate}
\item $x\in \mathrm{Li}(X)$ if and only if $\mathrm{crit}(x)$ is the greatest element in $\mathrm{crit}[X]$.

\item $x^{\sharp}(\alpha)\geq\alpha$ for all $\alpha$.

\item $\mathrm{crit}[X]$ is a linear ordering.

\item $x^{\sharp}(\alpha)>\alpha$ if and only if $\mathrm{crit}(x)\leq\alpha<\max(\mathrm{crit}[X])$.

\item $x^{\sharp}|_{A}$ is injective where $A=\{\alpha\in \mathrm{crit}[X]\mid x^{\sharp}(\alpha)<\max(\mathrm{crit}[X])\}$.
\end{enumerate}
\end{thm}
\begin{proof}
\begin{enumerate}
\item  $\rightarrow$ If $x\in \mathrm{Li}(X)$, then $y^{n}*x\in \mathrm{Li}(X)$ for all $n$, so $\mathrm{crit}(y)\leq \mathrm{crit}(x)$ for all $y\in X$. Therefore, $\mathrm{crit}(x)$ is the greatest element in $\mathrm{crit}[X].$

$\leftarrow$ We shall prove this direction by contrapositive. If $y\not\in\mathrm{Li}(X),x\in\mathrm{Li}(X)$, then $x^{n}*y=y\not\in \mathrm{Li}(X)$, so $\mathrm{crit}(x)\not\leq \mathrm{crit}(y)$. 

\item Suppose that $\alpha=\mathrm{crit}(y)$. Then $y^{n}*y\in \mathrm{Li}(X)$ for some $n$. Therefore, for such an $n$, we have
\[y^{n}*(x*y)=(y^{n}*x)*(y^{n}*y)\in \mathrm{Li}(X)\]
as well since $\mathrm{Li}(X)$ is a left ideal. Therefore,
\[\alpha=\mathrm{crit}(y)\leq \mathrm{crit}(x*y)=x^{\sharp}(\alpha).\]

\item We shall use the fact that $X$ is permutative here. For all $n$, we have
\[\mathrm{crit}(t_{n+2}(x,y))=\mathrm{crit}(t_{n+1}(x,y)*t_{n}(x,y))\geq t_{n}(x,y).\]
Therefore
\[\mathrm{crit}(t_{2n+2}(x,y))\geq \mathrm{crit}(t_{2}(x,y))=\mathrm{crit}(x)\]
and
\[\mathrm{crit}(t_{2n+1}(x,y))\geq \mathrm{crit}(t_{1}(x,y))=\mathrm{crit}(y).\]

Suppose now that $t_{N}(x,y)\in\mathrm{Li}(X)$. Then $\mathrm{crit}(t_{N+1}(x,y))\geq \mathrm{crit}(x)$ or
$\mathrm{crit}(t_{N+1}(x,y))\geq\mathrm{crit}(y)$.

Now suppose that $\mathrm{crit}(r)\leq \mathrm{crit}(t_{N+1}(x,y))$. Then there is some $n$ where
\[t_{N+1}(r^{n}*x,r^{n}*y)=r^{n}*t_{N+1}(x,y)\in \mathrm{Li}(X).\]
However, we also have
\[t_{N}(r^{n}*x,r^{n}*y)=r^{n}*t_{N}(x,y)\in \mathrm{Li}(X).\]
Therefore, by Lemma \ref{fewn2}, we have
$r^{n}*x,r^{n}*y\in \mathrm{Li}(X)$. Therefore, $\mathrm{crit}(r)\leq \mathrm{crit}(x)$ and $\mathrm{crit}(r)\leq \mathrm{crit}(y)$.
Therefore, we have $\mathrm{crit}(t_{N+1}(x,y))\leq \mathrm{crit}(x)$ and $\mathrm{crit}(t_{N+1}(x,y))\leq \mathrm{crit}(y)$.
Since $\mathrm{crit}(x)\leq\mathrm{crit}(t_{N+1}(x,y))$ or $\mathrm{crit}(y)\leq\mathrm{crit}(t_{N+1}(x,y))$, we conclude that 
\[\mathrm{crit}(x)=\mathrm{crit}(t_{N+1}(x,y))\leq\mathrm{crit}(y)\]
or
\[\mathrm{crit}(y)=\mathrm{crit}(t_{N+1}(x,y))\leq\mathrm{crit}(x).\]
Therefore $\mathrm{crit}[X]$ is a linearly ordered set.

\item Now if $\alpha=\max(\mathrm{crit}[X])$, then $x^{\sharp}(\alpha)\geq\alpha$ by property $2$, so $x^{\sharp}(\alpha)=\alpha$ in this case.

I now claim that if $\alpha<\max(\mathrm{crit}[X])$ and $x^{\sharp}(\alpha)=\alpha$, then
$\alpha<\mathrm{crit}(x)$. Let $y\in X$ be an element such that $\mathrm{crit}(y)=\alpha$. Then 
\[\mathrm{crit}(y)=\mathrm{crit}(x*y)=\mathrm{crit}(x^{2}*y)=\ldots=\mathrm{crit}(x^{n}*y),\]
so $x^{n}*y\not\in\mathrm{Li}(X)$ for all $n$. Therefore, $\alpha=\mathrm{crit}(y)<\mathrm{crit}(x)$.

By the proof of 3, we know that if $t_{N}(u,v)\in\mathrm{Li}(X)$, then 
\[\mathrm{crit}(t_{N+1}(u,v))=\min(\mathrm{crit}(u),\mathrm{crit}(v)).\]

We will now show that if $\alpha<\mathrm{crit}(x)$, then 
$\alpha=x^{\sharp}(\alpha)$. Suppose that $\mathrm{crit}(y)=\alpha$. Let $N$ be a natural number with $N>1$ such that
$t_{N}(x,y)\in \mathrm{Li}(X)$. Then, $t_{N-1}(x*y,x)\in\mathrm{Li}(X)$ as well.

Therefore,
\[\mathrm{crit}(y)=\min(\mathrm{crit}(x),\mathrm{crit}(y))=t_{N+1}(x,y)\]
\[=t_{N}(x*y,x)=\min(\mathrm{crit}(x*y),\mathrm{crit}(x)).\]

We conclude that $\mathrm{crit}(y)=\mathrm{crit}(x*y)$, so $\alpha=x^{\sharp}(\alpha)$.

\item Suppose that $\alpha,\beta\in A$ and $\alpha<\beta$. Then let $\alpha=\mathrm{crit}(y),\beta=\mathrm{crit}(z)$. Then since
$\alpha<\beta$, we have $\mathrm{crit}(y)=\mathrm{crit}(z*y)$, so 
\[\mathrm{crit}(x*y)=\mathrm{crit}(x*(z*y))=\mathrm{crit}((x*z)*(x*y)).\]
Therefore, since
\[\mathrm{crit}(x*y)=x^{\sharp}(\alpha)\neq\max(\mathrm{crit}[X]),\]
we have $\mathrm{crit}(x*y)<\mathrm{crit}(x*z)$, so $x^{\sharp}(\alpha)<x^{\sharp}(\beta)$.
\end{enumerate}
\end{proof}
We shall now work towards defining an operation on permutative LD-systems that resembles the composition of rank-into-rank embeddings.
However, this ``composition" operation only satisfies the desired identities when $|\mathrm{Li}(X)|=1$. Fortunately, whenever $X$ is a permutative LD-system, one can collapse $\mathrm{Li}(X)$ to a single element and then define the ``composition" operation on the resulting quotient algebra.

\begin{prop}
Suppose that $X$ is an LD-system such that $\mathrm{Li}(X)$ is a left-ideal. Let
$\simeq$ be the equivalence relation on $X$ that collapses $\mathrm{Li}(X)$ to a point. Then $\simeq$ is a congruence on $X$.
Furthermore, if $X$ is a permutative LD-system, then the equivalence class $\mathrm{Li}(X)$ is the only left-identity in the quotient algebra $X/\simeq$.
\label{t42huife}
\end{prop}
\begin{proof}
Suppose $x,y,z\in X$. Assume first that $y\simeq z$. If $y\not\in \mathrm{Li}(X)$, then $y=z$, so $x*y=x*z$. If $y\in \mathrm{Li}(X)$, then $z\in \mathrm{Li}(X)$ and hence $x*y,x*z\in \mathrm{Li}(X)$ also. Therefore, $x*y\simeq x*z$ in either case.

Now assume that $x\simeq y$. If $x\not\in \mathrm{Li}(X)$, then $x=y$, so $x*z=y*z$. Similarly, if
$x\in \mathrm{Li}(X)$, then $y\in \mathrm{Li}(X)$ and hence $x*z=z=y*z$. We therefore conclude that $\simeq$ is a congruence on $X$.

Assume that $X$ is a permutative LD-system. Then clearly the equivalence class $\mathrm{Li}(X)$ is a left-identity in $X/\simeq$.
If $x\not\in \mathrm{Li}(X)$, then $x*x\neq x$ since otherwise $t_{n}(x,x)=x$ for all $n$ which is a contradiction. Therefore,
$[x*x]\neq[x]$. We conclude that $[x]$ is not a left-identity in $X/\simeq$.
\end{proof}
\begin{defn}
We shall let $X/\mathrm{Li}(X)$ denote the reduced permutative LD-system $X/\simeq$ in Proposition \ref{t42huife}.
\end{defn}

\begin{defn}
We shall call a permutative LD-system \index{reduced}\emph{reduced} if $\mathrm{Li}(X)$ has only one element.
\end{defn}

In order to establish the associativity of the ``composition" operation on reduced permutative LD-systems, we will need to
refer to the action of the positive braid monoid on LD-systems. 
\begin{defn}
Let \index{$B_{n}^{+}$}$B_{n}^{+}$ be the monoid presented by the generators $\sigma_{1},\ldots,\sigma_{n-1}$ and relations $\sigma_{i}\sigma_{j}=\sigma_{j}\sigma_{i}$ whenever $|i-j|>1$ and
$\sigma_{i}\sigma_{i+1}\sigma_{i}=\sigma_{i+1}\sigma_{i}\sigma_{i+1}$ whenever $1\leq i<n-1$. The monoid $B_{n}^{+}$ is known as the $n$-strand \index{positive braid monoid}\emph{positive braid monoid}.
\end{defn}
\begin{defn}
Let $X$ be an LD-system. Then it is easy to show that there is a unique right action of the monoid $B_{n}^{+}$ on $X^{n}$ such that
$\mathbf{x}e=\mathbf{x}$ whenever $\mathbf{x}\in X^{n}$ and where
\[(x_{1},\ldots ,x_{n})\sigma_{i}=(x_{1},\ldots ,x_{i-1},x_{i}*x_{i+1},x_{i},x_{i+2},\ldots ,x_{n}).\]
This action is known as the \index{Hurwitz action}\emph{Hurwitz action}.
\end{defn}
\begin{prop}
Suppose that $r,s\in B_{n}^{+}$. Then there are $t,u\in B_{n}^{+}$ where $rt=su$.
\end{prop}
\begin{defn}
An \index{LD-monoid}\emph{LD-monoid} is an algebra $(X,*,\circ,1)$ such that $(X,\circ,1)$ is a monoid, and the following identities are satisfied:
\begin{enumerate}
\item \[(x\circ y)*z=x*(y*z),x*(y\circ z)=(x*y)\circ(x*z),\]
\[x\circ y=(x*y)\circ x,\]
and
\item \[1*x=x,x*1=1.\]
\end{enumerate}
\end{defn}
\begin{prop}
If $(X,*,\circ,1)$ is an LD-monoid, then $(X,*)$ is an LD-system.
\end{prop}
\begin{proof}
$x*(y*z)=(x\circ y)*z=((x*y)\circ x)*z=(x*y)*(x*z)$.
\end{proof}
\begin{thm}
Let $X$ be a permutative LD-system with a unique left-identity $1$. Then
\begin{enumerate}
\item there exists a unique binary operation $\circ$ on $X$ such that $(x*y)\circ x=x\circ y$ and where $x=x\circ 1$
for all $x,y\in X$,

\item the algebra $(X,*,\circ,1)$ is an LD-monoid, and

\item if 
\begin{enumerate}
\item $(x_{1},\ldots,x_{n}),(y_{1},\ldots ,y_{n})\in X^{n}$,

\item $s\in B_{n}^{+}$ is an element with $(x_{1},\ldots,x_{n})\cdot s=(y_{1},\ldots,y_{n})$, and

\item $j\in\{1,\ldots ,n\}$ is a natural number where $y_{i}=1$ for $i\neq j$,
\end{enumerate}
then
\[y_{j}=x_{1}\circ\ldots\circ x_{n}.\]
\end{enumerate}
\end{thm}
\begin{proof}
\begin{enumerate}
\item I first claim that there is at most one operation $\circ$. If $t_{n}(x,y)=1$, then since $x\circ y=(x*y)\circ x$, we have 
\[x\circ y=t_{n+1}(x,y)\circ t_{n}(x,y)=t_{n+1}(x,y).\]
We conclude that there is at most one operation $\circ$ on $X$ that satisfies $x\circ y=(x*y)\circ x$.
Now define $x\circ y$ so that if $t_{n}(x,y)=1$, then $t_{n+1}(x,y)=x\circ y$. 

I now claim that $x\circ y$ does not depend on the choice of $n$. Suppose that $n$ is the least natural number such that
$t_{n}(x,y)=1$. Then 
\[t_{n+2}(x,y)=t_{n+1}(x,y)*t_{n}(x,y)=t_{n}(x,y)*1=1\]
and
\[t_{n+3}(x,y)=t_{n+2}(x,y)*t_{n+1}(x,y)=1*t_{n+1}(x,y)=t_{n+1}(x,y).\]
Therefore, by induction, 
$t_{n+2k}(x,y)=1$ and $t_{n+2k+1}(x,y)=t_{n+1}(x,y)$ for all $k\geq 0$.

If $t_{n+1}(x,y)\neq 1$, then whenever $t_{m}(x,y)=1$ we have $m=n+2k$ for some $k$, so
$t_{m+1}(x,y)=t_{n}$

If $t_{n+1}(x,y)=1$, then whenever $t_{m}(x,y)=1$, we have $m\geq n$, so
$t_{m+1}(x,y)=1$

Now if $x,y\in X$, $n>1$ and $t_{n}(x,y)=1$, then $t_{n-1}(x*y,x)=t_{n}(x,y)=1$ as well. Therefore, we have
\[(x*y)\circ x=t_{n}(x*y,x)=t_{n+1}(x,y)=x\circ y.\]

\item
\begin{enumerate}
\item $1\circ x=x=x\circ 1$:

We have $t_{2}(1,x)=1$, so $1\circ x=t_{3}(1,x)=1*x=x$. Similarly, $t_{1}(x,1)=1$, so $x\circ 1=t_{2}(x,1)=x$.

\item $(x\circ y)*z=x*(y*z)$: Suppose that $t_{n}(x,y)=1$. Then
\[x*(y*z)=t_{n+1}(x,y)*(t_{n}(x,y)*z)\]
\[=t_{n+1}(x,y)*(1*z)=t_{n+1}(x,y)*z=(x\circ y)*z.\]

\item $x*(y\circ z)=(x*y)*(x*z)$:

Suppose now that $t_{n}(y,z)=1$. Then 
\[1=x*1=x*t_{n}(y,z)=t_{n}(x*y,x*z).\]
Therefore,
\[(x*y)\circ(x*z)=t_{n+1}(x*y,x*z)=x*t_{n+1}(y,z)=x*(y\circ z).\]

\item $(x\circ y)\circ z=x\circ(y\circ z)$:
There is some $n$ where 
\[(x,y,z)=(1,x\circ y,z)\cdot\sigma_{1}^{n}.\]
Therefore, there is some $m$ where
\[(x,y,z)\cdot\sigma_{1}^{n}\sigma_{2}^{m}=(1,(x\circ y)\circ z,1).\]

However, there is some $r$ where 
\[(x,y,z)=(x,y\circ z,1)\cdot\sigma_{2}^{r},\]
so there is some $s$ where
\[(x,y,z)=(1,x\circ(y\circ z),1)\cdot\sigma_{2}^{r}\sigma_{1}^{s}.\]

Now, there are $\mathbf{u},\mathbf{v}$ so that
\[\sigma_{1}^{n}\sigma_{2}^{m}\mathbf{u}=
\sigma_{2}^{r}\sigma_{1}^{s}\mathbf{v}\]
Thus,
\[(1,(x\circ y)\circ z,1)\cdot\mathbf{u}=(1,x\circ(y\circ z),1)\cdot\mathbf{v},\]
but this is only possible if 
\[(x\circ y)\circ z=x\circ(y\circ z).\]

\end{enumerate}
\end{enumerate}
\end{proof}
\begin{defn}
For our purposes, the canonical group operation $\cdot$ on the symmetry group $S_{n}$ shall be anti-composition instead of composition (i.e. $f\cdot g=g\circ f$). Define a right action of $(S_{n},\cdot)$ on $\{1,\ldots ,n\}$ by $x\cdot f=f(x)$.
Define a monoid homomorphism $\Sigma:B_{n}^{+}\rightarrow(S_{n},\cdot)$ by letting $\Sigma(s_{i})=(i,i+1)$ for $1\leq i<n$.
Suppose now that $X$ is a permutative LD-system. Let $\langle X\rangle^{n}\subseteq X^{n}$ be the set of all tuples
$(x_{1},\ldots,x_{n})$ such that $x_{i}\not\in\mathrm{Li}(X)$ for some $i\in\{1,\ldots,n\}$.
Now define a mapping $\Psi:\langle X\rangle^{n}\rightarrow\{1,\ldots,n\}$ by letting
$\Psi(x_{1},\ldots,x_{n})=i$ if $i$ is the least natural number such that $\mathrm{crit}(x_{i})=\min(\mathrm{crit}(x_{1}),\ldots,\mathrm{crit}(x_{n}))$.
\end{defn}

\begin{prop}
Suppose that $X$ is a permutative LD-system and $(x_{1},\ldots,x_{n})\in\langle X\rangle^{n}$. Then
\begin{enumerate}
\item $\Psi(x_{1},\ldots,x_{n})\cdot\Sigma(s)=\Psi((x_{1},\ldots,x_{n})\cdot s)$ whenever $s\in B_{n}^{+}$, and

\item there exists an $s\in B_{n}^{+}$ such that if $(x_{1},\ldots,x_{n})\cdot s=(y_{1},\ldots,y_{n})$, then there is precisely one
$i\in\{1,\ldots,n\}$ where $y_{i}\not\in\mathrm{Li}(X)$.
\end{enumerate}
\end{prop}

\begin{prop}
Let $X$ be a permutative LD-monoid, and suppose $x,y\in X,\alpha\in\mathrm{crit}[X]$. Then 
\begin{enumerate}
\item $x^{\sharp}(y^{\sharp}(\alpha))=(x\circ y)^{\sharp}(\alpha)$, and

\item $\mathrm{crit}(x\circ y)=\min(\mathrm{crit}(x),\mathrm{crit}(y))$.
\end{enumerate}
\end{prop}
\begin{proof}
\begin{enumerate}
\item Suppose that $\mathrm{crit}(z)=\alpha$. Then

\[x^{\sharp}(y^{\sharp}(\alpha))=x^{\sharp}(y^{\sharp}(\mathrm{crit}(z)))=x^{\sharp}(\mathrm{crit}(y*z))\]
\[=\mathrm{crit}(x*(y*z))=\mathrm{crit}((x\circ y)*z)=(x\circ y)^{\sharp}(\mathrm{crit}(z))=(x\circ y)^{\sharp}(\alpha).\]

\item Suppose $x,y\in X$. Let $N$ be a natural number where $t_{N}(x,y)=1$. Then by the proof of Item 3 in Theorem \ref{2j1298ur28tye}, we have
\[\mathrm{crit}(t_{N+1}(x,y))=\min(\mathrm{crit}(x),\mathrm{crit}(y)).\]
Therefore, since $x\circ y=t_{N+1}(x,y)$, we conclude that
\[\mathrm{crit}(x\circ y)=\min(\mathrm{crit}(x),\mathrm{crit}(y)).\]
\end{enumerate}
\end{proof}

If $(X,*)$ is a non-reduced permutative LD-system and $x\not\in\mathrm{Li}(X)$ or $y\not\in\mathrm{Li}(X)$, then we shall hold to the convention that $x\circ y=t_{n+1}(x,y)$ where $n$ is a natural number such that $t_{n}(x,y)\in\mathrm{Li}(X)$, but we shall leave $x\circ y$ undefined whenever $x,y\in\mathrm{Li}(X)$.

The following result shows that every finite LD-system with a reasonable notion of a critical point is already a permutative LD-system.
\begin{prop}
Suppose that $X$ is an LD-system, $L$ is a finite linear ordering, and $\Gamma:X\rightarrow L$ is a mapping such that
\begin{enumerate}
\item $\Gamma(x*y)\leq\Gamma(y)$ whenever $\Gamma(x)>\Gamma(y)$, and

\item $\Gamma(x*y)>\Gamma(y)$ whenever $\Gamma(x)\leq\Gamma(y)<\max(L)$.
\end{enumerate}
Suppose furthermore that $\mathrm{Li}(X)$ is a left-ideal in $X$ and $x\in\mathrm{Li}(X)$ if and only if $\Gamma(x)=\max(L)$.
Then $X$ is a permutative LD-system, and $\mathrm{crit}(x)\leq\mathrm{crit}(y)$ if and only if $\Gamma(x)\leq\Gamma(y)$.
\end{prop}
\begin{proof}
Suppose $x,y\in X$.

If $\Gamma(y)<\Gamma(x)$, then I claim that 
\[\Gamma(t_{2n-1}(x,y))\leq\Gamma(y)<\Gamma(x)\leq\Gamma(t_{2n}(x,y))\]
for all $n$, and this claim shall be proven by induction. The case where $n=1$ follows from the fact that $t_{1}(x,y)=y,t_{2}(x,y)=x$. Suppose now that $n>1$. Then \[\Gamma(t_{2n-1}(x,y))=\Gamma(t_{2n-2}(x,y)*t_{2n-3}(x,y))\leq\Gamma(t_{2n-3}(x,y))\leq\Gamma(y).\]
Furthermore,
\[\Gamma(t_{2n}(x,y))=\Gamma(t_{2n-1}(x,y)*t_{2n-2}(x,y))\geq\Gamma(t_{2n-2}(x,y)).\]
If $\Gamma(t_{2n-2}(x,y))<\max(L)$, then 
\[\Gamma(t_{2n}(x,y))=\Gamma(t_{2n-1}(x,y)*t_{2n-2}(x,y))>\Gamma(t_{2n-2}(x,y)).\]

Therefore, we conclude that the sequence $(\Gamma(t_{2n}(x,y)))_{n}$ is increasing and $\Gamma(t_{2N}(x,y))=\max(L)$ for some $N$. Therefore, we conclude that
$t_{2N}(x,y)\in\mathrm{Li}(X)$.

Suppose now that $\Gamma(y)\geq\Gamma(x)$. If $\Gamma(y)=\max(L)$, then $t_{1}(x,y)\in\mathrm{Li}(X)$.
If $\Gamma(y)<\max(L)$, then $\Gamma(x*y)>\Gamma(y)\geq\Gamma(x)$. Therefore,
\[t_{n+1}(x,y)=t_{n}(x*y,x)\in\mathrm{Li}(X)\]
for some $n$. We conclude that the LD-system $(X,*)$ is permutative.

Suppose now that $\Gamma(x)\leq\Gamma(y)$. Then there is some $n\geq 0$ with
\[\Gamma(y)<\Gamma(x*y)<\ldots <\Gamma(x^{n}*y)=\max(L).\]
Therefore, $x^{n}*y\in\mathrm{Li}(X)$, so $\mathrm{crit}(x)\leq\mathrm{crit}(y)$.

If $\Gamma(x)>\Gamma(y)$, then for all $n$, we have 
\[\Gamma(y)\geq\Gamma(x*y)\geq\ldots \geq\Gamma(x^{n}*y),\]
so
$\Gamma(x^{n}*y)\neq\max(L)$, hence $\Gamma(x^{n}*y)\not\in\mathrm{Li}(X)$, thus
$\mathrm{crit}(x)>\mathrm{crit}(y)$.
\end{proof}

\begin{cor}
Let $X$ be a permutative LD-system. Then
\begin{enumerate}
\item If $\mathrm{crit}(x)>\mathrm{crit}(y)$, then $\mathrm{crit}(t_{2n+1}(x,y))=\mathrm{crit}(y)$ for all $n$ and
\[\mathrm{crit}(t_{2}(x,y))<\mathrm{crit}(t_{4}(x,y))<\ldots<\mathrm{crit}(t_{2N}(x,y))\]
where $N$ is the least natural number such that $t_{2N}(x,y)\in \mathrm{Li}(X)$. Furthermore,
$t_{2m}(x,y)\in \mathrm{Li}(X)$ whenever $m\geq N$.

\item If $\mathrm{crit}(x)\leq \mathrm{crit}(y)$, then $\mathrm{crit}(t_{2n}(x,y))=\mathrm{crit}(x)$ for all $n$ and
\[\mathrm{crit}(t_{1}(x,y))<\mathrm{crit}(t_{3}(x,y))<\ldots<\mathrm{crit}(t_{2N+1}(x,y))\]
where $N$ is the least natural number such that $t_{2N+1}(x,y)\in \mathrm{Li}(X)$. In this case, $t_{2m+1}(x,y)\in \mathrm{Li}(X)$
for $m\geq N$.

\item $\mathrm{crit}(x)\leq \mathrm{crit}(y)$ if and only if $t_{n}(x,y)\in \mathrm{Li}(X)$ for sufficiently large odd numbers $n$.
\end{enumerate}
\end{cor}

\begin{defn}
A \index{Heyting semilattice}\emph{Heyting semilattice} is a \index{meet-semilattice}meet-semilattice $(L,\wedge)$ along with a binary operation $\rightarrow$ such that $x\leq y\rightarrow z$ if and only if $x\wedge y\leq z$.
\end{defn}
Every Heyting semilattice $(L,\wedge,\rightarrow)$ has a greatest element which we shall denote by 1. Heyting algebras satisfy the identity $x\rightarrow x=1$. Furthermore, the element $1$ is the unique left-identity in the Heyting semilattice $(X,\wedge,\rightarrow)$.
\begin{exam}
Suppose that $(X,\rightarrow,\wedge)$ is a Heyting semilattice with greatest element 1. Then $(X,\rightarrow,\wedge,1)$ is an LD-monoid.
\end{exam}
\begin{prop}
Suppose that $(X,\rightarrow,\wedge)$ is a Heyting semi-lattice with greatest element $1$. Then $(X,\rightarrow,\wedge)$ is a permutative LD-monoid
if and only if $X$ is a linear ordering. Furthermore, if $X$ is a linear ordering, then $\mathrm{crit}(x)\leq \mathrm{crit}(y)$ if and only if $x\leq y$.
\end{prop}
\begin{proof}
Suppose $X$ is a linear ordering and $x,y\in X$. If $x\leq y$, then $T^{\rightarrow}_{3}(x,y)=x\rightarrow y=1$. If $x>y$, then $t^{\rightarrow}_{3}(x,y)=x\rightarrow y=y$ and
\[t_{4}^{\rightarrow}(x,y)=t_{3}^{\rightarrow}(x,y)\rightarrow t_{2}^{\rightarrow}(x,y)=y\rightarrow x=1.\]
Thus, we conclude that $X$ is permutative.

Suppose now that $X$ is a permutative Heyting semi-lattice and $x,y\in X$. Without loss of generality, assume that $\mathrm{crit}(x)\leq\mathrm{crit}(y)$. Then there is some $n$ where $(x\wedge\ldots\wedge x)\rightarrow y=1$ where we apply the operation $\wedge$ $n$-times. Therefore, we have $x\rightarrow y=1$, so $x\leq y$. We conclude that $X$ is a linear ordering.

If $X$ is a linear ordering, then clearly $\mathrm{crit}(x)\leq \mathrm{crit}(y)$ iff $x\rightarrow y=1$ iff $x\leq y$.
\end{proof}

\begin{thm}
Suppose that $(X,*)$ is a permutative LD-system. If $\simeq$ is a congruence on $(X,*)$, then there is a partition $A,B$ of
$\mathrm{crit}[X]$ where

\begin{enumerate}
\item $A$ is downwards closed, and $B$ is upwards closed, and

\item $\mathrm{crit}(x)\in B$ if and only if $x\simeq y$ for some $y\in \mathrm{Li}(X)$,

\item if $\mathrm{crit}(x)\in A$ and $x\simeq y$, then $\mathrm{crit}(x)=\mathrm{crit}(y)$,

\item $[x]\in \mathrm{Li}(X/\simeq)$ if and only if $\mathrm{crit}(x)\in B$,

\item $X/\simeq$ is also a permutative LD-system,

\item $\mathrm{crit}([x])\leq \mathrm{crit}([y])$ if and only if $\mathrm{crit}(y)\in B$ or 
$\mathrm{crit}(x)\leq \mathrm{crit}(y)$, and

\item $\mathrm{crit}([x])=\mathrm{crit}([y])$ if and only if $\mathrm{crit}(x)=\mathrm{crit}(y)$
or $\mathrm{crit}(x),\mathrm{crit}(y)\in B.$
\end{enumerate}
\label{uh3d5tjh4uyyuin}
\end{thm}
\begin{proof}
I first claim that if $u\simeq v,v\in\mathrm{Li}(X),\mathrm{crit}(u)\leq\mathrm{crit}(x)$, then $x\simeq y$ for some
$y\in\mathrm{Li}(X)$. We have $x=v*x\simeq u*x\simeq u*(u*x)\simeq\ldots\simeq u^{n}*x$ for all $n$. Since
$\mathrm{crit}(u)\leq\mathrm{crit}(x)$, we conclude that $u^{n}*x\in\mathrm{Li}(X)$ and $x\simeq u^{n}*x$ for some $n$.

Let $B=\{\mathrm{crit}(u)\mid \text{$u\simeq v$ for some $v\in \mathrm{Li}(X)$}\}$, and let $A=\mathrm{crit}[X]\setminus B$.

\begin{enumerate}
\item Suppose that $\alpha\in B,\alpha\leq\beta$. Then there are $u,v,x$ with $u\simeq v,v\in\mathrm{Li}(X),\mathrm{crit}(u)=\alpha\leq\beta=\mathrm{crit}(x)$. Therefore, by the remarks at the beginning of this proof, there is some
$y\in\mathrm{Li}(X)$ with $x\simeq y$. Therefore $\beta\in B$ as well, so we conclude that $B$ is upwards closed.

\item The direction $\leftarrow$ follows from the definition of $B$. For the direction $\rightarrow$, assume that
$\mathrm{crit}(x)\in B$. Then there are $u,v\in X$ where $\mathrm{crit}(x)=\mathrm{crit}(u),u\simeq v,v\in\mathrm{Li}(X)$.
Thus, by the claim at the beginning of this proof, there is some $y\in\mathrm{Li}(X)$ and where $x\simeq y$. 

\item We shall proceed by contrapositive. Suppose that $\mathrm{crit}(x)<\mathrm{crit}(y),x\simeq y$. Then there is some $n$ with $x^{n}*x\in\mathrm{Li}(X)$. Therefore, since $x^{n}*x\simeq y^{n}*x$ and $\mathrm{crit}(y^{n}*x)=\mathrm{crit}(x)$, we conclude that $\mathrm{crit}(x)\in B$. Since $B$ is upwards closed, $\mathrm{crit}(y)\in B$ as well.

\item

$\leftarrow$ Suppose that $\mathrm{crit}(x)\in B$. Then we have $x\simeq y$ for some $y\in \mathrm{Li}(X)$. Therefore, for all $z\in X$ we have
$[x]*[z]=[y]*[z]=[z]$, so $[x]\in \mathrm{Li}(X/\simeq)$

$\rightarrow$ Suppose now that $\mathrm{crit}(x)\not\in B$. Then $[x]\cap \mathrm{Li}(X)=\emptyset$. However, we have
$x^{n}*x\in \mathrm{Li}(X)$ for some $n$, so $[x]\neq [x]^{n}*[x]$. Therefore, we have $[x]\not\in \mathrm{Li}(X/\simeq)$.

\item I first claim that $\mathrm{Li}(X/\simeq)$ is a left-ideal. Suppose that $[y]\in\mathrm{Li}(X/\simeq)$. Then
$\mathrm{crit}(y)\in B$. Therefore, for all $x\in X$, we have $\mathrm{crit}(x*y)\in B$, so
$[x]*[y]=[x*y]\in\mathrm{Li}(X/\simeq)$. We conclude that $\mathrm{Li}(X/\simeq)$ is a left-ideal.

Suppose that $x,y\in X$. Then there is some $n$ where $t_{n}(x,y)\in \mathrm{Li}(X)$. Therefore,
$t_{n}([x],[y])=[t_{n}(x,y)]\in \mathrm{Li}(X/\simeq)$ as well. Thus, $X/\simeq$ is also permutative.

\item $\leftarrow$ If $x,y\in X$ and $\mathrm{crit}(y)\in B$, then $[y]\in \mathrm{Li}(X/\simeq)$, so $\mathrm{crit}([x])\leq \mathrm{crit}([y])$. Similarly, if $\mathrm{crit}(x)\leq \mathrm{crit}(y)$, then $x^{n}*y\in\mathrm{Li}(X)$ for some $n$, so $[x]^{n}*[y]\in\mathrm{Li}(X/\simeq)$.

$\rightarrow$ Suppose now that $\mathrm{crit}([x])\leq \mathrm{crit}([y])$ and $\mathrm{crit}(y)\in A$. Then 
$[x]^{n}*[y]\in \mathrm{Li}(X/\simeq)$ for some $n$. Therefore $\mathrm{crit}(x^{n}*y)\in B$, but this is only possible if
$\mathrm{crit}(x)\leq \mathrm{crit}(y)$. Therefore, if $\mathrm{crit}([x])\leq \mathrm{crit}([y])$, then $\mathrm{crit}(x)\leq \mathrm{crit}(y)$ or $\mathrm{crit}(y)\in B$.

\item

$\leftarrow$ If $\mathrm{crit}(x)=\mathrm{crit}(y)$, then
\[\mathrm{crit}(x)\leq \mathrm{crit}(y)\leq \mathrm{crit}(x),\] 
so 
\[\mathrm{crit}([x])\leq \mathrm{crit}([y])\leq \mathrm{crit}([x]),\]
hence $\mathrm{crit}([x])=\mathrm{crit}([y])$. If $\mathrm{crit}(x),\mathrm{crit}(y)\in B$, then $[x],[y]\in \mathrm{Li}(X/\simeq)$, so $\mathrm{crit}([x])=\mathrm{crit}([y])$.

$\rightarrow$ Suppose $\mathrm{crit}([x])=\mathrm{crit}([y])$. If $[x]\in \mathrm{Li}(X/\simeq)$, then $[y]\in \mathrm{Li}(X/\simeq)$, so $\mathrm{crit}(x)\in B$ and $\mathrm{crit}(y)\in B$ as well.
Now assume that $[x]\not\in \mathrm{Li}(X/\simeq)$. Then $\mathrm{crit}(x),\mathrm{crit}(y)\in A$ and 
\[\mathrm{crit}([x])\leq \mathrm{crit}([y])\leq \mathrm{crit}([x]).\]
Therefore, we have
\[\mathrm{crit}(x)\leq \mathrm{crit}(y)\leq \mathrm{crit}(x),\]
so $\mathrm{crit}(x)=\mathrm{crit}(y).$
\end{enumerate}
\end{proof}
\begin{prop}
Suppose that $(X,*,\circ,1)$ is a permutative LD-monoid. If $\simeq$ is a congruence on $(X,*)$. Then
$\simeq$ is also a congruence with respect to the operation $\circ$.
\label{428tu0hj4ueri24tidef}
\end{prop}
\begin{proof}
Suppose that $\simeq$ is a congruence on $(X,*)$. It suffices to show that
\begin{equation}
\label{nhggedfeqekp}
x\simeq y\Rightarrow x\circ z\simeq y\circ z
\end{equation}
since if \ref{nhggedfeqekp} holds, then $y\simeq z$ implies that 
\[x\circ y=(x*y)\circ x\simeq(x*z)\circ x=x\circ z.\] 

Let $A,B$ be the partition of $\mathrm{crit}[X]$ where $\mathrm{crit}(x)\in B$ if and only if $[x]=[1]$.
Now suppose that $x,y,z\in X$, and assume that $x\simeq y$. Then we shall prove that $x\circ z\simeq y\circ z$ by cases.

\item[Case I: $\mathrm{crit}(x)\in B$ and $\mathrm{crit}(z)\in B$.]

We have $\mathrm{crit}(x),\mathrm{crit}(y),\mathrm{crit}(z)\in B$. Therefore, $\mathrm{crit}(x\circ z)=\min(\mathrm{crit}(x),\mathrm{crit}(z))\in B$, and $\mathrm{crit}(y\circ z)=\min(\mathrm{crit}(y),\mathrm{crit}(z))\in B.$
Therefore, we conclude that $x\circ z\simeq 1\simeq y\circ z$.

\item[Case II: $\mathrm{crit}(x)\not\in B$ or $\mathrm{crit}(z)\not\in B$.]

I claim that 
\begin{enumerate}
\item\label{421uefdh4kwnef} $\mathrm{crit}(x)\leq \mathrm{crit}(z)$ and $\mathrm{crit}(y)\leq \mathrm{crit}(z)$ or

\item\label{q4g1keddnh4ikws} $\mathrm{crit}(x)>\mathrm{crit}(z)$ and $\mathrm{crit}(y)>\mathrm{crit}(z)$.
\end{enumerate}
If $\mathrm{crit}(x)\in A$, then $\mathrm{crit}(x)=\mathrm{crit}(y)$ so the above claim trivially holds.
If $\mathrm{crit}(x)\in B$, then $\mathrm{crit}(y)\in B$ and $\mathrm{crit}(z)\in A.$ Therefore,
$\mathrm{crit}(x)>\mathrm{crit}(z)$ and $\mathrm{crit}(y)>\mathrm{crit}(z)$, so the claim holds in this case as well.

Suppose now that \ref{421uefdh4kwnef} holds. Then for sufficiently large even $n$, we have 
\[t_{n-1}(x,z)=1,t_{n}(x,z)=x\circ z\]
and
\[t_{n-1}(y,z)=1,t_{n}(y,z)=y\circ z.\]
Therefore, we have 
\[x\circ z=t_{n}(x,z)\simeq t_{n}(y,z)=y\circ z.\]

Now suppose that \ref{q4g1keddnh4ikws} holds. Then for sufficiently large odd $n$, we have $t_{n}(x,z)=x\circ z$ and
$t_{n}(y,z)=y\circ z$. Therefore, we conclude that
\[x\circ z=t_{n}(x,z)\simeq t_{n}(y,z)=y\circ z.\]
Thus, $\simeq$ is a congruence with respect to $\circ$ as well.
\end{proof}

\begin{prop}
\begin{enumerate}
\item A subalgebra of a permutative LD-system is also permutative.

\item If $i:(X,*)\rightarrow(Y,*)$ is an injective homomorphism between permutative LD-systems, then
$\mathrm{crit}(x)\leq \mathrm{crit}(y)$ if and only if $\mathrm{crit}(i(x))\leq \mathrm{crit}(i(y))$.
\end{enumerate}
\label{49tt4ngio224gf}
\end{prop}
\begin{proof}
\begin{enumerate}
\item Suppose that $(X,*),(Y,*)$ are LD-systems, $(Y,*)$ is permutative, and $i:(X,*)\rightarrow(Y,*)$ is a homomorphism.
I first claim that if $x\in \mathrm{Li}(X)$ if and only if $i(x)\in \mathrm{Li}(Y)$.
Suppose that $i(x)\in \mathrm{Li}(Y)$. Then whenever $y\in X$ we have $i(x*y)=i(x)*i(y)=i(y)$, so $x*y=y$, hence
$x\in \mathrm{Li}(X)$.
Now assume that $i(x)\not\in \mathrm{Li}(Y)$. Then since $\mathrm{Li}(Y)$ is permutative, we have $i(x)\neq i(x)*i(x)=i(x*x)$, so
$x\not\in \mathrm{Li}(X)$. 

Now I claim that $\mathrm{Li}(X)$ is a left-ideal. 
Suppose that $x\in \mathrm{Li}(X)$ and $r\in X$. Then $i(x)\in \mathrm{Li}(Y)$. Therefore, $i(r*x)=i(r)*i(x)\in \mathrm{Li}(Y)$ as well, so
$r*x\in\mathrm{Li}(X)$. We therefore conclude that $\mathrm{Li}(X)$ is a left-ideal.

Now assume that $x,y\in X$. Then there is some $n$ where $i(t_{n}(x,y))=t_{n}(i(x),i(y))\in \mathrm{Li}(Y)$. Therefore,
$t_{n}(x,y)\in \mathrm{Li}(X)$. We conclude that $(X,*)$ is permutative.

\item Suppose that $x,y\in X$. If $\mathrm{crit}(x)\leq \mathrm{crit}(y)$, then $x^{n}*y\in \mathrm{Li}(X)$ for some $n$. Therefore, we have
$i(x)^{n}*i(y)=i(x^{n}*y)\in \mathrm{Li}(Y)$. We conclude that $\mathrm{crit}(i(x))\leq \mathrm{crit}(i(y))$.

Now suppose that $\mathrm{crit}(i(x))\leq \mathrm{crit}(i(y))$. Then there is some $n$ where
\[i(x^{n}*y)=i(x)^{n}*i(y)\in\mathrm{Li}(Y).\]
Therefore, $x^{n}*y\in \mathrm{Li}(X)$. We conclude that $\mathrm{crit}(x)\leq \mathrm{crit}(y)$.
\end{enumerate}
\end{proof}

\begin{prop}
Let $\phi:X\rightarrow Y$ be a homomorphism between permutative LD-systems. Then $\mathrm{crit}(\phi(x))\leq \mathrm{crit}(\phi(y))$ if and only if $\mathrm{crit}(x)\leq \mathrm{crit}(y)$ or $\phi(y)\in \mathrm{Li}(Y)$.
\end{prop}
\begin{proof}
Let $s:X\rightarrow Z,i:Z\rightarrow Y$ be homomorphisms between permutative LD-systems $X,Y,Z$ such that $s$ is surjective
and $i$ is injective and $\phi=is$. Then by Theorem \ref{uh3d5tjh4uyyuin} and
Proposition \ref{49tt4ngio224gf},
$\mathrm{crit}(\phi(x))\leq \mathrm{crit}(\phi(y))$ if and only if
$\mathrm{crit}(s(x))\leq \mathrm{crit}(s(y))$ if and only if $\mathrm{crit}(x)\leq \mathrm{crit}(y)$ or $s(y)\in \mathrm{Li}(Z)$ if and only if
$\mathrm{crit}(x)\leq \mathrm{crit}(y)$ or $\phi(y)\in \mathrm{Li}(Y)$.
\end{proof}

\begin{prop}
Let $(X,*,\circ),(Y,*,\circ)$ be permutative LD-monoids and let $\phi:X\rightarrow Y$ be a homomorphism between LD-systems.
Then $\phi$ is also a homomorphism between LD-monoids.
\end{prop}
\begin{proof}
We shall prove this result in a couple of different cases.

\item[Case 1: $\mathrm{crit}(x)\leq \mathrm{crit}(y)$.]

If $\mathrm{crit}(x)\leq \mathrm{crit}(y)$, then $\mathrm{crit}(\phi(x))\leq \mathrm{crit}(\phi(y))$. Therefore, for large enough even $n$, we have
$t_{n}(x,y)=x\circ y$ and $t_{n}(\phi(x),\phi(y))=\phi(x)\circ\phi(y)$. Therefore,
\[\phi(x\circ y)=\phi(t_{n}(x,y))=t_{n}(\phi(x),\phi(y))=\phi(x)\circ\phi(y).\]

\item[Case 2: $\mathrm{crit}(x)>\mathrm{crit}(y),\phi(y)=1$.]
We have $\mathrm{crit}(\phi(x))\geq\mathrm{crit}(\phi(y))=\max(\mathrm{crit}[Y])$. Therefore,
$\phi(x)=1$ as well, so $\phi(x)\circ\phi(y)=1$. Similarly,
\[\mathrm{crit}(x\circ y)=\min(\mathrm{crit}(x),\mathrm{crit}(y))=\mathrm{crit}(y),\]
so
\[\mathrm{crit}(\phi(x\circ y))=\mathrm{crit}(\phi(y))=\max(\mathrm{crit}[Y]).\]
We conclude that $\phi(x\circ y)=1$ as well.

\item[Case 3: $\mathrm{crit}(x)>\mathrm{crit}(y),\phi(y)\neq 1$.]

In this case, $\mathrm{crit}(\phi(x))>\mathrm{crit}(\phi(y))$. Therefore,
for large enough odd $n$, we have $t_{n}(x,y)=x\circ y$ and $t_{n}(\phi(x),\phi(y))=\phi(x)\circ\phi(y)$. Therefore,
\[\phi(x\circ y)=\phi(t_{n}(x,y))=t_{n}(\phi(x),\phi(y))=\phi(x)\circ\phi(y).\]
\end{proof}

We shall now algebraize the congruences $\equiv^{\gamma}$ on $\mathcal{E}_{\lambda}$ from set theory to algebra.

\begin{lem}
Let $X$ be a permutative LD-system. Now suppose that $r,s\in X$, $\mathrm{crit}(r)\leq \mathrm{crit}(s)$, and $r*r,s*s\in \mathrm{Li}(X)$. Then $s*x=s*y$ implies that $r*x=r*y$. In particular, if $\mathrm{crit}(r)=\mathrm{crit}(s)$, then $r*x=r*y$ if and only if $s*x=s*y$.

\label{4th2utqebnjofa}
\end{lem}
\begin{proof}
Suppose that $s*x=s*y$. Then since $\mathrm{crit}(r)\leq \mathrm{crit}(s)$, we have $r*s\in \mathrm{Li}(X)$. Therefore,
\[r*x=(r*s)*(r*x)=r*(s*x)=r*(s*y)=(r*s)*(r*y)=r*y.\]
\end{proof}
\begin{defn}
Suppose that $X$ is a permutative LD-system. If $\alpha\in \mathrm{crit}[X]$, then define \index{$\equiv^{\alpha}$}$\equiv^{\alpha}$ to be the equivalence relation $X$ such that if
\[r\in X,r*r\in \mathrm{Li}(X),\mathrm{crit}(r)=\alpha,\]
then $x\equiv^{\alpha}y$ if and only if $r*x=r*y$.
\end{defn}
By Lemma \ref{4th2utqebnjofa}, the equivalence relation $\equiv^{\alpha}$ does not depend on the choice of element $r$.
The equivalence relation $\equiv^{\alpha}$ is a congruence on the permutative LD-system $X$. Furthermore, if $\alpha\leq\beta$ and $x\equiv^{\beta}y$, then $x\equiv^{\alpha}y$ by Lemma \ref{4th2utqebnjofa}. By the following proposition, we conclude that the congruences $\equiv^{\alpha}$ on permutative LD-systems $X$ generalize the congruences $\equiv^{\gamma}$ which we have defined on $\mathcal{E}_{\lambda}$ whenever $\gamma$ is a limit ordinal with $\gamma<\lambda$.

\begin{prop}
Suppose that $j,k,l\in\mathcal{E}_{\lambda}$ and $\gamma<\lambda$ is a limit ordinal with
$\mathrm{crit}(j*j)\geq\gamma\geq\mathrm{crit}(j)$. Then $k\equiv^{\mathrm{crit}(j)}l$ if and only if $j*k\equiv^{\gamma}j*l$.
\end{prop}
\begin{prop}
Suppose that $X$ is a permutative LD-system.
\begin{enumerate}
\item If $\alpha\in\mathrm{crit}[X],\mathrm{crit}(z)\geq\alpha,x\in X$, then $x\equiv^{\alpha}z*x$. In particular, if
$x\in X$ and $\alpha\leq\mathrm{crit}(x)$, then $x\equiv^{\alpha}y$ for some $y\in \mathrm{Li}(X)$.

\item If $x,y\in X$ and $x\equiv^{\alpha}y$, then $\mathrm{crit}(x)<\alpha$ if and only if $\mathrm{crit}(y)<\alpha$. Furthermore, if
$\mathrm{crit}(x)<\alpha$, then $\mathrm{crit}(x)=\mathrm{crit}(y)$.
\end{enumerate}
\end{prop}
\begin{proof}
\begin{enumerate}
\item Suppose that $r*r\in \mathrm{Li}(X),\mathrm{crit}(r)=\alpha$. Then $\mathrm{crit}(r*z)\geq\mathrm{crit}(r*r)=\max(\mathrm{crit}[X])$. Therefore,
$r*x=(r*z)*(r*x)=r*(z*x)$, so $x\equiv^{\alpha}z*x$.

If $x\in X,\alpha\leq\mathrm{crit}(x)$, then $x\equiv^{\alpha}r*x$ but
$\mathrm{crit}(r*x)=\max(\mathrm{crit}[X])$, so if $y=r*x$, then $x\equiv^{\alpha}y$ and $y\in\mathrm{Li}(X)$.

\item Suppose that $r*r\in \mathrm{Li}(X),\mathrm{crit}(r)=\alpha$ and $x\equiv^{\alpha}y$. Then
$\mathrm{crit}(x)<\alpha$ if and only if
\[\mathrm{crit}(r*x)=r^{\sharp}(\mathrm{crit}(x))\neq \max(\mathrm{crit}[X])\]
if and only if $r*x\not\in \mathrm{Li}(X)$. Similarly, $\mathrm{crit}(y)<\alpha$ if and only if $r*y\not\in \mathrm{Li}(X)$. Therefore, since $r*x=r*y$, we have $\mathrm{crit}(x)<\alpha$ if and only if $\mathrm{crit}(y)<\alpha$.

Now, if $\mathrm{crit}(x)<\alpha$, then 
\[r^{\sharp}(\mathrm{crit}(y))=\mathrm{crit}(r*y)
=\mathrm{crit}(r*x)=r^{\sharp}(\mathrm{crit}(x))=\mathrm{crit}(x).\]
This is only possible if $\mathrm{crit}(y)=\mathrm{crit}(x)$.
\end{enumerate}
\end{proof}
\begin{prop}
The congruence $\equiv^{\alpha}$ is the smallest congruence on $X$ such that if $\mathrm{crit}(x)\geq\alpha$, then
$x\equiv^{\alpha}y$ for some $y\in \mathrm{Li}(X)$.
\end{prop}
\begin{proof}
Suppose that $r\in X,r*r\in\mathrm{Li}(X),\mathrm{crit}(r)=\alpha.$

If $\mathrm{crit}(x)\geq\alpha$, then 
\[r*x=(r*r)*(r*x)=r*(r*x),\]
so $x\equiv^{\alpha}r*x$ and $r*x\in \mathrm{Li}(X)$.

Suppose now that $\simeq$ is a congruence on $X$ such that if $\mathrm{crit}(x)\geq\alpha$, then $x\simeq y$ for some
$y\in \mathrm{Li}(X)$. Let $x,y\in X$ be elements with $x\equiv^{\alpha}y$.
Then there is some $s\in \mathrm{Li}(X)$ with $r\simeq s$. Therefore, 
\[x=s*x\simeq r*x=r*y\simeq s*y=y.\]
\end{proof}

\begin{prop}
If $j,k,l\in\mathcal{E}_{\lambda}$, then $k\equiv^{\alpha}l$ if and only if $j*k\equiv^{j(\alpha)}j*l$.
\label{t42hwWEDhvei}
\end{prop}
\begin{proof}
This follows from the elementarity of the mapping $j$.
\end{proof}
Proposition \ref{t42hwWEDhvei} partially extends from set theory to algebra.
\begin{prop}
If $x,y\in X$ and $x\equiv^{\alpha}y$, then $v*x\equiv^{v^{\sharp}(\alpha)}v*y$.
\end{prop}
\begin{proof}
Suppose that $r\in X,\mathrm{crit}(r)=\alpha,r*r\in\mathrm{Li}(X)$. Then $r*x=r*y$. Let $s=v*r$. Then \[s*s=(v*r)*(v*r)=v*(r*r)\in\mathrm{Li}(X),\]
\[\mathrm{crit}(s)=v^{\sharp}(\mathrm{crit}(r))=v^{\sharp}(\alpha),\]
and
\[s*(v*x)=(v*r)*(v*x)=v*(r*x)\]
\[=v*(r*y)=(v*r)*(v*y)=s*(v*y),\]
so $v*x\equiv^{v^{\sharp}(\alpha)}v*y$.
\end{proof}

\begin{defn}
A permutative LD-system $X$ shall be called \index{critically simple}\emph{critically simple} if there is no non-trivial congruence
$\simeq$ on $X$ such that if $x\simeq y$, then $x\in \mathrm{Li}(X)$ iff $y\in \mathrm{Li}(X)$.
\end{defn}
It is easy to see that every critically simple permutative LD-system is reduced. If $X$ is a linear order with greatest element, then the Heyting algebra $(X,\rightarrow)$ is critically simple. Each classical Laver table is critically simple.
\begin{defn}
Now suppose that $X$ is a permutative LD-system. Then define an equivalence relation \index{$\simeq_{cmx}$}$\simeq_{cmx}$ on $X$ where we have $x\simeq_{cmx}y$ whenever 
\[a*x*b_{1}*\ldots*b_{n}\in \mathrm{Li}(X)\]
if and only if
\[a*y*b_{1}*\ldots*b_{n}\in \mathrm{Li}(X)\]
for each choice of $a,b_{1},\ldots,b_{n}$ and $n\in\omega$. In particular, if $x\simeq_{cmx}y$, then $a*x\in \mathrm{Li}(X)$ if and only if $a*y\in \mathrm{Li}(X)$ and $x\in \mathrm{Li}(X)$ if and only if $y\in \mathrm{Li}(X)$. 
\end{defn}
The letters ``cmx" stand for ``critically maximal."

\begin{thm}
Suppose that $X$ is a permutative LD-system.
\begin{enumerate}
\item $\simeq_{cmx}$ is a congruence on $X$.

\item $\simeq_{cmx}$ is the largest congruence on $X$ where if $x\simeq_{cmx}y$, then $x\in \mathrm{Li}(X)$ iff $y\in \mathrm{Li}(X)$. 

\item $X$ is critically simple if and only if $\simeq_{cmx}$ is the trivial congruence.

\item
\begin{enumerate}
\item \label{0gt2i0t2i4tr} $x\simeq_{cmx}y$ if and only if whenever $t$ is a term function in the language of LD-systems on $n+1$ variables and $a_{1},\ldots,a_{n}\in X$, then $t(a_{1},\ldots,a_{n},x)\in \mathrm{Li}(X)$ if and only if $t(a_{1},\ldots,a_{n},y)\in \mathrm{Li}(X)$.

\item \label{1gt2i0t2i4tr} Suppose that $X$ is an LD-monoid. Then $x\simeq_{cmx}y$ if and only if whenever $t$ is a term function in the language of 
LD-monoids of $n+1$ variables and $a_{1},\ldots,a_{n}\in X$ then $t(a_{1},\ldots,a_{n},x)=1$ if and only if
$t(a_{1},\ldots,a_{n},y)=1$.
\end{enumerate}
\item Suppose that $\simeq$ is a congruence on $X$ such that if $x\simeq y$, then $x\in\mathrm{Li}(X)$ iff $y\in\mathrm{Li}(X)$. Then
$X/\simeq$ is critically simple if and only if $\simeq$ is equal to $\simeq_{cmx}$.
\end{enumerate}
\end{thm}
\begin{proof}
\begin{enumerate}
\item Suppose that $x,y,z\in X$ and $x\simeq_{cmx}y$. Let $a,b_{1},\ldots,b_{n}\in X$. Then
\[a*(x*z)*b_{1}\ldots*b_{n}=(a*x)*(a*z)*b_{1}*\ldots*b_{n}\in \mathrm{Li}(X)\]
if and only if
\[a*(y*z)*b_{1}\ldots*b_{n}=(a*y)*(a*z)*b_{1}*\ldots*b_{n}\in \mathrm{Li}(X).\]
Therefore $x*z\simeq_{cmx}y*z$.

Suppose that $x,y,z\in X,y\simeq_{cmx}z$ and that $a,b_{1},\ldots,b_{n}\in X$. Let $N$ be a natural number such that
$t_{N}(a,x)\in\mathrm{Li}(X)$. Then
\[(a*(x*y))*b_{1}*\ldots*b_{n}=t_{N+1}(a,x)*(t_{N}(a,x)*y)*b_{1}*\ldots*b_{n}\]
\[=t_{N+1}(a,x)*y*b_{1}*\ldots*b_{n}.\]
Similarly,
\[(a*(x*z))*b_{1}*\ldots*b_{n}=t_{N+1}(a,x)*z*b_{1}*\ldots*b_{n}.\]

Therefore, since,
\[t_{N+1}(a,x)*y*b_{1}*\ldots*b_{n}\in\mathrm{Li}(X)\Leftrightarrow t_{N+1}(a,x)*z*b_{1}*\ldots*b_{n}\in
\mathrm{Li}(X),\]
we conclude that
\[(a*(x*y))*b_{1}*\ldots*b_{n}\in\mathrm{Li}(X)\Leftrightarrow(a*(x*z))*b_{1}*\ldots*b_{n}\in\mathrm{Li}(X).\]
Therefore $x*y\simeq_{cmx}x*z.$

\item \label{4h39t4ui0} Clearly $\simeq_{cmx}$ is a congruence where if $x\simeq_{cmx}y$, then $x\in\mathrm{Li}(X)$ iff $y\in \mathrm{Li}(X)$.
Now assume that $\simeq$ is another congruence where if $x\simeq y$, then $x\in\mathrm{Li}(X)$ iff $y\in\mathrm{Li}(X)$.
Then whenever $x\simeq y$ and $a,b_{1},\ldots,b_{n}\in X$, we have
\[a*x*b_{1}*\ldots*b_{n}\simeq a*y*b_{1}*\ldots*b_{n}.\]
Therefore, 
\[a*x*b_{1}*\ldots*b_{n}\in \mathrm{Li}(X)\]
if and only if
\[a*y*b_{1}*\ldots*b_{n}\in \mathrm{Li}(X),\]
so $x\simeq_{cmx}y$.

\item This follows from \ref{4h39t4ui0} and the definitions.

\item The direction $\leftarrow$ follows by the definition of $\simeq_{cmx}$. For $\rightarrow$ assume that $x\simeq_{cmx}y$.
Suppose that $t$ is an $n+1$-ary term in the language of LD-systems for $\ref{0gt2i0t2i4tr}$ and in the language of LD-monoids for
$\ref{1gt2i0t2i4tr}$. Then by Proposition \ref{428tu0hj4ueri24tidef}, the congruence $\simeq_{cmx}$ is also a congruence with respect to the operation $\circ$ whenever $X$ is an LD-monoid. Therefore, we have
\[t(a_{1},\ldots,a_{n},x)\simeq_{cmx}t(a_{1},\ldots,a_{n},y).\]
Therefore,
\[t(a_{1},\ldots,a_{n},x)\in \mathrm{Li}(X)\]
if and only if
\[t(a_{1},\ldots,a_{n},y)\in \mathrm{Li}(X).\]

\item 

$\leftarrow$ I claim that $X/\simeq_{cmx}$ is critically simple. We have $z\in \mathrm{Li}(X)$ if and only if
$[z]_{cmx}\in \mathrm{Li}(X/\simeq_{cmx})$. Suppose that
$[x]_{cmx},[y]_{cmx}\in X/\simeq_{cmx}$ and 
\[[a]_{cmx}*[x]_{cmx}*[b_{1}]_{cmx}*\ldots*[b_{n}]_{cmx}\in \mathrm{Li}(X/\simeq_{cmx})\]
if and only if 
\[[a]_{cmx}*[y]_{cmx}*[b_{1}]_{cmx}*\ldots*[b_{n}]_{cmx}\in \mathrm{Li}(X/\simeq_{cmx})\]
for all $a,b_{1},\ldots,b_{n}\in X$. Then
\[a*x*b_{1}*\ldots*b_{n}\in \mathrm{Li}(X)\]
if and only if 
\[a*y*b_{1}*\ldots*b_{n}\in \mathrm{Li}(X).\]
Therefore, $[x]_{cmx}=[y]_{cmx}$. We conclude that $X/\simeq_{cmx}$ is critically simple.

$\rightarrow$ Suppose that $\simeq$ is a congruence on $X$ where if $x\simeq y$ then $x\in \mathrm{Li}(X)$ if and only if $y\in \mathrm{Li}(X)$ but where $\simeq$ is not equal to $\simeq_{cmx}$.

Then there are elements $x,y$ with $x\simeq_{cmx}y$, but $x\not\simeq y$. Therefore, $[x]_{\simeq}\neq[y]_{\simeq}$.

However, since $x\simeq_{cmx}y$, whenever $a,b_{1},\ldots,b_{n}\in X$, we have
\[a*x*b_{1}*\ldots*b_{n}\in \mathrm{Li}(X)\]
if and only if 
\[a*y*b_{1}*\ldots*b_{n}\in \mathrm{Li}(X).\]
Therefore, whenever 
\[[a]_{\simeq},[b_{1}]_{\simeq},\ldots,[b_{n}]_{\simeq}\in X/\simeq,\]
we have
\[[a]_{\simeq}*[x]_{\simeq}*[b_{1}]_{\simeq}*\ldots*[b_{n}]_{\simeq}\in \mathrm{Li}(X/\simeq)\]
if and only if
\[[a]_{\simeq}*[y]_{\simeq}*[b_{1}]_{\simeq}*\ldots*[b_{n}]_{\simeq}\in \mathrm{Li}(X/\simeq),\]
but $[x]_{\simeq}\neq[y]_{\simeq}$. Therefore $X/\simeq$ is not critically simple.
\end{enumerate}
\end{proof}
\begin{exam}
A subalgebra of a critically simple permutative LD-system is not necessarily critically simple.
For example, every classical Laver table is critically simple. However, 
in the classical Laver table $A_{4}$, the subalgebra
\[B=\{2, 4, 6, 8, 10, 12, 14, 16\}\]
is not critically simple
since $B\models 2\simeq_{cmx}10$.
\end{exam}
\begin{prop}
Let $X$ be a permutative LD-system. Suppose that $x,x_{1},\ldots,x_{n}\in X$, and
if $1\leq m<n$, then either
\[(x*x_{1}*\ldots*x_{m})^{\sharp}(\mathrm{crit}(x))=\max(\mathrm{crit}[X])\]
or
\[(x*(x_{1}*\ldots*x_{m}))^{\sharp}(\mathrm{crit}(x))=\max(\mathrm{crit}[X]).\]
Then
\[x*x_{1}*\ldots*x_{n}=x*(x_{1}*\ldots*x_{n}).\]
\label{8u3y290rq}
\end{prop}
\begin{proof}
We shall prove this result by induction on $n$. For $n=1$, the result is clearly true. Now assume that $n>1$.
Then for $m<n$ either 
\[(x*x_{1}*\ldots*x_{n-1})^{\sharp}(\mathrm{crit}(x))=\max(\mathrm{crit}[X])\]
or
\[(x*(x_{1}*\ldots*x_{n-1}))^{\sharp}(\mathrm{crit}(x))=\max(\mathrm{crit}[X]).\]
However, by the induction hypothesis, we have
$x*(x_{1}*\ldots*x_{n-1})=x*x_{1}*\ldots*x_{n-1}$, so in either case,
\[(x*(x_{1}*\ldots* x_{n-1}))^{\sharp}(\mathrm{crit}(x))=\max(\mathrm{crit}[X]).\]
Therefore, we have
\[x*(x_{1}*\ldots*x_{n-1})*x\in \mathrm{Li}(X).\]
Thus,
\[x*x_{1}*\ldots*x_{n}=(x*x_{1}*\ldots*x_{n-1})*x_{n}=x*(x_{1}*\ldots*x_{n-1})*x_{n}\]
\[=(x*(x_{1}*\ldots*x_{n-1})*x)*(x*(x_{1}*\ldots*x_{n-1})*x_{n})\]
\[=(x*(x_{1}*\ldots*x_{n-1}))*(x*x_{n})=x*(x_{1}*\ldots*x_{n-1}*x_{n}).\]
\end{proof}
\begin{cor}
Suppose that $X$ is a permutative LD-system. Let $x,x_{1},\ldots,x_{n}\in X$.
Suppose that 
\[(x*(x_{1}*\ldots* x_{m}))^{\sharp}(\mathrm{crit}(x))\geq\beta\]
or
\[(x*x_{1}*\ldots*x_{m})^{\sharp}(\mathrm{crit}(x))\geq\beta\]
for $1\leq m<n$. Then
\[x*x_{1}*\ldots*x_{n}\equiv^{\beta}x*(x_{1}*\ldots* x_{n}).\]

\label{hu47u240pu2GNF}
\end{cor}

\begin{prop}
Let $X$ be a permutative LD-system. Suppose that $x,x_{1},\ldots,x_{n}\in X$ and
\[\mathrm{crit}(x_{1}*\ldots*x_{m})<\mathrm{crit}(x)\]
for $1\leq m<n$ and
\[\mathrm{crit}(x_{1}*\ldots*x_{n})\leq \mathrm{crit}(x).\]
Then
\[\mathrm{crit}(x*x_{1}*\ldots*x_{n})=x^{\sharp}(\mathrm{crit}(x_{1}*\ldots*x_{n})).\]
\label{t4j8i09t42huqa}
\end{prop}
\begin{proof}
If $n=1$, then the result is trivial, so assume that $n>1$. Let $\alpha=\mathrm{crit}(x)$.
Let 
\[\beta=\min((x*x_{1})^{\sharp}(\alpha),\ldots,(x*(x_{1}*\ldots*x_{n-1}))^{\sharp}(\alpha)).\]
Then
\[x*x_{1}*\ldots*x_{n}\equiv^{\beta}x*(x_{1}*\ldots*x_{n}).\]

If $\beta=\max(\mathrm{crit}[X])$, then 
\[x*x_{1}*\ldots*x_{n}=x*(x_{1}*\ldots*x_{n}),\]
so the proof is complete in this case. Now assume that $\beta<\max(\mathrm{crit}[X])$.

Let $r$ be a natural number such that $\beta=(x*(x_{1}*\ldots*x_{r}))^{\sharp}(\alpha)$. Then

Let $y=x_{1}*\ldots*x_{r}$. Then $\mathrm{crit}(y)<\alpha$, so let $\delta$ be a critical point with
$\delta<\alpha\leq y^{\sharp}(\delta)$. Then
\[x^{\sharp}(\alpha)\leq x^{\sharp}(y^{\sharp}(\delta))=(x*y)^{\sharp}(x^{\sharp}(\delta))=(x*y)^{\sharp}(\delta)<(x*y)^{\sharp}(\alpha)=\beta.\]

Therefore, since $\mathrm{crit}(x_{1}*\ldots*x_{n})\leq\alpha$, we have
\[\mathrm{crit}(x*(x_{1}*\ldots*x_{n}))\leq x^{\sharp}(\alpha)<\beta,\]
so since
\[x*x_{1}*\ldots*x_{n}\equiv^{\beta}x*(x_{1}*\ldots*x_{n}),\]
we conclude that
\[\mathrm{crit}(x*x_{1}*\ldots*x_{n})=\mathrm{crit}(x*(x_{1}*\ldots*x_{n})).\]
\end{proof}
\begin{prop}
Suppose that $X$ is a permutative LD-system and $x_{i}\in X$ for $i\in\{1,\ldots,2^{n}\}$. Then either
\begin{enumerate}
\item there is some $i$ with
\[x_{1}*\ldots*x_{i}\in\mathrm{Li}(X),\]
or
\item 
\[|\{\mathrm{crit}(x_{1}),\ldots,\mathrm{crit}(x_{1}*\ldots*x_{2^{n}})\}|\geq n+1.\]
\end{enumerate}
\end{prop}
\begin{proof}
Suppose that there is no $i$ with $x_{1}*\ldots*x_{i}\not\in\mathrm{Li}(X)$.

By the induction hypothesis, we know that
\[|\{\mathrm{crit}(x_{1}),\ldots,\mathrm{crit}(x_{1}*\ldots*x_{2^{n-1}})\}|\geq n.\]
If 
\[|\{\mathrm{crit}(x_{1}),\ldots,\mathrm{crit}(x_{1}*\ldots*x_{2^{n-1}})\}|>n,\]
then the proof is complete.
Therefore, assume that 
\[|\{\mathrm{crit}(x_{1}),\ldots,\mathrm{crit}(x_{1}*\ldots*x_{2^{n-1}})\}|=n.\]
Then let $r$ be the least natural number such that
\[\mathrm{crit}(x_{1}*\ldots*x_{r})=\max(\mathrm{crit}(x_{1}),\ldots,\mathrm{crit}(x_{1}*\ldots*x_{2^{n-1}})).\]

\item[Case 1: $\max(\mathrm{crit}(x_{r+1}),\ldots,\mathrm{crit}(x_{r+1}*\ldots*x_{r+2^{n-1}}))\geq \mathrm{crit}(x_{1}*\ldots*x_{r}).$]

Let $s$ be the least natural number such that 
\[\mathrm{crit}(x_{r+1}*\ldots*x_{r+s})\geq \mathrm{crit}(x_{1}*\ldots*x_{r}).\]

Let
\[\beta=\min\{((x_{1}*\ldots*x_{r})*(x_{r+1}*\ldots*x_{r+m}))^{\sharp}(\mathrm{crit}(x_{1}*\ldots*x_{r}))\mid 1\leq m<s\}.\]
Then $\beta>\mathrm{crit}(x_{1}*\ldots*x_{r})$ and
\[((x_{1}*\ldots*x_{r})*(x_{r+1}*\ldots*x_{r+m}))^{\sharp}(\mathrm{crit}(x_{1}*\ldots*x_{r}))\geq\beta\]
for $1\leq m<s$, so
\[(x_{1}*\ldots*x_{r})*(x_{r+1}*\ldots*x_{r+s})\equiv^{\beta}x_{1}*\ldots*x_{r}*x_{r+1}*\ldots*x_{r+s}\]

Therefore, since 
\[\mathrm{crit}((x_{1}*\ldots*x_{r})*(x_{r+1}*\ldots*x_{r+s}))=(x_{1}*\ldots*x_{r})^{\sharp}(\mathrm{crit}(x_{r+1}*\ldots*x_{r+s}))\]
\[>\mathrm{crit}(x_{1}*\ldots*x_{r}),\]
we have
\[\mathrm{crit}(x_{1}*\ldots*x_{r}*x_{r+1}*\ldots*x_{r+s})>\mathrm{crit}(x_{1}*\ldots*x_{r})\]
as well proving that
\[\{\mathrm{crit}(x_{1}),\ldots,\mathrm{crit}(x_{1}*\ldots*x_{2^{n}})\}\]
has at least $n+1$ elements.

\item[Case 2: $\max(\mathrm{crit}(x_{r+1}),\ldots,\mathrm{crit}(x_{r+1}*\ldots*x_{r+2^{n-1}}))<\mathrm{crit}(x_{1}*\ldots*x_{r})$.]

In this case, since 
\[|\{\mathrm{crit}(x_{r+1}),\ldots,\mathrm{crit}(x_{r+1}*\ldots*x_{r+2^{n-1}})\}|\geq n\]
and
\[|\{\mathrm{crit}(x_{1}),\ldots,\mathrm{crit}(x_{1}*\ldots*x_{2^{n-1}})\}|=n,\]
there is some $m$ where
\[\mathrm{crit}(x_{r+1}*\ldots*x_{r+m})\not\in\{\mathrm{crit}(x_{1}),\ldots,\mathrm{crit}(x_{1}*\ldots*x_{2^{n-1}})\}.\]

Let $\beta=\mathrm{crit}(x_{1}*\ldots*x_{r})$. Then
\[x_{1}*\ldots*x_{r}*x_{r+1}*\ldots*x_{r+m}\equiv^{\beta}x_{r+1}*\ldots*x_{r+m},\]
so
\[\mathrm{crit}(x_{1}*\ldots*x_{r}*x_{r+1}*\ldots*x_{r+m})\]
\[=\mathrm{crit}(x_{r+1}*\ldots*x_{r+m})\not\in\{\mathrm{crit}(x_{1}),\ldots,\mathrm{crit}(x_{1}*\ldots*x_{2^{n-1}})\}.\]

We therefore conclude that
\[\{\mathrm{crit}(x_{1}),\ldots,\mathrm{crit}(x_{1}*\ldots*x_{2^{n}})\}\]
has at least $n+1$ elements.
\end{proof}
\begin{defn}
Suppose that $X$ is a permutative LD-system. Then define a relation \index{$\preceq$}$\preceq$ on $X$ were we set
$x\preceq y$ if there is some $n\in\omega$ along with $a_{1},\ldots,a_{n}\in X$ such that
$y=x*a_{1}*\ldots*a_{n}$ and if $0\leq m<n$, then
$x*a_{1}*\ldots*a_{m}\not\in \mathrm{Li}(X)$. Then $\preceq$ is a pre-ordering on $X$.
\end{defn}
\begin{prop}
The relation $\preceq$ is a partial ordering on $X$.
\end{prop}
\begin{proof}
Suppose to the contrary that $\preceq$ is not a partial ordering. Then there are distinct $x,z$ with
$x\preceq z\preceq x$. Therefore, there are $a_{0},\ldots,a_{n}$ where
$x*a_{0}*\ldots*a_{n}=x$ but where $x*a_{0}*\ldots*a_{m}\not\in \mathrm{Li}(X)$ for all $m<n$.

Now, let $r$ be such that $r<n$ and $\mathrm{crit}(x*a_{0}*\ldots*a_{r})$ is maximal.

Let $y=x*a_{0}*\ldots*a_{r}$. Then we have $y=y*a_{r+1}*\ldots*a_{n}*a_{0}*\ldots*a_{r}$. Therefore, define
\[b_{0}=a_{r+1},\ldots,b_{n-r-1}=a_{n},b_{n-r}=a_{0},\ldots,b_{n}=a_{r}.\]
Then 
\[y=y*b_{0}*\ldots*b_{n}\]
and
\[\mathrm{crit}(y*b_{0}*\ldots*b_{i})\leq \mathrm{crit}(y)<\max(\mathrm{crit}[X])\]
for all $i$.

Let $v$ be the least natural number such that
\[\mathrm{crit}(y*b_{0}*\ldots*b_{v})=\mathrm{crit}(y).\]
Then let 
\[\beta=\min((y*b_{0})^{\sharp}(\mathrm{crit}(y)),\ldots,(y*b_{0}*\ldots*b_{v-1})^{\sharp}(\mathrm{crit}(y))).\]
Then $\beta>\mathrm{crit}(y)$ and since 
\[(y*b_{0}*\ldots*b_{i})^{\sharp}(\mathrm{crit}(y))\geq\beta\]
for $0\leq i<v$, by Corollary \ref{hu47u240pu2GNF}, we have
\[y*b_{0}*\ldots*b_{v}\equiv^{\beta}y*(b_{0}*\ldots*b_{v}).\]
Since
\[\mathrm{crit}(y)=\mathrm{crit}(y*b_{0}*\ldots*b_{v}),\]
and since 
\[y*b_{0}*\ldots*b_{v}\equiv^{\beta}y*(b_{0}*\ldots*b_{v}),\]
we have 
\[y^{\sharp}(b_{0}*\ldots*b_{v})=\mathrm{crit}(y*(b_{0}*\ldots*b_{v}))=\mathrm{crit}(y)\]
which is a contradiction.
\end{proof}
\begin{prop}
Let $X$ be a permutative LD-system, and suppose $x,y\in X\setminus\mathrm{Li}(X)$. 
If $n$ is the least natural number with $t_{n}(x,y)\in \mathrm{Li}(X)$, then 
\[x=t_{2}(x,y)\prec\ldots\prec t_{n-1}(x,y)\prec t_{n}(x,y)\in \mathrm{Li}(X),\]
and if $X$ is reduced, then
\[x=t_{2}(x,y)\prec\ldots\prec t_{n-1}(x,y)=x\circ y\prec t_{n}(x,y)=1.\]
\end{prop}
\begin{proof}
This follows from the fact that if $m>2$, then
\[t_{m}(x,y)=t_{2}(x,y)*t_{1}(x,y)*t_{2}(x,y)*\ldots*t_{m-2}(x,y)\]
\[=x*t_{1}(x,y)*\ldots*t_{m-2}(x,y).\]
\end{proof}

From the partial ordering $\preceq$, we obtain the following algorithm for computing the critically maximal congruence $\simeq_{cmx}$.

\begin{algo}$\textbf{Computing the critically maximal congruence:}$
Suppose that $X$ is a finite reduced permutative LD-system. Then define an equivalence relation $\simeq_{rm}$ on $X$ where
we have $x\simeq_{rm}y$ precisely when
\[x*a_{1}*\ldots*a_{n}\in \mathrm{Li}(X)\Leftrightarrow y*a_{1}*\ldots*a_{n}\in \mathrm{Li}(X)\]
whenever $n\in\omega,a_{1},\ldots ,a_{n}\in X$. The ``rm" in $\simeq_{rm}$ stands for ``right multiple." The following steps
may be used to compute the critically maximal congruence $\simeq_{cmx}$.

\begin{enumerate}
\item Select a bijective function $f:X\rightarrow\{0,\ldots ,N\}$ such that if
$x\prec y$, then $f(y)<f(x)$. Let $g$ be the inverse of the function $f$.

\item We now compute the partition $P_{n}$ of $\{g(0),\ldots ,g(n)\}$ for all $n\in\{0,\ldots ,N\}$ by recursion on $n$.
Let $P_{0}=\{1\}$. For the induction step, suppose that $P_{n}$ has been computed already and $n<N$. Suppose that $P_{n}=\{R_{1},\ldots ,R_{k}\}$. Then select $x_{i}\in R_{i}$ for $1\leq i\leq k$, and let $x=g(n+1)$.

If there is some $j\in\{1,\ldots ,k\}$ where for all $b\in X$ there is some $l\in\{1,\ldots ,k\}$ where $x*b,x_{j}*b\in R_{l}$, then
set
\[P_{n+1}=\{R_{1},\ldots ,R_{j-1},R_{j}\cup\{x\},R_{j+1},\ldots ,R_{k}\}.\]
Otherwise, set 
\[P_{n+1}=P_{n}\cup\{\{x\}\}.\]

\item The partition $P_{N}$ corresponds to the equivalence relation $\simeq_{rm}$. Therefore, once $P_{N}$ has been computed, one can compute $\simeq_{cmx}$ since $x\simeq_{cmx}y$ if and only if $a*x,a*y\in R\in P_{N}$ for some $R$.
\end{enumerate}
\end{algo}
\begin{defn}
An element $x$ in a permutative LD-system $X$ is said to be \index{involutive}\emph{involutive} if $x*x\in \mathrm{Li}(X)$. Let \index{$\mathcal{I}_{X}$}$\mathcal{I}_{X}$ be the set of all involutive elements in $X$. A permutative LD-system $X$ is said to be \index{involutive}\emph{involutive} if all of its elements are involutive.
\end{defn}
$\mathcal{I}_{X}$ is a left-ideal in $X$, and in particular a sub-LD-system of $X$.

\begin{prop}
Suppose $X$ is a permutative LD-monoid and $x\in\mathcal{I}_{X},y,z\in X$.
\begin{enumerate}
\item If $\mathrm{crit}(x)\leq \mathrm{crit}(y)$, then $x\circ y=x$.

\item If $\mathrm{crit}(x)>\mathrm{crit}(y)$, then $x\circ y=x*y$. 

\item $y\equiv^{\mathrm{crit}(x)}z$ if and only if $x\circ y=x\circ z$.

\item If $y^{n}\in\mathcal{I}_{X}$, then $y^{m}=y^{n}$ whenever $m\geq n$.
\end{enumerate}
\end{prop}
\begin{lem}
If $X$ is a permutative LD-system and $x\in\mathcal{I}_{X}$ with $\mathrm{crit}(x)=\gamma$, then
\begin{enumerate}
\item if $\mathrm{crit}(a)<\mathrm{crit}(x)$, then $(x*a)^{\sharp}(\beta)=\max(\mathrm{crit}[X])$ for some $\beta<\gamma$, and

\item if $\mathrm{crit}(a)<\mathrm{crit}(x)$, then we have $x*a*x\in \mathrm{Li}(X)$.
\end{enumerate}
\label{t24jirho}
\end{lem}
\begin{proof}
\begin{enumerate}
\item Suppose that $\mathrm{crit}(a)<\mathrm{crit}(x)$. Then there is some $\beta$ with
$\mathrm{crit}(a)\leq\beta<\mathrm{crit}(x)$, but where $a^{\sharp}(\beta)\geq \mathrm{crit}(x)$. Therefore, let
$b\in X$ be an element with $\mathrm{crit}(b)=\beta$. Then
\[(x*a)^{\sharp}(\beta)=(x*a)^{\sharp}(\mathrm{crit}(b))=(x*a)^{\sharp}(\mathrm{crit}(x*b))\]
\[=\mathrm{crit}((x*a)*(x*b))=\mathrm{crit}(x*(a*b))=x^{\sharp}(a^{\sharp}(\beta))=\max(\mathrm{crit}[X]).\]

\item We have
\[\mathrm{crit}(x*a*x)=(x*a)^{\sharp}(\mathrm{crit}(x))\geq
(x*a)^{\sharp}(\beta)=\max(\mathrm{crit}[X]).\]
\end{enumerate}
\end{proof}
\begin{lem}
If $x\in\mathcal{I}_{X}$ and for $1\leq m<n$ we have $x*x_{1}*\ldots*x_{m}\not\in \mathrm{Li}(X)$ or
$\mathrm{crit}(x_{1}*\ldots*x_{m})<\mathrm{crit}(x)$, then 
\[x*(x_{1}*\ldots*x_{n})=x*x_{1}*\ldots*x_{n}.\]
\label{32u8r9}
\end{lem}
\begin{proof}
We shall prove this result by induction on $n$. Suppose that $x\in\mathcal{I}_{X}$ and
$x*x_{1}*\ldots*x_{m}\not\in \mathrm{Li}(X)$ or $\mathrm{crit}(x_{1}*\ldots*x_{m})<\mathrm{crit}(x)$ for $1\leq m<n$. Then by the induction hypothesis,
for $1\leq m<n$, we have 
\[x*x_{1}*\ldots*x_{m}=x*(x_{1}*\ldots*x_{m}),\]
so $\mathrm{crit}(x_{1}*\ldots*x_{m})<\mathrm{crit}(x)$ in any case for $1\leq m<n$.
Therefore, by Lemma \ref{t24jirho}, we have
\[(x*(x_{1}*\ldots*x_{m}))^{\sharp}(\mathrm{crit}(x))=\max(\mathrm{crit}[X])\]
for $1\leq m<n$, so by Theorem
\ref{8u3y290rq}, we have $x*x_{1}*\ldots*x_{n}=x*(x_{1}*\ldots*x_{n})$.
\end{proof}
\begin{prop}
If $x,y\in\mathcal{I}_{X}$ and $\mathrm{crit}(x)=\mathrm{crit}(y)$, then $x\simeq_{cmx}y$.
\end{prop}
\begin{proof}
Suppose that $x,y\in\mathcal{I}_{X}$ and $\mathrm{crit}(x)=\mathrm{crit}(y)$. Then it suffices to show that
whenever $a,a_{1},\ldots,a_{n}\in X$, we have 
\[(a*x)*a_{1}*\ldots*a_{m}\not\in \mathrm{Li}(X)\]
whenever $1\leq m\leq n$ if and only if
\[(a*y)*a_{1}*\ldots*a_{m}\not\in \mathrm{Li}(X)\]
whenever $1\leq m\leq n$.

Suppose that 
\[(a*x)*a_{1}*\ldots*a_{m}\not\in \mathrm{Li}(X)\]
whenever $1\leq m\leq n$. Then since $x,y\in\mathcal{I}_{X}$, we have
$a*x,a*y\in\mathcal{I}_{X}$ as well.

If
\[(a*x)*a_{1}*\ldots*a_{m}\not\in \mathrm{Li}(X)\]
whenever $1\leq m\leq n$, then by Lemma \ref{32u8r9},
\[a*x*a_{1}*\ldots*a_{m}=a*x*(a_{1}*\ldots*a_{m}),\]
so
\[\mathrm{crit}(a_{1}*\ldots*a_{m})<\mathrm{crit}(a*x)\]
for $1\leq m\leq n$.
Therefore, we have 
\[\mathrm{crit}(a_{1}*\ldots*a_{m})<\mathrm{crit}(a*y)\]
for $1\leq m\leq n$, so again by Lemma \ref{32u8r9}, we conclude that
\[(a*y)*a_{1}*\ldots*a_{m}=a*y*(a_{1}*\ldots*a_{m})\not\in \mathrm{Li}(X)\]
for $1\leq m\leq n$. 

By the same argument, if
\[(a*y)*a_{1}*\ldots*a_{m}\not\in \mathrm{Li}(X)\]
whenever $1\leq m\leq n$, then
\[(a*x)*a_{1}*\ldots*a_{m}\not\in \mathrm{Li}(X)\] whenever $1\leq m\leq n$.
\end{proof}
\begin{prop}
Let $X$ be a permutative LD-system. Then $x\in\mathcal{I}_{X}$ if and only if $x$ is a maximal element in $\{y\in X\mid \mathrm{crit}(y)=\mathrm{crit}(x)\}$.
\end{prop}
\begin{proof}
$\leftarrow$. Suppose that $x\not\in\mathcal{I}_{X}$. Then $x*x\not\in \mathrm{Li}(X)$, so $x\prec x*x*x$, but $\mathrm{crit}(x*x*x)=\mathrm{crit}(x)$. Therefore
$x$ is not maximal in $\{y\in X\mid \mathrm{crit}(y)=\mathrm{crit}(x)\}$.

$\rightarrow$. Suppose that $x\in\mathcal{I}_{X}$ but $x$ is not maximal in $\{y\in X\mid \mathrm{crit}(y)=\mathrm{crit}(x)\}$. Then
$x\prec y$ for some $y$ with $\mathrm{crit}(x)=\mathrm{crit}(y)$. Therefore, there are $a_{1},\ldots,a_{n}$ with $x*a_{1}*\ldots*a_{n}=y$ but where
$x*a_{1}*\ldots*a_{m}\not\in \mathrm{Li}(X)$ whenever $1\leq m<n$.

Therefore, by Lemma \ref{32u8r9}, we have 
\[y=x*a_{1}*\ldots*a_{n}=x*(a_{1}*\ldots*a_{n}).\]
Therefore,
\[\mathrm{crit}(x)=\mathrm{crit}(y)=x^{\sharp}(\mathrm{crit}(a_{1}*\ldots*a_{n})),\]
which is impossible.
\end{proof}

\section{Multigenic Laver tables}
In this chapter, we shall begin to investigate multigenic Laver tables, and we shall also introduce the notion of a Laver-like LD-system.

Let $A$ be a set. Let \index{$A^{+}$}$A^{+}$ be the set of all non-empty words from the alphabet $A$. We shall use lowercase letters at the beginning of the alphabet such as $a,b,c,d$ to denote elements in the alphabet $A$ while we use boldface letters near the end of the alphabet such as $\mathbf{x},\mathbf{y},\mathbf{z}$ to denote strings. We shall write $|\mathbf{x}|$\index{$|\mathbf{x}|$} for
the length of the string $\mathbf{x}$ (for example, $|abcde|=5$). We shall let \index{$\preceq$}$\preceq$ denote the prefix ordering on $A^{+}$. For example, $abc\preceq abcde$. We shall let the concatenation operation precede other operations on strings when there are no parentheses. For example, $\mathbf{x}*\mathbf{y}a=\mathbf{x}*(\mathbf{y}a)$ and $\mathbf{x}\circ\mathbf{y}a=\mathbf{x}\circ(\mathbf{y}a)$.
\begin{defn}
Let $A$ be a finite set, and let $L\subseteq A^{+}$ be a finite subset which is downwards closed with respect to the prefix ordering $\preceq$. Suppose $M=\{\mathbf{x}a\mid\mathbf{x}\in L\}\cup A$ and $F=M\setminus L$.
We define a binary operation $*^{M}$ on $M$ by a double induction which is descending on $\mathbf{x}$ and for all $\mathbf{x}$ the induction will be ascending on $\mathbf{y}$ according to the following rules.
\begin{enumerate}
\item If $\mathbf{x}\in F$, then $\mathbf{x}*^{M}\mathbf{y}=\mathbf{y}$.

\item If $\mathbf{x}\in L$, then $\mathbf{x}*^{M}b=\mathbf{x} b$ whenever $b\in A$.

\item If $\mathbf{x},\mathbf{y}\in L$, then $\mathbf{x}*^{M}\mathbf{y} b=(\mathbf{x}*^{M}\mathbf{y})*^{M}\mathbf{x}b$ whenever $\mathbf{x}*^{M}\mathbf{y}$ has been defined previously
in the induction and where $(\mathbf{x}*^{M}\mathbf{y})*^{M}\mathbf{x} b$ has been previously defined in the induction process. Otherwise, we leave $\mathbf{x}*^{M}\mathbf{y} b$ undefined.
\end{enumerate}
The algebra $(M,*^{M})$ shall be known as a finite \index{pre-multigenic Laver table}\emph{pre-multigenic Laver table} over the alphabet $A$. Define \index{$\mathrm{Li}(M)$}$\mathrm{Li}(M)=F$.
\end{defn}

The proof of Theorems \ref{t4uh4aP} and \ref{m43ti0} are proven using the same sort of induction that is needed to show that the classical Laver tables are well-defined and self-distributive.

\begin{prop}
Let $(M,*)$ be a pre-multigenic Laver table. Then the operation $*$ is defined everywhere. Furthermore,
\begin{enumerate}[(i)]
\item $\mathbf{y}$ and $\mathbf{x}*\mathbf{y}$ both have the same last element,

\item $\mathbf{x}*\mathbf{z} b=\mathbf{x} b$ whenever $\mathbf{x}\in M\setminus\mathrm{Li}(M)$ and $\mathbf{x}*\mathbf{z}\in \mathrm{Li}(M)$,

\item $\mathbf{x}*\mathbf{z} b$ has $\mathbf{x}*\mathbf{z}$ as a proper initial segment whenever $\mathbf{x}*\mathbf{z}\in M\setminus\mathrm{Li}(M)$,

\item $\mathbf{x}$ is a proper initial segment of $\mathbf{x}*\mathbf{y}$ whenever $\mathbf{x}\in M\setminus\mathrm{Li}(M)$, and

\item the element $\mathbf{x}*\mathbf{y}$ only contains letters which are already contained in the word $\mathbf{x}\mathbf{y}$.
\end{enumerate}
\label{t4uh4aP}
\end{prop}
\begin{proof}
We shall prove this proposition for each $\mathbf{x}*\mathbf{y}$ by a double induction which is decreasing on $\mathbf{x}$, and for each $\mathbf{x}$, we shall proceed by an increasing induction on $\mathbf{y}$. The case where either $\mathbf{x}\in \mathrm{Li}(M)$ or $\mathbf{x}\not\in \mathrm{Li}(M),\mathbf{y}\in A$ follow directly from the definition of $*$. Assume therefore that $\mathbf{x}\not\in \mathrm{Li}(M),\mathbf{y}\not\in A,\mathbf{y}=\mathbf{z} b$. In this case, $\mathbf{x}*\mathbf{z}$ has been defined already by the induction hypothesis, and $\mathbf{x}$ is a proper prefix of $\mathbf{x}*\mathbf{z}$. Therefore, again by the induction hypothesis, $(\mathbf{x}*\mathbf{z})*\mathbf{x} b$ has been defined already. Therefore, $\mathbf{x}*\mathbf{y}$ is defined to be $(\mathbf{x}*\mathbf{z})*\mathbf{x} b$.
\begin{enumerate}[(i)]
\item $\mathbf{x}*\mathbf{z} b=(\mathbf{x}*\mathbf{z})*\mathbf{x} b$ has $b$ as its last element by the inductive hypothesis.

\item If $\mathbf{x}*\mathbf{z}\in \mathrm{Li}(M)$, then $\mathbf{x}*\mathbf{z} b=(\mathbf{x}*\mathbf{z})*\mathbf{x} b=\mathbf{x} b$.

\item If $\mathbf{x}*\mathbf{z}\in \mathrm{Li}(M)$, then $\mathbf{x}*\mathbf{z} b=(\mathbf{x}*\mathbf{z})*\mathbf{x} b$ has $\mathbf{x}*\mathbf{z}$ as a proper initial segment.

\item If $\mathbf{x}*\mathbf{z}\in \mathrm{Li}(M)$, then $\mathbf{x}*\mathbf{y}=\mathbf{x}*\mathbf{z} b=\mathbf{x} b$ which has $\mathbf{x}$ as a proper initial segment. If $\mathbf{x}*\mathbf{z}\in M\setminus \mathrm{Li}(M)$, then $\mathbf{x}*\mathbf{y}$ has $\mathbf{x}*\mathbf{z}$ as a proper initial segment. However,
$\mathbf{x}*\mathbf{z}$ has $\mathbf{x}$ as a proper initial segment.

\item We have $\mathbf{x}*\mathbf{y}=\mathbf{x}*\mathbf{z}b=(\mathbf{x}*\mathbf{z})*\mathbf{x}b$. The word, $\mathbf{x}*\mathbf{y}$ only contains
letters already in the words $\mathbf{x}*\mathbf{z}$  and $\mathbf{x}b$. However, the word $\mathbf{x}*\mathbf{z}$ only contains letters which are already present in $\mathbf{x}$ or $\mathbf{z}$. Therefore, $\mathbf{x}*\mathbf{y}$ only contains letters which are already present in
$\mathbf{x},\mathbf{z}$ or b.
\end{enumerate}
\end{proof}
\begin{prop}
If $\mathbf{z}\in M\setminus\mathrm{Li}(M),\mathbf{x}\in M$, then there is some $\mathbf{w}\in M\setminus\mathrm{Li}(M)$ such that $\mathbf{x}*\mathbf{z}b=\mathbf{w}b$ for all $b\in A$.

\label{t4jio}
\end{prop}
\begin{proof}
We shall prove this result by a descending induction on $|\mathbf{x}|$. The case where $\mathbf{x}\in\mathrm{Li}(M)$ follows from the fact that $\mathbf{x}*\mathbf{z}b=\mathbf{z}b$ for all $\mathbf{z}\in M\setminus\mathrm{Li}(M)$.

Now assume $\mathbf{x}\not\in\mathrm{Li}(M)$. Then by the induction hypothesis, there is some $\mathbf{w}$ such that $\mathbf{x}*\mathbf{z}b=(\mathbf{x}*\mathbf{z})*\mathbf{x} b=\mathbf{w}b$ for all $b\in B$.
\end{proof}
\begin{prop}
Let $M$ be a pre-multigenic Laver table. Suppose that $\mathbf{x},\mathbf{yz}\in M$
and $\mathbf{x}*\mathbf{y}\in\mathrm{Li}(M)$. Then
$\mathbf{x}*\mathbf{yz}=\mathbf{x}*\mathbf{z}$.
\label{G0Wuoeh2g2EHO2uo}
\end{prop}
\begin{proof}
We shall prove this result by induction on the length of $\mathbf{z}$. If $\mathbf{z}=a,|a|=1$, then
\[\mathbf{x}*\mathbf{y}a=(\mathbf{x}*\mathbf{y})*\mathbf{x}a=\mathbf{x}a=\mathbf{x}*a.\]

Now suppose that $\mathbf{z}=\mathbf{w}a$ and $|\mathbf{w}|>0$. Then
\[\mathbf{x}*\mathbf{yw}a=(\mathbf{x}*\mathbf{yw})*\mathbf{x}a
=(\mathbf{x}*\mathbf{w})*\mathbf{x}a=\mathbf{x}*\mathbf{w}a.\]
\end{proof}
\begin{cor}
\label{10W44eh2g2EHO2uo}
Let $M$ be a pre-multigenic Laver table. Suppose that $\mathbf{x},\mathbf{y}_{1}\ldots \mathbf{y}_{n}\mathbf{z}\in M$ and $\mathbf{x}*\mathbf{y}_{1},\ldots ,\mathbf{x}*\mathbf{y}_{n}\in\mathrm{Li}(M)$. Then $\mathbf{x}*\mathbf{y}_{1}\ldots \mathbf{y}_{n}\mathbf{z}
=\mathbf{x}*\mathbf{z}$.
\end{cor}
\begin{prop}
Let $M$ be a pre-multigenic Laver table over a finite alphabet $A$ and let $B\subseteq A$.
Let $M'=M\cap B^{+}$. Then whenever $\mathbf{x},\mathbf{y}\in M'$, we have $\mathbf{x}*^{M'}\mathbf{y}=\mathbf{x}*^{M}\mathbf{y}$.
\end{prop}
We shall now define pre-multigenic Laver tables over infinite alphabets.
\begin{defn}
Suppose now that $A$ is an infinite set, $L\subseteq A^{+}$ is downwards closed with respect to $\preceq$, $M=A\cup\{\mathbf{x}a\mid \mathbf{x}\in L,a\in A\}$, and $F=M\setminus L$. Then whenever $B,C$ are finite subsets of $A$ with
$B\subseteq C$, the algebra $(M\cap B^{+},*^{M\cap B^{+}})$ is a subalgebra of $(M\cap C^{+},*^{M\cap C^{+}})$. Define a binary operation $*^{M}$ on $M$ to be the union 
\[\bigcup_{B\in[A]^{<\omega}}*^{M\cap B^{+}}\]
where $[A]^{<\omega}$ denotes the collection of all finite subsets of $A$.  In other words, the algebra $(M,*^{M})$ is the direct limit of the algebras $(M\cap B^{+},*^{M\cap B^{+}})$ where $B$ ranges over all finite subsets of $A$. The algebra
$(M,*^{M})$ shall be called an \index{$\mathrm{crit}(x)$} infinite pre-multigenic Laver table.
\end{defn}
It is easy to see that theorems \ref{t4uh4aP} and \ref{t4jio} still hold for pre-multigenic Laver tables when $A$ is infinite.

\begin{thm}
Let $(M,*)$ be a pre-multigenic Laver table. Then the operation $*$ is self-distributive if and only if $\mathbf{x}*\mathbf{y}\in\mathrm{Li}(M)$ whenever $\mathbf{x}\in M,\mathbf{y}\in\mathrm{Li}(M)$.
\label{m43ti0}
\end{thm}
\begin{proof}
We only need to prove this result for when the alphabet $A$ is finite since the general case follows as an immediate corollary.

$\rightarrow$ Suppose that $(M,*)$ is self-distributive. Suppose now that $\mathbf{x}\in M$ and $\mathbf{y}\in\mathrm{Li}(M)$. Then we have
\[\mathbf{x}*\mathbf{y}=\mathbf{x}*(\mathbf{y}*\mathbf{y})=(\mathbf{x}*\mathbf{y})*(\mathbf{x}*\mathbf{y}).\]
Thus $\mathbf{x}*\mathbf{y}\in\mathrm{Li}(M)$ since otherwise $\mathbf{x}*\mathbf{y}$ would be a proper initial segment of $(\mathbf{x}*\mathbf{y})*(\mathbf{x}*\mathbf{y})$.

$\leftarrow$ We shall prove that $\mathbf{x}*(\mathbf{y}*\mathbf{z})=(\mathbf{x}*\mathbf{y})*(\mathbf{x}*\mathbf{z})$ by a triple induction; we shall proceed by a descending induction on $|\mathbf{x}|$, and for each $\mathbf{x}$ we shall use a descending induction on $|\mathbf{y}|$, and for each $\mathbf{y}$ we shall perform an ascending induction on $|\mathbf{z}|$.

\item[Case 1: $\mathbf{x}\in\mathrm{Li}(M)$.]

We have $\mathbf{x}*(\mathbf{y}*\mathbf{z})=\mathbf{y}*\mathbf{z}$ and $(\mathbf{x}*\mathbf{y})*(\mathbf{x}*\mathbf{z})=\mathbf{y}*\mathbf{z}$, so
\[\mathbf{x}*(\mathbf{y}*\mathbf{z})=(\mathbf{x}*\mathbf{y})*(\mathbf{x}*\mathbf{z}).\]

\item[Case 2: $\mathbf{x}\not\in\mathrm{Li}(M),\mathbf{y}\in\mathrm{Li}(M).$]

We have $\mathbf{x}*(\mathbf{y}*\mathbf{z})=\mathbf{x}*\mathbf{z}$.
We also have $\mathbf{x}*\mathbf{y}\in\mathrm{Li}(M)$, so $(\mathbf{x}*\mathbf{y})*(\mathbf{x}*\mathbf{z})=\mathbf{x}*\mathbf{z}$ as well. Therefore $\mathbf{x}*(\mathbf{y}*\mathbf{z})=(\mathbf{x}*\mathbf{y})*(\mathbf{x}*\mathbf{z})$ in this case.

\item[Case 3: $\mathbf{x}\not\in\mathrm{Li}(M),\mathbf{y}\not\in\mathrm{Li}(M),\mathbf{z}\in A$.]

Let $\mathbf{z}=c$. Then
\[\mathbf{x}*(\mathbf{y}*c)=\mathbf{x}*\mathbf{y} c=(\mathbf{x}*\mathbf{y})*\mathbf{x} c=(\mathbf{x}*\mathbf{y})*(\mathbf{x}*c).\]

\item[Case 4: $\mathbf{x}\not\in\mathrm{Li}(M),\mathbf{y}\not\in\mathrm{Li}(M),\mathbf{z}\not\in A$.]

Suppose that $\mathbf{z}=\mathbf{w} c$.

Then 
\begin{equation}
\label{42uvhovc4txsahuh922}
\mathbf{x}*(\mathbf{y}*\mathbf{z})=\mathbf{x}*(\mathbf{y}*\mathbf{w} c)=\mathbf{x}*((\mathbf{y}*\mathbf{w})*(\mathbf{y}*c)).
\end{equation}
Since $\mathbf{y}$ is a proper prefix of $\mathbf{y}*\mathbf{w}$, by the induction hypothesis, we have
\begin{equation}
\label{42uvhovc4txsahuh923}
\mathbf{x}*((\mathbf{y}*\mathbf{w})*(\mathbf{y}*c))=(\mathbf{x}*(\mathbf{y}*\mathbf{w}))*(\mathbf{x}*(\mathbf{y}*c)).
\end{equation}
Since $|\mathbf{w}|<|\mathbf{z}|$ and $|c|<|\mathbf{z}|$, we have
$\mathbf{x}*(\mathbf{y}*\mathbf{w})=(\mathbf{x}*\mathbf{y})*(\mathbf{x}*\mathbf{w})$ and $\mathbf{x}*(\mathbf{y}*c)=(\mathbf{x}*\mathbf{y})*(\mathbf{x}*c)$, so
\begin{equation}
\label{42uvhovc4txsahuh924}
(\mathbf{x}*(\mathbf{y}*\mathbf{w}))*(\mathbf{x}*(\mathbf{y}*c))
\end{equation}
\[=((\mathbf{x}*\mathbf{y})*(\mathbf{x}*\mathbf{w}))*((\mathbf{x}*\mathbf{y})*(\mathbf{x}*c))\]
\[=(\mathbf{x}*\mathbf{y})*((\mathbf{x}*\mathbf{w})*(\mathbf{x}*c))\]
\[=(\mathbf{x}*\mathbf{y})*(\mathbf{x}*\mathbf{w} c)=(\mathbf{x}*\mathbf{y})*(\mathbf{x}*\mathbf{z}).\]

Therefore by combining \ref{42uvhovc4txsahuh922},\ref{42uvhovc4txsahuh923},\ref{42uvhovc4txsahuh924}, we conclude that
$\mathbf{x}*(\mathbf{y}*\mathbf{z})=(\mathbf{x}*\mathbf{y})*(\mathbf{x}*\mathbf{z})$.
\end{proof}
\begin{defn}
A \index{multigenic Laver table}\emph{multigenic Laver table} is a left-distributive pre-multigenic Laver table.
\end{defn}
Most pre-multigenic Laver tables are not left-distributive since there are only 147 multigenic Laver tables
over the alphabet $\{0,1\}$ of cardinality at most 120.

\begin{thm}
\label{enbfwu92}
Suppose that $M$ is a pre-multigenic Laver table over an alphabet $A$. Let $X$ be an LD-system and suppose that
$x_{a}\in X$ whenever $a\in A$. Furthermore, suppose that whenever $a_{1}\ldots a_{n}\in\mathrm{Li}(M)$ and $a\in A$, we have
$x_{a}=x_{a_{1}}*\ldots*x_{a_{n}}*x_{a}$. Then define a mapping $\phi:M\rightarrow X$ by letting
$\phi(a_{1}\ldots a_{n})=x_{a_{1}}*\ldots *x_{a_{n}}$. Then $\phi$ is a homomorphism.
\end{thm}
\begin{proof}
We shall prove that $\phi(\mathbf{a}*\mathbf{b})=\phi(\mathbf{a})*\phi(\mathbf{b})$ by a double induction which is descending on
$\mathbf{a}$ and for each $\mathbf{a}$ we shall proceed by ascending induction on $\mathbf{b}$.
Let $\mathbf{a}=a_{1}\ldots a_{m},\mathbf{b}=b_{1}\ldots b_{n}$.

\item[Case 1: $\mathbf{a}\in\mathrm{Li}(M)$.]

Suppose that $\mathbf{a}\in\mathrm{Li}(M)$. Then 
\[\phi(\mathbf{a}*\mathbf{b})=\phi(\mathbf{b})=x_{b_{1}}*\ldots *x_{b_{n}}.\]

On the other hand, we have 
\[\phi(\mathbf{a})*\phi(\mathbf{b})=(x_{a_{1}}*\ldots*x_{a_{m}})*(x_{b_{1}}*\ldots*x_{b_{n}})\]
\[=(x_{a_{1}}*\ldots*x_{a_{m}}*x_{b_{1}})*\ldots*(x_{a_{1}}*\ldots*x_{a_{m}}*x_{b_{n}})=x_{b_{1}}*\ldots*x_{b_{n}}.\]
Therefore, we have
$\phi(\mathbf{a}*\mathbf{b})=\phi(\mathbf{a})*\phi(\mathbf{b})$ in this case.

\item[Case 2: $\mathbf{a}\not\in\mathrm{Li}(M),\mathbf{b}\in A$.]

Suppose that $\mathbf{b}=b$ for some $b\in A$. Then 
\[\phi(\mathbf{a}*b)=\phi(\mathbf{a} b)=x_{a_{1}}*\ldots*x_{a_{m}}*x_{b},\]
and
\[\phi(\mathbf{a})*\phi(b)=(x_{a_{1}}*\ldots*x_{a_{m}})*x_{b}=x_{a_{1}}*\ldots*x_{a_{m}}*x_{b}\]
as well (this case follows without using the induction hypothesis).

\item[Case 3: $\mathbf{a}\not\in\mathrm{Li}(M),\mathbf{b}\not\in A$.]

Suppose that $\mathbf{b}=\mathbf{c}b_{n}$. Thus, by applying the induction hypothesis several times,
we have

\[\phi(\mathbf{a}*\mathbf{b})=\phi(\mathbf{a}*\mathbf{c} b_{n})=\phi((\mathbf{a}*\mathbf{c})*\mathbf{a}b_{n})\]

\[=\phi(\mathbf{a}*\mathbf{c})*\phi(\mathbf{a}b_{n})\]

\[=(\phi(\mathbf{a})*\phi(\mathbf{c}))*(\phi(\mathbf{a})*\phi(b_{n}))\]

\[=\phi(\mathbf{a})*(\phi(\mathbf{c})*\phi(b_{n}))\]

\[=\phi(\mathbf{a})*\phi(\mathbf{c}b_{n})=\phi(\mathbf{a})*\phi(\mathbf{b}).\]
\end{proof}

\begin{defn}
An LD-system $X$ is said to be \index{Laver-like}\emph{Laver-like} if $\mathrm{Li}(X)$ is a left-ideal in $X$ and whenever
$x_{n}\in X$ for each $n\in\omega$ there is some $n$ where $x_{0}*\ldots*x_{n}\in \mathrm{Li}(X)$.
\end{defn}
In other words, the Laver-like LD-systems are precisely the LD-systems that satisfy the Laver-Steel Theorem
(theorems \ref{j3rti90} and \ref{235wfe}). Every Laver-like LD-system is permutative (see Proposition \ref{t4r4ghtj0qhyqfha}), and if $X$ is Laver-like, then $\mathrm{crit}[X]$ is well-ordered (otherwise, if $\mathrm{crit}(x_{0})>\mathrm{crit}(x_{1})>\ldots$, then $x_{0}*\ldots*x_{n}\not\in\mathrm{Li}(X)$ for all $n$).

\begin{defn}
An LD-system $X$ is said to be \index{locally Laver-like}\emph{locally Laver-like} if $\mathrm{Li}(X)$ is a left-ideal in $X$ and whenever $Y$ is a finitely generated subalgebra of $X$ and $x_{n}\in Y$ for each $n\in\omega$ there is some $n$ where $x_{0}*\ldots*x_{n}\in \mathrm{Li}(X)$.
\end{defn}
\begin{defn}
Suppose now that $X$ is an LD-system such that $\mathrm{Li}(X)$ is a left ideal in $X$. Suppose furthermore that $A$ is a set, $x_{a}\in X$ for each $a\in A$, and whenever $B$ is a finite subset of $A$ and $b_{n}\in B$ for all $n\in\omega$ there is some $n$ where $x_{b_{0}}*\ldots*x_{b_{n}}\in \mathrm{Li}(X)$. Then define \index{$\mathbf{M}(x_{a})_{a\in A}$}$\mathbf{M}(x_{a})_{a\in A}$ to be the set of all strings $a_{1}\ldots a_{n}$ such that if $1\leq m<n$, then $x_{a_{1}}*\ldots*x_{a_{m}}\not\in \mathrm{Li}(X)$. We shall write $\mathbf{M}((x_{a})_{a\in A},X)$ if the algebra $X$ is not clear from context. Take note that $\mathbf{M}(x_{a})_{a\in A}\cap B^{+}$ is finite whenever $B$ is a finite subset of $A$; otherwise if $\mathbf{M}(x_{a})_{a\in A}\cap B^{+}$ is infinite, then by Konig's lemma, there would be an infinite string $(a_{n})_{n\in\omega}\in B^{\omega}$ such that $a_{0}\ldots a_{n}\in\mathbf{M}(x_{a})_{a\in A}$ for all $n$ which forms a contradiction. The set of strings $\mathbf{M}(x_{a})_{a\in A}$ is the underlying set of a pre-multigenic Laver table.
\end{defn}

Recall that an algebra $X$ is locally finite if every finitely generated subalgebra of $X$ is finite.

\begin{cor}
Every locally Laver-like LD-system is a locally finite LD-system.
\label{4hw232s}
\end{cor}
\begin{proof}
Let $X$ be a locally Laver-like LD-system. Suppose $A$ is a finite set and $x_{a}\in X$ for $a\in A$. Then define
$\phi:\mathbf{M}((x_{a})_{a\in A})\rightarrow X$ to be the homomorphism where $\phi(a_{1}\ldots a_{n})=x_{a_{1}}*\ldots*x_{a_{n}}$ for each
$a_{1},\ldots ,a_{n}\in\mathbf{M}((x_{a})_{a\in A})$. Then $\mathbf{M}((x_{a})_{a\in A})$ is finite. Furthermore, $\phi[\mathbf{M}((x_{a})_{a\in A})]$ is the subalgebra of
$X$ generated by $\{x_{a}\mid a\in A\}$. Since each $\phi[\mathbf{M}((x_{a})_{a\in A})]$ is finite, we conclude that $X$ is locally finite.
\end{proof}

\begin{cor}
Suppose that $X$ is a finitely generated sub-LD-system or a finitely generated sub-LD-monoid of $\mathcal{E}_{\lambda}$. Then
\begin{enumerate}
\item $X/\equiv^{\gamma}$ is finite whenever $\gamma<\lambda$ is a limit ordinal, and

\item the set of ordinals $\{\mathrm{crit}(j)\mid j\in X\}$ has order type $\omega$.
\end{enumerate}
\end{cor}

\begin{thm}
Let $X$ be an LD-system such that $\mathrm{Li}(X)$ is a left-ideal. Suppose that $A$ is a set and $x_{a}\in X$ for 
$a\in A$. Suppose furthermore that whenever $B$ is a finite subset of $A$ and $b_{n}\in B$ for $n\in\omega$, there is some
$n$ where $x_{b_{0}}*\ldots *x_{b_{n}}\in\mathrm{Li}(X)$. Then $\mathbf{M}(x_{a})_{a\in A}$ is a multigenic Laver table.
\end{thm}
\begin{proof}
Let $M=\mathbf{M}(x_{a})_{a\in A}$. By Theorem \ref{enbfwu92}, the mapping $\phi:M\rightarrow X$ defined by $\phi(a_{1}\ldots a_{n})=x_{a_{1}}*\ldots*x_{a_{n}}$ is a homomorphism,
and it is easy to see that $\mathbf{a}\in\mathrm{Li}(M)$ if and only if $\phi(\mathbf{a})\in \mathrm{Li}(X)$. 

Now assume that $\mathbf{a}\in M,\mathbf{b}\in\mathrm{Li}(M)$. Then $\phi(\mathbf{b})\in \mathrm{Li}(X)$, so
$\phi(\mathbf{a}*\mathbf{b})=\phi(\mathbf{a})*\phi(\mathbf{b})\in \mathrm{Li}(X)$. Therefore, $\mathbf{a}*\mathbf{b}\in\mathrm{Li}(M)$ as well.
We conclude by Theorem \ref{m43ti0} that $M$ is self-distributive.
\end{proof}
\begin{rem}
It turns out that every multigenic Laver table is of the form $\mathbf{M}(x_{a})_{a\in A}$ since if
$M$ is a multigenic Laver table over an alphabet $A$, then since $A\subseteq M$, we have $M=\mathbf{M}((a)_{a\in A})$.
\end{rem}

\begin{prop}
Suppose that $X$ is a LD-system generated by a set $A$ such that $\mathrm{Li}(X)$ is a left-ideal. Then
\begin{enumerate}
\item $X$ is Laver-like if and only if whenever $a_{n}\in A$ for each $n\in\omega$, there is some $n$ where
$a_{0}*\ldots*a_{n}\in \mathrm{Li}(X)$, and

\item $X$ is locally Laver-like if and only if whenever $B$ is a finite subset of $A$ and $a_{n}\in B$ for each $n\in\omega,$ there is some $n$ where $a_{0}*\ldots*a_{n}\in \mathrm{Li}(X)$.
\label{4j0tuh9uowjmor}
\end{enumerate}
\end{prop}
\begin{proof}
\begin{enumerate}
\item The direction $\rightarrow$ is trivial, so we shall only prove the direction
$\leftarrow$. The algebra $\mathbf{M}((x_{a})_{a\in A})$ is a well-founded multigenic Laver table and hence a Laver-like algebra. Since the algebra $X$ is a quotient LD-system of $\mathbf{M}((x_{a})_{a\in A})$ and $\mathbf{M}((x_{a})_{a\in A})$ is a Laver-like algebra, we conclude that $X$ is also a Laver-like algebra.

\item As before, the direction $\rightarrow$ is trivial. The direction $\leftarrow$ follows from \ref{4j0tuh9uowjmor}.
\end{enumerate}
\end{proof}
\begin{prop}
Let $X$ be an LD-system. Then the following are equivalent.
\begin{enumerate}
\item $X$ is locally Laver-like.

\item $X$ is locally finite and permutative.

\item $X$ is permutative and $\langle a_{1},\ldots,a_{n}\rangle$ has only finitely many critical points whenever $a_{1},\ldots,a_{n}\in X$.
\end{enumerate}
\end{prop}

Let $X$ be a totally ordered set with greatest element $1$. Then the Heyting algebra $(X,\rightarrow)$ is a locally Laver-like LD-system.
\begin{prop}
Let $X$ be a totally ordered set with greatest element $1$. Let $(M,*)=\mathbf{M}((x)_{x\in X},(X,\rightarrow)$.
Then the operation $*$ can be completely described according to the following rules:
\begin{enumerate}
\item If $(x_{1},\ldots,x_{m})\in X^{+}$, then $(x_{1},\ldots,x_{m})\in M$ if and only if $x_{1}>x_{2}>\ldots>x_{m-1}$ and 
either $m=1$ or $x_{1}<1$.

\item $(x_{1},\ldots,x_{m})\in \mathrm{Li}(M)$ precisely when $m\geq 2,x_{m}\geq x_{m-1}$ or $m=1,x_{1}=1$.

\item If $(x_{1},\ldots,x_{m})\in \mathrm{Li}(M)$ and $(y_{1},\ldots,y_{n})\in M$, then 
\[(x_{1},\ldots,x_{m})*(y_{1},\ldots,y_{n})=(y_{1},\ldots,y_{n}).\]

\item If $(x_{1},\ldots,x_{m})\in M\setminus\mathrm{Li}(M)$ and $y_{n}<x_{m}$, then
\[(x_{1},\ldots,x_{m})*(y_{1},\ldots,y_{n})=(x_{1},\ldots,x_{m},y_{i},\ldots,y_{n})\]
where $i$ is the least natural number with
$y_{i}<x_{m}$.

\item If $(x_{1},\ldots,x_{m})\in M\setminus\mathrm{Li}(M)$, then
\[(x_{1},\ldots,x_{m})*(y_{1},\ldots,y_{n})=(x_{1},\ldots,x_{m},y_{n})\]
whenever $y_{n}\geq x_{m}$.
\end{enumerate}
\end{prop}
\begin{defn}
Let $n\in\omega$. Let $M=\{a,\ldots ,a^{2^{n}}\}$. Then $(M,*^{M})$ is a pre-multigenic Laver table.
Let $A_{n}=(\{1,\ldots,2^{n}\},*_{n})$ where $*_{n}$ is the unique operation such that $a^{x*^{n}y}=a^{x}*^{M}a^{y}$
for all $x,y\in\{1,\ldots,2^{n}\}$. We shall call the algebra $A_{n}$ the $n$-th classical Laver table.
\end{defn}
\begin{prop}
\begin{enumerate}
\item The classical Laver table $A_{n}$ is self-distributive for all natural numbers $n$.

\item If $A$ is a non-empty set, then $(A^{\leq 2^{n}})^{+}$ is a multigenic Laver table.
\end{enumerate}
\end{prop}
\begin{proof}
\begin{enumerate}
\item We shall prove this result by induction on $n$. The case where $n=0$ is trivial so assume $n>0$. Let $M=\{a^{1},\ldots ,a^{2^{n}}\}$ and let $N=\{a^{1},\ldots ,a^{2^{n-1}}\}$. Then by the induction hypothesis, $N$ is a multigenic Laver table. Define a mapping $\phi:M\rightarrow N$ by letting $\phi(a^{x})=a^{x}$ for $x\leq 2^{n-1}$ and $\phi(a^{x})=a^{x-2^{n-1}}$ for $x>2^{n-1}$. Then by Theorem \ref{enbfwu92}, the mapping $\phi$ is a homomorphism. Also, define $\phi:\{1,\ldots ,2^{n}\}\rightarrow\{1,\ldots ,2^{n-1}\}$ by letting $\phi(x)=(x)_{2^{n-1}}$

Let $x\in\{1,\ldots ,2^{n}\}$. Then let $y$ be the least natural number such that
$a^{\phi(x)}*^{N}a^{y}\in\mathrm{Li}(N)$. Then $y$ is a power of 2, since 
otherwise there are $p,q$ where $2^{n-1}=py+q$ and $1\leq q<y$, and therefore
\[a^{\phi(x)}*a^{2^{n-1}}=a^{\phi(x))}*a^{q}\not\in\mathrm{Li}(N)\]
by Corollary \ref{10W44eh2g2EHO2uo}, a contradiction. 

Now,
\[\phi(a^{x}*^{M}a^{y})=a^{\phi(x)}*^{N}a^{y}=a^{2^{n-1}}.\]
Therefore,
$a^{x}*^{M}a^{y}=a^{2^{n-1}}$ or $a^{x}*^{M}a^{y}=a^{2^{n}}$. If $a^{x}*^{M}a^{y}=a^{2^{n}}$, then our proof is complete. Now assume that $a^{x}*^{M}a^{y}=a^{2^{n-1}}$.

Then I claim that $a^{x}*^{M}a^{2y}=a^{2^{n}}$.

If $y<2^{n-1}$, then for $1\leq q<y$, we have
\[\phi(a^{x}*^{M}a^{y+q})=a^{\phi(x)}*^{N}a^{y+q}=a^{\phi(x)}*^{N}a^{q}\not\in
\mathrm{Li}(N).\]
If $y=2^{n-1}$, then 
\[\phi(a^{x}*^{M}a^{y+q})=a^{\phi(x)}*^{N}a^{q}\not\in\mathrm{Li}(N)\]
as well. Therefore, $a^{x}*^{M}a^{y+q}\not\in\mathrm{Li}(M)$ for $1\leq q<y$ in any case. Therefore, since $a^{x}*^{M}a^{s}\not\in\mathrm{Li}(M)$ whenever
$1\leq s<2y$, we conclude that $a^{x}*^{M}a^{y}$ is a proper substring of
$a^{x}*^{M}a^{2y}$. Therefore, since
$\phi(a^{x}*^{M}a^{2y})=a^{\phi(x)}*^{N}a^{\phi(2y)}\in\mathrm{Li}(N)$, we conclude that
$a^{x}*^{M}a^{2y}=a^{2^{n-1}}$ or $a^{x}*^{M}a^{2y}=a^{2^{n}}$. Since
$a^{x}*^{M}a^{y}$ is a proper prefix of $a^{x}*^{M}a^{2y}$, we deduce that
$a^{x}*^{M}a^{2y}=a^{2^{n}}$.  Now, since $a^{x}*^{M}a^{2y}=a^{2^{n}}$ and $2y$ is a factor of $2^{n}$, by Corollary \ref{10W44eh2g2EHO2uo}, we conclude that $a^{x}*^{M}a^{2^{n}}\in\mathrm{Li}(M)$. Therefore $M$ is self-distributive by Theorem \ref{m43ti0}.

\item It is easy to see that $(A^{\leq 2^{n}})^{+}=\mathbf{M}((1)_{a\in A},A_{n})$. Therefore, since each $A_{n}$ is Laver-like, the algebra $(A^{\leq 2^{n}})^{+}$ is also self-distributive.
\end{enumerate}
\end{proof}
Let $M$ be a multigenic Laver table and let $\alpha\in \mathrm{crit}[M]$. Then let $M\upharpoonright\alpha$ be the set of all strings $\mathbf{x}$ so that if $\mathbf{y}$ is a non-empty proper prefix of $\mathbf{x}$, then $\mathrm{crit}(\mathbf{y})<\alpha$. 

\begin{prop}
Suppose that $M$ is a finite multigenic Laver table and $\alpha\in\mathrm{crit}[M]\setminus\{\max(\mathrm{crit}[M])\}$. Let $\mathbf{x}\in M$ be a string such that $\mathbf{x}^{\sharp}(\beta)=\max(\mathrm{crit}[M])$ if and only if
$\beta\geq\alpha$.
\begin{enumerate}
\item $M\upharpoonright\alpha$ is a multigenic Laver table.

\item Define a mapping $\phi:M\upharpoonright\alpha\rightarrow M$ by letting $\phi(\mathbf{y})=\mathbf{x}*^{M}\mathbf{y}$. Then $\phi$ is an injective homomorphism.

\item The mapping $L:M\rightarrow M\upharpoonright\alpha$ defined by $L(a_{1}\ldots a_{n})=a_{1}*^{M\upharpoonright\alpha}\ldots*^{M\upharpoonright\alpha}a_{n}$ is a surjective homomorphism between algebras. \label{36ygqb389b5qv9tf93}

\item Define a mapping $\Delta:M\rightarrow M$ by $\Delta(\mathbf{y})=\mathbf{x}*\mathbf{y}$. Then $\Delta=\phi L$ and
\[\ker(\Delta)=\ker(L)=(\equiv^{\alpha}).\]
\end{enumerate}
\end{prop}
\begin{proof}
\begin{enumerate}
\item I claim that $M\upharpoonright\alpha=\mathbf{M}((\mathbf{x}a)_{a\in A},M)$.

Observe that $\mathbf{M}((\mathbf{x}a)_{a\in A},M)\subseteq M$; if $a_{1}\dots a_{n}\not\in M$, then $a_{1}\dots a_{m}\in\mathrm{Li}(M)$ for some $m<n$, so
\[\mathbf{x}a_{1}*^{M}\dots *^{M}\mathbf{x}a_{m}=\mathbf{x}*^{M}a_{1}\dots a_{m}\in\mathrm{Li}(M),\]
hence
\[a_{1}\dots a_{n}\not\in\mathbf{M}((\mathbf{x}a)_{a\in A},M).\]

Suppose now that $\mathbf{y}=a_{1}\dots a_{n}\in M$. Then $\mathbf{y}\in M\upharpoonright\alpha$ if and only if $\mathrm{crit}^{M}(\mathbf{z})<\alpha$ whenever $\varepsilon\neq\mathbf{z}\prec\mathbf{y}$ if and only if $\mathbf{x}*^{M}\mathbf{z}\not\in\mathrm{Li}(M)$ whenever $\varepsilon\neq\mathbf{z}\prec\mathbf{y}$ if and only if
\[\mathbf{x}a_{1}*^{M}\ldots*^{M}\mathbf{x}a_{m}\not\in\mathrm{Li}(M)\]
whenever $1\leq m<n$ if and only if
$\mathbf{y}\in\mathbf{M}((\mathbf{x}a)_{a\in A},M)$. Therefore, $M\upharpoonright\alpha=\mathbf{M}((\mathbf{x}a)_{a\in A},M)$, so $M\upharpoonright\alpha$
is a multigenic Laver table.

\item The mapping $\phi$ is a homomorphism by Theorem \ref{enbfwu92}. Suppose now that
$a_{1}\dots a_{m},b_{1}\dots b_{n}$ are distinct strings in $M\upharpoonright\alpha$. If 
$a_{1}\dots a_{m}$ is a prefix of $b_{1}\dots b_{n}$, then $\mathbf{x}*^{M}a_{1}\dots a_{m}$ is a proper prefix of $\mathbf{x}*^{M}b_{1}\dots b_{n}$. If $b_{1}\dots b_{n}$ is a proper prefix of $a_{1}\dots a_{m}$, then $\mathbf{x}*^{M}b_{1}\dots b_{n}$ is a proper prefix of
$\mathbf{x}*^{M}a_{1}\dots a_{m}$. If $a_{1}\dots a_{m},b_{1}\dots b_{n}$ are incomparable by $\preceq$, then let $r$ be the least natural number such that $a_{r}\neq b_{r}$. If $r=1$, then $\mathbf{x}a_{1}$ is incomparable with $\mathbf{x}b_{1}$, so
$\mathbf{x}*^{M}a_{1}\dots a_{m}$ must be incomparable with $\mathbf{x}*^{M}b_{1}\dots b_{n}$.
If $r>1$, then
\[\mathbf{x}*^{M}a_{1}\dots a_{r}=(\mathbf{x}\circ^{M}a_{1}\dots a_{r-1})a_{r}\]
is a prefix of $\mathbf{x}*^{M}a_{1}\dots a_{m}$ while
\[\mathbf{x}*^{M}a_{1}\dots a_{r-1}b_{r}
=(\mathbf{x}\circ^{M}a_{1}\dots a_{r-1})b_{r}\]
is a prefix of $\mathbf{x}*^{M}b_{1}\dots b_{m}$. We therefore conclude that
$\mathbf{x}*^{M}a_{1}\dots a_{m}$ and $\mathbf{x}*^{M}b_{1}\dots b_{n}$ are incomparable in this case. In any case, 
\[\phi(a_{1}\dots a_{m})=\mathbf{x}*^{M}a_{1}\dots a_{m}\neq\mathbf{x}*^{M}b_{1}\dots b_{n}=\phi(b_{1}\dots b_{n}),\]
so the mapping $\phi$ is injective.

\item Suppose $a_{1}\dots a_{n}\in\mathrm{Li}(M)$. Then

\[\phi L(a_{1}\dots a_{n})=\phi(a_{1}*^{M\upharpoonright\alpha}\dots *^{M\upharpoonright\alpha}a_{n})=\phi(a_{1})*^{M}\dots *^{M}\phi(a_{n})\]
\[=(x*^{M}a_{1})*^{M}\dots *^{M}(x*^{M}a_{n})=x*^{M}(a_{1}\dots a_{n})\in\mathrm{Li}(M).\]

Therefore, since $\phi$ is an injective homomorphism, we conclude that
\[L(a_{1}\dots a_{n})\in\mathrm{Li}(M\upharpoonright\alpha).\]
Therefore, the mapping $L$ is a homomorphism by Theorem \ref{enbfwu92}. The mapping $L$ is surjective since $L$ is a homomorphism, $L(a)=a$ for each $a\in A$, and $A$ generates $M\upharpoonright\alpha$.

\item By the proof of \ref{36ygqb389b5qv9tf93}, we know that $\Delta=\phi L$. Since $\phi$ is injective, $\ker(\Delta)=\ker(L)$. Now let $\mathbf{c}\in M,\mathrm{crit}^{M}(\mathbf{c})=\alpha,\mathbf{c}*^{M}\mathbf{c}\in\mathrm{Li}(M).$ Define a mapping
$\Delta':M\rightarrow M$ by letting $\Delta'(\mathbf{x})=\mathbf{c}*^{M}\mathbf{x}$. Then $\ker(L)=\ker(\Delta')=(\equiv^{\alpha}).$
\end{enumerate}
\end{proof}
\begin{thm}
Suppose that $M,N$ are finite multigenic Laver tables over the alphabet $A$ and $\phi:M\rightarrow N$ is a homomorphism such that $\phi(a)=a$ for each $a\in A$. Then $N=M\upharpoonright\alpha$ for some $\alpha\in\mathrm{crit}[M]$.
\end{thm}
\begin{proof}
Let $\alpha$ be the critical point in $\mathrm{crit}[M]$ such that
$\phi(\mathbf{x})\in\mathrm{Li}(N)$ if and only if $\mathrm{crit}^{M}(\mathbf{x})\geq\alpha$. Let $L:M\rightarrow M\upharpoonright\alpha$ be the mapping defined by
\[L(a_{1}\dots a_{n})=a_{1}*^{M\upharpoonright\alpha}\dots *^{M\upharpoonright\alpha}a_{n}.\]
Then since $\ker(L)=(\equiv^{\alpha})\subseteq\ker(\phi)$, there exists a unique homomorphism $\phi':M\upharpoonright\alpha\rightarrow N$ such that $\phi'L=\phi$. We shall now show that $\phi'$ is the identity function. First observe that
$\phi'(a)=a$ for $a\in A$ and that $\mathbf{x}\in\mathrm{Li}(M\upharpoonright\alpha)$ if and only if $\phi(\mathbf{x})\in\mathrm{Li}(M)$. We shall show that $\phi'(\mathbf{x})=\mathbf{x}$ by induction on $|\mathbf{x}|$. The case when $|\mathbf{x}|=1$ has already been established, so assume $\mathbf{x}=\mathbf{y}a,|\mathbf{x}|>1$. Then
\[\phi'(\mathbf{x})=\phi'(\mathbf{y}a)=\phi'(\mathbf{y}*^{M\upharpoonright a}a)
=\phi'(\mathbf{y})*^{N}\phi'(a)=\mathbf{y}*^{N}a=\mathbf{y}a=\mathbf{x}.\]
Therefore, $\phi'$ is the identity function. Thus, since $\phi'$ is the identity function, $N=M\upharpoonright\alpha$.
\end{proof}

The following theorem shows how one can extend a multigenic Laver table to a larger multigenic Laver table with higher critical points.

\begin{thm}
Let $M$ be finite a multigenic Laver table and suppose $\mathrm{crit}[X]\setminus\{\max(\mathrm{crit}[X])\}$ has a maximum $\alpha$. Let $N=M\upharpoonright\alpha$. Let $J=\mathrm{Li}(N)\setminus\mathrm{Li}(M)$. Then $M=N\cup\{\mathbf{xy}\mid\mathbf{x}\in J,\mathbf{y}\in N\}$. In particular, we have $|M|=|N|\cdot(|J|+1)$.
\end{thm}

\begin{cor}
Suppose that $M,N$ are finite multigenic Laver tables over an alphabet $A$ and
$\phi:M\rightarrow N$ is a homomorphism with $\phi(a)=a$ for each $a\in A$. Then
$\frac{|M|}{|N|}$ is a natural number.
\end{cor}
\begin{cor}
Let $M$ be a finite multigenic Laver table over an alphabet $A$. Let $|\mathrm{crit}[M]|=n+1$.
Then there are $x_{0},\ldots,x_{n}$ where $|M|=x_{0}\cdot\ldots\cdot x_{n}$, $x_{0}=|A|$, and for each
$i$, we have $2\leq x_{i+1}\leq x_{0}\cdot\ldots\cdot x_{i}(1-\frac{1}{x_{0}})+2$.
\end{cor}

The following result gives a duality between multigenic Laver tables and critically simple Laver-like LD-systems. Suppose that $M$ is a multigenic Laver table over an alphabet $A$. Then let $\Delta(M)=(([a]_{\simeq_{cmx}})_{a\in A},M/\simeq_{cmx})$ where $\simeq$ is the critically maximum congruence on $M$.

\begin{thm}
\begin{enumerate}
\item If $M$ is a multigenic Laver table over an alphabet $A$, then $\mathbf{M}(\Delta(M))=M$. 

\item If $X$ is a locally Laver-like LD-system, $x_{a}\in X$ for $a\in A$, and $(x_{a})_{a\in A}$ generates $X$, then there is a unique homomorphism $\iota:X\rightarrow\Delta(\mathbf{M}(X))$ such that $\iota(x_{a})=[a]_{\mathrm{cmx}}$ for each $a\in A$. The mapping $\iota$ is surjective, and
$\iota$ is an isomorphism precisely when $X$ is critically simple.
\end{enumerate}
\end{thm}

\section{Generalized distributivity and partially pre-endomorphic Laver tables}

In this chapter, we shall generalize the notion of a multigenic Laver table further to the universal algebraic setting of
partially endomorphic Laver tables and twistedly endomorphic Laver tables. The partially endomorphic Laver tables could have an arbitrary number of fundamental operations of arbitrary arity. Furthermore, only some of the operations in partially endomorphic Laver tables are required to be self-distributive. For example, if none of the operations in a partially endomorphic Laver table are declared to be self-distributive, then the partially endomorphic Laver table is simply a term algebra. The twistedly endomorphic Laver tables use associative operations to generalize the $n$-ary self-distributive identity in such a way that the twistedly endomorphic Laver tables still satisfy this identity. In the next chapter, we shall generalize the notions of a permutative LD-system and a Laver-like LD-systems to a universal algebraic setting, and we shall use these algebras to construct partially endomorphic Laver tables and twistedly endomorphic Laver tables.

We shall now motivate the notion of these partially endomorphic Laver tables by extending the notion of self-distributivity to algebras
of arbitrary type.

\begin{defn}
A \index{multi-LD-system}\emph{multi-LD-system} is an algebra $(X,(*_{a})_{a\in A})$ where $*_{a}$ is a binary operation on $X$ and where
$x*_{a}(y*_{b}z)=(x*_{a}y)*_{b}(x*_{a}z)$ whenever $a,b\in A,x,y,z\in X$. 
\end{defn}
For example, every distributive lattice $(X,\wedge,\vee)$ is a multi-LD-system.
\begin{defn}
Suppose that $(X,(*_{a})_{a\in A})$ is an algebra where each $*_{a}$ is a binary operation on $X$. Then define the \index{hull}\emph{hull} of $(X,(*_{a})_{a\in A})$ to be the algebra $(X\times A,*)$ where $*$ is defined by letting $(x,a)*(y,b)=(x*_{a}y,b)$ for $x,y\in X,a,b\in A$. 
\end{defn}
An algebra $(X,(*_{a})_{a\in A})$ is a multi-LD-system if and only if its hull $(X\times A,*)$ is an LD-system.
Furthermore, if $(X,(*_{a})_{a\in A})$ is a multi-LD-system freely generated by a set $F$, then
$(X\times A,*)$ is an LD-system freely generated by the set $F\times A$.

The notion of an LD-system can also be generalized to \index{ternary self-distributivity}ternary self-distributivity. Suppose that $X$ is a set and $t$ is a ternary operation on the set $X$. Then we shall say that $t$ is \emph{left-distributive} if $t(a,b,t(x,y,z))=t(t(a,b,x),t(a,b,y),t(a,b,z))$
for each $a,b,x,y,z\in X$. Suppose that $X$ is a set and $t$ is a ternary operation on $X$. Then for each $a,b\in X$, define a mapping
$L_{a,b}:X\rightarrow X$ by letting $L_{t,a,b}(x)=t(a,b,x)$ for all $x\in X$. Then $t$ is left-distributive if and only if
for all $a,b\in X$, the mapping $L_{t,a,b}$ is an endomorphism of $(X,t)$. Therefore the ternary self-distributive algebras
are motivated as being algebras with many inner endomorphisms. Ternary self-distributive algebras (in particular, ternary racks and quandles) have been studied in the papers \cite{EGM} and \cite{B}. Furthermore, multi-LD-systems were studied in \cite{D97}.

\begin{exam}
Suppose that $X$ is a set and $m$ is a ternary operation on $X$. For $a\in X$, define $*_{a}$ by
$x*_{a}y=m(a,x,y)$. Then $(X,m)$ is a \index{median algebra}median algebra if 
\[m(x,x,y)=x,m(x,y,z)=m(y,z,x)=m(x,z,y)\]
for each $x,y,z\in X$ and $(X,(*_{a})_{a\in A})$ is a multi-LD-system \cite{B83}. If $(X,m)$ is a median algebra, then the operation $m$ is a ternary
self-distributive operation. 

If $X$ is a distributive lattice and $m$ is defined by 
\[m(x,y,z)=(x\wedge y)\vee
(x\wedge z)\vee(y\wedge z),\]
then $(X,m)$ is a median algebra and 
\[m(x,y,z)=(x\vee y)\wedge(x\vee z)\wedge(y\vee z)\]
as well.

If $(X,E)$ is an unrooted tree, then whenever $x,y\in X$ there is a unique path $[x,y]$ from $x$ to $y$. Furthermore, if $x,y,z\in X$, then $[x,y]\cap[x,z]\cap[y,z]$ has a unique element. Therefore, define an operation $m$ on $X$ by letting $\{m(x,y,z)\}=[x,y]\cap[x,z]\cap[y,z]$. Then $(X,m)$ is a median algebra.
\end{exam}

As with multi-LD-systems, one can take the hull of a ternary self-distributive algebra. Suppose that $X$ is a set and $t$ is a ternary operation on $X$. Define a binary operation $*$ on $X^{2}$ by letting $(a,b)*(x,y)=(t(a,b,x),t(a,b,y))$. We shall call
$(X^{2},*)$ the hull of $(X,t)$. Then the operation $t$ is self-distributive if and only if $(X^{2},*)$ is an LD-system.
If $(X,t)$ is a ternary self-distributive algebra freely generated by a single element, then the hull
$(X^{2},*)$ contains a free left-distributive subalgebra on infinitely many generators.

\begin{defn}
Suppose that $\mathcal{F}$ is a set of function symbols, $E\subseteq\mathcal{F}$, $F=\mathcal{F}\setminus E$, and
$\mathcal{X}=(X,(t^{\mathcal{X}})_{t\in E},(f^{\mathcal{X}})_{f\in F})$. Then whenever $t\in E$ is $n+1$-ary and
$a_{1},\ldots,a_{n}\in X$, define a function $L_{t,a_{1},\ldots,a_{n}}:X\rightarrow X$ by letting
$L_{t,a_{1},\ldots,a_{n}}(x)=t(a_{1},\ldots,a_{n},x)$. Then we shall say that $\mathcal{X}$ is \index{partially endomorphic algebra}\emph{partially endomorphic} of type
$(E,F)$ if and only if whenever $t\in E$ is $n+1$-ary and $a_{1},\ldots,a_{n}\in X$, the mapping
$L_{t,a_{1},\ldots,a_{n}}$ is an endomorphism from $\mathcal{X}$ to $\mathcal{X}$. If $F=\emptyset$, then we shall call $\mathcal{X}$ an \index{endomorphic algebra}\emph{endomorphic algebra}.
\end{defn}
In other words, $\mathcal{X}$ is a partially endomorphic algebra if and only if whenever $t\in E$ is $n+1$-ary and
$g\in\mathcal{F}$ is $m$-ary, we have
\[t(a_{1},\ldots,a_{n},g(x_{1},\ldots,x_{m}))=g(t(a_{1},\ldots,a_{n},x_{1}),\ldots,t(a_{1},\ldots,a_{n},x_{m})).\]

\begin{defn}
If $\mathcal{X}=(X,(t^{\mathcal{X}})_{t\in E})$ is an algebra where each $t\in E$ has arity $n_{t}+1$, then define the \emph{hull}\index{hull} \index{$\Gamma(\mathcal{X})$}$\Gamma(\mathcal{X})$ to be the algebra $((\bigcup_{t\in E}\{t\}\times X^{n_{t}},*)$ where $*$ is defined by
\[(s,x_{1},\ldots,x_{n_{s}})*(t,y_{1},\ldots,y_{n_{t}})=(t,s(x_{1},\ldots,x_{n_{s}},y_{1}),\ldots,s(x_{1},\ldots,x_{n_{s}},y_{n_{t}})).\]
\end{defn}
The algebra $\mathcal{X}$ is endomorphic if and only if the hull $\Gamma(\mathcal{X})$ is an LD-system.

We shall now proceed to define the partially pre-endomorphic Laver tables.
\begin{defn}
If $(X_{1},\leq),\ldots,(X_{n},\leq)$ are posets, then the lexicographic ordering $\leq$ on $X_{1}\times\ldots\times X_{n}$ is the partial ordering where $(x_{1},\ldots,x_{n})<(y_{1},\ldots,y_{n})$ if and only if $(x_{1},\ldots,x_{n})\neq(y_{1},\ldots,y_{n})$ and $x_{i}<y_{i}$ where $i$ is the least natural number such that $x_{i}\neq y_{i}$.
\end{defn}
\begin{defn}
If $\mathcal{F}$ is a set of function symbols, and $X$ is a set of variables, then we shall write \index{$\mathbf{T}_{\mathcal{F}}(X)$}$\mathbf{T}_{\mathcal{F}}(X)$ for the set of all terms over the language $\mathcal{F}$ with variables in $X$.
\end{defn}
\begin{defn}
\label{3hqw3f6l8icg}
Now suppose that $\mathcal{F}$ is a set of function symbols, $X$ is a set of variables, $E\subseteq\mathcal{F}$, $F=\mathcal{F}\setminus E$, and each $f\in\mathcal{F}$ has arity $n_{f}$. We shall use lower-case fractur letters $\mathfrak{f},\mathfrak{g}$ to denote function symbols in $F$ while the letters $f,g,h$ shall denote function symbols in $E$. Suppose that $f_{x}$ is a function symbol of arity $n_{f}$ whenever $f\in E$. Let $\mathcal{G}=\{f_{x}\mid f\in E,x\in X\}\cup F$.

Give the set $\mathbf{T}_{\mathcal{G}}(X)$ the partial ordering $\preceq$ where we set $u\preceq v$ if $u$ is a subterm of
$v$.

Let $L\subseteq\mathbf{T}_{\mathcal{G}}(X)$ be a subset satisfying the following conditions:
\begin{enumerate}
\item $L$ is a downwards closed subset of $(\mathbf{T}_{\mathcal{G}}(X),\preceq)$.

\item $X\subseteq L$.

\item $f_{x}(\ell_{1},\ldots,\ell_{n})\in L$ if and only if $f_{y}(\ell_{1},\ldots,\ell_{n})\in L$.

\item If $\ell_{1},\ldots,\ell_{n}\in L$, then $\mathfrak{g}(\ell_{1},\ldots,\ell_{n})\in L$.
\end{enumerate}
Let $\Omega=\{(f,\ell_{1},\ldots,\ell_{n_{f}})\mid f\in E,\ell_{1},\ldots,\ell_{n_{f}}\in L\}$.

We shall now use transfinite induction to assign some of the elements in $\Omega$ an ordinal which we shall call
the rank, and during this induction process, we shall construct a partial function $g^{\sharp}:L^{n_{g}+1}\rightarrow L$ for each $g\in E$.
The element $g^{\sharp}(\ell_{1},\ldots ,\ell_{n},\ell)$ will be defined whenever $(g,\ell_{1},\ldots ,\ell_{n})$ has a rank.

If $f_{x}(\ell_{1},\ldots,\ell_{n_{f}})\not\in L$, then we shall say that the rank of $(f,\ell_{1},\ldots,\ell_{n_{f}})$ is $0$.
If the rank of $(f,\ell_{1},\ldots,\ell_{n_{f}})$ is $0$, then define $f^{\sharp}(\ell_{1},\ldots,\ell_{n_{f}},\ell)=\ell$.

Now suppose that $\alpha>0$ is an ordinal. Suppose now that $(g,\ell_{1},\ldots,\ell_{n})$ has not been assigned a rank yet. Then
define a partial function $Q_{(g,\ell_{1},\ldots,\ell_{n}),\alpha}:L\rightarrow L$ by induction on $\ell\in L$ by letting
\begin{enumerate}[(i)]
\item $Q_{(g,\ell_{1},\ldots,\ell_{n}),\alpha}(x)=g_{x}(\ell_{1},\ldots,\ell_{n})$ whenever $x\in X$,

\item if $\ell=\mathfrak{f}(u_{1},\ldots,u_{m})$, then
\[Q_{(g,\ell_{1},\ldots,\ell_{n}),\alpha}(\ell)=\mathfrak{f}(Q_{(g,\ell_{1},\ldots,\ell_{n}),\alpha}(u_{1}),\ldots,Q_{(g,\ell_{1},\ldots,\ell_{n}),\alpha}(u_{m}))\]
whenever
$Q_{(g,\ell_{1},\ldots,\ell_{n}),\alpha}(u_{1}),\ldots,Q_{(g,\ell_{1},\ldots,\ell_{n}),\alpha}(u_{n})$ have been defined already. Otherwise, we shall leave $Q_{(g,\ell_{1},\ldots,\ell_{n}),\alpha}(\ell)$ undefined, and

\item if $\ell=f_{x}(u_{1},\ldots,u_{m})$, then define
\[Q_{(g,\ell_{1},\ldots,\ell_{n}),\alpha}(\ell)=f^{\sharp}(Q_{(g,\ell_{1},\ldots,\ell_{n}),\alpha}(u_{1}),\ldots,Q_{(g,\ell_{1},\ldots,\ell_{n}),\alpha}(u_{n}),g_{x}(\ell_{1},\ldots,\ell_{n}))\]
whenever $Q_{(g,\ell_{1},\ldots,\ell_{n}),\alpha}(u_{1}),\ldots,Q_{(g,\ell_{1},\ldots,\ell_{n}),\alpha}(u_{n})$ have been defined already and $(f,Q_{(g,\ell_{1},\ldots,\ell_{n}),\alpha}(u_{1}),\ldots,Q_{(g,\ell_{1},\ldots,\ell_{n}),\alpha}(u_{n}))$ has rank less than $\alpha$, and we shall leave $Q_{(g,\ell_{1},\ldots,\ell_{n}),\alpha}(\ell)$ undefined otherwise.
\end{enumerate}

We say that the rank of $(g,\ell_{1},\ldots,\ell_{n})$ is $\alpha$ if $Q_{(g,\ell_{1},\ldots,\ell_{n}),\alpha}$ is a total function.
If $(g,\ell_{1},\ldots,\ell_{n})$ has rank $\alpha$, then define
\[g^{\sharp}(\ell_{1},\ldots,\ell_{n},\ell)=Q_{(g,\ell_{1},\ldots,\ell_{n}),\alpha}(\ell)\]
for each $\ell\in L$.

If for each $(f,\ell_{1},\ldots,\ell_{n})\in\Omega$, there is some ordinal $\alpha$ such that
\[\mathrm{rank}((f,\ell_{1},\ldots,\ell_{n}))=\alpha,\]
then we shall say that the set $L$ is \index{operable}\emph{operable}.
If $L$ is operable, then the operations $g^{\sharp}$ are all total operations.
If $L$ is operable, then we shall call $(L,(g^{\sharp})_{g\in E},F)$ a \index{partially pre-endomorphic Laver table}\emph{well-founded partially pre-endomorphic Laver table}.
\end{defn}
We shall write $\mathbf{On}$\index{\mathbf{On}} for the class of all ordinals.
\begin{lem}
Suppose that $L\subseteq\mathbf{T}_{\mathcal{G}}[X]$ is downwards closed with respect to $\preceq$. Furthermore, suppose that
$\nabla:\Omega\rightarrow\mathbf{On}$ be function such that $\nabla(f,\ell_{1},\ldots,\ell_{n})=0$ if and only if
$f_{x}(\ell_{1},\ldots,\ell_{n})\not\in L$. Then suppose that for each $f\in E$, $f^{\bullet}:L^{n_{f}+1}\rightarrow L$ is a function such that
\begin{enumerate}
\item $f^{\bullet}(\ell_{1},\ldots,\ell_{n},\ell)=\ell$ whenever $\nabla(f,\ell_{1},\ldots,\ell_{n})=0$,

\item If $\nabla(f,\ell_{1},\ldots ,\ell_{n})>0$, then $f^{\bullet}(\ell_{1},\ldots ,\ell_{n},x)=f_{x}(\ell_{1},\ldots ,\ell_{n})$,

\item if $\mathfrak{g}\in F$ and $u_{1},\ldots,u_{m}\in L$, then
\[f^{\bullet}(\ell_{1},\ldots,\ell_{n},\mathfrak{g}(u_{1},\ldots,u_{m}))=
\mathfrak{g}(f^{\bullet}(\ell_{1},\ldots,\ell_{n},u_{1}),\ldots,f^{\bullet}(\ell_{1},\ldots,\ell_{n},u_{m})),\]
and

\item if $g\in E,u_{1},\ldots,u_{m}\in L,x\in X$, then
\[f^{\bullet}(\ell_{1},\ldots,\ell_{n},g_{x}(u_{1},\ldots,u_{m}))\]
\[=g^{\bullet}(f^{\bullet}(\ell_{1},\ldots,\ell_{n},u_{1}),\ldots,f^{\bullet}(\ell_{1},\ldots,\ell_{n},u_{m}),
f_{x}(\ell_{1},\ldots,\ell_{n}))\]
and $\nabla(f,\ell_{1},\ldots,\ell_{n})>\nabla(g,f^{\bullet}(\ell_{1},\ldots,\ell_{n},u_{1}),\ldots,f^{\bullet}(\ell_{1},\ldots,\ell_{n},u_{m}))$.
\end{enumerate}
Then $L$ is operable, $f^{\bullet}=f^{\sharp}$ for each $f\in E$, and
\[\mathrm{rank}(f,\ell_{1},\ldots,\ell_{n})\leq\nabla(f,\ell_{1},\ldots,\ell_{n})\]
whenever $f\in E,\ell_{1},\ldots,\ell_{n}\in L$.
\end{lem}
\begin{exam}
Every pre-multigenic Laver table is isomorphic a pre-endomorphic Laver table with one fundamental operation of arity 2.
Let $X$ be a set. Suppose that $M$ is a multigenic Laver table over $X$. Then let $j$ be a unary function symbol. Let
$L$ be the set of all terms of the form $j_{x_{m}}\circ\ldots\circ j_{x_{2}}(x_{1})$ such that the string
$x_{1}x_{2}\ldots x_{m}$ belongs to $M$. Define a mapping $\phi:(M,*)\rightarrow(L,j^{\sharp})$ by letting
$\phi(x_{1}\ldots x_{m})=j_{x_{m}}\circ\ldots\circ j_{x_{2}}(x_{1})$. Then the mapping $\phi$ is an isomorphism of algebras.
This construction gives a one-to-one correspondence between the pre-multigenic Laver tables and the pre-endomorphic Laver tables with one fundamental operation of arity 2.
\label{4uok9et3gnjd}
\end{exam}
\begin{defn}
If $t$ is a term in $\mathbf{T}_{\mathcal{G}}(X)$, then we shall say a variable $x$ is \index{present}\emph{present} in $t$ if
$x$ is a subterm of $t$. We shall say that a variable $x$ is \index{operationally present}\emph{operationally present} in $t$ if some
term of the form $f_{x}(\ell_{1},\ldots,\ell_{n})$ is a subterm of $t$. For example, $y$ is present but not operationally present
in $f_{x}(y)$, but $x$ is operationally present but not present in $f_{x}(y)$.
We shall say that some function symbol $\mathfrak{g}\in\mathcal{G}$ is \index{present}\emph{present} in a term $t$ if
$\mathfrak{g}(u_{1},\ldots,u_{m})$ is a subterm of $t$ for some $u_{1},\ldots,u_{m}$. We shall say that a function symbol
$f\in E$ is present in $t$ if some $f_{x}$ is present in $t$. For example, $f$ is present in $f_{x}(y)$.
\end{defn}
\begin{prop}
Suppose that $L$ is operable and $f\in E,\ell_{1},\ldots,\ell_{n},\ell\in L$.
\begin{enumerate}
\item If $g\in E$ and $g$ is present in $f^{\sharp}(\ell_{1},\ldots,\ell_{n},\ell)$, then $g=f$ or $g$ is present in
$\ell_{1},\ldots,\ell_{n},\ell$.

\item If $\mathfrak{g}\in F$ and $\mathfrak{g}$ is present in $f^{\sharp}(\ell_{1},\ldots,\ell_{n},\ell)$, then
$\mathfrak{g}$ is present in $\ell_{1},\ldots,\ell_{n},\ell$.

\item If $x$ is present or functionally present in $f^{\sharp}(\ell_{1},\ldots,\ell_{n},\ell)$, then
$x$ is present or functionally present in $\ell_{1},\ldots,\ell_{n},\ell$.
\end{enumerate}
\end{prop}
\begin{defn}
Suppose now that $f\in E$ is $n$-ary, $\ell_{1},\ldots,\ell_{n}$ are terms in $\mathbf{T}_{\mathcal{G}}(X)$, and
$t(x_{1},\ldots,x_{l})$ is a term in $\mathbf{T}_{\mathcal{G}}(X)$. Then we shall write
\index{$t(x_{1},\ldots,x_{l})\otimes(f,\ell_{1},\ldots,\ell_{n})$}
$t(x_{1},\ldots,x_{l})\otimes(f,\ell_{1},\ldots,\ell_{n})$ for the term
$t(f_{x_{1}}(\ell_{1},\ldots,\ell_{n}),\ldots,f_{x_{l}}(\ell_{1},\ldots,\ell_{n}))$.
\end{defn}
\begin{prop}
If $\mathrm{rank}(f,\ell_{1},\ldots,\ell_{n})>0$ and $\ell\in L$, then there is some necessarily unique term $t$ such that
\[f^{\sharp}(\ell_{1},\ldots,\ell_{n},\ell)=t\otimes(f,\ell_{1},\ldots,\ell_{n}).\]
\label{t4gh2ui04gthn}
\end{prop}
\begin{proof}
We shall prove this result by induction on $(\mathrm{rank}(f,\ell_{1},\ldots,\ell_{n}),\ell)\in\mathbf{On}\setminus\{0\}\times L$ in several different cases.

\item[Case I: $\ell=x\in X$.]

In this case, since $\mathrm{rank}(f,\ell_{1},\ldots,\ell_{n})\neq 0$, we have
\[f^{\sharp}(\ell_{1},\ldots,\ell_{n},\ell)=f_{x}(\ell_{1},\ldots,\ell_{n}).\]

\item[Case II: $\ell=\mathfrak{g}(u_{1},\ldots,u_{m})$.]

By the induction hypothesis, there are term functions $\Gamma_{1},\ldots,\Gamma_{m}$ along with $x_{1},\ldots,x_{l}$ so that
for $1\leq i\leq m$ we have
\[f^{\sharp}(\ell_{1},\ldots,\ell_{n},u_{i})=\Gamma_{i}(f_{x_{1}}(\ell_{1},\ldots,\ell_{n}),\ldots,f_{x_{l}}(\ell_{1},\ldots,\ell_{n})).\]

We have 
\[f^{\sharp}(\ell_{1},\ldots,\ell_{n},\mathfrak{g}(u_{1},\ldots,u_{m}))\]
\[=\mathfrak{g}(f^{\sharp}(\ell_{1},\ldots,\ell_{n},u_{1}),\ldots,f^{\sharp}(\ell_{1},\ldots,\ell_{n},u_{m}))\]
\[=\mathfrak{g}(\Gamma_{1}(f_{x_{1}}(\ell_{1},\ldots,\ell_{n}),\ldots,f_{x_{l}}(\ell_{1},\ldots,\ell_{n})),\ldots,
\Gamma_{m}(f_{x_{1}}(\ell_{1},\ldots,\ell_{n}),\ldots,f_{x_{l}}(\ell_{1},\ldots,\ell_{n})))\]

\item[Case III: $\ell=g_{x}(u_{1},\ldots,u_{m})$.]

\item[Case IIIA: $\mathrm{rank}(g,f^{\sharp}(\ell_{1},\ldots,\ell_{n},u_{1}),\ldots,f^{\sharp}(\ell_{1},\ldots,\ell_{n},u_{m}))=0$.]

\[f^{\sharp}(\ell_{1},\ldots,\ell_{n},\ell)\]
\[=g^{\sharp}(f^{\sharp}(\ell_{1},\ldots,\ell_{n},u_{1}),\ldots,f^{\sharp}(\ell_{1},\ldots,\ell_{n},u_{m}),f_{x}(\ell_{1},\ldots,\ell_{n}))\]
\[=f_{x}(\ell_{1},\ldots,\ell_{n}).\]

\item[Case IIIB: $\mathrm{rank}(g,f^{\sharp}(\ell_{1},\ldots,\ell_{n},u_{1}),\ldots,f^{\sharp}(\ell_{1},\ldots,\ell_{n},u_{m}))>0$.]

In this case, by the induction hypothesis, there are term function $\Gamma_{1},\ldots,\Gamma_{m}$ and $x_{1},\ldots,x_{k}$ such that
\[f^{\sharp}(\ell_{1},\ldots,\ell_{n},u_{i})=\Gamma_{i}(f_{x_{1}}(\ell_{1},\ldots,\ell_{n}),\ldots,f_{x_{k}}(\ell_{1},\ldots,\ell_{n}))\]
for $1\leq i\leq m$. 

Furthermore, again by the induction hypothesis, there is some term function $\Gamma$ such that

\[g^{\sharp}(f^{\sharp}(\ell_{1},\ldots,\ell_{n},u_{1}),\ldots,f^{\sharp}(\ell_{1},\ldots,\ell_{n},u_{m}),f_{x}(\ell_{1},\ldots,\ell_{n}))\]
\[=\Gamma(f^{\sharp}(\ell_{1},\ldots,\ell_{n},u_{1}),\ldots,f^{\sharp}(\ell_{1},\ldots,\ell_{n},u_{m})).\]

We therefore conclude that

\[f^{\sharp}(\ell_{1},\ldots,\ell_{n},\ell)=f^{\sharp}(\ell_{1},\ldots,\ell_{n},g_{x}(u_{1},\ldots,u_{m}))\]

\[=g^{\sharp}(f^{\sharp}(\ell_{1},\ldots,\ell_{n},u_{1}),\ldots,f^{\sharp}(\ell_{1},\ldots,\ell_{n},u_{m}),f_{x}(\ell_{1},\ldots,\ell_{n}))\]

\[=\Gamma(f^{\sharp}(\ell_{1},\ldots,\ell_{n},u_{1}),\ldots,f^{\sharp}(\ell_{1},\ldots,\ell_{n},u_{m}))\]

\[=\Gamma(\Gamma_{1}(f_{x_{1}}(\ell_{1},\ldots,\ell_{n}),\ldots,f_{x_{k}}(\ell_{1},\ldots,\ell_{n})),\ldots\]
\[,\Gamma_{m}(f_{x_{1}}(\ell_{1},\ldots,\ell_{n}),\ldots,f_{x_{k}}(\ell_{1},\ldots,\ell_{n}))).\]
\end{proof}

\begin{cor}
Suppose that $L$ is a partially endomorphic Laver table. If
\[\mathrm{rank}(g,f^{\sharp}(\ell_{1},\ldots,\ell_{m},u_{1}),\ldots,f^{\sharp}(\ell_{1},\ldots,\ell_{m},u_{n}))>0,\]
then
\[f^{\sharp}(\ell_{1},\ldots,\ell_{m},g_{x}(u_{1},\ldots,u_{n}))=t\otimes
(g,f^{\sharp}(\ell_{1},\ldots,\ell_{m},u_{1}),\ldots,f^{\sharp}(\ell_{1},\ldots,\ell_{m},u_{n}))\]
for some unique term $t$.
\end{cor}
\begin{proof}
We have
\[f^{\sharp}(\ell_{1},\ldots,\ell_{m},g_{x}(u_{1},\ldots,u_{n}))\]
\[=g^{\sharp}(f^{\sharp}(\ell_{1},\ldots,\ell_{m},u_{1}),\ldots,f^{\sharp}(\ell_{1},\ldots,\ell_{m},u_{n}),f_{x}(\ell_{1},\ldots,\ell_{m})).\]

Therefore, by Proposition \ref{t4gh2ui04gthn}, there is a term $t$ such that
\[f^{\sharp}(\ell_{1},\ldots,\ell_{m},g_{x}(u_{1},\ldots,u_{n}))\]
\[=t\otimes(f^{\sharp}(\ell_{1},\ldots,\ell_{m},u_{1}),\ldots,f^{\sharp}(\ell_{1},\ldots,\ell_{m},u_{n})).\]
\end{proof}
\begin{defn}
\index{$f^{-}(\ell_{1},\ldots,\ell_{n},\ell)$}
We shall write $f^{-}(\ell_{1},\ldots,\ell_{n},\ell)$ for the term such that
\[f^{-}(\ell_{1},\ldots,\ell_{n},\ell)\otimes(f,\ell_{1},\ldots,\ell_{n})=f^{\sharp}(\ell_{1},\ldots,\ell_{n},\ell).\]
\end{defn}






\begin{prop}
Suppose that $f\in E$ and $\ell_{1},\ldots,\ell_{n},\ell\in L$. Then
if a variable $x$ is present in $f^{-}(\ell_{1},\ldots,\ell_{n},\ell)$, then $x$ is either present in $\ell$ or functionally present in $\ell$.
\end{prop}
\begin{defn}
Let $X,E,F,\mathcal{G},\ldots$ be defined the same way as in Definition \ref{3hqw3f6l8icg}.
Let $L\subseteq\mathbf{T}_{\mathcal{G}}[X]$ be downwards closed with respect to $\preceq$. If $Y\subseteq X,U\subseteq E,V\subseteq F$ are finite subsets, then let $\mathcal{G}_{Y,U,V}=\{f_{y}\mid f\in U,y\in Y\}\cup V$. Let $L_{Y,U,V}=\mathbf{T}_{\mathcal{G}_{Y,U,V}}[Y]\cap L$. Then we shall say that $L$ is \index{locally operable}\emph{locally operable} if
$L_{Y,U,V}$ is locally operable whenever $Y\subseteq X,U\subseteq E,V\subseteq F$ are finite subsets. 

If $L$ is locally operable, then as with the multigenic Laver tables, there is a system of operations $(g^{\sharp})_{g\in E}$ where
$g^{\sharp}:L^{n_{g}+1}\rightarrow L$ for each $g\in E$ and such that $(L_{Y,U,V},(g^{\sharp})_{g\in E},F)$ is a well-founded partially pre-endomorphic Laver table whenever $Y\subseteq X,U\subseteq E,V\subseteq F$ are finite subsets. We shall call the algebra $(L,(g^{\sharp})_{g\in E},F)$ a partially pre-endomorphic Laver table.
\end{defn}
\begin{thm}
Suppose that $L$ is locally operable and $(L,(f^{\sharp})_{f\in E},F)$ is a partially pre-endomorphic Laver table.
Then $(L,(f^{\sharp})_{f\in E},F)$ is endomorphic if and only if
$\mathrm{rank}(g,u_{1},\ldots,u_{n})=0$ implies that
\[\mathrm{rank}(g,f^{\sharp}(\ell_{1},\ldots,\ell_{m},u_{1}),\ldots,f^{\sharp}(\ell_{1},\ldots,\ell_{m},u_{n}))=0.\]
\label{2ijr50o}
\end{thm}
\begin{proof}
We shall only prove this result in the case that the set $L$ is operable. The general case when $L$ is locally operable will follow as an
immediate corollary.

$\rightarrow$. Suppose that $g_{x}(u_{1},\ldots,u_{n})\not\in L$. Then $g^{\sharp}(u_{1},\ldots,u_{n},u_{i})=u_{i}$ for $1\leq i\leq n$. Therefore, we have
\[f^{\sharp}(\ell_{1},\ldots,\ell_{m},u_{i})\]
\[=f^{\sharp}(\ell_{1},\ldots,\ell_{m},g^{\sharp}(u_{1},\ldots,u_{n},u_{i}))\]
\[=g^{\sharp}(f^{\sharp}(\ell_{1},\ldots,\ell_{m},u_{1}),\ldots,f^{\sharp}(\ell_{1},\ldots,\ell_{m},u_{n}),
f^{\sharp}(\ell_{1},\ldots,\ell_{m},u_{i})\]
for $1\leq i\leq n$. This implies that
\[g_{x}(f^{\sharp}(\ell_{1},\ldots,\ell_{m},u_{1}),\ldots,f^{\sharp}(\ell_{1},\ldots,\ell_{m},u_{n}))\not\in L.\]

$\leftarrow$. By the definition of $g^{\sharp}$, we have
\[g^{\sharp}(\ell_{1},\ldots,\ell_{m},\mathfrak{f}(u_{1},\ldots,u_{n}))\]
\[=\mathfrak{f}(g^{\sharp}(\ell_{1},\ldots,\ell_{m},u_{1}),\ldots,
g^{\sharp}(\ell_{1},\ldots,\ell_{m},u_{n})\]
whenever $\mathfrak{f}\in F$. It therefore suffices to show that whenever $f,g\in E$, we have
\[f^{\sharp}(\ell_{1},\ldots,\ell_{m},g^{\sharp}(u_{1},\ldots,u_{n},\ell))\]
\begin{equation}
=g^{\sharp}(f^{\sharp}(\ell_{1},\ldots,\ell_{m},u_{1}),\ldots,
f^{\sharp}(\ell_{1},\ldots,\ell_{m},u_{n}),f^{\sharp}(\ell_{1},\ldots,\ell_{m},\ell)).
\label{453r8m}
\end{equation}
We shall prove \ref{453r8m} by induction on
\[((f,\ell_{1},\ldots,\ell_{m}),(g,u_{1},\ldots,u_{n}),\ell)\in(\Omega,<)\times(\Omega,<)\times L\]
in five different cases. To reduce the amount of notation, define \[\overline{\ell}=
(\ell_{1},\ldots,\ell_{m}),\overline{u}=(u_{1},\ldots ,u_{n}).\]

Case I: $f_{x}(\ell_{1},\ldots,\ell_{m})\not\in L$.

We have 
\[f^{\sharp}(\overline{\ell},g^{\sharp}(u_{1},\ldots,u_{n},\ell))=g^{\sharp}(u_{1},\ldots,u_{n},\ell)\]
and
\[g^{\sharp}(f^{\sharp}(\overline{\ell},u_{1}),\ldots,f^{\sharp}(\overline{\ell},u_{n}),f^{\sharp}(
\overline{\ell},\ell))=g^{\sharp}(u_{1},\ldots,u_{n},\ell).\]

Case II: $f_{x}(\ell_{1},\ldots,\ell_{m})\in L,g_{x}(u_{1},\ldots,u_{n})\not\in L$.

We have 
\[f^{\sharp}(\ell_{1},\ldots,\ell_{m},g^{\sharp}(u_{1},\ldots,u_{n},\ell))=f^{\sharp}(\ell_{1},\ldots,\ell_{n},\ell).\] 

On the other hand,
\[g_{x}(f^{\sharp}(\ell_{1},\ldots,\ell_{m},u_{1}),\ldots,f^{\sharp}(\ell_{1},\ldots,\ell_{m},u_{n}))\not\in L.\]

Therefore,
\[g^{\sharp}(f^{\sharp}(\ell_{1},\ldots,\ell_{m},u_{1}),\ldots ,f^{\sharp}(\ell_{1},\ldots,\ell_{m},u_{n}),f^{\sharp}(\ell_{1},\ldots,\ell_{m},\ell))\]
\[=f^{\sharp}(\ell_{1},\ldots,\ell_{m},\ell)\]

Case III: $f_{x}(\ell_{1},\ldots,\ell_{m})\in L,g_{x}(u_{1},\ldots,u_{n})\in L,\ell=x$.

\[f^{\sharp}(\overline{\ell},g^{\sharp}(u_{1},\ldots,u_{n},x))\]
\[=f^{\sharp}(\overline{\ell},g_{x}(u_{1},\ldots,u_{n}))\]
\[=g^{\sharp}(f^{\sharp}(\overline{\ell},u_{1}),\ldots,f^{\sharp}(\overline{\ell},u_{n}),f_{x}(\ell))\]
\[=g^{\sharp}(f^{\sharp}(\overline{\ell},u_{1}),\ldots,f^{\sharp}(\overline{\ell},u_{n}),f^{\sharp}(\ell,x)).\]

Case IV: $f_{x}(\ell_{1},\ldots,\ell_{m})\in L,g_{x}(u_{1},\ldots,u_{n})\in L,\ell=\mathfrak{h}(v_{1},\ldots,
v_{r})$ for some $\mathfrak{h}\in F$.

\[f^{\sharp}(\overline{\ell},g^{\sharp}(\overline{u},\ell))\]

\[=f^{\sharp}(\overline{\ell},g^{\sharp}(\overline{u},\mathfrak{h}(v_{1},\ldots,v_{r})))\]

\[=f^{\sharp}(\overline{\ell},\mathfrak{h}(g^{\sharp}(\overline{u},v_{1}),\ldots,g^{\sharp}(\overline{u},v_{r})))\]

\[=\mathfrak{h}(f^{\sharp}(\overline{\ell},g^{\sharp}(\overline{u},v_{1})),\ldots,f^{\sharp}(\overline{\ell},g^{\sharp}(\overline{u},v_{r})))\]

\[=\mathfrak{h}(g^{\sharp}(f^{\sharp}(\overline{\ell},u_{1}),\ldots,f^{\sharp}(\overline{\ell},u_{n}),
f^{\sharp}(\overline{\ell},v_{1})),\ldots,g^{\sharp}(f^{\sharp}(\overline{\ell},u_{1}),\ldots,f^{\sharp}(\overline{\ell},u_{n}),
f^{\sharp}(\overline{\ell},v_{r})))\]

\[=g^{\sharp}(f^{\sharp}(\overline{\ell},u_{1}),\ldots,f^{\sharp}(\overline{\ell},u_{n}),\mathfrak{h}(f^{\sharp}(
\overline{\ell},v_{1}),\ldots,f^{\sharp}(\overline{\ell},v_{r})))\]

\[=g^{\sharp}(f^{\sharp}(\overline{\ell},u_{1}),\ldots,f^{\sharp}(\overline{\ell},u_{n}),f^{\sharp}(\overline{\ell},\mathfrak{h}(v_{1},\ldots,v_{r})))\]

\[=g^{\sharp}(f^{\sharp}(\overline{\ell},u_{1}),\ldots,f^{\sharp}(\overline{\ell},u_{n}),f^{\sharp}(\overline{\ell},\ell)).\]

Case V: $f_{x}(\ell_{1},\ldots,\ell_{m})\in L,g_{x}(u_{1},\ldots,u_{n})\in L,\ell=
h_{x}(v_{1},\ldots,v_{r})$ for some $h\in E,x\in X$.

\[f^{\sharp}(\overline{\ell},g^{\sharp}(\overline{u},\ell))
=f^{\sharp}(\overline{\ell},g^{\sharp}(\overline{u},h_{x}(v_{1},\ldots,v_{r})))\]

\[=f^{\sharp}(\overline{\ell},h^{\sharp}(g^{\sharp}(\overline{u},v_{1}),\ldots,g^{\sharp}(\overline{u},v_{r}),g_{x}(\overline{u})))\]

\[=h^{\sharp}(f^{\sharp}(\overline{\ell},g^{\sharp}(\overline{u},v_{1})),\ldots,f^{\sharp}(\overline{\ell},
g^{\sharp}(\overline{u},v_{r})),f^{\sharp}(\overline{\ell},g_{x}(\overline{u})))\]

\[=h^{\sharp}(g^{\sharp}(f^{\sharp}(\overline{\ell},u_{1}),\ldots,f^{\sharp}(\overline{\ell},u_{n}),f^{\sharp}(\overline{\ell},v_{1})),\ldots,\]
\[g^{\sharp}(f^{\sharp}(\overline{\ell},u_{1}),\ldots,f^{\sharp}(\overline{\ell},u_{n}),f^{\sharp}(\overline{\ell},v_{r})),\]
\[g^{\sharp}(f^{\sharp}(\overline{\ell},u_{1}),\ldots,f^{\sharp}(\overline{\ell},u_{n}),f_{x}(\overline{\ell})))\]

\[=g^{\sharp}(f^{\sharp}(\overline{\ell},u_{1}),\ldots,f^{\sharp}(\overline{\ell},u_{n}),
h^{\sharp}(f^{\sharp}(\overline{\ell},v_{1}),\ldots,f^{\sharp}(\overline{\ell},v_{r}),f_{x}(\overline{\ell})))\]

\[=g^{\sharp}(f^{\sharp}(\overline{\ell},u_{1}),\ldots,f^{\sharp}(\overline{\ell},u_{n}),
f^{\sharp}(\overline{\ell},h_{x}(v_{1},\ldots,v_{r})))\]

\[=g^{\sharp}(f^{\sharp}(\overline{\ell},u_{1}),\ldots,f^{\sharp}(\overline{\ell},u_{n}),f^{\sharp}(\overline{\ell},\ell)).\]
\end{proof}
\begin{defn}
Let $(L,(f^{\sharp})_{f\in E},(\mathfrak{g})_{g\in F})$ be a partially pre-endomorphic Laver table.
If each operation $f^{\sharp}$ is endomorphic, then we shall call
$(L,(f^{\sharp})_{f\in E},(\mathfrak{g})_{g\in F})$ a \index{partially endomorphic Laver table}\emph{partially endomorphic Laver table}.
\end{defn}
\begin{rem}
The proof of Theorem \ref{2ijr50o} illuminates why one uses a descending, descending, ascending triple induction in the proof of self-distributivity in Theorem \ref{m43ti0}. In Theorem \ref{2ijr50o}, the directions of induction are more clear since one has no reason to use any other induction except for the induction on the well-founded poset $(\Omega,<)\times(\Omega,<)\times(L,\prec)$ with the lexicographic ordering. The isomorphisms in Example \ref{4uok9et3gnjd} between pre-multigenic Laver tables and some pre-endomorphic Laver tables translate the proof of Theorem \ref{2ijr50o} to a proof of self-distributivity in Theorem \ref{m43ti0} using a descending, descending, ascending triple induction.
\end{rem}
\begin{defn}
Using the notation in Definition \ref{3hqw3f6l8icg}, define a mapping $\Psi:L\rightarrow T_{F}(X)$ by letting
\index{$\Psi:L\rightarrow T_{F}(X)$}
\begin{enumerate}
\item $\Psi(x)=x$,

\item $\Psi(f_{x}(\ell_{1},\ldots,\ell_{n}))=x$, and

\item $\Psi(\mathfrak{g}(\ell_{1},\ldots,\ell_{n}))=\mathfrak{g}(\Psi(\ell_{1}),\ldots,\Psi(\ell_{n}))$
\end{enumerate}
whenever $x\in X,\ell_{1},\ldots,\ell_{n}\in L,f\in E,\mathfrak{g}\in F.$
\end{defn}
\begin{prop}
Let $(L,(g^{\sharp})_{g\in E},F)$ be an endomorphic Laver table.
\begin{enumerate}
\item $\Psi(f^{\sharp}(\ell_{1},\ldots,\ell_{n},\ell))=\Psi(\ell)$ whenever $\ell_{1},\ldots,\ell_{n},\ell\in L$.

\item Suppose $u,v\in L$. Then $\Psi(u)=\Psi(v)$ if and only if
\[L_{f_{1},a_{1,1},\ldots,a_{1,n_{1}}}\circ\ldots\circ L_{f_{r},a_{r,1},\ldots,a_{r,n_{r}}}(u)\]
\[=L_{g_{1},b_{1,1},\ldots,b_{1,m_{1}}}\circ\ldots\circ L_{g_{s},b_{s,1},\ldots,b_{s,m_{s}}}(v)\] for some $f_{1},\ldots,f_{r},g_{1},\ldots,g_{s}\in E$ and some $a_{i,j}\in L$ and $b_{i,j}\in L$ for appropriate $i,j$.
\end{enumerate}
\label{24rouh}
\end{prop}
\begin{proof}
\begin{enumerate}
\item We shall prove this result by induction on 
\[((f,\ell_{1},\ldots,\ell_{n}),\ell)\in(\Omega,<)\times(L,\prec)\]
in four cases.

Case I: $f_{x}(\ell_{1},\ldots,\ell_{n})\not\in L$.

Since $f^{\sharp}(\ell_{1},\ldots,\ell_{n},\ell)=\ell$, we have
\[\Psi(f^{\sharp}(\ell_{1},\ldots,\ell_{n},\ell))=\Psi(\ell).\]

Case II: $f_{x}(\ell_{1},\ldots,\ell_{n})\in L,\ell=x$:

We have 
\[\Psi(f^{\sharp}(\ell_{1},\ldots,\ell_{n},\ell))=\Psi(f_{x}(\ell_{1},\ldots,\ell_{n}))=\Psi(x).\]

Case III: $f_{x}(\ell_{1},\ldots,\ell_{n})\in L,\ell=\mathfrak{g}(u_{1},\ldots,u_{m})$

\[\Psi(f^{\sharp}(\ell_{1},\ldots,\ell_{n},\ell))\]
\[=\Psi(\mathfrak{g}(f^{\sharp}(\ell_{1},\ldots,\ell_{n},u_{1}),\ldots,f^{\sharp}(\ell_{1},\ldots,\ell_{n},u_{m})))\]
\[=\mathfrak{g}(\Psi(f^{\sharp}(\ell_{1},\ldots,\ell_{n},u_{1})),\ldots,\Psi(f^{\sharp}(\ell_{1},\ldots,\ell_{n},u_{m})))\]
\[=\mathfrak{g}(\Psi(u_{1}),\ldots,\Psi(u_{m}))=\Psi(\ell)\]

Case IV: $f_{x}(\ell_{1},\ldots,\ell_{n})\in L,\ell=g_{x}(u_{1},\ldots,u_{m})$.

\[\Psi(f^{\sharp}(\ell_{1},\ldots,\ell_{n},\ell))=\Psi(f^{\sharp}(\ell_{1},\ldots,\ell_{n},g_{x}(u_{1},\ldots,u_{m})))\]
\[=\Psi(g^{\sharp}(f^{\sharp}(\ell_{1},\ldots,\ell_{n},u_{1}),\ldots,f^{\sharp}(\ell_{1},\ldots,\ell_{n},u_{m}),
f_{x}(\ell_{1},\ldots,\ell_{n})))\]
\[=\Psi(f_{x}(\ell_{1},\ldots,\ell_{n}))=\Psi(x)=\Psi(g_{x}(u_{1},\ldots,u_{m}))=\Psi(\ell)\]

\item $\leftarrow$. This direction follows from part 1 immediately.

$\rightarrow$ Suppose that $(f,\ell_{1},\ldots,\ell_{n})\in\Omega$ is minimal with respect to the property that
$f_{x}(\ell_{1},\ldots,\ell_{n})\in L$.

Define a mapping $\Gamma:T_{F}(X)\rightarrow L$ by letting $\Gamma(x)=f_{x}(\ell_{1},\ldots,\ell_{n})$ and
\[\Gamma(\mathfrak{g}(u_{1},\ldots,u_{m}))=\mathfrak{g}(\Gamma(u_{1}),\ldots,\Gamma(u_{m})).\]

Then I claim that $f^{\sharp}(\ell_{1},\ldots,\ell_{n},\ell)=\Gamma(\Psi(\ell))$ for all terms $\ell$ and we shall prove this claim by induction on $\ell\in L$.
\begin{enumerate}
\item $f^{\sharp}(\ell_{1},\ldots,\ell_{n},x)=f_{x}(\ell_{1},\ldots,\ell_{n})=\Gamma(x)=\Gamma(\Psi(x)).$

\item Suppose now that $\ell=\mathfrak{g}(u_{1},\ldots,u_{m})$. Then

\[f^{\sharp}(\ell_{1},\ldots,\ell_{n},\ell)=f^{\sharp}(\ell_{1},\ldots,\ell_{n},\mathfrak{g}(u_{1},\ldots,u_{m}))\]
\[=\mathfrak{g}(f^{\sharp}(\ell_{1},\ldots,\ell_{n},u_{1}),\ldots,f^{\sharp}(\ell_{1},\ldots,\ell_{n},u_{m}))\]
\[=\mathfrak{g}(\Gamma(\Psi(u_{1})),\ldots,\Gamma(\Psi(u_{m})))\]
\[=\Gamma(\mathfrak{g}(\Psi(u_{1}),\ldots,\Psi(u_{m})))\]
\[=\Gamma(\Psi(\mathfrak{g}(u_{1},\ldots,u_{m})))=\Gamma(\Psi(\ell)).\]

\item Now let $\ell=g_{x}(u_{1},\ldots,u_{m})$. Then since
\[(g,f^{\sharp}(\ell_{1},\ldots,\ell_{n},u_{1}),\ldots,f^{\sharp}(\ell_{1},\ldots,\ell_{n},u_{m}))<(f,\ell_{1},\ldots,\ell_{n}),\]
we have
\[g_{x}(f^{\sharp}(\ell_{1},\ldots,\ell_{n},u_{1}),\ldots,f^{\sharp}(\ell_{1},\ldots,\ell_{n},u_{m}))\not\in L.\]
Therefore,
\[f^{\sharp}(\ell_{1},\ldots,\ell_{n},\ell)=f^{\sharp}(\ell_{1},\ldots,\ell_{n},g_{x}(u_{1},\ldots,u_{m}))\]
\[=g^{\sharp}(f^{\sharp}(\ell_{1},\ldots,\ell_{n},u_{1}),\ldots,f^{\sharp}(\ell_{1},\ldots,\ell_{n},u_{m}),f_{x}(\ell_{1},\ldots,\ell_{n}))\]
\[=f_{x}(\ell_{1},\ldots,\ell_{n})=\Gamma(x)=\Gamma(\Psi(\ell)).\]

Therefore, if $\Psi(u)=\Psi(v)$, then 
\[f^{\sharp}(\ell_{1},\ldots,\ell_{n},u)=f^{\sharp}(\ell_{1},\ldots,\ell_{n},v).\]
\end{enumerate}
\end{enumerate}
\end{proof}
\begin{prop}
Let $(L,(f^{\sharp})_{f\in E},(\mathfrak{g})_{g\in F})$ be a pre-endomorphic Laver table. Then let $\simeq$ be the equivalence relation on $L$ where $u\simeq v$ if and only if $u,v\in X$ or $u=f_{x}(u_{1},\ldots ,u_{m}),v=f_{y}(u_{1},\ldots ,u_{m})$ for some $u_{1},\ldots ,u_{m}\in L$ and $x,y\in L$. Then whenever $u\simeq v$, we have
\[f^{\sharp}(\ell_{1},\ldots,\ell_{n},u)\simeq f^{\sharp}(\ell_{1},\ldots,\ell_{n},v).\]
\label{34t9o9403}
\end{prop}
\begin{proof}
We shall prove this result by induction on the rank of $(f,\ell_{1},\ldots,\ell_{n})$. If
$\mathrm{rank}(f,\ell_{1},\ldots,\ell_{n})=0$, then 
$f^{\sharp}(\ell_{1},\ldots,\ell_{n},u)=u\simeq v=f^{\sharp}(\ell_{1},\ldots,\ell_{n},v)$. Now assume that
$\mathrm{rank}(f,\ell_{1},\ldots,\ell_{n})>0$. 

Then either $u=v$ or $u=g_{x}(u_{1},\ldots,u_{m}),v=g_{y}(u_{1},\ldots,u_{m})$ for some $u_{1},\ldots,u_{m}\in L,g\in E$ or
$u=x,v=y$. If $u=v$ or $u=x,v=y$, then it is easy to see that
$f^{\sharp}(\ell_{1},\ldots,\ell_{n},u)\simeq f^{\sharp}(\ell_{1},\ldots,\ell_{n},v)$.

Now assume that $u=g_{x}(u_{1},\ldots,u_{m}),v=g_{y}(u_{1},\ldots,u_{m})$. Then

\[f^{\sharp}(\ell_{1},\ldots,\ell_{n},u)=f^{\sharp}(\ell_{1},\ldots,\ell_{n},g_{x}(u_{1},\ldots,u_{m}))\]
\[=g^{\sharp}(f^{\sharp}(\overline{\ell},u_{1}),\ldots,f^{\sharp}(\overline{\ell},u_{m}),f_{x}(\ell_{1},\ldots,\ell_{n}))\]
\[\simeq g^{\sharp}(f^{\sharp}(\overline{\ell},u_{1}),\ldots,f^{\sharp}(\overline{\ell},u_{m}),f_{y}(\ell_{1},\ldots,\ell_{n}))\]
\[=f^{\sharp}(\ell_{1},\ldots,\ell_{n},g_{x}(u_{1},\ldots,u_{m}))=f^{\sharp}(\ell_{1},\ldots,\ell_{n},u).\]
\end{proof}

We shall now give a correspondence between pre-endomorphic Laver tables of many generators and pre-endomorphic Laver tables
of one generator. The trade off in this correspondence is that the pre-endomorphic Laver tables with one generator have more function
symbols. One obvious advantage of pre-endomorphic Laver tables of one generator is that the notation for the pre-endomorphic Laver tables of one generator is simpler than the notation for the pre-endomorphic Laver tables of multiple generators.

\begin{defn}
Suppose that $E$ is a set of function symbols and $X$ is a set of variables. Let $e$ be a variable with $e\not\in X$.
Let $\mathcal{G}=\{f_{x}\mid f\in E,x\in X\}$ such that each $f_{x}$ is a function symbol where $f_{x}$ and $f$ both have the same arity $n_{f}$. Let $\mathcal{H}=\{f_{x_{1},\ldots,x_{n_{f}}}\mid f\in E,x_{1},\ldots,x_{n_{f}}\in X\}$. 

Define a mapping \index{$\Lambda:X\times\mathbf{T}_{\mathcal{H}}[\{e\}]\rightarrow\mathbf{T}_{\mathcal{G}}[X]$}
$\Lambda:X\times\mathbf{T}_{\mathcal{H}}[\{e\}]\rightarrow\mathbf{T}_{\mathcal{G}}[X]$ by letting
\begin{enumerate}
\item $\Lambda(x,e)=x$ whenever $x\in X$, and

\item $\Lambda(x,f_{x_{1},\ldots,x_{n_{f}}}(u_{1},\ldots,u_{n_{f}}))
=f_{x}(\Lambda(x_{1},u_{1}),\ldots,\Lambda(x_{n_{f}},u_{n_{f}}))$
whenever $f\in E,x_{1},\ldots,x_{n_{f}}\in X,u_{1},\ldots,u_{n_{f}}\in\mathbf{T}_{\mathcal{H}}[\{e\}]$.
\end{enumerate}

The function $\Lambda$ is a bijection. Let 
\[\pi_{1}:X\times\mathbf{T}_{\mathcal{H}}[\{e\}]\rightarrow X,
\pi_{2}:X\times\mathbf{T}_{\mathcal{H}}[\{e\}]\rightarrow\mathbf{T}_{\mathcal{H}}[\{e\}]\]
be the projections. Then
\begin{enumerate}
\item $\Lambda^{-1}(x)=(x,e)$, and

\item $\Lambda^{-1}(f_{x}(v_{1},\ldots,v_{n}))$
\[=(x,f_{\pi_{1}\Lambda^{-1}(v_{1}),\ldots,\pi_{1}\Lambda^{-1}(v_{n})}(\pi_{2}\Lambda^{-1}(v_{1}),\ldots,\pi_{2}\Lambda^{-1}(v_{1})).\]
\end{enumerate}
\end{defn}

\begin{thm}
Let $L\subseteq T_{\mathcal{H}}[\{e\}]$.
\begin{enumerate}
\item $L$ is operable if and only if $\Lambda[X\times L]$ is operable. 

\item If $L$ is operable, then 
\[\mathrm{rank}(f_{x_{1},\ldots,x_{n}},\ell_{1},\ldots,\ell_{n})=\mathrm{rank}(f,\Lambda(x_{1},\ell_{1}),\ldots,\Lambda(x_{n},\ell_{n})).\]
whenever $\ell_{1},\ldots ,\ell_{n},\in L,x_{1},\ldots ,x_{n}\in X$.

\item $L\subseteq T_{\mathcal{H}}[\{e\}]$ is locally operable if and only if $\Lambda[X\times L]$ is locally operable. \label{htyv93h5n5yq54cy3}

\item If $L\subseteq T_{\mathcal{H}}[\{e\}]$ is locally operable, then
\[\Lambda(x,f_{x_{1},\ldots,x_{n}}^{\sharp}(\ell_{1},\ldots,\ell_{n},\ell))
=f^{\sharp}(\Lambda(x_{1},\ell_{1}),\ldots,\Lambda(x_{n},\ell_{n}),\Lambda(x,\ell))\]
whenever $\ell_{1},\ldots ,\ell_{n},\ell\in L,x_{1},\ldots ,x_{n},x\in X$.
\label{htyv93h5n5yq54cy4}

\item If $L\subseteq T_{\mathcal{H}}[\{e\}]$ is locally operable, then $L$ is endomorphic if and only if $\Lambda[X\times L]$ is endomorphic. \label{htyv93h5n5yq54cy5}
\end{enumerate}
\end{thm}

\begin{proof}
For simplicity, we shall omit the proof of \ref{htyv93h5n5yq54cy3} and we shall only prove \ref{htyv93h5n5yq54cy4},\ref{htyv93h5n5yq54cy5} in the case that the algebra $L$ is operable.

Clearly if $e\in L$, then $x=\Lambda(x,e)\in\Lambda[X\times L]$ for each $x\in X$. Therefore, $e\in L$ if and only if $X\subseteq\Lambda[X\times L]$. 

Claim: $L$ is downwards closed under $\preceq$ if and only if $\Lambda[X\times L]$ is downwards closed under $\preceq$.

$\rightarrow$ Suppose that $L$ is closed under $\preceq$. Now assume that
$f_{x}(v_{1},\ldots,v_{n})\in\Lambda[X\times L]$ and $v_{j}=\Lambda(x_{j},u_{j})$ for $1\leq j\leq n$. Then $\Lambda^{-1}(f_{x}(v_{1},\ldots,v_{n}))\in X\times L$, so
\[\Lambda^{-1}f_{x}(v_{1},\ldots,v_{n})=\Lambda^{-1}f_{x}(\Lambda(x_{1},u_{1}),\ldots,\Lambda(x_{n},u_{n}))\]
\[=(x,f_{x_{1},\ldots,x_{n}}(u_{1},\ldots,u_{n})).\]
Therefore, 
\[f_{x_{1},\ldots,x_{n}}(u_{1},\ldots,u_{n})\in L,\]
hence $u_{i}\in L$ for $1\leq i\leq n$. Thus, $(x_{i},u_{i})\in X\times L$, so
\[v_{i}=\Lambda(x_{i},u_{i})\in\Lambda[X\times L]\]
for $1\leq i\leq n$. Therefore $\Lambda[X\times L]$ is downwards closed with respect to $\preceq$.

$\leftarrow$ Suppose that $\Lambda[X\times L]$ is downwards closed under $\preceq$.
Let 
\[f_{x_{1},\ldots,x_{n}}(u_{1},\ldots,u_{n})\in L.\]
Then
\[f_{x}(\Lambda(x_{1},u_{1}),\ldots,\Lambda(x_{n},u_{n}))=\Lambda(x,f_{x_{1},\ldots,x_{n}}(u_{1},\ldots,u_{n}))\in\Lambda[X\times L],\]
so
$\Lambda(x_{i},u_{i})\in\Lambda[X\times L]$, hence $u_{i}\in L$ for $1\leq i\leq n$.

Development 1: Suppose that $\Lambda[X\times L]$ is operable. Then take note that
$u=v$ if and only if $\Lambda(x,u)\simeq\Lambda(y,v)$ where $\simeq$ is the equivalence relation described in Proposition \ref{34t9o9403} and $x,y\in X$.
Furthermore, $\Psi(\Lambda(x,u))=x$ where $\Psi$ is the operation in Proposition \ref{24rouh}. Suppose now that $\ell_{1},\ldots,\ell_{n},\ell\in L$ and $f_{x_{1},\ldots,x_{n}}\in\mathcal{H}$. Then by Proposition \ref{24rouh}, we have
\[f^{\sharp}(\Lambda(x_{1},\ell_{1}),\ldots,\Lambda(x_{n},\ell_{n}),\Lambda(x,\ell))=\Lambda(x,u)\]
for some $u$.

Furthermore, if 
\[f^{\sharp}(\Lambda(x_{1},\ell_{1}),\ldots,\Lambda(x_{n},\ell_{n}),\Lambda(y,\ell))=\Lambda(y,v),\]
then since
$\Lambda(x,\ell)\simeq\Lambda(y,\ell)$, by Proposition \ref{34t9o9403}
we have $\Lambda(x,u)\simeq\Lambda(y,v)$. Therefore, $u=v$.

Therefore, define a function $f^{\bullet}_{x_{1},\ldots,x_{n}}:L^{n+1}\rightarrow L$ by letting
\[\Lambda(x,f^{\bullet}_{x_{1},\ldots,x_{n}}(\ell_{1},\ldots,\ell_{n},\ell))=f^{\sharp}(\Lambda(x_{1},\ell_{1}),\ldots,\Lambda(x_{n},\ell_{n}),
\Lambda(x,\ell)).\]

Now define a function $\nabla:\Gamma(L)\rightarrow\mathbf{On}$ by letting
\[\nabla(f_{x_{1},\ldots,x_{n}},\ell_{1},\ldots,\ell_{n})=\mathrm{rank}(f,\Lambda(x_{1},\ell_{1}),\ldots,\Lambda(x_{n},\ell_{n})).\]

i. \[\nabla(f_{x_{1},\ldots,x_{n}},\ell_{1},\ldots,\ell_{n})=0\] iff
\[f_{x}(\Lambda(x_{1},\ell_{1}),\ldots,\Lambda(x_{n},\ell_{n}))\not\in\Lambda[X\times L]\] iff
\[\Lambda(x,f_{x_{1},\ldots,x_{n}}(\ell_{1},\ldots,\ell_{n}))\not\in\Lambda[X\times L]\] iff
\[f_{x_{1},\ldots,x_{n}}(\ell_{1},\ldots,\ell_{n})\not\in L.\]

ii. If $f_{x_{1},\ldots,x_{n}}(\ell_{1},\ldots,\ell_{n})\not\in L$, then
\[f_{x}(\Lambda(x_{1},\ell_{1}),\ldots,\Lambda(x_{n},\ell_{n}))\not\in\Lambda[X\times L],\]
so
\[\Lambda(x,\ell)=f^{\sharp}(\Lambda(x_{1},\ell_{1}),\ldots,\Lambda(x_{n},\ell_{n}),\Lambda(x,\ell))\]
\[=\Lambda(x,f_{x_{1},\ldots,x_{n}}^{\bullet}(\ell_{1},\ldots,\ell_{n},\ell)).\]
Therefore, 
\[f_{x_{1},\ldots,x_{n}}^{\bullet}(\ell_{1},\ldots,\ell_{n},\ell)=\ell.\]

iii. I now claim that
\[f^{\bullet}_{x_{1},\ldots,x_{n}}(\ell_{1},\ldots,\ell_{n},e)=f_{x_{1},\ldots,x_{n}}(\ell_{1},\ldots,\ell_{n})\]
whenever $f_{x_{1},\ldots,x_{n}}(\ell_{1},\ldots,\ell_{n})\in L$.

We have
\[\Lambda(x,f_{x_{1},\ldots,x_{n}}^{\bullet}(\ell_{1},\ldots,\ell_{n},e))\]
\[=f^{\sharp}(\Lambda(x_{1},\ell_{1}),\ldots,\Lambda(x_{n},\ell_{n}),\Lambda(x,e))\]
\[=f^{\sharp}(\Lambda(x_{1},\ell_{1}),\ldots,\Lambda(x_{n},\ell_{n}),x)\]
\[=f_{x}(\Lambda(x_{1},\ell_{1}),\ldots,\Lambda(x_{n},\ell_{n}),x)\]
\[=\Lambda(x,f_{x_{1},\ldots,x_{n}}(\ell_{1},\ldots,\ell_{n})).\]
Therefore,
\[f_{x_{1},\ldots,x_{n}}^{\bullet}(\ell_{1},\ldots,\ell_{n},e)=f_{x_{1},\ldots,x_{n}}(\ell_{1},\ldots,\ell_{n}).\]

iv. Suppose now that $f_{x_{1},\ldots,x_{n}}\in\mathcal{H},\ell_{1},\ldots,\ell_{n},\in L$ and
$g_{y_{1},\ldots,y_{m}}(u_{1},\ldots,u_{m})\in L$, and $f_{x_{1},\ldots,x_{n}}(\ell_{1},\ldots,\ell_{n})\in L$.
Then I claim that
\[f_{x_{1},\ldots,x_{n}}^{\bullet}(\ell_{1},\ldots,\ell_{n},g_{y_{1},\ldots,y_{m}}(u_{1},\ldots,u_{m}))\]

\[=g_{y_{1},\ldots,y_{m}}^{\bullet}(f_{x_{1},\ldots,x_{n}}^{\bullet}(\ell_{1},\ldots,\ell_{n},u_{1}),\ldots,\]
\[f_{x_{1},\ldots,x_{n}}^{\bullet}(\ell_{1},\ldots,\ell_{n},u_{m}),f_{x_{1},\ldots,x_{n}}(\ell_{1},\ldots,\ell_{n})).\]
We have

\[\Lambda(x,f_{x_{1},\ldots,x_{n}}^{\bullet}(\ell_{1},\ldots,\ell_{n},g_{y_{1},\ldots,y_{m}}(u_{1},\ldots,u_{m})))\]

\[=f^{\sharp}(\Lambda(x_{1},\ell_{1}),\ldots,\Lambda(x_{n},\ell_{n}),\Lambda(x,g_{y_{1},\ldots,y_{m}}(u_{1},\ldots,u_{m})))\]

\[=f^{\sharp}(\Lambda(x_{1},\ell_{1}),\ldots,\Lambda(x_{n},\ell_{n}),g_{x}(\Lambda(y_{1},u_{1}),\ldots,\Lambda(y_{m},u_{m})))\]

\[=g^{\sharp}(f^{\sharp}(\Lambda(x_{1},\ell_{1}),\ldots,\Lambda(x_{n},\ell_{n}),\Lambda(y_{1},u_{1})),\ldots,\]
\[f^{\sharp}(\Lambda(x_{1},\ell_{1}),\ldots,\Lambda(x_{n},\ell_{n}),\Lambda(y_{m},u_{m})),f_{x}(\Lambda(x_{1},\ell_{1}),\ldots,\Lambda(x_{n},\ell_{n}))\]

\[=g^{\sharp}(\Lambda(y_{1},f_{x_{1},\ldots,x_{n}}^{\bullet}(\ell_{1},\ldots,\ell_{n},u_{1})),\ldots,
\Lambda(y_{m},f_{x_{1},\ldots,x_{n}}^{\bullet}(\ell_{1},\ldots,\ell_{n},u_{m})),\]
\[\Lambda(x,f_{x_{1},\ldots,x_{n}}(\ell_{1},\ldots,\ell_{n})))\]

\[=\Lambda(x,g_{y_{1},\ldots,y_{m}}^{\bullet}(f_{x_{1},\ldots,x_{n}}^{\bullet}(\ell_{1},\ldots,\ell_{n},u_{1}),\ldots,f_{x_{1},\ldots,x_{n}}^{\bullet}(\ell_{1},\ldots,\ell_{n},u_{m}),\]
\[f_{x_{1},\ldots,x_{n}}(\ell_{1},\ldots,\ell_{n}))).\]
Therefore,
\[f_{x_{1},\ldots,x_{n}}^{\bullet}(\ell_{1},\ldots,\ell_{n},g_{y_{1},\ldots,y_{m}}(u_{1},\ldots,u_{m}))\]

\[=g_{y_{1},\ldots,y_{m}}^{\bullet}(f_{x_{1},\ldots,x_{n}}^{\bullet}(\ell_{1},\ldots,\ell_{n},u_{1}),\ldots,f_{x_{1},\ldots,x_{n}}^{\bullet}(\ell_{1},\ldots,\ell_{n},u_{m}),f_{x_{1},\ldots,x_{n}}(\ell_{1},\ldots,\ell_{n})).\]

v. I claim that
\[\nabla(f_{x_{1},\ldots,x_{n}},\ell_{1},\ldots,\ell_{n})\]
\[>\nabla(g_{y_{1},\ldots,y_{m}},f_{x_{1},\ldots,x_{n}}^{\bullet}(\ell_{1},\ldots,\ell_{n},u_{1}),\ldots,f_{x_{1},\ldots,x_{n}}^{\bullet}(\ell_{1},\ldots,\ell_{n},u_{m})).\]
We have
\[\nabla(f_{x_{1},\ldots,x_{n}},\ell_{1},\ldots,\ell_{n})\]

\[=\mathrm{rank}(f,\Lambda(x_{1},\ell_{1}),\ldots,\Lambda(x_{n},\ell_{n}))\]

\[>\mathrm{rank}(g,f^{\sharp}(\Lambda(x_{1},\ell_{1}),\ldots,\Lambda(x_{n},\ell_{n}),\Lambda(y_{1},u_{1})),\ldots,\]
\[f^{\sharp}(\Lambda(x_{1},\ell_{1}),\ldots,\Lambda(x_{n},\ell_{n}),\Lambda(y_{m},u_{m})))\]

\[=\mathrm{rank}(g,\Lambda(y_{1},f_{x_{1},\ldots,x_{n}}^{\bullet}(\ell_{1},\ldots,\ell_{n},u_{1})),\ldots,\]
\[\Lambda(y_{m},f_{x_{1},\ldots,x_{n}}^{\bullet}(\ell_{1},\ldots,\ell_{n},u_{m})))\]

\[=\nabla(g_{y_{1},\ldots,y_{m}},f_{x_{1},\ldots,x_{n}}^{\bullet}(\ell_{1},\ldots,
\ell_{n},u_{1}),\ldots,\]
\[f_{x_{1},\ldots,x_{n}}^{\bullet}(\ell_{1},\ldots,\ell_{n},u_{m})).\]

Therefore, from these facts that we just have proven, we conclude that $L$ is operable and $(L,(f_{x_{1},\dots,x_{n_{f}}}^{\bullet})_{f\in E,x_{1},\dots,x_{n_{f}}\in X})$ is an pre-endomorphic Laver table and
\[\mathrm{rank}(f^{\bullet}_{x_{1},\ldots,x_{n}},\ell_{1},\ldots,\ell_{n})\leq\nabla(f^{\bullet}_{x_{1},\ldots,x_{n}},\ell_{1},\ldots,\ell_{n})\]
\[=\mathrm{rank}(f,\Lambda(x_{1},\ell_{1}),\ldots,\Lambda(x_{n},\ell_{n})).\]

vi. I now claim that if the algebra $\Lambda[X\times L]$ is endomorphic, then so is the algebra $L$.

Suppose therefore that $\Lambda[X\times L]$ is endomorphic. Suppose that $f_{x_{1},\ldots,x_{n}},g_{y_{1},\ldots,y_{m}}\in\mathcal{H},x\in X,\ell_{1},\ldots,\ell_{n},u_{1},\ldots,u_{m},u\in L$.
Then
\[\Lambda(x,f_{x_{1},\ldots,x_{n}}^{\bullet}(\ell_{1},\ldots,\ell_{n},g_{y_{1},\ldots,y_{m}}^{\bullet}(u_{1},\ldots,u_{m},u)))\]

\[=f^{\sharp}(\Lambda(x_{1},\ell_{1}),\ldots,\Lambda(x_{n},\ell_{n}),\Lambda(x,g_{y_{1},\ldots,y_{m}}^{\bullet}(u_{1},\ldots,u_{m},u)))\]

\[=f^{\sharp}(\Lambda(x_{1},\ell_{1}),\ldots,\Lambda(x_{n},\ell_{n}),g^{\sharp}(\Lambda(y_{1},u_{1}),\ldots,\Lambda(y_{m},u_{m}),\Lambda(x,u)))\]

\[=g^{\sharp}(f^{\sharp}(\Lambda(x_{1},\ell_{1}),\ldots,\Lambda(x_{n},\ell_{n}),\Lambda(y_{1},u_{1})),\ldots,\]
\[f^{\sharp}(\Lambda(x_{1},\ell_{1}),\ldots,\Lambda(x_{n},\ell_{n}),\Lambda(y_{m},u_{m})),
f^{\sharp}(\Lambda(x_{1},\ell_{1}),\ldots,\Lambda(x_{n},\ell_{n}),\Lambda(x,u)))\]

\[=g^{\sharp}(\Lambda(y_{1},f_{x_{1},\ldots,x_{n}}^{\bullet}(\ell_{1},\ldots,\ell_{n},u_{1})),\ldots,\]
\[\Lambda(y_{m},f_{x_{1},\ldots,x_{n}}^{\bullet}(\ell_{1},\ldots,\ell_{n},u_{m})),
\Lambda(x,f_{x_{1},\ldots,x_{n}}^{\bullet}(\ell_{1},\ldots,\ell_{n},u))\]

\[=\Lambda(x,g_{y_{1},\ldots,y_{m}}^{\bullet}(f_{x_{1},\ldots,x_{n}}^{\bullet}(\ell_{1},\ldots,\ell_{n},u_{1}),\ldots,\]
\[f_{x_{1},\ldots,x_{n}}^{\bullet}(\ell_{1},\ldots,\ell_{n},u_{m}),f_{x_{1},\ldots,x_{n}}^{\bullet}(\ell_{1},\ldots,\ell_{n},u)).\]
Therefore,
\[f_{x_{1},\ldots,x_{n}}^{\bullet}(\ell_{1},\ldots,\ell_{n},g_{y_{1},\ldots,y_{m}}^{\bullet}(u_{1},\ldots,u_{m},u)))\]

\[=g_{y_{1},\ldots,y_{m}}^{\bullet}(f_{x_{1},\ldots,x_{n}}^{\bullet}(\ell_{1},\ldots,\ell_{n},u_{1}),\ldots,\]
\[f_{x_{1},\ldots,x_{n}}^{\bullet}(\ell_{1},\ldots,\ell_{n},u_{m}),f_{x_{1},\ldots,x_{n}}^{\bullet}(\ell_{1},\ldots,\ell_{n},u)).\]

Development 2: Suppose now that $L$ is operable. Then Let $M=\Lambda[X\times L]$. Then
define a function $\nabla:\Gamma(M)\rightarrow\mathbf{On}$ by letting
\[\nabla(f,p_{1},\ldots,p_{n})\]
\[=\mathrm{rank}(f_{\pi_{1}\Lambda^{-1}(p_{1}),\ldots,\pi_{1}\Lambda^{-1}(p_{n})},\pi_{2}\Lambda^{-1}(p_{1}),\ldots,\pi_{2}\Lambda^{-1}(p_{n}))\]
\[=\mathrm{rank}(f_{\pi_{1}\Lambda^{-1}(\overline{p})},\pi_{2}\Lambda^{-1}(\overline{p}))\]
whenever $f\in E,\ell_{1},\ldots,\ell_{n}\in M$.

The definition of $\nabla$ is equivalent to saying that
\[\nabla(f,\Lambda(x_{1},\ell_{1}),\ldots,\Lambda(x_{n},\ell_{n}))
=\mathrm{rank}(f_{x_{1},\ldots,x_{n}},\ell_{1},\ldots,\ell_{n})\]
whenever $x_{1},\ldots,x_{n}\in X,\ell_{1},\ldots,\ell_{n}\in L,f\in E$.

Define
\[f^{\bullet}(\overline{p},p)\]
\[=\Lambda(\pi_{1}\Lambda^{-1}(p),f^{\sharp}_{\pi_{1}\Lambda^{-1}(\overline{p})}(\pi_{2}\Lambda^{-1}(\overline{p}),\pi_{2}\Lambda^{-1}(p))).\]

Said differently,
\[f^{\bullet}(\Lambda(x_{1},\ell_{1}),\ldots,\Lambda(x_{n},\ell_{n}),\Lambda(x,\ell))\]
\[=\Lambda(x,f_{x_{1},\ldots,x_{n}}^{\sharp}(\ell_{1},\ldots,\ell_{n},\ell).\]

i. I claim that $\nabla(f,\overline{p})=0$ if and only if $f_{x}(\overline{p})\not\in M$.

\[\nabla(f,\overline{p})=0\]
if and only if
\[\mathrm{rank}(f_{\pi_{1}\Lambda^{-1}(\overline{p})},\pi_{2}\Lambda^{-1}(\overline{p}))=0\]
if and only if
\[f_{\pi_{1}\Lambda^{-1}(\overline{p})}(\pi_{2}\Lambda^{-1}(\overline{p}))\not\in L\]
if and only if
\[(x,f_{\pi_{1}\Lambda^{-1}(\overline{p})}(\pi_{2}\Lambda^{-1}(\overline{p})))\not\in X\times L\]
if and only if
\[\Lambda(x,f_{\pi_{1}\Lambda^{-1}(\overline{p})}(\pi_{2}\Lambda^{-1}(\overline{p})))\not\in\Lambda[X\times L]=M\]
if and only if
\[f_{x}(\overline{p})\not\in M.\]

ii. I claim that $f^{\bullet}(\overline{p},p)=p$ whenever $f_{x}(\overline{p})\not\in M$.
We have
\[f^{\bullet}(\overline{p},p)
=\Lambda(\pi_{1}\Lambda^{-1}(p),f^{\sharp}_{\pi_{1}\Lambda^{-1}(\overline{p})}(\pi_{2}\Lambda^{-1}(\overline{p}),\pi_{2}\Lambda^{-1}(p)))\]
\[=\Lambda(\pi_{1}\Lambda^{-1}(p),\pi_{2}\Lambda^{-1}(p))=p.\]

iii. I claim that if $f_{x}(\overline{p})\in M$, then
\[f^{\bullet}(\overline{p},x)=f_{x}(\overline{p}).\]

Recall that $\Lambda^{-1}(x)=(x,e)$. Therefore,
\[f^{\bullet}(\overline{p},x)
=\Lambda(\pi_{1}\Lambda^{-1}(x),f_{\pi_{1}\Lambda^{-1}(\overline{p})}^{\sharp}(\pi_{2}\Lambda^{-1}(\overline{p}),\pi_{2}\Lambda^{-1}(x)))\]
\[=\Lambda(x,f_{\pi_{1}\Lambda^{-1}(\overline{p})}^{\sharp}(\pi_{2}\Lambda^{-1}(\overline{p}),e))
=\Lambda(x,f_{\pi_{1}\Lambda^{-1}(\overline{p})}(\pi_{2}\Lambda^{-1}(\overline{p})))
=f_{x}(\overline{p}).\]

iv. I claim that
\[f^{\bullet}(\Lambda(x_{1},\ell_{1}),\ldots,\Lambda(x_{n},\ell_{n}),g_{x}(\Lambda(y_{1},u_{1}),\ldots,\Lambda(y_{m},u_{m})))\]
\[=g^{\bullet}(f^{\bullet}(\Lambda(x_{1},\ell_{1}),\ldots,\Lambda(x_{n},\ell_{n}),\Lambda(y_{1},u_{1})),\ldots,\]
\[f^{\bullet}(\Lambda(x_{1},\ell_{1}),\ldots,\Lambda(x_{n},\ell_{n}),\Lambda(y_{m},u_{m})),
f_{x}(\Lambda(x_{1},\ell_{1}),\ldots,\Lambda(x_{n},\ell_{n}))).\]
We have
\[f^{\bullet}(\Lambda(x_{1},\ell_{1}),\ldots,\Lambda(x_{n},\ell_{n}),g_{x}(\Lambda(y_{1},u_{1}),\ldots,\Lambda(y_{m},u_{m})))\]
\[=f^{\bullet}(\Lambda(x_{1},\ell_{1}),\ldots,\Lambda(x_{n},\ell_{n}),\Lambda(x,g_{y_{1},\ldots,y_{m}}(u_{1},\ldots,u_{m})))\]
\[=\Lambda(x,f_{x_{1},\ldots,x_{n}}^{\sharp}(\ell_{1},\ldots,\ell_{n},g_{y_{1},\ldots,y_{m}}(u_{1},\ldots,u_{m})))\]
\[=\Lambda(x,g_{y_{1},\ldots,y_{n}}^{\sharp}(f_{x_{1},\ldots,x_{n}}^{\sharp}(\ell_{1},\ldots,\ell_{n},u_{1}),\ldots,\]
\[f_{x_{1},\ldots,x_{n}}^{\sharp}(\ell_{1},\ldots,\ell_{n},u_{m}),f_{x_{1},\ldots,x_{n}}(\ell_{1},\ldots,\ell_{n})))\]
\[=g^{\bullet}(\Lambda(y_{1},f_{x_{1},\ldots,x_{n}}^{\sharp}(\ell_{1},\ldots,\ell_{n},u_{1})),\ldots,\]
\[\Lambda(y_{m},f_{x_{1},\ldots,x_{n}}^{\sharp}(\ell_{1},\ldots,\ell_{n},u_{m})),\Lambda(x,f_{x_{1},\ldots,x_{n}}(\ell_{1},\ldots,\ell_{n})))\]
\[=g^{\bullet}(f^{\bullet}(\Lambda(x_{1},\ell_{1}),\ldots,\Lambda(x_{n},\ell_{n}),\Lambda(y_{1},u_{1})),\ldots,\]
\[f^{\bullet}(\Lambda(x_{1},\ell_{1}),\ldots,\Lambda(x_{n},\ell_{n}),\Lambda(y_{m},u_{m})),f_{x}(\Lambda(x_{1},\ell_{1}),\ldots,\Lambda(x_{n},\ell_{n}))).\]

v. I claim that if $\nabla(f,\Lambda(x_{1},\ell_{1}),\ldots,\Lambda(x_{n},\ell_{n}))>0$, then
\[\nabla(f,\Lambda(x_{1},\ell_{1}),\ldots,\Lambda(x_{n},\ell_{n}))\]
\[>\nabla(g,f^{\bullet}(\Lambda(x_{1},\ell_{1}),\ldots,\Lambda(x_{n},\ell_{n}),\Lambda(y_{1},u_{1})),\ldots,\]
\[f^{\bullet}(\Lambda(x_{1},\ell_{1}),\ldots,\Lambda(x_{n},\ell_{n}),\Lambda(y_{m},u_{m}))).\]
We have 
\[\nabla(f,\Lambda(x_{1},\ell_{1}),\ldots,\Lambda(x_{n},\ell_{n}))\]
\[=\mathrm{rank}(f_{x_{1},\ldots,x_{n}},\ell_{1},\ldots,\ell_{n})\]
\[>\mathrm{rank}(g_{y_{1},\ldots,y_{m}},f_{x_{1},\ldots,x_{n}}^{\sharp}(\ell_{1},\ldots,\ell_{n},u_{1}),\ldots,
f_{x_{1},\ldots,x_{n}}^{\sharp}(\ell_{1},\ldots,\ell_{n},u_{m}))\]
\[=\nabla(g,\nabla(y_{1},f_{x_{1},\ldots,x_{n}}^{\sharp}(\ell_{1},\ldots,\ell_{n},u_{1}),\ldots,
\nabla(y_{1},f_{x_{1},\ldots,x_{n}}^{\sharp}(\ell_{1},\ldots,\ell_{n},u_{m})))\]
\[=\nabla(g,f^{\bullet}(\Lambda(x_{1},\ell_{1}),\ldots,\Lambda(x_{n},\ell_{n}),\Lambda(y_{1},u_{1})),\ldots,\]
\[f^{\bullet}(\Lambda(x_{1},\ell_{1}),\ldots,\Lambda(x_{n},\ell_{n}),\Lambda(y_{m},u_{m}))).\]

vi. I claim that if $L$ is operable, then so is $M$. Therefore, suppose that $L$ is operable. Then

\[f^{\bullet}(\Lambda(x_{1},\ell_{1}),\ldots,\Lambda(x_{n},\ell_{n}),g^{\bullet}(\Lambda(y_{1},u_{1}),\ldots,
\Lambda(y_{m},u_{m}),\Lambda(y,u)))\]
\[=f^{\bullet}(\Lambda(x_{1},\ell_{1}),\ldots,\Lambda(x_{n},\ell_{n}),\Lambda(y,g_{y_{1},\ldots,y_{m}}^{\sharp}(u_{1},\ldots,u_{m},u)))\]
\[=\Lambda(y,f_{x_{1},\ldots,x_{n}}^{\sharp}(\ell_{1},\ldots,\ell_{n},g_{y_{1},\ldots,y_{m}}^{\sharp}(u_{1},\ldots,u_{m},u)))\]
\[=\Lambda(y,g_{y_{1},\ldots,y_{m}}^{\sharp}(f_{x_{1},\ldots,x_{n}}^{\sharp}(\ell_{1},\ldots,\ell_{n},u_{1}),\ldots,\]
\[f_{x_{1},\ldots,x_{n}}^{\sharp}(\ell_{1},\ldots,\ell_{n},u_{m}),f_{x_{1},\ldots,x_{n}}^{\sharp}(\ell_{1},\ldots,\ell_{n},u)))\]
\[=g^{\bullet}(\Lambda(y_{1},f_{x_{1},\ldots,x_{n}}^{\sharp}(\ell_{1},\ldots,\ell_{n},u_{1})),\ldots,\]
\[\Lambda(y_{m},f_{x_{1},\ldots,x_{n}}^{\sharp}(\ell_{1},\ldots,\ell_{n},u_{m}),\Lambda(y,f_{x_{1},\ldots,x_{n}}^{\sharp}(\ell_{1},\ldots,\ell_{n},u))))\]
\[=g^{\bullet}(f^{\bullet}(\Lambda(x_{1},\ell_{1}),\ldots,\Lambda(x_{n},\ell_{n}),\Lambda(y_{1},u_{1})),\ldots\]
\[,f^{\bullet}(\Lambda(x_{1},\ell_{1}),\ldots,\Lambda(x_{n},\ell_{n}),\Lambda(y_{m},u_{m})),
f^{\bullet}(\Lambda(x_{1},\ell_{1}),\ldots,\Lambda(x_{n},\ell_{n}),\Lambda(y,u))).\]
\end{proof}

\begin{thm}
Suppose that $(L,(f^{\sharp})_{f\in\mathcal{F}},(\mathfrak{g})_{\mathfrak{g}\in\mathcal{G}})$ is an endomorphic Laver table. Suppose that
$(f,\ell_{1},\dots,\ell_{n_{f}})*(f,\ell_{1},\dots,\ell_{n_{f}})\in\mathrm{Li}(\Gamma(L))$
and $t(x_{1},\dots,x_{k})$ is a term such that if $g_{y}(u_{1},\dots,u_{n_{g}})$ is a subterm of $t(x_{1},\dots,x_{k})$, then
$\mathrm{crit}(u_{1},\dots,u_{n_{g}})<\mathrm{crit}(\ell_{1},\dots,\ell_{n_{f}})$. Then
\[f^{\sharp}(\ell_{1},\dots,\ell_{n_{f}},t(x_{1},\dots,x_{k}))\]
\[=t(f_{x_{1}}(\ell_{1},\dots,\ell_{n_{f}}),\dots,f_{x_{k}}(\ell_{1},\dots,\ell_{n_{f}})).\]
\end{thm}
\begin{proof}
We shall prove this result by induction on the complexity of the term $t(x_{1},\dots,x_{k})$. We have $f^{\sharp}(\ell_{1},\dots,\ell_{n_{f}},x)=f_{x}(\ell_{1},\dots,\ell_{n_{f}})$.

Suppose that $\mathfrak{g}\in\mathcal{G}$ and the result holds for
$t_{1}(x_{1},\dots,x_{k}),\dots,t_{n_{\mathfrak{g}}}(x_{1},\dots,x_{k})$. Then

\[f^{\sharp}(\ell_{1},\dots,\ell_{n_{f}},\mathfrak{g}(t_{1}(x_{1},\dots,x_{k}),\dots,t_{n_{\mathfrak{g}}}(x_{1},\dots,x_{k})))\]

\[=\mathfrak{g}(f^{\sharp}(\ell_{1},\dots,\ell_{n_{f}},t_{1}(x_{1},\dots,x_{k})),\dots,
f^{\sharp}(\ell_{1},\dots,\ell_{n_{f}},t_{n_{\mathfrak{g}}}(x_{1},\dots,x_{k})))\]

\[=\mathfrak{g}(t_{1}(f_{x_{1}}(\ell_{1},\dots,\ell_{n_{f}}),\dots,f_{x_{k}}(\ell_{1},\dots,\ell_{n_{f}})),\]
\[\dots,t_{n_{\mathfrak{g}}}(f_{x_{1}}(\ell_{1},\dots,\ell_{n_{f}}),\dots,f_{x_{k}}(\ell_{1},\dots,\ell_{n_{f}}))).\]

Suppose now that $t(x_{1},\dots,x_{k})=g_{y}(t_{1}(x_{1},\dots,x_{k}),\dots,t_{n_{g}}(x_{1},\dots,x_{k}))$. Then, 
\[(f,\ell_{1},\dots,\ell_{n_{f}})\circ(g,t_{1}(x_{1},\dots,x_{k}),\dots,t_{n_{g}}(x_{1},\dots,x_{k}))\]
\[=(f,\ell_{1},\dots,\ell_{n_{f}})\circ(g,t_{1}(x_{1},\dots,x_{k}),\dots,t_{n_{g}}(x_{1},\dots,x_{k}))\]
\[=(g,f^{\sharp}(\ell_{1},\dots,\ell_{n_{f}},t_{1}(x_{1},\dots,x_{k})),\dots,
f^{\sharp}(\ell_{1},\dots,\ell_{n_{f}},t_{n_{g}}(x_{1},\dots,x_{k})))\]

\[=(g,t_{1}(f_{x_{1}}(\ell_{1},\dots,\ell_{n_{f}}),\dots,
f_{x_{k}}(\ell_{1},\dots,\ell_{n_{f}})),\dots,\]
\[t_{n_{g}}(f_{x_{1}}(\ell_{1},\dots,\ell_{n_{f}}),\dots,
f_{x_{k}}(\ell_{1},\dots,\ell_{n_{f}})))\]

, so

\[f^{\sharp}(\ell_{1},\dots,\ell_{n_{f}},t(x_{1},\dots,x_{k}))\]

\[=f^{\sharp}(\ell_{1},\dots,\ell_{n_{f}},g_{y}(t_{1}(x_{1},\dots,x_{k}),\dots,t_{n_{g}}(x_{1},\dots,x_{k})))\]

\[=g_{y}(t_{1}(f_{x_{1}}(\ell_{1},\dots,\ell_{n_{f}}),\dots,\]
\[f_{x_{k}}(\ell_{1},\dots,\ell_{n_{f}}),\dots,t_{n_{g}}(f_{x_{1}}(\ell_{1},\dots,\ell_{n_{f}}),\dots,f_{x_{k}}(\ell_{1},\dots,\ell_{n_{f}})))\]

\[=t(f_{x_{1}}(\ell_{1},\dots,\ell_{n_{f}}),\dots,f_{x_{k}}(\ell_{1},\dots,\ell_{n_{f}})).\]
\end{proof}

\section{Permutative and locally Laver-like partially endomorphic algebras}
We shall now generalize the notion of a permutative LD-system, Laver-like LD-system, and a locally Laver-like
LD-system to partially endomorphic algebras, and we shall use these generalizations to construct partially endomorphic Laver tables.

\begin{defn}
A partially endomorphic algebra $\mathcal{X}=(X,(f^{\mathcal{X}})_{f\in E},(g^{\mathcal{X}})_{g\in F})$ shall be called \index{permutative partially endomorphic algebra}\emph{permutative} if the hull $\Gamma(X,(f^{\mathcal{X}})_{f\in E})$ is a permutative LD-system.
\end{defn}
Notice how the definition of a permutative partially endomorphic LD-system does not depend on the functions $g^{\mathcal{X}}$ whenever
$\mathcal{X}$ is a partially endomorphic LD-system.
\begin{defn}
A partially endomorphic algebra $\mathcal{X}=(X,(f^{\mathcal{X}})_{f\in E},(g^{\mathcal{X}})_{g\in F})$ shall be called
\index{Laver-like partially endomorphic algebra}\emph{Laver-like} if the hull $\Gamma(X,(f^{\mathcal{X}})_{f\in E})$ is Laver-like.
\end{defn}
\begin{defn}
A partially endomorphic algebra $\mathcal{X}=(X,(f^{\mathcal{X}})_{f\in E},(g^{\mathcal{X}})_{g\in F})$ shall be called \index{locally Laver-like partially endomorphic algebra}\emph{locally
Laver-like} if whenever $U\subseteq E,V\subseteq F$ are finite subsets, then every finitely generated subalgebra of the reduct
$(X,(f^{\mathcal{X}})_{f\in U},(g^{\mathcal{X}})_{g\in V})$ is Laver-like.
\end{defn}
The definition of a locally Laver-like partially endomorphic algebra depends on the function symbols $(g^{\mathcal{X}})_{g\in F}$, and the notion of a locally Laver-like partially endomorphic algebra is stronger than just requiring the hull to be locally Laver-like.

The following examples illustrate a method for generating new permutative LD-systems from old permutative LD-systems.
\begin{exam}
Suppose that $(X,*,\circ)$ is an LD-monoid and $t:X\rightarrow X$ is a function given by a term in the language of LD-monoids. Then define an operation $+$ on $X$ by $x+y=t(x)*y$. Then $+$ is a self-distributive operation on $X$. Furthermore, we have
$t(x+y)=t(t(x)*y)=t(x)*t(y)$. Therefore, the mapping $t$ is a homomorphism from $(X,+)$ to $(X,*)$. From these facts, one can easily conclude that 
\begin{enumerate}
\item $x\in \mathrm{Li}^{+}(X)$ if and only if $t(x)\in \mathrm{Li}^{*}(X)$,

\item if $\mathrm{Li}^{*}(X)$ is a left-ideal in $(X,*)$, then $\mathrm{Li}^{+}(X)$ is a left-ideal in $(X,+)$,

\item if $(X,*,\circ)$ is permutative, then $(X,+)$ is also permutative,

\item if $(X,*,\circ)$ is Laver-like, then $(X,+)$ is also Laver-like,

\item if $(X,*,\circ)$ is locally Laver-like, then $(X,+)$ is also locally Laver-like,

\item if $(X,*,\circ)$ is permutative and $x,y\in X$, then $\mathrm{crit}^{+}(x)\leq \mathrm{crit}^{+}(y)$ if and only if
$\mathrm{crit}^{*}(t(x))\leq \mathrm{crit}^{*}(t(y))$.
\end{enumerate}
\label{024efrji}
\end{exam}
\begin{exam}
Example \ref{024efrji} generalizes to multi-LD-systems. Suppose that $(X,*,\circ,1)$ is an LD-monoid.
Suppose furthermore that $I$ is an index set and $t_{i}$ is a term of one variable in the language of LD-monoids for all $i\in I$.
Then define an operation $+_{i}$ for each $i\in I$ by $x+_{i}y=t_{i}(x)*y$. Then $(X,(*_{i})_{i\in I})$ is a multi-LD-system.
Furthermore, if $(X,*,\circ,1)$ is permutative, Laver-like, or locally Laver-like, then $(X,(*_{i})_{i\in I})$ is also permutative, Laver-like, or locally Laver-like respectively, and if $(X,*,\circ,1)$ is permutative, then $\mathrm{crit}(x,i)=\mathrm{crit}(t_{i}(x))$ whenever $x\in X,i\in I$.

\label{3ogug7w3r4tt6}
\end{exam}

The following development extends the method found in examples \ref{024efrji} and \ref{3ogug7w3r4tt6}
of generating new permutative LD-systems from old algebras to the context of partially endomorphic algebras.

\begin{defn}
If $\mathcal{X}=(X,(f^{\mathcal{X}})_{f\in E},(g^{\mathcal{X}})_{g\in F})$ is a partially endomorphic algebra, then a \index{basic inner term function}\emph{basic inner term function} is a function $t:X^{n+1}\rightarrow X$ such that for some $f\in E$ and term functions $t_{1},\ldots,t_{m}$ of $n$ variables, we have
\[t(a_{1},\ldots,a_{n},x)=f^{\mathcal{X}}(t_{1}(a_{1},\ldots,a_{n}),\ldots,t_{m}(a_{1},\ldots,a_{n}),x).\]
An inner term is a function $t$ where 
\[t(a_{1},\ldots,a_{n},x)=L_{t_{1},a_{1},\ldots,a_{n}}\circ\ldots\circ L_{t_{v},a_{1},\ldots,a_{n}}(x)\]
for some basic inner term functions
$t_{1},\ldots,t_{v}$. 
\end{defn}
\begin{thm}
Suppose that $\mathcal{X}=(X,(f^{\mathcal{X}})_{f\in E},(g^{\mathcal{X}})_{g\in F})$ is a permutative partially endomorphic algebra.
Let $I$ be an index set, and for each $i\in I$, let $h_{i}:X^{n_{i}+1}\rightarrow X$ be the function where
if $\mathbf{a}\in X^{n_{i}}$, then
\[h_{i}(\mathbf{a},x)=L_{t_{i,1},\mathbf{a}}\circ\ldots\circ L_{t_{i,m_{i}},\mathbf{a}}(x)\]
for some basic inner term functions
$t_{i,1},\ldots,t_{i,m_{i}}$. Furthermore, suppose that whenever $i\in I,1\leq r\leq m_{i}$, we have
\[t_{i,r}(\mathbf{a},x)=f_{i,r}^{\mathcal{X}}(s_{i,r,1}(\mathbf{a}),\ldots,s_{i,r,p_{i,r}}(\mathbf{a}),x)\]
for some term functions
$s_{i,r,1},\ldots,s_{i,r,p_{i,r}}$ and $f_{i,r}\in E$. Let $J$ be an index set and let $l_{j}$ be a term for each $j\in J$.
Let $\mathcal{X}^{\sharp}=(X,(h_{i})_{i\in I},(l_{j})_{j\in J})$.
Then define a mapping 
\[\Lambda:\Gamma(\mathcal{X}^{\sharp})\rightarrow\Gamma(\mathcal{X}))/\mathrm{Li}(\Gamma(\mathcal{X}))\]
 by letting
\[\Lambda(h_{i},\mathbf{a})=[f_{i,1},s_{i,1,1}(\mathbf{a}),\ldots,s_{i,1,p_{i,1}}(\mathbf{a})]\circ\ldots\circ
[f_{i,m_{i}},s_{i,m_{i},1}(\mathbf{a}),\ldots,s_{i,m_{i},p_{i,m_{i}}}(\mathbf{a})].\] Then
\begin{enumerate}
\item $\Lambda$ is an LD-system homomorphism,

\item $(h_{i},\mathbf{a})\in \mathrm{Li}(\Gamma(\mathcal{X}^{\sharp}))$ if and only if $\Lambda(h_{i},\mathbf{a})=1$,

\item $\mathcal{X}^{\sharp}$ is permutative,

\item
\[\mathrm{crit}((h_{i},\mathbf{a}))\leq \mathrm{crit}((h_{j},\mathbf{b}))\]
if and only if
\[\mathrm{crit}(\Lambda(h_{i},\mathbf{a}))\leq \mathrm{crit}(\Lambda(h_{j},\mathbf{b})),\]

\item if $\mathcal{X}$ is Laver-like, then $\mathcal{X}^{\sharp}$ is also Laver-like, and

\item if $\mathcal{X}$ is locally Laver-like, then $\mathcal{X}^{\sharp}$ is also locally Laver-like.
\end{enumerate}
\end{thm}
\begin{proof}
1. Suppose that $\mathbf{a}=(a_{1},\ldots,a_{n_{i}})$ and $\mathbf{b}=(b_{1},\ldots,b_{n_{j}})$. Then
\[(h_{i},\mathbf{a})*(h_{j},\mathbf{b})=(h_{j},h_{i}(\mathbf{a},b_{1}),\ldots,h_{i}(\mathbf{a},b_{n_{j}})).\]

Let 
\[\mathbf{c}=(h_{i}(\mathbf{a},b_{1}),\ldots,h_{i}(\mathbf{a},b_{n_{j}})).\]
Then for each $n_{j}$-ary term function $s$, we have
\[s(\mathbf{c})=s(h_{i}(\mathbf{a},b_{1}),\ldots,h_{i}(\mathbf{a},b_{n_{j}}))\]
\[=h_{i}(\mathbf{a},s(b_{1},\ldots,b_{n_{j}}))=h_{i}(\mathbf{a},s(\mathbf{b})).\]

Then 
\[\Lambda((h_{i},\mathbf{a})*(h_{j},\mathbf{b}))\]
\[=\Lambda(h_{j},h_{i}(\mathbf{a},b_{1}),\ldots,h_{i}(\mathbf{a},b_{n_{j}}))=\Lambda(h_{j},\mathbf{c})\]
\[=[f_{j,1},s_{j,1,1}(\mathbf{c}),\ldots,s_{j,1,p_{j,1}}(\mathbf{c})]\circ\ldots\circ
[f_{j,m_{j}},s_{j,m_{j},1}(\mathbf{c}),\ldots,s_{j,m_{j},p_{j,m_{j}}}(\mathbf{c})].\]

On the other hand,
\[\Lambda(h_{i},\mathbf{a})*\Lambda(h_{j},\mathbf{b})\]
\[=\Lambda(h_{i},\mathbf{a})
*([f_{j,1},s_{j,1,1}(\mathbf{b}),\ldots,s_{j,1,p_{j,1}}(\mathbf{b})]\]
\[\circ\ldots\circ[f_{j,m_{j}},s_{j,m_{j},1}(\mathbf{b}),\ldots,
s_{j,m_{j},p_{j,m_{j}}}(\mathbf{b})])\]
\[=(\Lambda(h_{i},\mathbf{a})*[f_{j,1},s_{j,1,1}(\mathbf{b}),\ldots,s_{j,1,p_{j,1}}(\mathbf{b})])\]
\[\circ\ldots\circ(\Lambda(h_{i},\mathbf{a})*[f_{j,m_{j}},s_{j,m_{j},1}(\mathbf{b}),\ldots,s_{j,m_{j},p_{j,m_{j}}}(\mathbf{b})]).\]

However, for $1\leq r\leq m_{j}$, we have

\[\Lambda(h_{i},\mathbf{a})*[f_{j,r},s_{j,r,1}(\mathbf{b}),\ldots,s_{j,r,p_{j,r}}(\mathbf{b})]\]
\[=([f_{i,1},s_{i,1,1}(\mathbf{a}),\ldots,s_{i,1,p_{i,1}}(\mathbf{a})]\circ\ldots\circ\]
\[
[f_{i,m_{i}},s_{i,m_{i},1}(\mathbf{a}),\ldots,s_{i,m_{i},p_{i,m_{i}}}(\mathbf{a})])*[f_{j,r},s_{j,r,1}(\mathbf{b}),\ldots,s_{j,r,p_{j,r}}(\mathbf{b})]\]

\[=[f_{j,r},L_{f_{i,1},s_{i,1,1}(\mathbf{a}),\ldots,s_{i,1,p_{i,1}}(\mathbf{a})}\circ\ldots\circ
L_{f_{i,m_{i}},s_{i,m_{i},1}(\mathbf{a}),\ldots,s_{i,m_{i},p_{i,m_{i}}}(\mathbf{a})}(s_{j,r,1}(\mathbf{b})),\]
\[\ldots,L_{f_{i,1},s_{i,1,1}(\mathbf{a}),\ldots,s_{i,1,p_{i,1}}(\mathbf{a})}\circ\ldots\circ
L_{f_{i,m_{i}},s_{i,m_{i},1}(\mathbf{a}),\ldots,s_{i,m_{i},p_{i,m_{i}}}(\mathbf{a})}(s_{j,r,p_{j,r}}(\mathbf{b}))]\]
\[=[f_{j,r},h_{i}(\mathbf{a},s_{j,r,1}(\mathbf{b})),\ldots,h_{i}(\mathbf{a},s_{j,r,p_{j,r}}(\mathbf{b}))]\]
\[=[f_{j,r},s_{j,r,1}(\mathbf{c}),\ldots,s_{j,r,p_{j,r}}(\mathbf{c})].\]

We therefore conclude that
\[\Lambda((h_{i},\mathbf{a})*(h_{j},\mathbf{b}))=(\Lambda(h_{i},\mathbf{a}))*(\Lambda(h_{j},\mathbf{b})).\]

2.

$\rightarrow$ Suppose that $(h_{i},\mathbf{a})\in \mathrm{Li}(\Gamma(\mathcal{X}^{\sharp}))$. Then $(h_{i},\mathbf{a})*(h_{j},\mathbf{b})=(h_{j},\mathbf{b})$ for all $j,\mathbf{b}$. Therefore, $h_{i}(\mathbf{a},x)=x$ for all $x\in X$. Thus,
\[L_{t_{i,1},\mathbf{a}}\circ\ldots\circ L_{t_{i,m_{i}},\mathbf{a}}(x)=x\]
for all $x\in X$. Therefore, whenever

\[[f_{i},\mathbf{b}]\in\Gamma(\mathcal{X})/\mathrm{Li}(\Gamma(\mathcal{X})),\]

we have

\[\Lambda(h_{i},\mathbf{a})*[f_{j},\mathbf{b}]\]
\[=([f_{i,1},s_{i,1,1}(\mathbf{a}),\ldots,s_{i,1,p_{i,1}}(\mathbf{a})]\circ\ldots\circ
[s_{i,m_{i}},s_{i,m_{i},1}(\mathbf{a}),\ldots,s_{i,m_{i},p_{i,m_{i}}}(\mathbf{a})])*[f_{j},\mathbf{b}]\]
\[=[f_{j},L_{t_{i,1},\mathbf{a}}\circ\ldots\circ L_{t_{i,m_{i}},\mathbf{a}}(\mathbf{b})]=[f_{j},\mathbf{b}].\]

Therefore, since $1$ is the only right identity in $\Gamma(\mathcal{X})/\mathrm{Li}(\Gamma(\mathcal{X}))$, we have $[f_{i},\mathbf{b}]=1$.

$\leftarrow$ If 
\[(h_{i},\mathbf{a})\not\in \mathrm{Li}(\Gamma(\mathcal{X}^{*})),\]
then
\[(h_{j},h_{i}(\mathbf{a},\mathbf{b}))=(h_{i},\mathbf{a})*(h_{j},\mathbf{b})\neq(h_{j},\mathbf{b})\]
for some $h_{j},\mathbf{b}$, so
$L_{h_{i},\mathbf{a}}(\mathbf{b})\neq\mathbf{b}$, so
\[L_{t_{i,1},\mathbf{a}}\circ\ldots\circ L_{t_{i,m_{i}},\mathbf{a}}(x)\neq x\]
for some $x$.

Therefore, there is some $j\in\{1,\ldots,m_{i}\}$ where $L_{t_{i,j},\mathbf{a}}$ is not the identity mapping.

Therefore there is some $x\in X$ with
\[x\neq t_{i,j}(\mathbf{a},x)=f_{i,j}(s_{i,j,1}(\mathbf{a}),\ldots,s_{i,j,p_{i,j}(\mathbf{a})},x),\]
hence
\[(f_{i,j},s_{i,j,1}(\mathbf{a}),\ldots,s_{i,j,p_{i,j}}(\mathbf{a}))\not\in \mathrm{Li}(\Gamma(\mathcal{X})),\]
and thus
\[[f_{i,j},s_{i,j,1}(\mathbf{a}),\ldots,s_{i,j,p_{i,j}}(\mathbf{a})]\neq 1.\]
Therefore

\[\Lambda(h_{i},\mathbf{a})\]
\[=[f_{i,1},s_{i,1,1}(\mathbf{a}),\ldots,s_{i,1,p_{i,1}}(\mathbf{a})]\circ\ldots\circ
[f_{i,m_{i}},s_{i,m_{i},1}(\mathbf{a}),\ldots,s_{i,m_{i},p_{i,m_{i}}}(\mathbf{a})]\neq 1.\]

3-6. These statements are left to the reader.
\end{proof}
\begin{prop}
(partial hull) Suppose that $\mathcal{X}=(X,(f^{\mathcal{X}})_{f\in E})$ is an algebra.
Suppose that each function symbol $f$ has arity $mn_{f}+1$. Then for each $f\in E$, let
$f_{[\![m]\!]}:(X^{m})^{n_{f}+1}\rightarrow X^{m}$ be the function where
\[f_{[\![m]\!]}((x_{1,1},\ldots,x_{1,m}),\ldots,(x_{n_{f},1},\ldots,x_{n_{f},m}),(y_{1},\ldots,y_{m}))\]
\[=(f(x_{1,1},\ldots,x_{1,m},\ldots,x_{n_{f},1},\ldots,x_{n_{f},m},y_{1}),\]
\[\ldots,f(x_{1,1},\ldots,x_{1,m},\ldots,x_{n_{f},1},\ldots,x_{n_{f},m},y_{m})).\]
Let $\Gamma_{m}(\mathcal{X})=(X^{m},(f_{[\![m]\!]})_{f\in E})$. Then $\Gamma_{n_{f}}(\mathcal{X})$ is endomorphic if and only if
$\mathcal{X}$ is endomorphic. Furthermore, the mapping $\phi:\Gamma(\mathcal{X})\rightarrow\Gamma(\Gamma_{m}(\mathcal{X}))$ defined by
\[\phi(f,a_{1,1},\ldots,a_{1,m},\ldots.,a_{n_{f},1},\ldots,a_{n_{f},m})\]
\[=(f_{[\![m]\!]},(a_{1,1},\ldots,a_{1,m}),\ldots,(a_{n_{f},1},\ldots,a_{n_{f},m}))\]
is an isomorphism of LD-systems. In particular, if $\mathcal{X}$ is endomorphic, then
\begin{enumerate}
\item $\Gamma_{m}(\mathcal{X})$ is permutative if and only if $\mathcal{X}$ is permutative,

\item $\Gamma_{m}(\mathcal{X})$ is Laver-like if and only if $\mathcal{X}$ is Laver-like, and

\item $\Gamma_{m}(\mathcal{X})$ is locally Laver-like whenever $\mathcal{X}$ is locally Laver-like.
\end{enumerate}
\end{prop}
\begin{thm}
Suppose that $E,F$ are sets of function symbols, each $f\in E\cup F$ has arity $n_{f}$, and $f^{\sharp}$ has arity $n_{f}+1$ for each
$f\in E$. Let $\mathcal{V}=(V,(f^{\sharp})_{f\in E},(\mathfrak{g}^{\mathcal{V}})_{\mathfrak{g}\in F})$ be a partially endomorphic algebra.
Let $f_{x}$ be a function symbol of arity $n_{f}$ whenever $f\in E,x\in X$ and let $\mathcal{G}=\{f_{x}\mid f\in E,x\in X\}\cup F$. Suppose now that $(L,(f^{\sharp})_{f\in E},(\mathfrak{g})_{\mathfrak{g}\in F})$ is a partially pre-endomorphic Laver table for some
$L\subseteq T[\mathcal{G}]$. Suppose now that $v_{x}\in V$ for each $x\in X$. Define a mapping $\phi:L\rightarrow V$ by induction on $(L,\preceq)$ by letting
\begin{enumerate}
\item $\phi(x)=v_{x}$ for each $x\in X$,

\item\label{4itjiot1hnqoqfi}\[\phi(\mathfrak{g}(\ell_{1},\ldots,\ell_{n_{\mathfrak{g}}}))=\mathfrak{g}^{\mathcal{V}}(\phi(\ell_{1}),\ldots,\phi(\ell_{n_{\mathfrak{g}}}))\]
whenever $\ell_{1},\ldots,\ell_{n_{\mathfrak{g}}}\in L$, and

\item \[\phi(f_{x}(\ell_{1},\ldots,\ell_{n_{f}}))=f^{\sharp}(\phi(\ell_{1}),\ldots,\phi(\ell_{n_{f}}),v_{x})\]
whenever $f_{x}(\ell_{1},\ldots,\ell_{n_{f}})\in L$.
\end{enumerate}
Suppose now that whenever $\ell_{1},\ldots,\ell_{n_{f}}\in L$ but $f_{x}(\ell_{1},\ldots,\ell_{n_{f}})\not\in L$, we have
\[f^{\sharp}(\phi(\ell_{1}),\ldots,\phi(\ell_{n_{f}}),v_{x})=v_{x}.\]
Then the mapping $\phi$ is a homomorphism between algebras.
\end{thm}
\begin{proof}
By property \ref{4itjiot1hnqoqfi}, we conclude that $\phi$ is a homomorphism with respect to the function symbols $\mathfrak{g}\in F$. We shall now show that
$\phi$ is a homomorphism with respect to the function symbols $f^{\sharp}$ for each $f\in E$ by showing that
\[\phi(f^{\sharp}(\ell_{1},\ldots,\ell_{n},\ell))=f^{\sharp}(\phi(\ell_{1}),\ldots,\phi(\ell_{n}),\phi(\ell))\]
for all $\ell_{1},\ldots,\ell_{n},\ell$  by induction on
\[(\mathrm{rank}(f,\ell_{1},\ldots,\ell_{n}),\ell)\in\mathbf{On}\times(L,\preceq)\]
using the usual cases.

Case 1: $f_{x}(\ell_{1},\ldots,\ell_{n})\not\in L,\ell=x$.

We have 
\[\phi(f^{\sharp}(\ell_{1},\ldots,\ell_{n},x))=\phi(x)=v_{x}\]
and
\[f^{\sharp}(\phi(\ell_{1}),\ldots,\phi(\ell_{n}),\phi(x))\]
\[=f^{\sharp}(\phi(\ell_{1}),\ldots,\phi(\ell_{n}),v_{x})=v_{x}.\]
Therefore,
\[\phi(f^{\sharp}(\ell_{1},\ldots,\ell_{n},x))=f^{\sharp}(\phi(\ell_{1}),\ldots,\phi(\ell_{n}),\phi(x))\]
in this case.

Case 2: $f_{x}(\ell_{1},\ldots,\ell_{n})\not\in L,\ell=\mathfrak{g}(u_{1},\ldots,u_{m})$.

\[\phi(f^{\sharp}(\ell_{1},\ldots,\ell_{n},\ell))=\phi(\ell)=\phi(\mathfrak{g}(u_{1},\ldots,u_{m}))\]

\[=\mathfrak{g}(\phi(f^{\sharp}(\overline{\ell},u_{1})),\ldots,\phi(f^{\sharp}(\overline{\ell},u_{m})))\]

\[=\mathfrak{g}(f^{\sharp}(\phi(\ell_{1}),\ldots,\phi(\ell_{n}),\phi(u_{1})),\ldots,f^{\sharp}(\phi(\ell_{1}),\ldots,\phi(\ell_{n}),\phi(u_{m})))\]

\[=f^{\sharp}(\phi(\ell_{1}),\ldots,\phi(\ell_{n}),\mathfrak{g}(\phi(u_{1}),\ldots,\phi(u_{m})))\]

\[=f^{\sharp}(\phi(\ell_{1}),\ldots,\phi(\ell_{n}),\phi(\mathfrak{g}(u_{1},\ldots,u_{m})))\]

\[=f^{\sharp}(\phi(\ell_{1}),\ldots,\phi(\ell_{n}),\phi(\ell)).\]

Case 3: $f_{x}(\ell_{1},\ldots,\ell_{n})\not\in L,\ell=g_{x}(u_{1},\ldots,u_{m})$.

\[f^{\sharp}(\phi(\ell_{1}),\ldots,\phi(\ell_{n}),\phi(\ell))\]

\[=f^{\sharp}(\phi(\ell_{1}),\ldots,\phi(\ell_{n}),\phi(g_{x}(u_{1},\ldots,u_{m})))\]

\[=f^{\sharp}(\phi(\ell_{1}),\ldots,\phi(\ell_{n}),g^{\sharp}(\phi(u_{1}),\ldots,\phi(u_{m}),\phi(x)))\]

\[=g^{\sharp}(f^{\sharp}(\phi(\ell_{1}),\ldots,\phi(\ell_{n}),\phi(u_{1})),\ldots,f^{\sharp}(\phi(\ell_{1}),\ldots,\phi(\ell_{n}),\phi(u_{m})),
f^{\sharp}(\phi(\ell_{1}),\ldots,\phi(\ell_{n}),\phi(x)))\]

\[=g^{\sharp}(\phi(f^{\sharp}(\overline{\ell},u_{1})),\ldots,\phi(f^{\sharp}(\overline{\ell},u_{m})),\phi(x))\]

\[=g^{\sharp}(\phi(u_{1}),\ldots,\phi(u_{m}),\phi(x))\]

\[=\phi(g_{x}(u_{1},\ldots,u_{m}))=\phi(\ell)=\phi(f^{\sharp}(\ell_{1},\ldots,\ell_{n},\ell)).\]

Case 4: $f_{x}(\ell_{1},\ldots,\ell_{n})\in L,\ell=x$ for some $x\in X$.

We have
\[\phi(f^{\sharp}(\ell_{1},\ldots,\ell_{n},x))=\phi(f_{x}(\ell_{1},\ldots,\ell_{n}))\]
\[=f^{\sharp}(\phi(\ell_{1}),\ldots,\phi(\ell_{n}),\phi(x)).\]

Case 5: $f_{x}(\ell_{1},\ldots,\ell_{n})\in L,\ell=\mathfrak{g}(u_{1},\ldots,u_{m})$.

\[\phi(f^{\sharp}(\ell_{1},\ldots,\ell_{n},\ell))\]

\[=\phi(f^{\sharp}(\ell_{1},\ldots,\ell_{n},\mathfrak{g}(u_{1},\ldots,u_{m})))\]

\[=\phi(\mathfrak{g}(f^{\sharp}(\overline{\ell},u_{1}),\ldots,f^{\sharp}(\overline{\ell},u_{m})))\]

\[=\mathfrak{g}(\phi(f^{\sharp}(\overline{\ell},u_{1})),\ldots,\phi(f^{\sharp}(\overline{\ell},u_{1})))\]

\[=\mathfrak{g}(f^{\sharp}(\phi(\ell_{1}),\ldots,\phi(\ell_{n}),\phi(u_{1})),\ldots,f^{\sharp}(\phi(\ell_{1}),\ldots,\phi(\ell_{n}),\phi(u_{m})))\]

\[=f^{\sharp}(\phi(\ell_{1}),\ldots,\phi(\ell_{n}),\mathfrak{g}(\phi(u_{1}),\ldots,\phi(u_{m})))\]

\[=f^{\sharp}(\phi(\ell_{1}),\ldots,\phi(\ell_{n}),\phi(\mathfrak{g}(u_{1},\ldots,u_{m})))\]

\[=f^{\sharp}(\phi(\ell_{1}),\ldots,\phi(\ell_{n}),\phi(\ell)).\]

Case 6: $f_{x}(\ell_{1},\ldots,\ell_{n})\in L,\ell=g_{x}(u_{1},\ldots,u_{m})$.

\[\phi(f^{\sharp}(\ell_{1},\ldots,\ell_{n},\ell))\]

\[=\phi(f^{\sharp}(\ell_{1},\ldots,\ell_{n},g_{x}(u_{1},\ldots,u_{m})))\]

\[=\phi(g^{\sharp}(f^{\sharp}(\overline{\ell},u_{1}),\ldots,f^{\sharp}(\overline{\ell},u_{m}),f^{\sharp}(\overline{\ell},x)))\]

\[=g^{\sharp}(\phi(f^{\sharp}(\overline{\ell},u_{1})),\ldots,\phi(f^{\sharp}(\overline{\ell},u_{m})),\phi(f^{\sharp}(\overline{\ell},x)))\]

\[=g^{\sharp}(f^{\sharp}(\phi(\ell_{1}),\ldots,\phi(\ell_{n}),\phi(u_{1})),\ldots,f^{\sharp}(\phi(\ell_{1}),\ldots,\phi(\ell_{n}),\phi(u_{m})),
f^{\sharp}(\phi(\ell_{1}),\ldots,\phi(\ell_{n}),\phi(x)))\]

\[=f^{\sharp}(\phi(\ell_{1}),\ldots,\phi(\ell_{n}),g^{\sharp}(\phi(u_{1}),\ldots,\phi(u_{m}),\phi(x)))\]

\[=f^{\sharp}(\phi(\ell_{1}),\ldots,\phi(\ell_{n}),\phi(g_{x}(u_{1},\ldots,u_{m})))\]

\[=f^{\sharp}(\phi(\ell_{1}),\ldots,\phi(\ell_{n}),\phi(\ell)).\]
\end{proof}
\begin{thm}
Let $\mathcal{V}=(V,(f^{\sharp})_{f\in E},(\mathfrak{g}^{\mathcal{V}})_{\mathfrak{g}\in F})$ be a locally Laver-like partially endomorphic algebra. Let $X$ be a set, and for each $x\in X$, let $v_{x}\in V$. Let $f_{x}$ be a function symbol whenever $f\in E,x\in X$, and let $\mathcal{G}=\{f_{x}\mid f\in E,x\in X\}\cup F$. Define a subset $L\subseteq\mathbf{T}_{\mathcal{G}}[X]$ and a function
$\phi:L\rightarrow V$ by induction on $\preceq$ according to the following rules.
\begin{enumerate}[(i)]
\item $x\in L$ and $\phi(x)=v_{x}$ for each $x\in X$.

\item $\mathfrak{g}(\ell_{1},\ldots,\ell_{n})\in L$ if and only if $\ell_{1},\ldots,\ell_{n}\in L$. Furthermore, if
$\mathfrak{g}(\ell_{1},\ldots,\ell_{n})\in L$, then 
\[\phi(\mathfrak{g}(\ell_{1},\ldots,\ell_{n}))=\mathfrak{g}^{\mathcal{V}}(\phi(\ell_{1}),\ldots,\phi(\ell_{n})).\]

\item $f_{x}(\ell_{1},\ldots,\ell_{n})\in L$ if and only if $\ell_{1},\ldots,\ell_{n}\in L$ and
\[(f,\phi(\ell_{1}),\ldots,\phi(\ell_{n}))\not\in \mathrm{Li}(\Gamma(X)).\]
Furthermore, if $f_{x}(\ell_{1},\ldots,\ell_{n})\in L$, then define 
\[\phi(f_{x}(\ell_{1},\ldots,\ell_{n}))=f^{\sharp}(\phi(\ell_{1}),\ldots,\phi(\ell_{n}),v_{x}).\]
\end{enumerate}
Then 
\begin{enumerate}
\item The set $L$ is locally operable.

\item The mapping $\phi$ is a homomorphism.

\item The algebra $(L,(f^{\sharp})_{f\in E},(\mathfrak{g})_{\mathfrak{g}\in F})$ is a partially endomorphic Laver table.

\item If $\mathcal{V}$ is a Laver-like partially endomorphic algebra, then $L$ is operable.
\end{enumerate}
\label{23rijo}
\end{thm}
\begin{proof}
For simplicity, in this proof, we shall assume that $\mathcal{V}=(V,(f^{\sharp})_{f\in E},(g^{\mathcal{V}})_{g\in F})$ is Laver-like. Then define a map
$\mho:\Gamma(\mathcal{V})\rightarrow\mathbf{On}$ by letting $\mho(f,a_{1},\ldots,a_{n})$ denote the rank of
$(f,a_{1},\ldots,a_{n})$ in the well-founded partial ordering $(\Gamma(\mathcal{V}),\preceq^{*})$.
Define a mapping $\nabla:\Gamma(L)\rightarrow\mathbf{On}$ by letting 
\[\nabla(f,\ell_{1},\ldots,\ell_{n})=\mho(f,\phi(\ell_{1}),\ldots,\phi(\ell_{n})).\]

Claim 1: I claim that $\nabla(f,\ell_{1},\ldots,\ell_{n})=0$ if and only if $f_{x}(\ell_{1},\ldots,\ell_{n})=0$.

We have 
\[\nabla(f,\ell_{1},\ldots,\ell_{n})=0\]
if and only if
\[\mho(f,\phi(\ell_{1}),\ldots,\phi(\ell_{n}))=0\]
if and only if
\[(f,\phi(\ell_{1}),\ldots,\phi(\ell_{n}))\in \mathrm{Li}(\Gamma(\mathcal{V}))\]
if and only if
\[f_{x}(\ell_{1},\ldots,\ell_{n})\not\in L.\]

Now define operations $f^{\sharp}$ for each $f\in E$ by defining
$f^{\sharp}(\ell_{1},\ldots,\ell_{n},\ell)$ by induction on $(\nabla(f,\ell_{1},\ldots,\ell_{n}),\ell)\in\mathbf{On}\times L$ by the following rules:

\begin{enumerate}[(i)]

\item If $\nabla(f,\ell_{1},\ldots,\ell_{n})=0$, then define 
\[f^{\sharp}(\ell_{1},\ldots,\ell_{n},\ell)=\ell.\]

\item If $\nabla(f,\ell_{1},\ldots,\ell_{n})>0$ and $\ell=x$, then 
\[f^{\sharp}(\ell_{1},\ldots,\ell_{n},x)=f_{x}(\ell_{1},\ldots,\ell_{n}).\]

\item If $\nabla(f,\ell_{1},\ldots,\ell_{n})>0$ and $\ell=\mathfrak{g}(u_{1},\ldots,u_{m})$, then
\[f^{\sharp}(\ell_{1},\ldots,\ell_{n},\ell)\]
\[=\mathfrak{g}(f^{\sharp}(\ell_{1},\ldots,\ell_{n},u_{1}),\ldots,f^{\sharp}(\ell_{1},\ldots,\ell_{n},u_{m}))\]
whenever 
\[f^{\sharp}(\ell_{1},\ldots,\ell_{n},u_{1}),\ldots,f^{\sharp}(\ell_{1},\ldots,\ell_{n},u_{m})\]
have been defined previously in the induction, and leave $f^{\sharp}(\ell_{1},\ldots,\ell_{n},\ell)$ undefined otherwise.

\item If $\nabla(f,\ell_{1},\ldots,\ell_{n})>0$ and $\ell=g_{x}(u_{1},\ldots,u_{m})$, then

\[f^{\sharp}(\ell_{1},\ldots,\ell_{n},\ell)\]
\[=g^{\sharp}(f^{\sharp}(\ell_{1},\ldots,\ell_{n},u_{1}),\ldots,f^{\sharp}(\ell_{1},\ldots,\ell_{n},u_{m}),f_{x}(\ell_{1},\ldots,\ell_{n}))\]
whenever
\[g^{\sharp}(f^{\sharp}(\ell_{1},\ldots,\ell_{n},u_{1}),\ldots,f^{\sharp}(\ell_{1},\ldots,\ell_{n},u_{m}),f_{x}(\ell_{1},\ldots,\ell_{n}))\]
has been defined previously in the induction and
$f^{\sharp}(\ell_{1},\ldots,\ell_{n},\ell)$ shall be left undefined otherwise.
\end{enumerate}

We shall now show that $f^{\sharp}(\ell_{1},\ldots,\ell_{n},\ell)$ is defined and
\[\phi(f^{\sharp}(\ell_{1},\ldots,\ell_{n},\ell))=f^{\sharp}(\phi(\ell_{1}),\ldots,\phi(\ell_{n}),\phi(\ell))\]
by induction on
\[(\nabla(f,\ell_{1},\ldots,\ell_{n}),\ell)\in\mathbf{On}\times L.\]

\begin{enumerate}
\item[Case I:] $\nabla(f,\ell_{1},\ldots,\ell_{n})=0$.

We have $f^{\sharp}(\ell_{1},\ldots,\ell_{n},\ell)=\ell$, so
$f^{\sharp}(\ell_{1},\ldots,\ell_{n},\ell)$ is defined.

We have
\[\phi(f^{\sharp}(\ell_{1},\ldots,\ell_{n},\ell))=\phi(\ell)\]
and since
\[(f,\phi(\ell_{1}),\ldots,\phi(\ell_{n}))\in \mathrm{Li}(\mathcal{V}),\]
we also have
\[f^{\sharp}(\phi(\ell_{1}),\ldots,\phi(\ell_{n}),\phi(\ell))=\phi(\ell).\]

\item[Case II:] $\nabla(f,\ell_{1},\ldots,\ell_{n})>0,\ell=x\in X$.

We have 
\[f^{\sharp}(\ell_{1},\ldots,\ell_{n},\ell)=f_{x}(\ell_{1},\ldots,\ell_{n})\]
in this case, so
$f^{\sharp}(\ell_{1},\ldots,\ell_{n},\ell)$ is defined.

Furthermore,
\[\phi(f^{\sharp}(\ell_{1},\ldots,\ell_{n},\ell))=\phi(f_{x}(\ell_{1},\ldots,\ell_{n}))\]
\[=f^{\sharp}(\phi(\ell_{1}),\ldots,\phi(\ell_{n}),v_{x})=f^{\sharp}(\phi(\ell_{1}),\ldots,\phi(\ell_{n}),\phi(\ell)).\]

\item[Case III:] $\nabla(f,\ell_{1},\ldots,\ell_{n})>0,\ell=\mathfrak{g}(u_{1},\ldots,u_{m}),\mathfrak{g}\in F.$

We have
\[f^{\sharp}(\ell_{1},\ldots,\ell_{n},\ell)\]
\[=\mathfrak{g}(f^{\sharp}(\ell_{1},\ldots,\ell_{n},u_{1}),\ldots,f^{\sharp}(\ell_{1},\ldots,\ell_{n},u_{m})),\]
so $f^{\sharp}(\ell_{1},\ldots,\ell_{n},\ell)$ is defined in this case. Furthermore,

\[\phi(f^{\sharp}(\ell_{1},\ldots,\ell_{n},\ell))\]
\[=\phi(\mathfrak{g}(f^{\sharp}(\ell_{1},\ldots,\ell_{n},u_{1}),\ldots,f^{\sharp}(\ell_{1},\ldots,\ell_{n},u_{m})))\]
\[=\mathfrak{g}(\phi(f^{\sharp}(\ell_{1},\ldots,\ell_{n},u_{1})),\ldots,\phi(f^{\sharp}(\ell_{1},\ldots,\ell_{n},u_{m})))\]
\[=\mathfrak{g}(f^{\sharp}(\phi(\ell_{1}),\ldots,\phi(\ell_{n}),\phi(u_{1})),\ldots,f^{\sharp}(\phi(\ell_{1}),\ldots,\phi(\ell_{n}),\phi(u_{m})))\]
\[=f^{\sharp}(\phi(\ell_{1}),\ldots,\phi(\ell_{n}),\mathfrak{g}(\phi(u_{1}),\ldots,\phi(u_{m})))\]
\[=f^{\sharp}(\phi(\ell_{1}),\ldots,\phi(\ell_{n}),\phi(\mathfrak{g}(u_{1},\ldots,u_{m})))\]
\[=f^{\sharp}(\phi(\ell_{1}),\ldots,\phi(\ell_{n}),\phi(\ell)).\]

\item[Case IV:] $\nabla(f,\ell_{1},\ldots,\ell_{n})>0,\ell=g_{x}(u_{1},\ldots,u_{m})$.

By the induction hypothesis, 

\[\nabla(g,f^{\sharp}(\ell_{1},\ldots,\ell_{n},u_{1}),\ldots,f^{\sharp}(\ell_{1},\ldots,\ell_{n},u_{m}))\]
\[=\mho(g,\phi(f^{\sharp}(\ell_{1},\ldots,\ell_{n},u_{1})),\ldots,\phi(f^{\sharp}(\ell_{1},\ldots,\ell_{n},u_{m})))\]
\[=\mho(g,f^{\sharp}(\phi(\ell_{1}),\ldots,\phi(\ell_{n}),\phi(u_{1})),\ldots,f^{\sharp}(\phi(\ell_{1}),\ldots,\phi(\ell_{n}),\phi(u_{m})))\]
\[=\mho((f,\phi(\ell_{1}),\ldots,\phi(\ell_{n}))*(g,\phi(u_{1}),\ldots,\phi(u_{m})))\]
\[<\mho(f,\phi(\ell_{1}),\ldots,\phi(\ell_{n}))\]
\[=\nabla(f,\ell_{1},\ldots,\ell_{n}).\]

Therefore, 
\[g^{\sharp}(f^{\sharp}(\ell_{1},\ldots,\ell_{n},u_{1}),\ldots,f^{\sharp}(\ell_{1},\ldots,\ell_{n},u_{m}),f_{x}(\ell_{1},\ldots,\ell_{n}))\]
has already been constructed in the induction process and
\[f^{\sharp}(\ell_{1},\ldots,\ell_{n},\ell)\]
\[=g^{\sharp}(f^{\sharp}(\ell_{1},\ldots,\ell_{n},u_{1}),\ldots,f^{\sharp}(\ell_{1},\ldots,\ell_{n},u_{m}),f_{x}(\ell_{1},\ldots,\ell_{n})).\]

We also have

\[\phi(f^{\sharp}(\ell_{1},\ldots,\ell_{n},\ell))\]

\[=\phi(g^{\sharp}(f^{\sharp}(\ell_{1},\ldots,\ell_{n},u_{1}),\ldots,f^{\sharp}(\ell_{1},\ldots,\ell_{n},u_{m}),f_{x}(\ell_{1},\ldots,\ell_{n})))\]

\[=g^{\sharp}(\phi(f^{\sharp}(\ell_{1},\ldots,\ell_{n},u_{1})),\ldots,\phi(f^{\sharp}(\ell_{1},\ldots,\ell_{n},u_{m})),
\phi(f_{x}(\ell_{1},\ldots,\ell_{n})))\]

\[=g^{\sharp}(f^{\sharp}(\phi(\ell_{1}),\ldots,\phi(\ell_{n}),\phi(u_{1})),\ldots,f^{\sharp}(\phi(\ell_{1}),\ldots,\phi(\ell_{n}),\phi(u_{m})),
f^{\sharp}(\phi(\ell_{1}),\ldots,\phi(\ell_{n}),\phi(x)))\]

\[=f^{\sharp}(\phi(\ell_{1}),\ldots,\phi(\ell_{n}),g^{\sharp}(\phi(u_{1}),\ldots,\phi(u_{m}),\phi(x)))\]

\[=f^{\sharp}(\phi(\ell_{1}),\ldots,\phi(\ell_{n}),\phi(g_{x}(u_{1},\ldots,u_{m})))\]

\[=f^{\sharp}(\phi(\ell_{1}),\ldots,\phi(\ell_{n}),\phi(\ell)).\]

\end{enumerate}

Claim: I claim that $L$ is partially endomorphic.

Suppose that $\nabla(g,u_{1},\ldots,u_{m})=0$. Then
\[\mho(g,\phi(u_{1}),\ldots,\phi(u_{m}))=0,\]
hence 
\[(g,\phi(u_{1}),\ldots,\phi(u_{m}))\in \mathrm{Li}(\Gamma(\mathcal{V})).\]
Therefore,
\[(f,\phi(\ell_{1}),\ldots,\phi(\ell_{n}))*(g,\phi(u_{1}),\ldots,\phi(u_{m}))\in \mathrm{Li}(\Gamma(\mathcal{V})),\]
but
\[(f,\phi(\ell_{1}),\ldots,\phi(\ell_{n}))*(g,\phi(u_{1}),\ldots,\phi(u_{m}))\]
\[=(g,f^{\sharp}(\phi(\ell_{1}),\ldots,\phi(\ell_{n}),\phi(u_{1})),\ldots,f^{\sharp}(\phi(\ell_{1}),\ldots,\phi(\ell_{n}),\phi(u_{m}))\]
\[=(g,\phi(f^{\sharp}(\ell_{1},\ldots,\ell_{n},u_{1})),\ldots,\phi(f^{\sharp}(\ell_{1},\ldots,\ell_{n},u_{m})),\]
so
\[(g,\phi(f^{\sharp}(\ell_{1},\ldots,\ell_{n},u_{1})),\ldots,\phi(f^{\sharp}(\ell_{1},\ldots,\ell_{n},u_{m}))\in \mathrm{Li}(\Gamma(\mathcal{V})),\]
hence
\[\mho(g,\phi(f^{\sharp}(\ell_{1},\ldots,\ell_{n},u_{1})),\ldots,\phi(f^{\sharp}(\ell_{1},\ldots,\ell_{n},u_{m}))=0,\]
thus
\[\nabla(g,f^{\sharp}(\ell_{1},\ldots,\ell_{n},u_{1}),\ldots,f^{\sharp}(\ell_{1},\ldots,\ell_{n},u_{m}))=0\]
as well. We therefore conclude that $L$ is partially endomorphic.
\end{proof}
We shall write $\mathbf{L}((v_{x})_{x\in X},\mathcal{V},(f^{\sharp})_{f\in E},(\mathfrak{g}^{\mathcal{V}})_{\mathfrak{g\in F}})$ for the algebra denoted by $L$ in Theorem \ref{23rijo}.
\begin{prop}
Let $\mathcal{X}$ be a partially endomorphic algebra.
\begin{enumerate}
\item $\mathcal{X}$ is Laver-like if and only if $\mathcal{X}$ is the surjective homomorphic image of some well-founded partially endomorphic Laver table.

\item $\mathcal{X}$ is locally Laver-like if and only if $\mathcal{X}$ is the surjective homomorphic image of some partially endomorphic Laver table.
\end{enumerate}
\end{prop}

\begin{proof}
$\rightarrow$ This follows from the facts that every partially endomorphic Laver table is locally Laver-like,
every well-founded partially endomorphic Laver table is Laver-like, and the quotient of a (locally) Laver-like partially endomorphic algebra is a (locally) Laver-like partially endomorphic algebra.

$\leftarrow$ This follows from the fact that the mapping $\phi:L\rightarrow\mathcal{X}$ described in Theorem \ref{23rijo} is a surjective homomorphism, and $L$ is well-founded whenever $\mathcal{X}$ is Laver-like.
\end{proof}
\begin{prop}
\begin{enumerate}
\item Every partially endomorphic Laver-table is of the form
\[\mathbf{L}((v_{x})_{x\in X},\mathcal{V},(f^{\sharp})_{f\in E},(\mathfrak{g}^{\mathcal{V}})_{\mathfrak{g\in F}})\]
for some locally Laver-like partially endomorphic algebra
\[(\mathcal{V},(f^{\sharp})_{f\in E},(\mathfrak{g})_{\mathfrak{g\in F}}).\]

\item Every well-founded partially endomorphic Laver table is of the form
\[\mathbf{L}((v_{x})_{x\in X},\mathcal{V},(f^{\sharp})_{f\in E},(\mathfrak{g}^{\mathcal{V}})_{\mathfrak{g\in F}})\]
for some Laver-like partially endomorphic algebra
\[(\mathcal{V},(f^{\sharp})_{f\in E},(\mathfrak{g})_{\mathfrak{g\in F}}).\]
\end{enumerate}
\end{prop}

Suppose that $(X,*)$ is a permutative LD-system and $\mathcal{F}$ is a set of operations on $X$ such that for each
$f\in\mathcal{F}$ the identity $f(x*x_{1},\ldots,x*x_{n})=x*f(x_{1},\ldots,x_{n})$ is satisfied.
Then define an equivalence relation $\simeq_{\mathrm{cmx}}^{*,\mathcal{F}}$ on $X$ where we set
$x\simeq y$ if and only if 
\[t(a_{1},\ldots,a_{n},x)\in \mathrm{Li}(X)\leftrightarrow t(a_{1},\ldots,a_{n},y)\in \mathrm{Li}(X)\]
whenever $a_{1},\ldots,a_{n}\in X$. Then $\simeq_{\mathrm{cmx}}^{*,\mathcal{F}}$ is a congruence on $(X,*,\mathcal{F})$, and
$\simeq_{\mathrm{cmx}}^{*,\mathcal{F}}$ is the largest congruence such that if
$x\simeq_{\mathrm{cmx}}^{*,\mathcal{F}}y$, then $\mathrm{crit}(x)=\mathrm{crit}(y)$. If $\simeq_{\mathrm{cmx}}^{*,\mathcal{F}}$
is the identity relation, then we shall say that $(X,*,\mathcal{F})$ is critically simple over $*$. The following result gives a duality

\begin{prop}
\begin{enumerate}
\item Suppose that $*$ is a binary operation on a set $V$, and $\mathcal{V}=(V,*,\mathcal{F})$ is an locally Laver-like endomorphic algebra which is critically simple over $*$. Furthermore, suppose that $X$ is a set, $v_{x}\in V$ for each $x\in X$, and that $(v_{x})_{x\in X}$ generates $V$ in the algebra $(V,*,\mathcal{F})$. Then there is a unique isomorphism 
\[j:(V,*,\mathcal{F})\rightarrow\mathbf{L}((v_{x})_{x\in X},*,\mathcal{F})/\equiv^{*,\mathcal{F}}_{\mathrm{cmx}}\]
such that $j(v_{x})=[x]$ for each $x\in X$.

\item Suppose that $(L,*,\mathcal{F})$ is a partially endomorphic Laver table where $*$ is a binary operation on $L$. Then
\[(L,*,\mathcal{F})=\mathbf{L}(([x])_{x\in X},(L,*,\mathcal{F})/\equiv^{*,\mathcal{F}}_{\mathrm{cmx}}).\]
\end{enumerate}
\end{prop}

The following result shows that all endomorphic algebras can be easily obtained from algebras of the form $(X,*,(t_{i})_{i\in I})$ where $(X,*)$ is an LD-system and each $t_{i}$ satisfies the identity \[x*t_{i}(x_{1},...,x_{n})=t_{i}(x*x_{1},...,x*x_{n}).\]

\begin{prop}
Suppose that $(X,\mathcal{F},\mathcal{G})$ is a partially endomorphic algebra where $\mathcal{F}$ contains at least one fundamental operation. Then there is some 
algebra 
\[(Y,*,(f^{+})_{f\in\mathcal{F}},(\mathfrak{g}^{+})_{\mathfrak{g}\in\mathcal{G}})\] along with a function $\phi:X\rightarrow Y$ where
\begin{enumerate}
\item $(Y,*)$ is an LD-system,
	
\item $f^{+}$ has arity $n_{f}$ for each $f\in\mathcal{F}$,

\item $\mathfrak{g}^{+}$ has arity $n_{\mathfrak{g}}$ for each $\mathfrak{g}\in\mathcal{G}$,

\item $y*f^{+}(y_{1},\dots,y_{n_{f}})=f^{+}(y*y_{1},\dots,y*y_{n_{f}})$ whenever $f\in\mathcal{F},y,y_{1},...,y_{n_{f}}\in Y$

\item $y*\mathfrak{g}^{+}(y_{1},\dots,y_{n_{\mathfrak{g}}})=
\mathfrak{g}^{+}(y*y_{1},\dots,y*y_{n_{\mathfrak{g}}})$ whenever
$\mathfrak{g}\in\mathcal{G}$ and $y,y_{1},...,y_{n_{\mathfrak{g}}}\in\mathcal{G}$.

Define a mapping $f^{++}:Y^{n_{f}+1}\rightarrow Y$ by letting
\[f^{++}(y_{1},\dots,y_{n_{f}},y)=f^{+}(y_{1},\dots,y_{n_{f}})*y.\]

\item The function $\phi$ is an injective homomorphism from \[(X,\mathcal{F},\mathcal{G})\]
to
\[(Y,(f^{++})_{f\in\mathcal{F}},(\mathfrak{g}^{+})_{\mathfrak{g}\in\mathcal{G}}).\]

\item If $(X,\mathcal{F},\mathcal{G})$ is permutative, then
$(Y,*)$ is also permutative.

\item If $(X,\mathcal{F},\mathcal{G})$ is Laver-like, then $(Y,*)$ is also
Laver-like.
\end{enumerate}
\end{prop}
\begin{proof}
Let $Y=\Gamma(X,\mathcal{F})$. Let $t\in\mathcal{F}$ be an operation. Define
$\phi:X\rightarrow Y$ by letting $\phi(x)=(t,x,\dots,x)$, and define a mapping
\[f^{+}((f_{1},x_{1,1},\dots,x_{1,n_{f_{1}}}),\dots,(
f_{n_{f}},x_{n_{f},1},\dots,x_{n_{f},n_{f_{n_{f}}}}))\]
\[=(f,x_{1,1},\dots,x_{n_{f_{1}},1}).\]

Let
\[\mathfrak{g}^{+}((f_{1},x_{1,1},\dots,x_{1,n_{f_{1}}}),\dots,(
f_{n_{\mathfrak{g}}},x_{n_{\mathfrak{g}},1},\dots,x_{n_{\mathfrak{g}},n_{f_{n_{\mathfrak{g}}}}}))\]
\[=(t,\mathfrak{g}(x_{1,1},\dots,x_{n_{\mathfrak{g}},1}),\dots,
\mathfrak{g}(x_{1,1},\dots,x_{n_{\mathfrak{g}},1})).\]

We have

\[(h,x_{1},\dots,x_{n_{h}})*f^{+}((f_{1},x_{1,1},\dots,x_{1,n_{f_{1}}}),\dots,(
f_{n_{f}},x_{n_{f},1},\dots,x_{n_{f},n_{f_{n_{f}}}}))\]
\[=(h,x_{1},\dots,x_{n_{h}})*(f,x_{1,1},\dots,x_{n_{f},1})\]
\[=(f,h(x_{1},\dots,x_{n_{h}},x_{1,1}),\dots,h(x_{1},\dots,x_{n_{h}},x_{n_{f},1})).\]

We have

\[f^{+}((h,x_{1},\dots,x_{n_{h}})*(f_{1},x_{1,1},\dots,x_{1,n_{f_{1}}}),\dots,
(h,x_{1},\dots,x_{n_{h}})*(f_{n_{f}},x_{n_{f},1},\dots,x_{n_{f},n_{f_{n_{f}}}}))\]
\[=f^{+}((f_{1},h(x_{1},\dots,x_{n_{h}},x_{1,1}),\dots,h(x_{1},\dots,x_{n_{h}},x_{1,n_{f}})),\dots,\]
\[(f_{n_{f}},h(x_{1},\dots,x_{n_{h}},x_{n_{f},1}),\dots,h(x_{1},\dots,x_{n_{h}},x_{n_{f},n_{f_{n_{f}}}}))\]
\[=(f,h(x_{1},\dots,x_{n_{h}},x_{1,1}),\dots,h(x_{1},\dots,x_{n_{h}},x_{n_{f},1})).\]

We therefore conclude that
\[(h,x_{1},\dots,x_{n_{h}})*f^{+}((f_{1},x_{1,1},\dots,x_{1,n_{f_{1}}}),\dots,(
f_{n_{f}},x_{n_{f},1},\dots,x_{n_{f},n_{f_{n_{f}}}}))\]
\[=f^{+}((h,x_{1},\dots,x_{n_{h}})*(f_{1},x_{1,1},\dots,x_{1,n_{f_{1}}}),\dots,
(h,x_{1},\dots,x_{n_{h}})*(f_{n_{f}},x_{n_{f},1},\dots,x_{n_{f},n_{f_{n_{f}}}})).\]

We have

\[(h,x_{1},\dots,x_{n_{h}})*\mathfrak{g}^{+}((f_{1},x_{1,1},\dots,x_{1,n_{f_{1}}}),\dots,
(f_{n_{\mathfrak{g}}},x_{n_{\mathfrak{g}},1},\dots,x_{n_{\mathfrak{g}},n_{f_{n_{\mathfrak{g}}}}}))\]

\[=(h,x_{1},\dots,x_{n_{\mathfrak{g}}})*(t,\mathfrak{g}(x_{1,1},\dots,x_{n_{\mathfrak{g}},1}),\dots,\mathfrak{g}(x_{1,1},\dots,x_{n_{\mathfrak{g}},1}))\]

\[=(t,h(x_{1},\dots,x_{n_{\mathfrak{g}}},\mathfrak{g}(x_{1,1},\dots,x_{n_{\mathfrak{g}},1})),\dots,h(x_{1},\dots,x_{n_{\mathfrak{g}}},\mathfrak{g}(x_{1,1},\dots,x_{n_{\mathfrak{g}},1})))\]

\[=(t,\mathfrak{g}(h(x_{1},\dots,x_{n_{\mathfrak{g}}},x_{1,1}),\dots,\mathfrak{g}(h(x_{1},\dots,x_{n_{\mathfrak{g}}},x_{n_{\mathfrak{g}},1}))))\]

We have 

\[\mathfrak{g}^{+}((h,x_{1},\dots,x_{n_{h}})*(f_{1},x_{1,1},\dots,x_{1,n_{f_{1}}}),\dots,
(h,x_{1},\dots,x_{n_{h}})*(f_{n_{\mathfrak{g}}},x_{n_{\mathfrak{g}},1},\dots,x_{n_{\mathfrak{g}},n_{f_{n_{\mathfrak{g}}}}}))\]

\[=\mathfrak{g}^{+}((f_{1},h(x_{1},\dots,x_{n_{h}},x_{1,1}),\dots,
h(x_{1},\dots,x_{n_{h}},x_{1,n_{f_{1}}})),\dots,\]
\[(f_{n_{\mathfrak{g}}},h(x_{1},\dots,x_{n_{h}},x_{n_{\mathfrak{g}},1}),\dots,
h(x_{1},\dots,x_{n_{h}},x_{n_{\mathfrak{g}},n_{f_{n_{\mathfrak{g}}}}})))\]

\[=(t,\mathfrak{g}(h(x_{1},\dots,x_{n_{h}},x_{1,1}),\dots,h(x_{1},\dots,x_{n_{h}},x_{n_{\mathfrak{g}},1})).\]

We therefore conclude that

\[(h,x_{1},\dots,x_{n_{h}})*\mathfrak{g}^{+}((f_{1},x_{1,1},\dots,x_{1,n_{f_{1}}}),\dots,
(f_{n_{\mathfrak{g}}},x_{n_{\mathfrak{g}},1},\dots,x_{n_{\mathfrak{g}},n_{f_{n_{\mathfrak{g}}}}}))\]
\[=\mathfrak{g}^{+}((h,x_{1},\dots,x_{n_{h}})*(f_{1},x_{1,1},\dots,x_{1,n_{f_{1}}}),\dots,\]
\[(h,x_{1},\dots,x_{n_{h}})*(f_{n_{\mathfrak{g}}},x_{n_{\mathfrak{g}},1},\dots,x_{n_{\mathfrak{g}},n_{f_{n_{\mathfrak{g}}}}})).\]

We have

\[\phi(f(x_{1},\dots,x_{n_{f}},x))\]
\[=(t,f(x_{1},\dots,x_{n_{f}},x),\dots,f(x_{1},\dots,x_{n_{f}},x)).\]

On the other hand,

\[f^{++}(\phi(x_{1}),\dots,\phi(x_{n_{f}}),\phi(x))\]

\[=f^{+}(\phi(x_{1}),\dots,\phi(x_{n_{f}}))*\phi(x)\]

\[=f^{+}((t,x_{1},\dots,x_{1}),\dots,(t,x_{n_{f}},\dots,x_{n_{f}}))*(t,x,\dots,x)\]

\[=(f,x_{1},\dots,x_{n_{f}})*(t,x,\dots,x)\]

\[=(t,f(x_{1},\dots,x_{n_{f}},x),\dots,f(x_{1},\dots,x_{n_{f}},x))\]

\[=\phi(f(x_{1},\dots,x_{n_{f}},x)).\]

We have

\[\phi(\mathfrak{g}(x_{1},\dots,x_{n_{\mathfrak{g}}}))\]
\[=(t,\mathfrak{g}(x_{1},\dots,x_{n_{\mathfrak{g}}}),\dots,
\mathfrak{g}(x_{1},\dots,x_{n_{\mathfrak{g}}}))\]
and

\[\mathfrak{g}^{+}(\phi(x_{1}),\dots,\phi(x_{n_{\mathfrak{g}}}))\]

\[=\mathfrak{g}^{+}((t,x_{1},\dots,x_{1}),\dots,(t,x_{n_{\mathfrak{g}}},\dots,x_{n_{\mathfrak{g}}}))\]

\[=(t,\mathfrak{g}(x_{1},\dots,x_{n_{\mathfrak{g}}})).\]

We have
\[\phi(\mathfrak{g}(x_{1},\dots,x_{n_{\mathfrak{g}}}))\]
\[=\mathfrak{g}^{+}(\phi(x_{1}),\dots,\phi(x_{n_{\mathfrak{g}}})).\]

We therefore conclude that $\phi$ is a homomorphism.
\end{proof}

\subsection{Twistedly endomorphic algebras}
We shall now generalize the notion of a Laver table to the notion of a twistedly endomorphic Laver table.
The twistedly endomorphic Laver tables satisfy modified versions of the self-distributivity identities such as
\[t(a,b,t(x,y,z))=t(t(a,b,x),t(b,a,y),t(a,b,z)).\]

\begin{defn}
\label{utgowevbatoiuwbutri}
Let $E$ be a set of function symbols. Let $X$ be a set. Let $\mathcal{G}=\{f_{x}\mid x\in X\}$.
Let $L\subseteq\mathbf{T}_{\mathcal{G}}(X)$ be a downwards closed subset where $X\subseteq L$ and
$f_{x}(\ell_{1},\ldots,\ell_{n_{f}})\in L$ if and only if $f_{y}(\ell_{1},\ldots,\ell_{n_{f}})\in L$. Now suppose that each
function symbol $f$ is $n_{f}$-ary.

Suppose that if $f,g\in E$ and $i\in\{1,\ldots,n_{f}\}$, then 
\[\sigma_{f,g,i}:\{1,\ldots,n_{g}\}\rightarrow\{1,\ldots,n_{g}\}\]
is a function.

Suppose now that $*$ is the operation on $\bigcup_{f\in E}\{f\}\times\{1,\ldots,n_{f}\}$ defined by
$(f,i)*(g,j)=(g,\sigma_{g,f,i}(j))$. The operation $*$ is associative if and only if $\sigma_{f,h,i}\circ\sigma_{f,g,j}=\sigma_{f,g,\sigma_{g,h,i}(j)}$ whenever $f,g,h\in E$ and $i\in\{1,\ldots,n_{h}\}$ and $j\in\{1,\ldots,n_{g}\}$.
Now suppose that $*$ is an associative operation on $\bigcup_{f\in E}\{f\}\times\{1,\ldots,n_{f}\}$.

Let $\Omega=\{(f,\ell_{1},\ldots,\ell_{n})\mid f\in E,\ell_{1},\ldots,\ell_{n}\in L\}$.
Let $\nabla:\Omega\rightarrow\mathbf{On}$ be a function such that
$\nabla(f,\ell_{1},\ldots,\ell_{n})=0$ if and only if $f_{x}(\ell_{1},\ldots,\ell_{n})\not\in L$.

Suppose that for each $f\in E$ there is an $n+1$-ary operation $f^{\sharp}:L^{n+1}\rightarrow L$ such that
\begin{enumerate}
\item $f^{\sharp}(\ell_{1},\ldots,\ell_{n},\ell)=\ell$ whenever $\nabla(f,\ell_{1},\ldots,\ell_{n})=0$,

\item $f^{\sharp}(\ell_{1},\ldots,\ell_{n},x)=f_{x}(\ell_{1},\ldots,\ell_{n})$ whenever $\nabla(f,\ell_{1},\ldots,\ell_{n})>0$,

\item If $\nabla(f,\ell_{1},\ldots,\ell_{m})>0$, then
\[f^{\sharp}(\ell_{1},\ldots,\ell_{m},g_{x}(u_{1},\ldots,u_{n}))\]
\[=g^{\sharp}(f^{\sharp}(\ell_{\sigma_{f,g,1}(1)},\ldots,\ell_{\sigma_{f,g,1}(m)},u_{1}),\ldots
,f^{\sharp}(\ell_{\sigma_{f,g,n}(1)},\]
\[\ldots,\ell_{\sigma_{f,g,n}(m)},u_{n}),f_{x}(\ell_{1},\ldots,
\ell_{m})),\]
and

\item $\nabla(f,\ell_{1},\ldots,\ell_{m})$
\[>\nabla(g,f^{\sharp}(\ell_{\sigma_{f,g,1}(1)},\ldots,\ell_{\sigma_{f,g,1}(m)},u_{1}),\ldots,
f^{\sharp}(\ell_{\sigma_{f,g,n}(1)},\ldots,\ell_{\sigma_{f,g,n}(m)},u_{n}))\] whenever
$(f,\ell_{1},\ldots,\ell_{m})>0$.
\end{enumerate}
Then we shall call the algebra $(L,(f^{\sharp})_{f\in E})$ a well-founded twistedly pre-endomorphic Laver table of type $*$.

Suppose now that $V$ is a set and $f^{\bullet}:V^{n_{f}+1}\rightarrow V$ is a function for all $f\in E$.
Then we shall call the algebra $(V,(f^{\bullet})_{f\in E})$ a \index{twistedly endomorphic algebra}\emph{twistedly endomorphic algebra} of type $*$ if for all $f,g$, the following identity is satisfied:
\[f^{\bullet}(x_{1},\ldots,x_{n_{f}},g^{\bullet}(y_{1},\ldots,y_{n_{g}},z))\]
\[=g^{\bullet}(f^{\bullet}(x_{\sigma_{f,g,1}(1)},\ldots,x_{\sigma_{f,g,1}(n_{f})},y_{1}),\]
\[\ldots,f^{\bullet}(x_{\sigma_{f,g,n_{g}}(1)},\ldots,x_{\sigma_{f,g,n_{g}}(n_{f})},y_{n_{g}}),f^{\bullet}(x_{1},\ldots,x_{n_{f}},z)).\]
Define the \index{twisted hull}\emph{twisted hull} $\Gamma^{*}(V,(f^{\bullet})_{f\in E})$ to be the algebra with underlying set $\bigcup_{f\in E}\{f\}\times V^{n_{f}}$ and an operation $*$ defined by letting
\[(f,x_{1},\ldots,x_{n_{f}})*(g,y_{1},\ldots,y_{n_{g}})\]
\[=(g,f^{\bullet}(x_{\sigma_{f,g,1}(1)},\ldots,x_{\sigma_{f,g,1}(n_{f})},y_{1}),\ldots\]
\[,f^{\bullet}(x_{\sigma_{f,g,n_{g}}(1)},\ldots,x_{\sigma_{f,g,n_{g}}(n_{f})},y_{n_{g}})).\]
\end{defn}

\begin{prop}
Suppose that $V$ is a set and $(V,F)$ is an algebra. Then $(V,F)$ is a well-founded twistedly endomorphic algebra of type
$*$ if and only if the twisted hull $\Gamma^{*}(V,F)$ is an LD-system.
\end{prop}

\begin{defn}
If $(V,F)$ is an algebra and the twisted hull $\Gamma^{*}(V,F)$ is Laver-like, then we shall call
$(V,F)$ a \index{twistedly Laver-like algebra}\emph{twistedly Laver-like algebra of type $*$}. We shall call $(L,(f^{\sharp})_{f\in E})$ a twistedly endomorphic Laver table of type $*$ if $L$ is a \index{twistedly endomorphic algebra}\emph{twistedly endomorphic algebra of type $*$}.
\end{defn}
\begin{thm}
The algebra $(L,(f^{\sharp})_{f\in E})$ is a twistedly endomorphic Laver table if and only if
$\nabla(g,u_{1},\ldots,u_{n})=0$ implies that
\[\nabla(g,f^{\sharp}(\ell_{\sigma_{f,g,1}(1)},\ldots,\ell_{\sigma_{f,g,1}(m)},u_{1}),\ldots
,f^{\sharp}(\ell_{\sigma_{f,g,n}(1)},
\ldots,\ell_{\sigma_{f,g,n}(m)},u_{n}))=0.\]
\end{thm}
\begin{proof}
The proof follows the same steps as Theorem \ref{2ijr50o} with obvious modifications.
\end{proof}
\begin{thm}
Unless otherwise defined, let us use the notation found in Definition \ref{utgowevbatoiuwbutri}.
Suppose now that the algebra $(V,(f^{\bullet})_{f\in E})$ is a twistedly Laver-like algebra of type $*$. Let $v_{x}\in V$ for each $x\in X$. Now define the subset $L\subseteq T$ and a function $\phi:L\rightarrow V$ inductively by the following rules:
\begin{enumerate}
\item $x\in L$ and $\phi(x)=v_{x}$ for all $x\in X$.

\item $f_{x}(\ell_{1},\ldots,\ell_{n})\in L$ if and only if $\ell_{1},\ldots,\ell_{n}\in L$ and $(f,\phi(\ell_{1}),\ldots,\phi(\ell_{n}))\not\in \mathrm{Li}(\Gamma^{*}(X,E))$.

\item If $f_{x}(\ell_{1},\ldots,\ell_{n})\in L$, then 
\[\phi(f_{x}(\ell_{1},\ldots,\ell_{n}))=f^{\bullet}(\phi(\ell_{1}),\ldots ,\phi(\ell_{n}),v_{x}).\]
\end{enumerate}
Then one can endow $L$ with operations $f^{\sharp}$ for each $f\in E$ such that $(L,(f^{\sharp})_{f\in E})$ is a well-founded twistedly endomorphic Laver table. Furthermore, the mapping $\phi:L\rightarrow V$ is a homomorphism.
\end{thm}

Unlike the case with the endomorphic Laver tables, we currently do not have very good techniques for producing twistedly endomorphic Laver tables.

\begin{exam}
Suppose that $(X,*)$ is a Laver-like LD-system. Define an $n+1$-ary operation on $X$ by letting $t(x_{1},\ldots ,x_{n},x)=x_{1}*x$.
Then the algebra $(X,f)$ is a twistedly endomorphic algebra of type $\bullet$ whenever $\bullet$ is an associative operation on $\{1,\ldots,n\}$ with
$1\bullet 1=1$.
\end{exam}

\section{More multigenic Laver tables and the final matrix}
In this chapter, we shall continue our investigations of the multigenic Laver tables with a special focus on the multigenic Laver tables of the form $(A^{\leq 2^{n}})^{+}$. In particular, we shall establish combinatorial properties of multigenic Laver tables as well as efficient algorithms for computing the self-distributive operation in multigenic Laver tables.

\subsection{The final matrix}
While the multigenic Laver tables $(A^{\leq 2^{n}})^{+}$ have rather large cardinality,
all of the combinatorial complexity in the algebra $(A^{\leq 2^{n}})^{+}$ is contained inside the classical Laver table $A_{n}$ along with another $2^{n}\times 2^{n}$ table which we shall call the $n$-th final matrix. Furthermore, in practice, individual entries in the $n$-th final matrix can be computed efficiently as long as one is able to compute the classical Laver table $A_{n}$. On the other hand, the final matrices contain many combinatorial intricacies which have not been found in the classical Laver tables.

We shall write \index{$\mathbf{x}[n]$}$\mathbf{x}[n]$ for the letter in the $n$-th position in the string $\mathbf{x}$ (for example $abcde[4]=d$).

\begin{prop}
Whenever $x,y\in\{1,\ldots,2^{n}\}$ and $1\leq\ell\leq x*_{n}y$, either
\begin{enumerate}[(i)]
\item there is some $i\in\{1,\ldots,x\}$ where $\mathbf{x}*_{n}\mathbf{y}[\ell]=\mathbf{x}[i]$ whenever
$|\mathbf{x}|=x,|\mathbf{y}|=y$, or

\item there is some $i\in\{1,\ldots,y\}$ where $\mathbf{x}*_{n}\mathbf{y}[\ell]=\mathbf{y}[i]$ whenever $|\mathbf{x}|=x,|\mathbf{y}|=y$.
\end{enumerate}
\end{prop}
\begin{defn}
Define \index{$M_{n}(x,y,\ell)$}
\[M_{n}:\{(x,y,\ell)\mid x,y\in\{1,\ldots,2^{n}\},1\leq\ell\leq x*_{n}y\}\rightarrow\mathbb{Z}\]
to be the function where
\begin{enumerate}
\item if $M_{n}(x,y,\ell)<0$, then
\[\mathbf{x}*_{n}\mathbf{y}[\ell]=\mathbf{x}[-M_{n}(x,y,\ell)],\]
and
\item if $M_{n}(x,y,\ell)>0$, then
\[\mathbf{x}*_{n}\mathbf{y}[\ell]=\mathbf{y}[M_{n}(x,y,\ell)].\]
\end{enumerate}

In other words, $M_{n}$ is the function where we always have
\[a_{-1}\ldots a_{-x}*_{n}a_{1}\ldots a_{y}=a_{M_{n}(x,y,1)}\ldots a_{M_{n}(x,y,x*y)}.\]
\end{defn}
\begin{prop}
The function $M_{n}$ is the unique function with domain 
\[\{(x,y,\ell)\mid x,y\in\{1,\ldots,2^{n}\},1\leq\ell\leq x*_{n}y\}\]
that satisfies the following properties.
\begin{enumerate}
\item $M_{n}(2^{n},y,\ell)=\ell$ whenever $1\leq\ell\leq x*_{n}y=y$.

\item $M_{n}(x,1,\ell)=-\ell$ whenever $\ell\leq x<2^{n}$.

\item $M_{n}(x,1,x+1)=1$ whenever $x<2^{n}$.

\item $M_{n}(x,y+1,\ell)=-M_{n}(x*y,x+1,\ell)$ whenever $0<M_{n}(x*y,x+1,\ell)\leq x$ and $x,y<2^{n}$.

\item $M_{n}(x,y+1,\ell)=y+1$ whenever $M_{n}(x*y,x+1,\ell)=x+1$ and $x,y<2^{n}$.

\item $M_{n}(x,y+1,\ell)=M_{n}(x,y,-M_{n}(x*y,x+1,\ell))$ whenever $M_{n}(x*y,x+1,\ell)<0$ and $x,y<2^{n}$.
\end{enumerate}
\end{prop}
\begin{defn}
Define \index{$FM_{n}^{-},FM_{n}^{+}$}
\[FM_{n}^{-},FM_{n}^{+}:\{1,\ldots,2^{n}\}^{2}\rightarrow\mathbb{Z}\]
to be the functions where
\[FM_{n}^{+}(x,y)=M_{n}(x,2^{n},y)\]
and
\[FM_{n}^{-}(x,y)=M_{n}(x,O_{n}(x),y).\]
We shall call the functions $FM_{n}^{-},FM_{n}^{+}$ the final matrices.
\end{defn}
The functions $FM_{n}^{-},FM_{n}^{+}$ can both be obtained from each other by using the following proposition.

\begin{prop}
$FM_{n}^{+}(x,y)>0$ if and only if $FM_{n}^{-}(x,y)>0$. Furthermore,
\begin{enumerate}
\item if $FM_{n}^{+}(x,y)<0$, then $FM_{n}^{+}(x,y)=FM_{n}^{-}(x,y)$, and

\item if $FM_{n}^{-}(x,y)>0$, then $FM_{n}^{+}(x,y)=FM_{n}^{-}(x,y)+2^{n}-O_{n}(x)$.
\end{enumerate}
\end{prop}

The function $M_{n}$ can be computed from $FM_{n}^{-}$ and $FM_{n}^{+}$ using the following facts.
\begin{prop}
Suppose that $x,y\in\{1,\ldots,2^{n}\}$ and $1\leq\ell\leq x*_{n}y$.
\begin{enumerate}
\item $M_{n}(x,y,\ell)=FM_{n}^{+}(x,\ell)=FM_{n}^{-}(x,\ell)$
whenever $FM_{n}^{+}(x,\ell)<0$.

\item $M_{n}(x,y,\ell)=FM_{n}^{+}(x,\ell)-\lfloor(2^{n}-y)\cdot O_{n}(x)^{-1}\rfloor\cdot O_{n}(x)$
whenever $FM_{n}(x,y)>0$.

\item $M_{n}(x,y,\ell)=FM_{n}^{-}(x,\ell)+2^{n}-O_{n}(x)-\lfloor(2^{n}-y)\cdot O_{n}(x)^{-1}\rfloor\cdot O_{n}(x)$
whenever $FM_{n}(x,y)>0$.

\item if $1\leq y\leq O_{n}(x)$, then $M_{n}(x,y,\ell)=FM_{n}^{-}(x,\ell)$.
\end{enumerate}
\end{prop}
The following tables are the multiplication tables for the final matrices $FM_{n}^{-}$ for $n\leq 4$.

\begin{math}\begin{array}{r|r}

FM_{0}^{-} & 1 \\

\hline
1 &  1
\end{array}
\end{math}
\begin{math}\begin{array}{r|rr}

FM_{1}^{-} & 1 & 2 \\

\hline

1& -1 & 1 \\
2&  1 & 2
\end{array}
\end{math}

\begin{math}\begin{array}{r|rrrr}

FM_{2}^{-} & 1 & 2 &  3 &  4\\

\hline

1 & -1 & 1 &-1 & 2 \\
2 & -1 &-2 & 1 & 2 \\
3 & -1 &-2 &-3 & 1 \\
4 &  1 & 2 & 3 & 4
\end{array}
\end{math}
\begin{math}\begin{array}{r|rrrrrrrr}

FM_{3}^{-} & 1 & 2 &  3 &  4 &  5 & 6 &  7 & 8 \\

\hline

1 & -1 & 1 &-1 & 2 &-1 & 3 &-1 & 4 \\
2 & -1 &-2 & 1 & 2 &-1 &-2 & 3 & 4 \\
3 & -1 &-2 &-3 & 1 &-1 &-2 &-3 & 2 \\
4 & -1 &-2 &-3 &-4 & 1 & 2 & 3 & 4 \\
5 & -1 &-2 &-3 &-4 &-5 & 1 &-5 & 2 \\
6 & -1 &-2 &-3 &-4 &-5 &-6 & 1 & 2 \\
7 & -1 &-2 &-3 &-4 &-5 &-6 &-7 & 1 \\
8 &  1 & 2 & 3 & 4 & 5 & 6 & 7 & 8
\end{array}
\end{math}

\begin{math}\begin{array}{r|rrrrrrrrrrrrrrrr}

FM_{4}^{-} & 1 & 2 &  3 &  4 &  5 & 6 &  7 & 8 &  9 & 10 &  11 & 12 &  13 & 14 &  15 & 16\\

\hline
1 & -1 &  1 & -1 & -1 & -1 &  1 & -1 &  1 & -1 &  1 & -1 &  2 & -1 &  3 & -1 &  4 \\
2 & -1 & -2 &  1 & -1 & -1 & -2 &  1 & -2 & -1 & -2 &  1 &  2 & -1 & -2 &  3 &  4 \\
3 & -1 & -2 & -3 &  1 & -1 & -2 & -3 &  2 & -1 & -2 & -3 &  3 & -1 & -2 & -3 &  4 \\
4 & -1 & -2 & -3 & -4 &  1 &  2 &  3 &  4 & -1 & -2 & -3 & -4 &  5 &  6 &  7 &  8 \\
5 & -1 & -2 & -3 & -4 & -5 &  1 & -5 &  2 & -1 & -2 & -3 & -4 & -5 &  3 & -5 &  4 \\
6 & -1 & -2 & -3 & -4 & -5 & -6 &  1 &  2 & -1 & -2 & -3 & -4 & -5 & -6 &  3 &  4 \\
7 & -1 & -2 & -3 & -4 & -5 & -6 & -7 &  1 & -1 & -2 & -3 & -4 & -5 & -6 & -7 &  2 \\
8 & -1 & -2 & -3 & -4 & -5 & -6 & -7 & -8 &  1 &  2 &  3 &  4 &  5 &  6 &  7 &  8 \\
9 & -1 & -2 & -3 & -4 & -5 & -6 & -7 & -8 & -9 &  1 & -9 &  2 & -9 &  3 & -9 &  4 \\
10& -1 & -2 & -3 & -4 & -5 & -6 & -7 & -8 & -9 &-10 &  1 &  2 & -9 &-10 &  3 &  4 \\
11& -1 & -2 & -3 & -4 & -5 & -6 & -7 & -8 & -9 &-10 &-11 &  1 & -9 &-10 &-11 &  2 \\
12& -1 & -2 & -3 & -4 & -5 & -6 & -7 & -8 & -9 &-10 &-11 &-12 &  1 &  2 &  3 &  4 \\
13& -1 & -2 & -3 & -4 & -5 & -6 & -7 & -8 & -9 &-10 &-11 &-12 &-13 &  1 &-13 &  2 \\
14& -1 & -2 & -3 & -4 & -5 & -6 & -7 & -8 & -9 &-10 &-11 &-12 &-13 &-14 &  1 &  2 \\
15& -1 & -2 & -3 & -4 & -5 & -6 & -7 & -8 & -9 &-10 &-11 &-12 &-13 &-14 &-15 &  1 \\
16&  1 &  2 &  3 &  4 &  5 &  6 &  7 &  8 &  9 & 10 & 11 & 12 & 13 & 14 & 15 & 16
\end{array}
\end{math}

Let $((A^{\leq r})^{+},*)$ be a pre-multigenic Laver table. Let \index{$BFM_{r}$} $BFM_{r}:\{1,\ldots,r\}^{2}\rightarrow\mathbb{Z}$ be the mapping
where if $x\in\{1,\ldots r\}$ and $y$ is the least natural number such that $|0^{x}*0^{y}|=r$, then
$a_{-1}\ldots a_{-x}*a_{1}\ldots a_{y}=a_{BFM_{r}(1)}\ldots a_{BFM_{r}(r)}.$ In particular $BFM_{2^{n}}=FM_{n}^{-}$ for all $n\in\omega$. 
The following table is the multiplication table for $BFM_{15}$.
\begin{math}\begin{array}{r|rrrrrrrrrrrrrrr}

BFM_{15} & 1 & 2 &  3 &  4 &  5 & 6 &  7 & 8 &  9 & 10 &  11 & 12 &  13 & 14 &  15\\

\hline

1  & -1 &  1 & -1 & -1 & -1 & -1 &  1 & -1 & -1 &  1 & -1 &  1 &  2 & -1 &  3 \\
2  & -1 & -2 &  1 & -1 & -1 & -1 & -2 & -1 & -1 & -2 &  1 & -2 &  2 &  3 &  4 \\
3  & -1 & -2 & -3 &  1 & -1 &  1 & -2 &  1 &  1 & -2 & -3 &  2 &  3 & -3 &  4 \\
4  & -1 & -2 & -3 & -4 &  1 & -4 &  2 & -4 & -4 &  2 &  3 &  4 & -4 &  5 &  6 \\
5  & -1 & -2 & -3 & -4 & -5 &  1 &  2 & -4 &  1 &  2 & -5 &  2 &  2 & -5 &  3 \\
6  & -1 & -2 & -3 & -4 & -5 & -6 &  1 &  2 & -6 &  3 & -5 & -6 &  3 &  4 &  5 \\
7  & -1 & -2 & -3 & -4 & -5 & -6 & -7 &  1 & -6 & -7 &  2 & -7 & -7 &  2 &  3 \\
8  & -1 & -2 & -3 & -4 & -5 & -6 & -7 & -8 &  1 &  2 &  3 &  4 & -7 & -8 &  5 \\
9  & -1 & -2 & -3 & -4 & -5 & -6 & -7 & -8 & -9 &  1 & -9 &  2 &  3 & -9 &  4 \\
10 & -1 & -2 & -3 & -4 & -5 & -6 & -7 & -8 & -9 &-10 &  1 &  2 &-10 &  3 &  4 \\
11 & -1 & -2 & -3 & -4 & -5 & -6 & -7 & -8 & -9 &-10 &-11 &  1 &-10 &-11 &  2 \\
12 & -1 & -2 & -3 & -4 & -5 & -6 & -7 & -8 & -9 &-10 &-11 &-12 &  1 &  2 &  3 \\
13 & -1 & -2 & -3 & -4 & -5 & -6 & -7 & -8 & -9 &-10 &-11 &-12 &-13 &  1 &  2 \\
14 & -1 & -2 & -3 & -4 & -5 & -6 & -7 & -8 & -9 &-10 &-11 &-12 &-13 &-14 &  1 \\
15 &  1 &  2 &  3 &  4 &  5 &  6 &  7 &  8 &  9 & 10 & 11 & 12 & 13 & 14 & 15
\end{array}
\end{math}
The following image is picture of $FM_{6}^{-}$ where the positive entries are colored black and the negative entries
are colored white.
\begin{center}
\includegraphics[width=9 cm,height=9 cm]{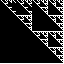}
\end{center}
\pagebreak
The following image is picture of $FM_{7}^{-}$.
\begin{center}
\includegraphics[width=9 cm,height=9 cm]{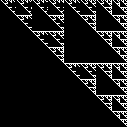}
\end{center}
The following image is picture of $FM_{8}^{-}$.
\begin{center}
\includegraphics[width=9 cm,height=9 cm]{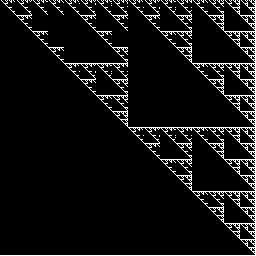}
\label{t2hi48v2brhbar}
\end{center}
\pagebreak

The following diagram gives a side by side comparison of the 128x128-classical Laver table with the
128x128-final matrix.
\begin{center}
\includegraphics[width=12 cm,height=6 cm]{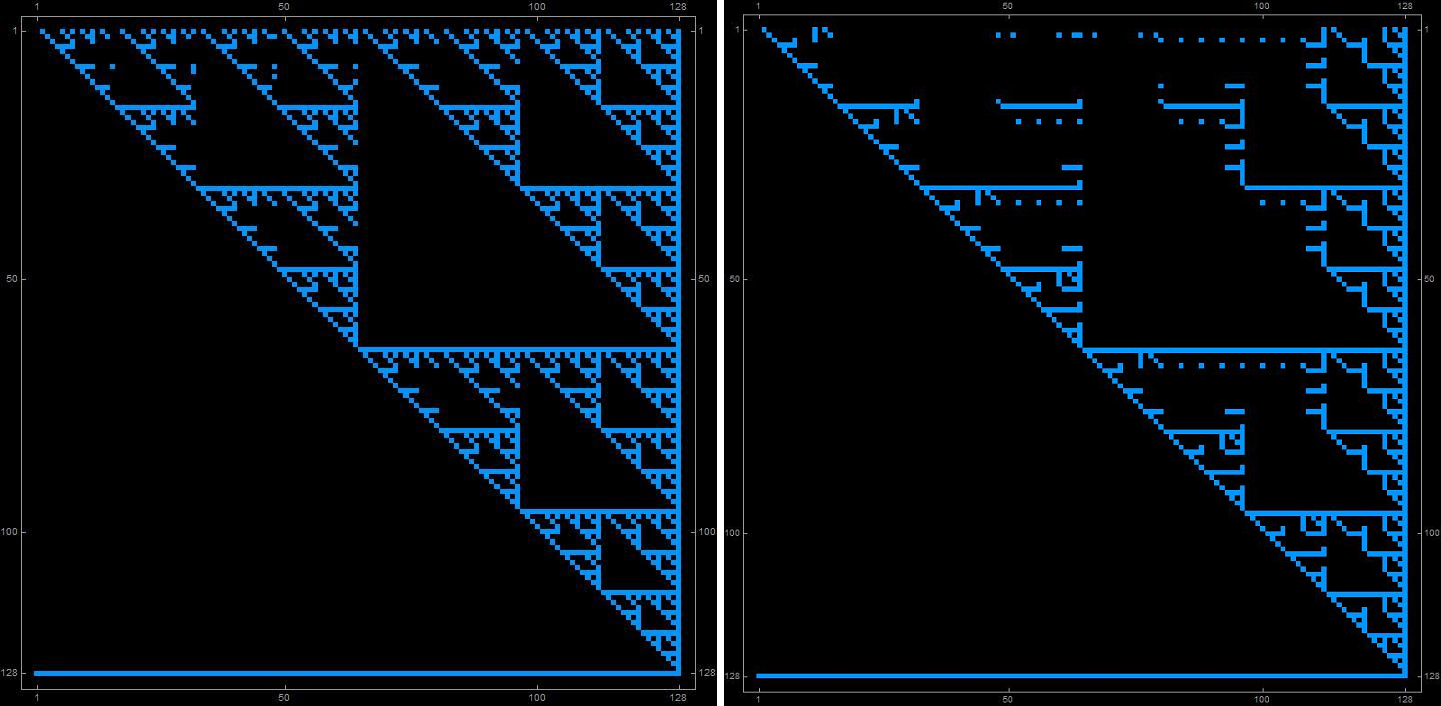}
\end{center}












\begin{lem}
Suppose that $FM_{n}^{+}(x,\ell)=-x$ and $x>1$. Then 
\[FM_{n}^{+}(x-1,\ell)>0\]
and
\[FM_{n}^{+}(1,FM_{n}^{+}(x-1,\ell))=-1.\]
\label{4t24tj0i}
\end{lem}
\begin{proof}
Suppose that $FM_{n}^{+}(x,\ell)=-x$. Then, in $(A^{\leq 2^{n}})^{+}$, we have
\[0^{x-1}*(1*a^{2^{n}})[\ell]=0^{x-1}1*(0^{x-1}*a^{2^{n}})[\ell]=1,\]
so $FM_{n}^{-}(x-1,\ell)>0$. Furthermore,
\[1=0^{x-1}*(1*a^{2^{n}})[\ell]=1*a^{2^{n}}[FM_{n}^{+}(x-1,\ell)],\]
so 
\[FM_{n}^{+}(1,FM_{n}^{+}(x-1,\ell))=-1.\]
\end{proof}

\begin{lem}
Suppose that $FM_{n}^{+}(x,\ell)=-j$ for some $j\in\{1,\ldots,x-1\}$. Then
\[FM_{n}^{+}(x,\ell)=FM_{n}^{+}(x-1,\ell).\]
\label{4th2ur09}
\end{lem}
\begin{proof}
In $(A^{\leq 2^{n}})^{+}$, we have 
\[a_{1}\ldots a_{x-1}*1^{2^{n}}[\ell]=a_{1}\ldots a_{x-1}*(1*1^{2^{n}})[\ell]\]
\[=a_{1}\ldots a_{x-1}1*(a_{1}\ldots a_{x-1}*1^{2^{n}})[\ell]=a_{j}.\]
Therefore, $FM_{n}^{+}(x-1,\ell)=-j$ as well.
\end{proof}
\begin{thm}
\label{t4hu824t09}
Suppose that $FM_{n}^{+}(x,\ell)=-j$. Then 
\begin{enumerate}
\item $FM_{n}^{+}(y,\ell)=-j$ whenever $j\leq y\leq x$.

\item If $j>1$, then
\[FM_{n}^{+}(j-1,\ell)>0,\] 
and 
\[FM_{n}^{+}(1,FM_{n}^{+}(j-1,\ell))=-1.\]
\end{enumerate}
\end{thm}
\begin{proof}
\begin{enumerate}
\item This result follows from applying Lemma \ref{4th2ur09} repeatedly by a descending induction on $y$.

\item From $1$, we have $FM_{n}^{+}(j,\ell)=-j$. Therefore, by Lemma \ref{4t24tj0i}, we conclude that
$FM_{n}^{+}(j-1,\ell)>0$ and $FM_{n}^{+}(1,FM_{n}^{+}(j-1,\ell))=-1$.
\end{enumerate}
\end{proof}
Theorem \ref{t4hu824t09} may be used to efficiently compute images of the final matrix (such as Image \ref{t2hi48v2brhbar}) by computing one column at a time.
Furthermore, from Theorem \ref{t4hu824t09}, one may conclude that all of the combinatorial complexity of the multigenic Laver tables
$(A^{\leq 2^{n}})^{+}$ is contained in the positive entries of the final matrix $FM_{n}^{-}$. 

Proposition \ref{3thuiqefhbk} is proven the same way that Theorem \ref{t4hu824t09} is proven.

\begin{prop}
\label{3thuiqefhbk}
Suppose that $x,y\in\{1,\ldots,2^{n}\},y\leq O_{n}(x),x+y=x*_{n}y$, and
$FM_{n}^{-}(x+y,\ell)=-(x+i)$ for some $i\in\{1,\ldots,y\}$. Then $FM_{n}^{-}(x,\ell)>0$, and
$FM_{n}^{-}(y,FM_{N}^{-}(x,\ell))=-i$.
\end{prop}
\begin{prop}
\label{4j4hu29gegouetfdfghjd}
Suppose that $M$ is a multigenic Laver table and $a_{1}\ldots a_{x},b_{1}\ldots b_{y}\in M$ where
$y>1$. Suppose also that $|a_{1}\ldots a_{x}*^{M}b_{1}\ldots b_{y}|=\ell$. Then
\begin{enumerate}
\item $a_{1}\ldots a_{x}*^{M}b_{1}\ldots b_{y}[\ell-1]=a_{x}$ whenever
$\mathrm{crit}(a_{1}\ldots a_{x})\leq\mathrm{crit}(b_{1}\ldots b_{y-1})$, and

\item $a_{1}\ldots a_{x}*^{M}b_{1}\ldots b_{y}[\ell-1]=b_{y-1}$ whenever
$\mathrm{crit}(a_{1}\ldots a_{x})>\mathrm{crit}(b_{1}\ldots b_{y-1}).$
\end{enumerate}
\end{prop}
\begin{proof}
I claim that the last letter of $t_{n}(a_{1}\ldots a_{x},b_{1}\ldots b_{y-1})$ is $b_{y-1}$ whenever $n$ is odd and
the last letter of $t_{n}(a_{1}\ldots a_{x},b_{1}\ldots b_{y-1})$ is $a_{x}$ whenever $n$ is even. We shall prove this claim by
induction on $n$.

If $n=1$, then $t_{n}(a_{1}\ldots a_{x},b_{1}\ldots b_{y-1})=b_{1}\ldots b_{y-1}$ which has $b_{y-1}$ as its last letter.
If $n=2$, then $t_{n}(a_{1}\ldots a_{x},b_{1}\ldots b_{y-1})=a_{1}\ldots a_{x}$ whose last letter is $a_{x}$. If $n>2$, then 
\[t_{n}(a_{1}\ldots a_{x},b_{1}\ldots b_{y-1})=t_{n-1}(a_{1}\ldots a_{x},b_{1}\ldots b_{y-1})*
t_{n-2}(a_{1}\ldots a_{x},b_{1}\ldots b_{y-1}).\]
If $n$ is odd, then since $t_{n-2}(a_{1}\ldots a_{x},b_{1}\ldots b_{y-1})$ has $b_{y-1}$ as its last letter, $t_{n}(a_{1}\ldots a_{x},b_{1}\ldots b_{y-1})$ also has $b_{y-1}$ as its last letter. If $n$ is even, then since
$a_{x}$ is the last letter of $t_{n-2}(a_{1}\ldots a_{x},b_{1}\ldots b_{y-1})$, we conclude that $a_{x}$ is also the last letter
of $t_{n}(a_{1}\ldots a_{x},b_{1}\ldots b_{y-1})$.

Now take note that $a_{1}\ldots a_{x}*^{M}b_{1}\ldots b_{y}[\ell-1]$ is the last letter of
$a_{1}\ldots a_{x}^{M}\circ b_{1}\ldots b_{y-1}$.

If $\mathrm{crit}(a_{1}\ldots a_{x})\leq\mathrm{crit}(b_{1}\ldots b_{y-1})$, then
$a_{1}\ldots a_{x}\circ b_{1}\ldots b_{y-1}=t_{n}(a_{1}\ldots a_{x},b_{1}\ldots b_{y-1})$ for some
even $n$. Therefore, the last letter of $a_{1}\ldots a_{x}\circ b_{1}\ldots b_{y-1}$ is $a_{x}$.
If $\mathrm{crit}(a_{1}\ldots a_{x})>\mathrm{crit}(b_{1}\ldots b_{y-1})$, then
\[a_{1}\ldots a_{x}\circ b_{1}\ldots b_{y-1}=t_{n}(a_{1}\ldots a_{x},b_{1}\ldots b_{y-1})\]
for some
odd $n$, so the last letter of $a_{1}\ldots a_{x}\circ b_{1}\ldots b_{y-1}$ is
$b_{y-1}$ in this case.
\end{proof}
\begin{cor}
Suppose that $x\in A_{n}$ and $1<y\leq O_{n}(x)$.
\begin{enumerate}
\item If $\gcd(2^{n},x)>\gcd(2^{n},y-1)$, then
\[FM_{n}^{-}(x,(x*_{n}y)-1)=y-1.\]
\item If $\gcd(2^{n},x)\leq\gcd(2^{n},y-1)$, then 
\[FM_{n}^{-}(x,(x*_{n}y)-1)=-x,\]
and
\[FM_{n}^{-}(x-1,(x*_{n}y)-1)>0,\]
and
\[FM_{n}^{+}(1,FM_{n}^{+}(x-1,(x*_{n}y)-1))=-1.\]
\end{enumerate}
\end{cor}
\begin{proof}
Suppose that $x\in A_{n}$ and $1<y\leq O_{n}(x)$. Then let $a_{1},\ldots,a_{x},b_{1},\ldots,b_{y}$ be letters.
\begin{enumerate}
\item If $\gcd(2^{n},x)>\gcd(2^{n},y-1)$, then $\mathrm{crit}(a_{1}\ldots a_{x})>\mathrm{crit}(b_{1}\ldots b_{y-1})$, so
by Proposition \ref{4j4hu29gegouetfdfghjd}, we have $a_{1}\ldots a_{x}*b_{1}\ldots b_{y}[x*y-1]=b_{y-1}$. Therefore,
$FM_{n}^{-}(x,(x*_{n}y)-1)=y-1$.

\item If $\gcd(2^{n},x)\leq\gcd(2^{n},y-1)$, then $\mathrm{crit}(a_{1}\ldots a_{x})\leq\mathrm{crit}(b_{1}\ldots b_{y-1})$, so
by Proposition \ref{4j4hu29gegouetfdfghjd}, we have $a_{1}\ldots a_{x}*b_{1}\ldots b_{y}[x*y-1]=a_{x}$. Therefore, we conclude that
$FM_{n}^{-}(x,(x*_{n}y)-1)=-x$. The other conclusions follow from Lemma \ref{4t24tj0i}.
\end{enumerate}
\end{proof}

We shall now give counterexamples to conjectures that one may make when observing the final matrices.
\begin{exam}
For $n<7$, if $FM_{n}^{+}(x,y)>0$ and $y\leq 2^{n-1}$, then $FM_{n}^{+}(x,y+2^{n-1})>0$ as well.
However, we have $FM_{7}^{+}(8,16)=1$ but $FM_{7}^{+}(8,16+2^{6})=-4$.
\end{exam}
\begin{exam}
If $n<9$ and $FM_{n}^{-}(i,j)>0,FM_{n}^{-}(i,j+1)>0$, then
\[FM_{n}^{-}(i,j+1)=FM_{n}^{-}(i,j)+1.\]
On the other hand,  
\[FM_{9}^{-}(8,175)=1,FM_{9}^{-}(8,176)=3,FM_{9}^{-}(8,191)=1,FM_{9}^{-}(8,192)=4,\]
\[FM_{9}^{-}(8,143)=FM_{9}^{-}(8,144)=1.\]
\end{exam}
\begin{exam}
If $n<11$ and $FM_{n}^{-}(i,j)>0,FM_{n}^{-}(i,j+1)>0$, then
\[FM_{n}^{-}(i,j+1)\geq FM_{n}^{-}(i,j)-1.\]
To the contrary, $FM_{11}^{-}(8,255)=3$ and $FM_{11}^{-}(8,256)=1$.
\end{exam}
\begin{exam}
For $n<4$, we have $\gcd(2^{n},x,y)\leq\gcd(2^{n},|FM_{n}^{-}(x,y)|)$,
but $FM_{4}^{-}(2,4)=1$.
\end{exam}
\begin{exam}
For $n<7$, if $FM_{n}^{-}(x,y)>0$, then
\[\gcd(2^{n},x,y)\leq\gcd(2^{n},FM_{n}^{-}(x,y)),\]
yet $FM_{7}^{-}(8,16)=1$. 
\end{exam}

Let $X$ be a permutative LD-system. Then if $x,y\in X$, then let $x\leq_{\mathrm{ST}}y$ if and only if
$x/\equiv^{\alpha}\preceq y/\equiv^{\alpha}$ for all critical points $\alpha\in\mathrm{crit}[X]$. 

\begin{exam}
Suppose that $a_{1}\ldots a_{x},b_{1}\ldots b_{y}\in (A^{\leq 2^{n}})^{+}$ are words where the letters
$a_{1},\ldots,a_{x},b_{1},\ldots,b_{y}$ are distinct. Let $a_{1}\ldots a_{x}*_{n}b_{1}\ldots b_{y}=c_{1}\ldots c_{z}$, and suppose that
$a_{1}\ldots a_{x}*_{n}b_{1}\ldots b_{y}[\ell]=b_{j}$.

If $n\in\{0,\ldots,6\}$, then $a_{1}\ldots a_{x}*_{n}b_{1}\ldots b_{j}\leq_{\mathrm{ST}}c_{1}\ldots c_{\ell}$.
However, if $n=7,x=1,y=4$, then $a_{1}*_{7}b_{1}b_{2}b_{3}b_{4}[32]=b_{1}$, but
$a_{1}b_{1}=a_{1}*_{7}b_{1}\not\leq_{\mathrm{ST}}c_{1}\ldots c_{32}$.
Let $\phi:(A^{\leq 32})^{+}\rightarrow(A^{\leq 4})^{+}$ be the canonical homomorphism. Then
$\phi(a_{1}*b_{1})=a_{1}b_{1}\not\preceq a_{1}b_{3}a_{1}b_{1}=\phi(c_{1}\ldots c_{32})$ which implies that
$a_{1}*b_{1}\not\leq_{\mathrm{ST}}c_{1}\ldots c_{32}$.
\end{exam}
\begin{exam}
Suppose $a_{1}\ldots a_{x},b_{1}\ldots b_{y}\in (A^{\leq 2^{n}})^{+}$ and all the letters $a_{1},\ldots,a_{x},b_{1},\ldots,b_{y}$ are
distinct. Let $a_{1}\ldots a_{x}*b_{1}\ldots b_{y}=c_{1}\ldots c_{z}$, and let $a_{1}\ldots a_{x}*b_{1}\ldots b_{y}[\ell]=a_{i}$.

If $n\leq 3$, then $a_{1}\ldots a_{i}\leq_{\mathrm{ST}}c_{1}\ldots c_{\ell}$. However, if $n=4$ and $x=2,y=4$, then $a_{1}a_{2}*b_{1}b_{2}b_{3}b_{4}[4]=a_{1}$, but $a_{1}\not\leq_{\mathrm{ST}}c_{1}c_{2}c_{3}c_{4}$.
\end{exam}

\subsection{Applications of induction to multigenic Laver tables}

\begin{thm}
Suppose that $M$ is a multigenic Laver table over an alphabet $A$,
$a_{1},\ldots,a_{x},b_{1},\ldots,b_{y}$ are all distinct letters in $A$,
$a_{1}\ldots a_{x}*^{M}b_{1}\ldots b_{y}=c_{1}\ldots c_{z}$, and $1\leq l\leq z$.
\begin{enumerate}
\item If $c_{l}=a_{i}$, then $\mathrm{crit}(a_{1}\ldots a_{i})\leq \mathrm{crit}(c_{1}\ldots c_{l})$.

\item If $c_{l}=b_{i}$, then $\mathrm{crit}(b_{1}\ldots b_{i})\leq \mathrm{crit}(c_{1}\ldots c_{l})$.
\end{enumerate}
\label{4tij0poqlke}
\end{thm}
\begin{proof}
Without loss of generality, assume that $N$ is a generalized Laver table over an alphabet $B$,
$f:A\rightarrow B$ is a function where $f^{-1}[\{b\}]$ is infinite for each $b\in B$, and
$M=\mathbf{M}((f(a))_{a\in A},N)$. Let $\phi:M\rightarrow N$ be the unique homomorphism that extends $f$, and let
$\simeq$ be the congruence on $M$ where $\mathbf{x}\simeq\mathbf{y}$ if and only if $\phi(\mathbf{x})=\phi(\mathbf{y})$. The motivation
behind our use of $M$ is that all of the letters in the set $f^{-1}[\{b\}]$ are simply isomorphic copies of one another and that in this proof we will need to replace letters with these copies.

We shall prove this result by our usual double induction which is descending on $a_{1}\ldots a_{x}$ but which is ascending
on $b_{1}\ldots b_{y}$ with respect to $\preceq$.
\item[Case 1: $a_{1}\ldots a_{x}\in \mathrm{Li}(M)$.]

We have $a_{1}\ldots a_{x}*b_{1}\ldots b_{y}=b_{1}\ldots b_{y}$. If $c_{l}=b_{i}$, then $l=i$, so $b_{1}\ldots b_{i}=c_{1}\ldots c_{l}$, hence
\[\mathrm{crit}(b_{1}\ldots b_{i})\leq\mathrm{crit}(c_{1}\ldots c_{l}).\]

\item[Case 2: $a_{1}\ldots a_{x}\not\in \mathrm{Li}(M),y=1$.]

We have $c_{1}\ldots c_{z}=a_{1}\ldots a_{x}b_{1}$. If $\ell\leq x$, then
$c_{1}\ldots c_{\ell}=a_{1}\ldots a_{\ell}$, so $c_{\ell}=a_{\ell}$ and
\[\mathrm{crit}(a_{1}\ldots a_{\ell})\leq \mathrm{crit}(c_{1}\ldots c_{\ell}).\]

Now assume that $\ell=x+1$. Then $c_{\ell}=b_{1}$, but
\[\mathrm{crit}(b_{1})\leq \mathrm{crit}(a_{1}\ldots a_{x}b_{1})=\mathrm{crit}(c_{1}\ldots c_{\ell}).\]

\item[Case 3: $a_{1}\ldots a_{x}\not\in \mathrm{Li}(M),y>1$.]

We have 
\[a_{1}\ldots a_{x}*b_{1}\ldots b_{y}=(a_{1}\ldots a_{x}*b_{1}\ldots b_{y-1})*a_{1}\ldots a_{x}b_{y}.\]

Suppose now that
\[d_{1}\ldots d_{r}\simeq a_{1}\ldots a_{x}*b_{1}\ldots b_{y-1}\]
where the letters $d_{1}\ldots d_{r}$ are all distinct and distinct from $a_{1}\ldots a_{x},b_{1}\ldots b_{y}$.
Suppose now that $d_{1}\ldots d_{r}*a_{1}\ldots a_{x}b_{y}=e_{1}\ldots e_{z}$. 

\item[Case 3a: $e_{\ell}=d_{j}$.]

If $e_{\ell}=d_{j}$, then
\[\mathrm{crit}(c_{1}\ldots c_{j})=\mathrm{crit}(d_{1}\ldots d_{j})\leq \mathrm{crit}(e_{1}\ldots e_{\ell})=\mathrm{crit}(c_{1}\ldots c_{\ell}).\]

Take note that $c_{j}=c_{\ell}$. Now, if $c_{j}=a_{i}$, then
\[\mathrm{crit}(a_{1}\ldots a_{i})\leq \mathrm{crit}(c_{1}\ldots c_{j})\leq \mathrm{crit}(c_{1}\ldots c_{\ell}).\]

If $c_{j}=b_{i}$, then 
\[\mathrm{crit}(b_{1}\ldots b_{i})\leq \mathrm{crit}(c_{1}\ldots c_{j})\leq \mathrm{crit}(c_{1}\ldots c_{\ell}).\]

\item[Case 3b: $e_{\ell}=b_{y}.$]
$\ell=z$, so
\[\mathrm{crit}(b_{1}\ldots b_{y})\leq \mathrm{crit}(a_{1}\ldots a_{x}*b_{1}\ldots b_{y})=\mathrm{crit}(c_{1}\ldots c_{z}).\]

\item[Case 3c: $e_{\ell}=a_{i}$ for some $i\in\{1,\ldots,x\}$.]
\[\mathrm{crit}(a_{1}\ldots a_{i})\leq \mathrm{crit}(e_{1}\ldots e_{\ell})=\mathrm{crit}(c_{1}\ldots c_{\ell}).\]
\end{proof}
\begin{thm}
Suppose that $M$ is a multigenic Laver table over an alphabet $A$, $\mathbf{x}_{i}=a_{i,1}\ldots a_{i,n_{i}}$ for $1\leq i\leq n$, and all the letters $a_{1,1},\ldots,a_{1,n_{1}},\ldots,a_{n,1},\ldots,a_{n,n_{n}}$ are distinct. Let $t$ be an $n$-ary term in the language of
LD-systems, $t(\mathbf{x}_{1},\ldots,\mathbf{x}_{n})=c_{1}\ldots c_{z}$, and $c_{\ell}=a_{i,j}$. Then
$\mathrm{crit}(a_{i,1}\ldots a_{i,j})\leq \mathrm{crit}(c_{1}\ldots c_{\ell})$.
\end{thm}
\begin{proof}
For this proof, assume that $M,N,f,\phi,\simeq$ are the same as in Theorem \ref{4tij0poqlke}.

We shall prove this result by induction on the complexity of the term $t$. If $t(x)=x$, then this result
is trivial. Now assume that 
\[t(\mathbf{x}_{1},\ldots,\mathbf{x}_{n})=t_{1}(\mathbf{x}_{1},\ldots,\mathbf{x}_{n})*t_{2}(\mathbf{x}_{1},\ldots,\mathbf{x}_{n}).\]

Let $t_{1}(\mathbf{x}_{1},\ldots,\mathbf{x}_{n})\simeq r_{1}\ldots r_{p}$ and $t_{2}(\mathbf{x}_{1},\ldots,\mathbf{x}_{n})\simeq s_{1}\ldots s_{q}$ where
all the letters in $\mathbf{x}_{1},\ldots,\mathbf{x}_{n},r_{1}\ldots r_{p},s_{1}\ldots s_{q}$ are distinct.

Suppose that $t_{1}(\mathbf{x}_{1},\ldots,\mathbf{x}_{n})=e_{1}\ldots e_{p}$ and $t_{2}(\mathbf{x}_{1},\ldots,\mathbf{x}_{n})=f_{1}\ldots f_{q}$.

Then suppose that 
\[r_{1}\ldots r_{p}*s_{1}\ldots s_{q}=d_{1}\ldots d_{z}.\] 

Suppose that $d_{\ell}=r_{k}$. Then $e_{k}=a_{i,j}$ and 
\[\mathrm{crit}(a_{i,1}\ldots a_{i,j})\leq \mathrm{crit}(e_{1}\ldots e_{k})=\mathrm{crit}(r_{1}\ldots r_{k})\leq \mathrm{crit}(d_{1}\ldots d_{\ell})=\mathrm{crit}(c_{1}\ldots c_{\ell}).\]

Now suppose that $d_{\ell}=s_{k}$. Then $f_{k}=a_{i,j}$ and
\[\mathrm{crit}(a_{i,1}\ldots a_{i,j})\leq \mathrm{crit}(f_{1}\ldots f_{k})=\mathrm{crit}(s_{1}\ldots s_{k})\leq \mathrm{crit}(d_{1}\ldots d_{\ell})=\mathrm{crit}(c_{1}\ldots c_{\ell}).\]
\end{proof}
\begin{thm}
Let $M$ be a multigenic Laver table. Suppose that $a_{1}\ldots a_{x},b_{1}\ldots b_{y}\in M$,
the letters $a_{1},\ldots,a_{x},b_{1},\ldots,b_{y}$ are all distinct, and
$a_{1}\ldots a_{x}*b_{1}\ldots b_{y}=c_{1}\ldots c_{z}$. If $c_{\ell}=b_{i}$, then
$\mathrm{crit}(a_{1}\ldots a_{x}*b_{1}\ldots b_{i})\leq \mathrm{crit}(c_{1}\ldots c_{\ell})$.
\end{thm}
\begin{proof}
For this proof, assume that $M,N,f,\phi,\simeq$ are the same as in Theorem \ref{4tij0poqlke}.
Suppose now that $a_{1}\ldots a_{x}*b_{1}\ldots b_{i}\simeq d_{1}\ldots d_{v}$ and the letters
\[d_{1},\ldots ,d_{v},a_{1},\ldots ,a_{x},b_{1},\ldots ,b_{i}\]
are all distinct. Then suppose that
\[d_{1}\ldots d_{v}*a_{1}\ldots a_{x}b_{i+1}*\ldots *a_{1}\ldots a_{x}b_{y}=e_{1}\ldots e_{z}.\]
Then $e_{\ell}=d_{v}$, so 
\[\mathrm{crit}(a_{1}\ldots a_{x}*b_{1}\ldots b_{i})=\mathrm{crit}(d_{1}\ldots d_{v})\leq \mathrm{crit}(e_{1}\ldots e_{\ell}).\]
\end{proof}
\begin{cor}
We have $\gcd(2^{n},\ell)\geq \gcd(2^{n},|M_{n}(x,y,\ell)|)$. In particular, 
\[\gcd(2^{n},y)\geq \gcd(2^{n},|FM_{n}^{-}(x,y)|).\]
\end{cor}
\begin{cor}
Suppose that $M_{n}(x,y,\ell)>0$. Then 
\[\gcd(2^{n},x*M_{n}(x,y,\ell))\leq \gcd(2^{n},\ell).\]
\end{cor}
\begin{cor}
If $x<2^{n}$ and $M_{n}(x,y,\ell)>0$, then $\gcd(2^{n},x+1)\leq\gcd(2^{n},\ell)$.
\end{cor}
\begin{defn}
Let $X$ be a locally Laver-like LD-system. Let $M$ be the set of all sequences $(x_{1},\ldots,x_{n})$ of elements in $X$ where there is no
$m<n$ where $x_{1}*\ldots*x_{m}\in \mathrm{Li}(X)$. Then $M$ is a multigenic Laver table. Define a mapping $\Phi:M\rightarrow X^{*}$ by letting
\[\Phi(x_{1},\ldots,x_{n})=(x_{1},x_{1}*x_{2},\ldots,x_{1}*\ldots*x_{n}).\]
Define an equivalence relation $\simeq$ on $M$ by letting
$(x_{1},\ldots,x_{m})\simeq(y_{1},\ldots,y_{n})$ if and only if $\Phi(x_{1},\ldots,x_{m})=\Phi(y_{1},\ldots,y_{m})$.
\label{t48stwv67qfda}
\end{defn}
\begin{thm}
The equivalence relation $\simeq$ in Definition \ref{t48stwv67qfda} is a congruence on $M$.
\end{thm}
\begin{proof}
We shall show that 
\item[A:] $(x_{1},\ldots,x_{m})\simeq(y_{1},\ldots,y_{m})$ implies that
\[(x_{1},\ldots,x_{m})*(z_{1},\ldots,z_{n})\simeq(y_{1},\ldots,y_{m})*(z_{1},\ldots,z_{n})\]
and
\item[B:] $(y_{1},\ldots,y_{n})\simeq(z_{1},\ldots,z_{n})$ implies that
\[(x_{1},\ldots,x_{m})*(y_{1},\ldots,y_{n})\simeq(x_{1},\ldots,x_{m})*(z_{1},\ldots,z_{n})\]
by a double induction which is descending on $m$ but ascending on $n$ using our usual cases.

\item[Case 1: $x_{1}*\ldots*x_{m}\in \mathrm{Li}(X)$.]
 
\item[A:] Suppose that $(x_{1},\ldots,x_{m})\simeq(y_{1},\ldots,y_{m})$. We have 
\[(x_{1},\ldots,x_{m})*(z_{1},\ldots,z_{n})=(z_{1},\ldots,z_{n})\]
and
\[(y_{1},\ldots,y_{m})*(z_{1},\ldots,z_{n})=(z_{1},\ldots,z_{n}),\]
so
\[(x_{1},\ldots,x_{m})*(z_{1},\ldots,z_{n})\simeq(y_{1},\ldots,y_{m})*(z_{1},\ldots,z_{n}).\]

\item[B:] Suppose now that $(y_{1},\ldots,y_{n})\simeq(z_{1},\ldots,z_{n})$. Then
\[(x_{1},\ldots,x_{m})*(y_{1},\ldots,y_{n})=(y_{1},\ldots,y_{n})\simeq(z_{1},\ldots,z_{n})=(x_{1},\ldots,x_{m})*(z_{1},\ldots,z_{n}).\]

\item[Case 2: $x_{1}*\ldots*x_{m}\not\in \mathrm{Li}(X)$ and $n=1$.]

\item[A:] Suppose that $(x_{1},\ldots,x_{m})\simeq(y_{1},\ldots,y_{m})$. Since 
\[x_{1}*\ldots*x_{m}=y_{1}*\ldots*y_{m},\] we have 
\[x_{1}*\ldots*x_{m}*z_{1}=y_{1}*\ldots*y_{m}*z_{1}.\]
Since
\[x_{1}*\ldots*x_{i}=y_{1}*\ldots*y_{i}\]
for $1\leq i\leq m$ as well, we conclude that
\[(x_{1},\ldots,x_{m})*(z_{1})=(x_{1},\ldots,x_{m},z_{1})\simeq(y_{1},\ldots,y_{m},z_{1})=(y_{1},\ldots,y_{m})*(z_{1}).\]

\item[B:] Suppose that $(y_{1})\simeq(z_{1})$. Then $y_{1}=z_{1}$, so
\[(x_{1},\ldots,x_{m})*(y_{1})=(x_{1},\ldots,x_{m},y_{1})=(x_{1},\ldots,x_{m},z_{1})=(x_{1},\ldots,x_{m})*(z_{1}).\]

\item[Case 3: $x_{1}*\ldots*x_{m}\not\in\mathrm{Li}(X)$ and $n>1$.]

\item[A:] Suppose that $(x_{1},\ldots,x_{m})\simeq(y_{1},\ldots,y_{m})$. Then
since $(x_{1},\ldots,x_{m})\simeq(y_{1},\ldots,y_{m})$, we have
\[(x_{1},\ldots,x_{m})*(z_{1},\ldots,z_{n-1})\simeq(y_{1},\ldots,y_{m})*(z_{1},\ldots,z_{n-1})\]
by the inductive hypothesis.

Therefore, by applying the inductive hypotheses two more times, we obtain
\[(x_{1},\ldots,x_{m})*(z_{1},\ldots,z_{n})\]
\[=(x_{1},\ldots,x_{m})*(z_{1},\ldots,z_{n-1})*(x_{1},\ldots,x_{m},z_{n})\]
\[\simeq(y_{1},\ldots,y_{m})*(z_{1},\ldots,z_{n-1})*(x_{1},\ldots,x_{m},z_{n})\]
\[\simeq(y_{1},\ldots,y_{m})*(z_{1},\ldots,z_{n-1})*(y_{1},\ldots,y_{m},z_{n})\]
\[=(y_{1},\ldots,y_{m})*(z_{1},\ldots,z_{n}).\]

\item[B:] Suppose that $(y_{1},\ldots,y_{n})\simeq(z_{1},\ldots,z_{n})$. Then

\[(x_{1},\ldots,x_{m})*(y_{1},\ldots,y_{n-1})\simeq (x_{1},\ldots,x_{m})*(z_{1},\ldots,z_{n-1}).\] 

Therefore,

\[(x_{1},\ldots,x_{m})*(y_{1},\ldots,y_{n-1})*(x_{1},\ldots,x_{m},y_{n})\]
\[\simeq(x_{1},\ldots,x_{m})*(z_{1},\ldots,z_{n-1})*(x_{1},\ldots,x_{m},y_{n}).\]

I now claim that 
\[(x_{1},\ldots,x_{m})*(z_{1},\ldots,z_{n-1})*(x_{1},\ldots,x_{m},y_{n})\]
\[\simeq(x_{1},\ldots,x_{m})*(z_{1},\ldots,z_{n-1})*(x_{1},\ldots,x_{m},z_{n}).\]

Let $x=x_{1}*\ldots*x_{m},r=z_{1}*\ldots*z_{n-1}$. Then since $(y_{1},\ldots,y_{n})\simeq(z_{1},\ldots,z_{n})$, we have
$r*y_{n}=r*z_{n}$, so 
\[(x\circ r)*y_{n}=x*(r*y_{n})=x*(r*z_{n})=(x\circ r)*z_{n}.\]
Let
\[(x_{1},\ldots,x_{m})*(z_{1},\ldots,z_{n-1})*(x_{1},\ldots,x_{m},y_{n})=(t_{1},\ldots,t_{l},y_{n}).\]
Then
$t_{1}*\ldots*t_{l}=x\circ r$, so 
\[t_{1}*\ldots*t_{l}*y_{n}=(x\circ r)*y_{n}=(x\circ r)*z_{n}=t_{1}*\ldots*t_{l}*z_{n}.\]
Therefore,

\[(x_{1},\ldots,x_{m})*(z_{1},\ldots,z_{n-1})*(x_{1},\ldots,x_{m},y_{n})=(t_{1},\ldots,t_{l},y_{n})\]
\[\simeq(t_{1},\ldots,t_{l},z_{n})=(x_{1},\ldots,x_{m})*(z_{1},\ldots,z_{n-1})*(x_{1},\ldots,x_{m},z_{n}).\]

We therefore, conclude that
\[(x_{1},\ldots,x_{m})*(y_{1},\ldots,y_{n})\]
\[=(x_{1},\ldots,x_{m})*(y_{1},\ldots,y_{n-1})*(x_{1},\ldots,x_{m},y_{n})\]
\[\simeq(x_{1},\ldots,x_{m})*(z_{1},\ldots,z_{n-1})*(x_{1},\ldots,x_{m},y_{n})\]
\[\simeq(x_{1},\ldots,x_{m})*(z_{1},\ldots,z_{n-1})*(x_{1},\ldots,x_{m},z_{n})\]
\[=(x_{1},\ldots x_{m})*(z_{1},\ldots,z_{n}).\]
\end{proof}
Theorem \ref{t48stwv67qfda} also extends to endomorphic Laver tables.

\begin{thm}
\label{thanhz9auntsc2nj52}

Suppose the following:
\begin{enumerate}
\item $\mathcal{V}=(V,E,F)$ is a locally Laver-like partially endomorphic algebra.

\item $X$ is a set of variables.

\item $v_{x}\in V$ for each $x\in X$. 

\item $L=\mathbf{L}((v_{x})_{x\in X},\mathcal{V})$. 

\item $\Phi:L\rightarrow V$ is the canonical endomorphic algebra homomorphism.

\item $\mathbf{I}:L\rightarrow\mathbf{T}_{E\cup F}(\{e\})$ is the mapping where 
\begin{enumerate}
\item $\mathbf{I}(v_{x})=e$ for all $x\in X$, 

\item $\mathbf{I}(\mathfrak{g}(\ell_{1},\ldots,\ell_{n}))=\mathfrak{g}(\mathbf{I}(\ell_{1}),\ldots,\mathbf{I}(\ell_{n}))$ for each $\mathfrak{g}\in F$, and 

\item $\mathbf{I}(f_{x}(\ell_{1},\ldots,\ell_{n}))=f(\mathbf{I}(\ell_{1}),\ldots,\mathbf{I}(\ell_{n}))$. 
\end{enumerate}
\item $\simeq$ is the equivalence relation on $L$ where $\mathfrak{l}\simeq\mathfrak{m}$ if and only if $\mathbf{I}(\mathfrak{l})=\mathbf{I}(\mathfrak{m})$ and
\begin{enumerate}
\item if $\mathfrak{l}=v_{x},\mathfrak{m}=v_{y}$, then $\mathfrak{l}=\mathfrak{m}$ if and only if $\mathbf{I}(\mathfrak{l})=\mathbf{I}(\mathfrak{m})$,

\item if $\mathfrak{l}=\mathfrak{g}(u_{1},\ldots,u_{n}),\mathfrak{m}=\mathfrak{g}(v_{1},\ldots,v_{n})$, then
$\mathfrak{l}\simeq\mathfrak{m}$ if and only if $u_{1}\simeq v_{1},\ldots,u_{n}\simeq v_{n}$, and

\item if $\mathfrak{l}=f_{x}(u_{1},\ldots,u_{n}),\mathfrak{m}=f_{y}(v_{1},\ldots,v_{n})$, then
$\mathfrak{l}\simeq\mathfrak{m}$ if and only if $u_{1}\simeq v_{1},\ldots,u_{n}\simeq v_{n}$ and
$\Phi(f_{x}(u_{1},\ldots,u_{n}))=\Phi(f_{y}(v_{1},\ldots,v_{n}))$.
\end{enumerate}
\end{enumerate}
Then $\simeq$ is a congruence on the partially endomorphic Laver table $\mathbf{L}((v_{x})_{x\in X},\mathcal{V})$.
\end{thm}

\begin{defn}
Let $X$ be a locally Laver-like LD-system. A \index{chain}\emph{chain} is a sequence $(r_{1},\ldots,r_{n})$ where each for all $i\in\{2,\ldots,n\}$ there is some $x$ with $r_{i}=r_{i-1}*x$ and where $r_{m}\not\in \mathrm{Li}(X)$ whenever $m<n$. In other words, a chain is an element in the image of the mapping $\Phi$. Let $\mathbf{Chain}(X)$ denote the set of all chains in $X$. The \index{chain algebra}\emph{chain algebra} over $(X,*)$ is the algebra
\index{$(\mathbf{Chain}(X),\otimes)$} $(\mathbf{Chain}(X),\otimes)$ where
\[\Phi(\mathbf{x})\otimes\Phi(\mathbf{y})=\Phi(\mathbf{x}*\mathbf{y}).\]
\end{defn}
\begin{defn}
Define a partial mapping \index{$\odot$}$\odot$ by 
\[\Phi(\mathbf{x})\odot\Phi(\mathbf{y})=\Phi(\mathbf{x}\circ\mathbf{y})\]
whenever $\mathbf{x}\circ\mathbf{y}$ is defined.
\end{defn}
\begin{prop}
Let $(X,*)$ be a locally Laver-like LD-system, and suppose that
\[(r_{1},\ldots,r_{m}),(s_{1},\ldots,s_{n})\in\mathbf{Chain}(X).\]
\begin{enumerate}
\item If $r_{m}\in \mathrm{Li}(X)$, then 
\[(r_{1},\ldots,r_{m})\otimes(s_{1},\ldots,s_{n})=(s_{1},\ldots,s_{n}).\]

\item If $n=1,r_{m}\not\in \mathrm{Li}(X)$, then
\[(r_{1},\ldots,r_{m})\otimes(s_{1})=(r_{1},\ldots,r_{m},r_{m}*s_{1}).\]

\item If $r_{m}\not\in\mathrm{Li}(X)$ and $n>1$, then
\[(r_{1},\ldots,r_{m})\otimes(s_{1},\ldots,s_{n})=((r_{1},\ldots,r_{m})\odot(s_{1},\ldots,s_{n-1}),r_{m}*s_{n}).\]
\end{enumerate}
\end{prop}
\begin{proof}
Suppose that $\Phi(x_{1},\ldots,x_{m})=(r_{1},\ldots,r_{m})$ and $\Phi(y_{1},\ldots,y_{n})=(s_{1},\ldots,s_{n})$.
\begin{enumerate}
\item If $r_{m}\in \mathrm{Li}(X)$, then
\[(r_{1},\ldots,r_{m})\otimes(s_{1},\ldots,s_{n})=\phi((x_{1},\ldots,x_{m}))*\phi((y_{1},\ldots,y_{n}))\]
\[=\Phi((x_{1},\ldots,x_{m})*(y_{1},\ldots,y_{n}))=\Phi(y_{1},\ldots,y_{n})=(s_{1},\ldots,s_{n}).\]

\item If $n=1,r_{m}\not\in \mathrm{Li}(X)$, then
\[(r_{1},\ldots,r_{m})\otimes(s_{1})=\Phi((x_{1},\ldots,x_{m})*(y_{1}))=\Phi(x_{1},\ldots,x_{m},y_{1})\]
\[=(x_{1},\ldots,x_{1}*\ldots*x_{m},x_{1}*\ldots*x_{m}*y_{1})=(r_{1},\ldots,r_{m},r_{m}*s_{1}).\]

\item \[(r_{1},\ldots,r_{m})\otimes(s_{1},\ldots,s_{n})\]
\[=\Phi((x_{1},\ldots,x_{m})*(y_{1},\ldots,y_{n}))=\Phi((x_{1},\ldots,x_{m})*((y_{1},\ldots,y_{n-1})*(y_{n})))\]
\[=\Phi((x_{1},\ldots,x_{m})\circ(y_{1},\ldots,y_{n-1}))*y_{n}))\]
\[=((r_{1},\ldots,r_{m})\odot(s_{1},\ldots,s_{n-1}))\otimes(y_{n})\]
\[=(((r_{1},\ldots,r_{m})\odot(s_{1},\ldots,s_{n-1})),(r_{m}\circ s_{n-1})*y_{n})\]
\[=((r_{1},\ldots,r_{m})\odot(s_{1},\ldots,s_{n-1}),r_{m}*(s_{n-1}*y_{n}))\]
\[=((r_{1},\ldots,r_{m})\odot(s_{1},\ldots,s_{n-1}),r_{m}*s_{n}).\]
\end{enumerate}
\end{proof}
\begin{thm}
If $\ell>1$ and $M_{n}(x,y,\ell)>1$, then
\[o_{n}(\ell-1)\leq o_{n}(|M_{n}(x,y,\ell)|-1).\]
\end{thm}
\begin{proof}
We have 
\[1^{x}\simeq(1,1+O_{n}(1),\ldots,1+O_{n}(x-1)),\]
and
\[1^{y}\simeq(1,1+O_{n}(1),\ldots,1+O_{n}(y-1)).\]

Therefore,
\[1^{x*y}=1^{x}*1^{y}\simeq(1,1+O_{n}(1),\ldots,1+O_{n}(x-1))*(1,1+O_{n}(1),\ldots,1+O_{n}(y-1)).\]

If $|M_{n}(x,y,\ell)|=i>1$, then 
\[(1,1+O_{n}(1),\ldots,1+O_{n}(x-1))*(1,1+O_{n}(1),\ldots,1+O_{n}(y-1))[\ell]=1+O_{n}(i-1).\]

Therefore, we have $(\ell-1)*1=(\ell-1)*(1+O_{n}(i-1))$, so $o_{n}(\ell-1)\leq o_{n}(i-1)$.
\end{proof}
\begin{cor}
If $|FM_{n}^{-}(x,y)|>1$, then
\[o_{n}(y-1)\leq o_{n}(|FM_{n}^{-}(x,y)|-1).\]
\end{cor}

\begin{conj}
If $|FM_{n}^{-}(x,y)|>1$, then
\[\gcd(2^n,(|FM_{n}^{-}(x,y)|-1)*_{n}a)\leq\gcd(2^n,(y-1)*_{n}a).\]
\end{conj}
\begin{conj}
Suppose $M$ is a multigenic Laver table, $t$ is an $n$-ary term in the language of LD-systems,
$\mathbf{x}_{i}=a_{i,1}\ldots a_{i,n_{i}}$ for $1\leq i\leq n$, $t(\mathbf{x}_{1},\ldots,\mathbf{x}_{n})=c_{1}\ldots c_{z}$, and
each symbol $a_{i,j}$ is distinct. If $c_{\ell}=a_{i,j}$ and $j>1,\ell>1$, then
\[(a_{i,1}\ldots a_{i,j-1})^{\sharp}(\alpha)\leq(c_{1}\ldots c_{\ell-1})^{\sharp}(\alpha).\]
\end{conj}
\begin{prop}
\label{hfqhq4hwbujddxn}
Suppose that $X$ is a locally Laver-like LD-system and $\simeq$ is the congruence found in Definition \ref{t48stwv67qfda}. If
\[(x_{1}*\ldots*x_{m}*y_{1},\ldots,x_{1}*\ldots*x_{m}*y_{n})\simeq(x_{1}*\ldots*x_{m}*z_{1},\ldots,x_{1}*\ldots*x_{m}*z_{n}),\] then
\[(x_{1},\ldots,x_{m})*(y_{1},\ldots,y_{n})\simeq(x_{1},\ldots,x_{m})*(z_{1},\ldots,z_{n}).\]
\end{prop}
\begin{prop}
If $FM_{n}^{-}(x,y)>0$, then
\[o_{n}(y-1)\leq o_{n}((x*FM_{n}^{-}(x,y))-1).\]
\end{prop}
\begin{proof}
Suppose that $FM_{n}^{-}(x,y)>0$ and $m$ is the least natural number with $x*2^{m}=2^{n}$. Then
by Proposition \ref{hfqhq4hwbujddxn},
\[\underbrace{(1,\ldots,1)}_{2^{n}-times}=\underbrace{(1,\ldots,1)}_{x-times}*\underbrace{(1,\ldots,1)}_{2^{m}-times}\]
\[\simeq\underbrace{(1,\ldots,1)}_{x-times}*(1+O_{n}(x*1),\ldots,1+O_{n}(x*2^{m})).\]
Therefore, since
\[\underbrace{(1,\ldots,1)}_{2^{n}-times}[y]=1,\]
and
\[\underbrace{(1,\ldots,1)}_{x-times}*(1+O_{n}(x*1),\ldots,1+O_{n}(x*2^{m}))[y]=1+O_{n}(x*FM_{n}^{-}(x,y)),\]
we conclude that
\[o_{n}(y-1)\leq o_{n}((x*FM_{n}^{-}(x,y))-1).\]
\end{proof}
\subsection{Nightmare Laver tables}
We shall now give an example of a finite multigenic Laver table $M$ such that there do not exist rank-into-rank embeddings
$(j_{a})_{a\in A}$ in $\mathcal{E}_{\lambda}$ and some $\gamma<\lambda$ such that
$\mathbf{M}(([j_{a}])_{a\in A},\mathcal{E}_{\lambda}/\equiv^{\gamma})=M$. The following lemma gives a condition as to which
Laver-like LD-systems embed into $\mathcal{E}_{\lambda}/\equiv^{\gamma}$.
\begin{lem}
Suppose that $j:V_{\lambda}\rightarrow V_{\lambda}$ is an elementary embedding. Then
\[(j*j)(\alpha)\leq j(\alpha)\]
for each $\alpha<\lambda$.
\label{t42ijt0t24j}
\end{lem}
\begin{proof}
Suppose that $\alpha<\lambda$. Then let $\beta$ be the least ordinal with $j(\beta)>\alpha$. Then
\[V_{\lambda}\models\forall x<\beta,j(x)\leq\alpha.\]
Therefore, by applying the elementary embedding $j$, we have
\[V_{\lambda}\models\forall x<j(\beta),(j*j)(x)\leq j(\alpha).\]
Therefore, since $\alpha<j(\beta)$, we conclude that
\[(j*j)(\alpha)\leq j(\alpha).\]
\end{proof}
\begin{defn}
A multigenic Laver table $M$ is said to be a \index{nightmare}\emph{nightmare} Laver table if there exists some $\mathbf{x}\in M$ and some
critical point $\alpha$ such that $(\mathbf{x}*\mathbf{x})^{\sharp}(\alpha)>\mathbf{x}^{\sharp}(\alpha)$.
\end{defn}
By Lemma \ref{t42ijt0t24j}, if $M$ is a nightmare Laver table over an alphabet $A$, then there do not exist
elementary embeddings $j_{a}\in\mathcal{E}_{\lambda}$ for $a\in A$ and some $\gamma<\lambda$ such that
$\mathbf{M}(([j_{a}])_{a\in A},\mathcal{E}_{\lambda}/\equiv^{\gamma})=M$.

We chose the term ``nightmare" mainly because the nightmare Laver tables seem to suggest that there may be an inconsistency around the level I3-I0 on the large cardinal hierarchy. If one were to ever produce from a large cardinal axiom rank-into-rank embeddings $(j_{a})_{a\in A}$ and some $\gamma<\lambda$ such that $\mathbf{M}(([j_{a}])_{a\in A},\mathcal{E}_{\lambda}/\equiv^{\gamma})$ were a nightmare Laver table, then such a large cardinal axiom would be inconsistent.

\begin{exam}
The following set is one of the two nightmare Laver tables over the alphabet 0,1 of cardinality 64. The other nightmare Laver table
over the alphabet 0,1 of cardinality 64 is obtained from the following table by replacing each instance of a 0 with an instance of a 1.
There does not exist a nightmare Laver table cardinality less than 64.

\{ 0, 1, 00, 10, 01, 000, 100, 010, 011, 1000, 0000, 0100, 0110, 00000, 00001, 01000, 01100, 000000, 000010, 011000, 010000, 0000000, 0000100,
  0100000, 0100001, 00001000, 01000000, 01000010, 010000000, 010000100, 0100001000, 11, 001, 101, 0001, 1001, 0101, 0111, 10000, 10001, 01001, 01101,
  000001, 000011, 010001, 011001, 0000001, 0000101, 0110000, 0110001, 00000000, 00000001, 00001001, 01000001, 01000011, 000010000, 000010001, 010000001,
  010000101, 0100000000, 0100000001, 0100001001, 01000010000, 01000010001 \}
\end{exam}

\subsection{Computing in multigenic Laver tables}
In this section, we shall present efficient algorithms for computing $FM_{n}^{-}(x,y)$ and $\mathbf{x}*_{n}\mathbf{y}$. The limiting factor in these algorithms is computing the application in the classical Laver table $A_{n}$. Since $A_{48}$ is the largest classical Laver table computed \cite{RD95}, we shall only be able to compute $FM_{n}^{-}(x,y)$ and $\mathbf{x}*_{n}\mathbf{y}$ if $n\leq 48$. 

Even though the double recursive algorithm for computing the application operation in a multigenic Laver table $M$ is extremely inefficient, if one can easily compute whether $\mathbf{x}*\mathbf{y}\in\mathrm{Li}(M)$ or not, then by
Theorem \ref{t4h2u9qwnu93qr} and Corollary \ref{t497h32nuofqnu}, a slight improvement of this double recursion algorithm can compute $\mathbf{x}*\mathbf{y}$ in precisely $2(|\mathbf{x}*\mathbf{y}|-|\mathbf{x}|)-1$ steps.

\begin{defn}
Let $M$ be a pre-multigenic Laver table. Suppose that $\mathbf{x}_{1},\ldots,\mathbf{x}_{n}\in M$. Then we say that
$(\mathbf{x}_{1},\ldots,\mathbf{x}_{n})$ is a \index{forward sequence}\emph{forward sequence} if whenever $1\leq m<n$,
\begin{enumerate}
\item $\mathbf{x}_{1}*\ldots*\mathbf{x}_{m}\not\in\mathrm{Li}(M)$, and

\item if $\mathbf{z}$ is a non-empty proper prefix of $\mathbf{x}_{m+1}$ then
$\mathbf{x}_{1}*\ldots*\mathbf{x}_{m}*\mathbf{z}\not\in\mathrm{Li}(M)$.
\end{enumerate}
Let $\mathbf{FS}(M)$ denote the set of all forward sequences in the pre-multigenic Laver table $M$.
\end{defn}
\begin{defn}
Let \index{$\mathrm{Red}(\mathbf{x},\mathbf{y})$}$\mathrm{Red}(\mathbf{x},\mathbf{y})$ denote the shortest nonempty string where for some string $\mathbf{z}$ we have $\mathbf{y}=\mathbf{z}\mathrm{Red}(\mathbf{x},\mathbf{y})$ and either $\mathbf{x}*\mathbf{z}\in\mathrm{Li}(M)$ or $|\mathbf{z}|=0$.
\end{defn}
\begin{defn}
Let \index{$\mathrm{Eval}:\mathbf{FS}(M)\rightarrow\mathbf{FS}(M)$}$\mathrm{Eval}:\mathbf{FS}(M)\rightarrow\mathbf{FS}(M)$ by the mapping where
\begin{enumerate}
\item \[\mathrm{Eval}(\mathbf{x})=\mathbf{x},\]

\item \[\mathrm{Eval}(\mathbf{x},a,\mathbf{x}_{1},\ldots,\mathbf{x}_{n})=(\mathbf{x}a,\mathbf{x}_{1},\ldots,\mathbf{x}_{n}),\]
and
\item \[\mathrm{Eval}(\mathbf{x},\mathbf{y}a,\mathbf{x}_{1},\ldots,\mathbf{x}_{n})
=(\mathbf{x},\mathbf{y},\mathrm{Red}(\mathbf{x}*\mathbf{y},\mathbf{x}a),\mathbf{x}_{1},\ldots,\mathbf{x}_{n})\]
whenever these equations make sense.
\end{enumerate}
\end{defn}
Take note that if
\[\mathrm{Eval}^{k}(\mathbf{x}_{1},\ldots,\mathbf{x}_{m})=(\mathbf{y}_{1},\ldots,\mathbf{y}_{n}),\]
then
\[\mathbf{x}_{1}*\ldots*\mathbf{x}_{m}=\mathbf{y}_{1}*\ldots*\mathbf{y}_{n}.\]
Therefore, if $\mathrm{Eval}^{k}(\mathbf{x}_{1},\ldots,\mathbf{x}_{m})=(\mathbf{x})$, then
$\mathbf{x}=\mathbf{x}_{1}*\ldots*\mathbf{x}_{m}$. One could thus compute the application
$\mathbf{x}*\mathbf{y}$ whenever $\mathbf{x}\not\in Li(M)$ by iteratively evaluating $\mathrm{Eval}^{k}(\mathbf{x},\text{Red}(\mathbf{x},\mathbf{y}))$ until one obtains $\mathrm{Eval}^{k}(\mathbf{x},\text{Red}(\mathbf{x},\mathbf{y}))=(\mathbf{z})$ for some $\mathbf{z}$.

\begin{prop}
Let $(\mathbf{x}_{1},\ldots,\mathbf{x}_{n})\in\mathbf{FS}(M)$, and suppose that
\[m=2\cdot(|\mathbf{x}_{1}*\ldots*\mathbf{x}_{n}|-|\mathbf{x}_{1}|)-(n-1).\]
Then $m$ is the least natural number such that $\mathrm{Eval}^{m}(\mathbf{x}_{1},\ldots,\mathbf{x}_{n})$
has length $1$.
\label{t4h2u9qwnu93qr}
\end{prop}
\begin{proof}
I first claim that $m=0$ if and only if $n=1$.
If $n=1$, then
\[m=2(|\mathbf{x}_{1}*\ldots*\mathbf{x}_{n}|-|\mathbf{x}_{1}|)-(n-1)=0-0=0.\]
Now if $n>1$, then
\[m=2(|\mathbf{x}_{1}*\ldots*\mathbf{x}_{n}|-|\mathbf{x}_{1}|)-(n-1)\geq 2(n-1)-(n-1)=n-1>0.\]
Therefore, $m=0$ iff $n=1$.

We shall now prove this result by induction on $m$. The case where $m=0$ follows from the above remarks.
Now assume that $m>0$. Then $n>1$, and we perform the inductive step in two cases.

\item[Case 1: $|\mathbf{x}_{2}|=1$.]
We have
\[\mathrm{Eval}(\mathbf{x}_{1},\ldots,\mathbf{x}_{n})=(\mathbf{x}_{1}\mathbf{x}_{2},\mathbf{x}_{3},\ldots,\mathbf{x}_{n}).\]
Let 
\[m_{1}=2(|\mathbf{x}_{1}\mathbf{x}_{2}*\mathbf{x}_{3}*\ldots*\mathbf{x}_{n}|-|\mathbf{x}_{1}\mathbf{x}_{2}|)-((n-1)-1).\]

Then $m_{1}=m-1$, so by the induction hypothesis, $m_{1}$ is the least natural number where
$\mathrm{Eval}^{m_{1}}(\mathbf{x}_{1}\mathbf{x}_{2},\mathbf{x}_{3},\ldots,\mathbf{x}_{n})$ has length $1$. Therefore,
$m=m_{1}+1$ is the least natural number where $\mathrm{Eval}^{m}(\mathbf{x}_{1},\ldots,\mathbf{x}_{n})$ has length 1.

\item[Case 2: $|\mathbf{x}_{2}|>1$.]
In this case, suppose that $\mathbf{x}_{2}=\mathbf{y}a$. Then
\[\mathrm{Eval}(\mathbf{x}_{1},\ldots,\mathbf{x}_{n})
=(\mathbf{x}_{1},\mathbf{y},\mathrm{Red}(\mathbf{x}_{1}*\mathbf{y},\mathbf{x}_{1}a),\mathbf{x}_{3},\ldots,\mathbf{x}_{n}).\]
Suppose now that 
\[m_{1}=2(|\mathbf{x}_{1}*\mathbf{y}*\mathrm{Red}(\mathbf{x}_{1}*\mathbf{y},\mathbf{x}_{1}a)*x_{3}*\ldots*x_{n}|-|x_{1}|)-(n+1-1).\]
Then
\[m_{1}=2(|x_{1}*\ldots*x_{n}|-|x_{1}|)-(n-1)-1=m-1.\]
Therefore, by the induction hypothesis, $m_{1}$ is the least natural number such that
\[\mathrm{Eval}^{m_{1}}(\mathrm{Eval}(\mathbf{x}_{1},\ldots,\mathbf{x}_{n}))=\mathrm{Eval}^{m_{1}}(\mathbf{x}_{1},\mathbf{y},
\mathrm{Red}(\mathbf{x}_{1}*\mathbf{y},\mathbf{x}_{1}a),\mathbf{x}_{3},\ldots,\mathbf{x}_{n})\]
has length $1$. Thus, $m$ is the least natural number such that $\mathrm{Eval}^{m}(\mathbf{x}_{1},\ldots,\mathbf{x}_{n})$ has
length $1$.
\end{proof}
\begin{cor}
Let $\mathbf{x}\in M\setminus\mathrm{Li}(M),\mathbf{y}\in M$ and let $m=2(|\mathbf{x}*\mathbf{y}|-|\mathbf{x}|)-1$.
Then $m$ is the least natural number such that
$\mathrm{Eval}^{m}(\mathbf{x},\mathrm{Red}(\mathbf{x},\mathbf{y}))$ has length $1$. Furthermore, if
$\mathrm{Eval}^{m}(\mathbf{x},\mathrm{Red}(\mathbf{x},\mathbf{y}))=(\mathbf{z})$, then
$\mathbf{x}*\mathbf{y}=\mathbf{z}$.
\label{t497h32nuofqnu}
\end{cor}
In some cases, one is able to compute $\mathrm{Eval}^{k}$ without having to
apply $k$-times the operation $\mathrm{Eval}$; if $(\mathbf{x},\mathbf{y},\mathbf{x}_{1},\ldots,\mathbf{x}_{n})$ is a forward sequence and
$|\mathbf{x}*\mathbf{y}|=|\mathbf{x}|+|\mathbf{y}|$, then $\mathbf{x}*\mathbf{y}=\mathbf{xy}$. It follows from Theorem
\ref{t4h2u9qwnu93qr} that if $|\mathbf{x}*\mathbf{y}|=|\mathbf{x}|+|\mathbf{y}|$, then
\[\mathrm{Eval}^{k}(\mathbf{x},\mathbf{y},\mathbf{x}_{1},\ldots,\mathbf{x}_{n})
=(\mathbf{xy},\mathbf{x}_{1},\ldots,\mathbf{x}_{n})\]
where $k=2\cdot|\mathbf{y}|-1$. The following algorithm is based on this technique for quickly computing $\mathrm{Eval}^{k}$.

We shall write \index{$(\mathbf{x})_{r}$}$(\mathbf{x})_{r}$ for the unique non-empty string such that
$|(\mathbf{x})_{r}|\leq r$ and $\mathbf{x}=\mathbf{z} (\mathbf{x})_{r}$ for some $\mathbf{z}$ with $|\mathbf{z}|=0\mod r$.

\begin{algo}
\label{r2g466ui8b2we}
The following steps may be used to evaluate the operation $*_{n}$ in $(A^{\leq 2^{n}})^{+}$.
\begin{enumerate}
\item Evaluate $\mathbf{x}*_{n}\mathbf{y}$ to $\mathbf{y}$ whenever $|\mathbf{x}|=2^{n}$.

\item Evaluate $\mathbf{x}*_{n}\mathbf{y}$ to $\mathbf{x}*_{n}\mathrm{Red}(\mathbf{x},\mathbf{y})$ whenever $\mathbf{y}\neq\mathrm{Red}(\mathbf{x},\mathbf{y}),|\mathbf{x}|<2^{n}$.

\item Evaluate $\mathbf{x}*_{n}\mathbf{y}$ to $\mathbf{xy}$ whenever $|\mathbf{x}|*_{n}|\mathbf{y}|=|\mathbf{x}|+|\mathbf{y}|$ and
$|\mathbf{y}|\leq O_{n}(|\mathbf{x}|)$.

\item Evaluate $\mathbf{x}*_{n}\mathbf{y}a$ to $\mathbf{x}*_{n}\mathbf{y}*_{n}(\mathbf{x}a)_{O_{n}(|\mathbf{x}|*_{n}|\mathbf{y}|)}$ whenever the above steps are not applicable.
\end{enumerate}
\end{algo}

Using Algorithm \ref{r2g466ui8b2we} to compute in $(A^{\leq 2^{n}})^{+}$ (where Dougherty's algorithm is used to compute in $A_{n}$),
we have been able to compute $a*_{27}a_{1}\ldots a_{16}$ in 66 seconds on a standard laptop computer in the language GAP even though
$|a*_{27}a_{1}\ldots a_{16}|=2^{27}$.

The following few results are needed to refine Algorithm \ref{r2g466ui8b2we} for computing $*_{n}$ in 
$(A^{\leq 2^{n}})^{+}$.

\begin{thm}
Assume that $0<s\leq n$. Then the mapping $L:A_{n-s}\rightarrow A_{n}$ defined by $L(x)=x\cdot 2^{s}$ is a homomorphism if and only if
$n-s\leq \gcd(2^{n},2s)$. 
\end{thm}
\begin{proof}
See \cite{AD93}.
\end{proof}
\begin{cor}
Assume that $0<s\leq n,n-s\leq \gcd(2^{n},2s)$. Then $2^{s}x*_{n}y=2^{s}x+y$ whenever $1\leq y\leq 2^{s}$.
\label{ytv9h4n4tr4wurGNHW8E}
\end{cor}

\begin{lem}
Suppose that 
\[0<s\leq n,n-s\leq \gcd(2^{n},2s)\]
and
\[|\mathbf{xy}|<2^{n},|\mathbf{x}|\neq 0,0<|\mathbf{y}|<2^{s},|\mathbf{x}|=0\mod 2^{s}.\]
Then \[\mathbf{xy}=\mathbf{x}\circ_{n}\mathbf{y}.\]

\label{53rguh9pg34hyu89p4t2}
\end{lem}
\begin{proof}
By Corollary \ref{ytv9h4n4tr4wurGNHW8E}, we have $|\mathbf{x}|*_{n}|\mathbf{y}a|=|\mathbf{x}|+|\mathbf{y}a|$ and
$|\mathbf{y}a|\leq O_{n}(|\mathbf{x}|)$. Therefore,
\[(\mathbf{x}\circ_{n}\mathbf{y})a=\mathbf{x}*_{n}\mathbf{y}a=\mathbf{x}\mathbf{y}a.\]
We conclude that 
\[\mathbf{x}\circ_{n}\mathbf{y}=\mathbf{x}\mathbf{y}.\]
\end{proof}
\begin{lem}
Suppose that $0<s\leq n,n-s\leq \gcd(2^{n},2s)$. Define a mapping 
\[\iota:((A^{2^{s}})^{\leq 2^{n-s}})^{+}\rightarrow(A^{\leq 2^{n}})^{+}\] by letting
\[\iota(\langle a_{1,1},\ldots,a_{1,2^{s}}\rangle\ldots\langle a_{r,1},\ldots,a_{r,2^{s}}\rangle)
=a_{1,1}\ldots a_{1,2^{s}}\ldots a_{r,1}\ldots a_{r,2^{s}}.\]
Then $\iota$ is a homomorphism between multigenic Laver tables.
\label{h42ur9}
\end{lem}
\begin{proof}
Take note that if $|\mathbf{x}|$ is a multiple of $2^{s}$ with $|\mathbf{x}|<2^{n}$, and $|\mathbf{y}|=2^{s}$, then since
$|\mathbf{x}|*_{n}|\mathbf{y}|=|\mathbf{x}|+|\mathbf{y}|$ and $|\mathbf{y}|\leq O_{n}(|\mathbf{x}|)$ we have
$\mathbf{x}*_{n}\mathbf{y}=\mathbf{x}\mathbf{y}$.

Therefore, by induction, we have
\[(a_{1,1}\ldots a_{2^{s},1})*_{n}\ldots*_{n}(a_{1,r}\ldots a_{2^{s},r})=(a_{1,1}\ldots a_{2^{s},1})\ldots(a_{1,r}\ldots a_{2^{s},r})\]
whenever
$1\leq r\leq 2^{n-s}$.
Therefore, we have
\[(a_{1,1}\ldots a_{2^{s},1})*_{n}\ldots*_{n}(a_{1,2^{n-s}}\ldots a_{2^{s},2^{n-s}})*_{n}(a_{1}\ldots a_{2^{s}})\]
\[=(a_{1,1}\ldots a_{2^{s},1})\ldots(a_{1,2^{n-s}}\ldots a_{2^{s},2^{n-s}})*_{n}(a_{1}\ldots a_{2^{s}})\]
\[=a_{1}\ldots a_{2^{s}}.\]

By Theorem \ref{enbfwu92}, we may therefore conclude that the mapping $\iota$ is a homomorphism between LD-systems.
\end{proof}
\begin{lem}
Suppose that $0<s\leq n,n-s\leq \gcd(2^{n},2s)$. Suppose furthermore that $u,v\in\{1,\ldots,2^{n-s}\},w\in\{1,\ldots,2^{s}\}$, and
\[\mathbf{x}=(x_{1,1}\ldots x_{1,2^{s}})\ldots(x_{u,1}\ldots x_{u,2^{s}}),\]
and
\[\mathbf{y}=(y_{1,1}\ldots y_{1,2^{s}})\ldots(y_{v-1,1}\ldots y_{v-1,2^{s}})(y_{v,1}\ldots y_{v,w}),\]
and
\[\langle x_{1,1},\ldots,x_{1,2^{s}}\rangle\ldots\langle x_{u,1},\ldots,x_{u,2^{s}}\rangle\]
\[*_{n-s}\langle y_{1,1},\ldots,y_{1,2^{s}}\rangle\ldots\langle y_{v-1,1},\ldots,y_{v-1,2^{s}}\rangle\langle y_{v,1},\ldots,y_{v,w}\rangle\]
\[=\langle z_{1,1},\ldots,z_{1,2^{s}}\rangle\ldots\langle z_{p-1,1},\ldots,z_{p-1,2^{s}}\rangle\langle z_{p,1},\ldots,z_{p,w}\rangle.\]

Then
\[\mathbf{x}*_{n}\mathbf{y}=(z_{1,1}\ldots z_{1,2^{s}})\ldots(z_{p-1,1}\ldots z_{p-1,2^{s}})(z_{p,1}\ldots z_{p,w}).\]
\end{lem}
\begin{proof}
If $w=2^{s}$, the the result follows directly from Proposition \ref{h42ur9}.

If $w=1$, then by Proposition \ref{h42ur9}, we have
\[\mathbf{x}*_{n}\mathbf{y}=(\mathbf{x}\circ_{n}(y_{1,1}\ldots y_{1,2^{s}})\ldots(y_{v-1,1}\ldots y_{v-1,2^{s}}))y_{v,1}\]

\[=\iota(\langle x_{1,1},\ldots,x_{1,2^{s}}\rangle\ldots\langle x_{u,1},\ldots,x_{u,2^{s}}\rangle\circ_{n-s}
\langle y_{1,1},\ldots,y_{1,2^{s}}\rangle\ldots\langle y_{v-1,1},\ldots,y_{v-1,2^{s}}\rangle)z_{p,1}\]

\[=\iota(\langle z_{1,1},\ldots,z_{1,2^{s}}\rangle\ldots\langle z_{p-1,1},\ldots,z_{p-1,2^{s}}\rangle)z_{p,1}\]

\[=(z_{1,1}\ldots z_{1,2^{s}})\ldots(z_{p-1,1}\ldots z_{p-1,2^{s}})z_{p,1}.\]

We shall now prove the general case of this result. Since $o_{n}(|\mathbf{x}|)\geq s$, we conclude that
\[\mathbf{x}*_{n}(y_{1,1}\ldots y_{1,2^{s}})\ldots(y_{v-1,1}\ldots y_{v-1,2^{s}})(y_{v,1}\ldots y_{v,w})\]
is a proper extension of the string
\[\mathbf{x}*_{n}(y_{1,1}\ldots y_{1,2^{s}})\ldots(y_{v-1,1}\ldots y_{v-1,2^{s}})(y_{v,1}\ldots y_{v,w-1})\]
whenever $1<w\leq 2^{s}$, and the last symbol of
\[\mathbf{x}*_{n}(y_{1,1}\ldots y_{1,2^{s}})\ldots(y_{v-1,1}\ldots y_{v-1,2^{s}})(y_{v,1}\ldots y_{v,w})\]
is $y_{v,w}$, and since
\[\mathbf{x}*_{n}(y_{1,1}\ldots y_{1,2^{s}})\ldots(y_{v-1,1}\ldots y_{v-1,2^{s}})y_{v,1}\]
\[=(z_{1,1}\ldots z_{1,2^{s}})\ldots(z_{p-1,1}\ldots z_{p-1,2^{s}})y_{v,1}\]
and
\[\mathbf{x}*_{n}(y_{1,1}\ldots y_{1,2^{s}})\ldots(y_{v-1,1}\ldots y_{v-1,2^{s}})(y_{v,1}\ldots y_{v,2^{s}})\]
\[=(z_{1,1}\ldots z_{1,2^{s}})\ldots(z_{p-1,1}\ldots z_{p-1,2^{s}})(y_{v,1}\ldots y_{v,2^{s}}),\]
we conclude that
\[\mathbf{x}*_{n}(y_{1,1}\ldots y_{1,2^{s}})\ldots(y_{v-1,1}\ldots y_{v-1,2^{s}})(y_{v,1}\ldots y_{v,w})\]
\[=(z_{1,1}\ldots z_{1,2^{s}})\ldots (z_{p-1,1}\ldots z_{p-1,2^{s}})(y_{v,1}\ldots y_{v,w})\]
for $1\leq w\leq 2^{s}$.
\end{proof}

The following algorithm uses the above lemmas to quickly compute the application operation in $(A^{\leq 2^{n}})^{+}$ often more quickly 
than Algorithm \ref{r2g466ui8b2we}.
\begin{algo}
\label{t42akm1q9uo}
The below algorithm may be used to evaluate $\mathbf{x}*_{n}\mathbf{y}$ in $(A^{\leq 2^{n}})^{+}$.
We first select some $s$ with $0<s\leq n,n-s\leq\gcd(2^{\infty},2s)$. We then evaluate the following conditional statement.
\begin{enumerate}
\item If $|\mathbf{y}|>O_{n}(|\mathbf{x}|)$, then evaluate $\mathbf{x}*_{n}\mathbf{y}$ to
$\mathbf{x}*_{n}(\mathbf{y})_{O_{n}(|\mathbf{x}|)}$,

\item else if $|\mathbf{x}|=2^{n}$, then evaluate $\mathbf{x}*_{n}\mathbf{y}$ to $\mathbf{y}$,

\item else if $|\mathbf{x}|*_{n}|\mathbf{y}|=|\mathbf{x}|+|\mathbf{y}|$, then evaluate $\mathbf{x}*_{n}\mathbf{y}$ to $\mathbf{xy}$,

\item else if $|\mathbf{x}|<2^{s}$ then let $\mathbf{y}=\mathbf{z}a$ and evaluate $\mathbf{x}*_{n}\mathbf{y}$ to
\[(\mathbf{x}*_{n}\mathbf{z})*_{n}(\mathbf{x}a)_{O_{n}(|\mathbf{x}|*_{n}|\mathbf{z}|)},\]

\item else if $|\mathbf{x}|$ is a multiple of $2^{s}$, then let 
\[\mathbf{x}=(x_{1,1}\ldots x_{1,2^{s}})\ldots(x_{u,1}\ldots x_{u,2^{s}}),\]
and let
\[\mathbf{y}=(y_{1,1}\ldots y_{1,2^{s}})\ldots(y_{v-1,1}\ldots y_{v-1,2^{s}})(y_{v,1}\ldots y_{v,w}),\]
and suppose that
\[\langle x_{1,1},\ldots,x_{1,2^{s}}\rangle\ldots\langle x_{u,1},\ldots,x_{u,2^{s}}\rangle\]
\[*_{n-s}\langle y_{1,1}\ldots y_{1,2^{s}}\rangle\ldots\langle y_{v-1,1}\ldots y_{v-1,2^{s}}\rangle\langle y_{v,1}\ldots y_{v,w}\rangle\]
\[=\langle z_{1,1},\ldots,z_{1,2^{s}}\rangle\ldots\langle z_{p-1,1},\ldots,z_{p-1,2^{s}}\rangle\langle z_{p,1},\ldots,z_{p,w}\rangle,\]
and evaluate $\mathbf{x}*_{n}\mathbf{y}$ to
\[(z_{1,1}\ldots z_{1,2^{s}})\ldots(z_{p-1,1}\ldots z_{p-1,2^{s}})(z_{p,1}\ldots z_{p,w}),\]

\item else let $\mathbf{x}=\mathbf{x}_{1}\mathbf{x}_{2}$ where $|\mathbf{x}_{1}|$ is a necessarily non-zero multiple of $2^{s}$ and
$|\mathbf{x}_{2}|<2^{s}$, and then evaluate $\mathbf{x}*_{n}\mathbf{y}$ to $\mathbf{x}_{1}*_{n}((\mathbf{x}_{2})_{2^{m}}*_{m}(\mathbf{y})_{2^{m}})$ where $m=o_{n}(|\mathbf{x}_{1}|)$. This last step is valid since $\mathbf{x}=\mathbf{x}_{1}\circ_{n}\mathbf{x}_{2}$ and hence
\[\mathbf{x}*_{n}\mathbf{y}=(\mathbf{x}_{1}\circ_{n}\mathbf{x}_{2})*_{n}\mathbf{y}=\mathbf{x}_{1}*_{n}(\mathbf{x}_{2}*_{n}\mathbf{y})
=\mathbf{x}_{1}*_{n}((\mathbf{x}_{2})_{2^{m}}*_{m}(\mathbf{y})_{2^{m}}).\]
\end{enumerate}
\end{algo}

In the language GAP, we were able to compute $a*_{24}a_{1}\ldots a_{16}$ in 1203 ms using Algorithm \ref{t42akm1q9uo}. In comparison, it took
8672 ms to compute $a*_{24}a_{1}\ldots a_{16}$ using Algorithm \ref{r2g466ui8b2we}. However,
using Algorithm \ref{t42akm1q9uo}, we took 17563 ms to compute $a*_{25}a_{1}\ldots a_{16}$.

We shall now give algorithms corresponding to \ref{r2g466ui8b2we} and \ref{t42akm1q9uo} to quickly compute $M_{n}(x,y,\ell)$
and $FM_{n}^{-}(x,y)$ whenever $n\leq 48$.

\begin{algo}
\label{ajujiuhu4h4ef3ie}

The following recursive conditional statement may be used to quickly compute $M_{n}(x,y,\ell)$ whenever one can efficiently compute in $A_{n}$. Let $s$ be a natural number with $0<s\leq n,n-s\leq\gcd(2^{\infty},2s)$.
\begin{enumerate}
\item If $y>O_{n}(x)$ and $M_{n}(x,(y)_{O_{n}(x)},\ell)<0$, then return $M_{n}(x,(y)_{O_{n}(x)},\ell)$,

\item else if $y>O_{n}(x)$ and $M_{n}(x,(y)_{O_{n}(x)},\ell)>0$, then return
\[M_{n}(x,(y)_{O_{n}(x)},\ell)+y-(y)_{O_{n}(x)},\]

\item else let $y'$ be the least natural number with $\ell\leq x*_{n}y'$, and if $y'<y$, then return $M_{n}(x,y',\ell)$,

\item else if $x=2^{n}$, then return $\ell$,

\item else if $\ell\leq x$, then return $-\ell$,

\item else if $\ell=x*_{n}y$, then return $y$,

\item else let $q=M_{n}(x*_{n}(y-1),x+1,\ell)$, and
\begin{enumerate}
\item if $q=x+1$, then return $y$,

\item else if $q>0$ then return $-q$,

\item else return $M_{n}(x,y-1,-q)$.
\end{enumerate}
\end{enumerate}
\end{algo}
\begin{algo}
\label{5v4qpcterhivtryuavw}
The following recursive conditional statement may be used to quickly compute $M_{n}(x,y,\ell)$.

Let $s$ be a natural number with $0<s\leq n,n-s\leq\gcd(2^{\infty},2s)$.
\begin{enumerate}
\item If $\ell>O_{n}(x)$ and $M_{n}(x,(y)_{O_{n}(x)},\ell)<0$, then return $M_{n}(x,(y)_{O_{n}(x)},\ell)$,

\item else if $\ell>O_{n}(x)$ and $M_{n}(x,(y)_{O_{n}(x)},\ell)>0$, then return
\[M_{n}(x,(y)_{O_{n}(x)},\ell)+y-(y)_{O_{n}(x)},\]

\item else if $x=2^{n}$, then return $\ell$,

\item else if $\ell\leq x$, then return $-\ell$,

\item else if $\ell=x*_{n}y$, then return $y$,

\item else if $x<2^{s}$, then let $q=M_{n}(x*_{n}(y-1),x+1,\ell)$, and
\begin{enumerate}
\item if $q=x+1$, then return $y$,

\item else if $q>0$, then return $-q$,

\item else return $M_{n}(x,y-1,-q)$,
\end{enumerate}
\item else if $x=0\mod 2^{s}$, then let
\[j=(l)_{2^{s}},i=(l-j)\cdot 2^{-s}+1,v=M_{n-s}(x\cdot 2^{-s},2^{o_{n}(x)-s},i),\]
and
\begin{enumerate}
\item if $v>0$ then return $j+(v-1)*2^{s}$,

\item else return $-j+(v+1)*2^{s}$,
\end{enumerate}

\item else let $j=(x)_{2^{s}},i=x-j,q=M_{n}(i,j*_{n}y,\ell)$, and
\begin{enumerate}
\item if $q<0$ then return $q$,

\item else if $M_{n}(j,y,q)>0$ then return $M_{n}(j,y,q)$,

\item else return $M_{n}(j,y,q)-i$.
\end{enumerate}
\end{enumerate}
\end{algo}

Algorithm \ref{5v4qpcterhivtryuavw} when implemented in GAP computes a random entry $FM_{48}^{-}(x,y)$ in 1 ms on average.
Algorithm \ref{5v4qpcterhivtryuavw} is about 20 times faster at computing random entries in
$FM_{48}^{-}(x,y)$ than Algorithm \ref{ajujiuhu4h4ef3ie}. By comparison, we were able to compute approximately 1200000 random values $x*_{48}y$ in $A_{48}$ in a minute.

\section{Vastness and free algebras}
In this section, we shall investigate inverse limits of permutative LD-systems. We shall show that if one assumes large cardinal hypotheses, then free left-distributive algebras on an arbitrary number of generators embed into inverse limits
of multigenic Laver tables. Furthermore, we shall define a class of algebraic structures that model the algebra $\mathcal{E}_{\lambda}$ in the same way that the permutative LD-systems model the quotient algebra $\mathcal{E}_{\lambda}/\equiv^{\gamma}$.

\begin{defn}
\label{4rh4c48bwrt8wgy}
Suppose that $(X_{n})_{n\in\omega}$ is an inverse system of permutative LD-systems with transitional mappings
$\phi_{n}:X_{n+1}\rightarrow X_{n}$. Then we shall say that $((X_{n})_{n\in\omega},(\phi_{n})_{n\in\omega})$ is \index{vast}\emph{vast} if
\begin{enumerate}
\item for each $n$, there is some $\alpha_{n}\in \mathrm{crit}[X_{n+1}]$ such that
$\ker(\phi_{n})$ is equal to $\equiv^{\alpha}$, and

\item for each $x,y\in X_{n}$, there is some $m\geq n$ and $x',y'\in X_{m}$ such that $\phi_{m,n}(x')=x,\phi_{m,n}(y')=y$ and
$x'*y'\not\in \mathrm{Li}(X_{m})$ where $\phi_{m,n}$ is the transitional mapping between $X_{m}$ to $X_{n}$.
\end{enumerate}
\end{defn}

If $(X_{n})_{n\in\omega}$ is an inverse system of permutative LD-systems that satisfies $(1)$ in Definition \ref{4rh4c48bwrt8wgy}, then we shall hold to the conventions that $\mathrm{crit}[X_{n}]\subseteq\mathrm{crit}[X_{n+1}]$,
$\mathrm{crit}(\phi_{n}(x))=\mathrm{crit}(x)$ whenever $\phi_{n}(x)\not\in \mathrm{Li}(X_{n})$, and $\mathrm{crit}[X_{n}]$ is an initial segment of $\mathrm{crit}[X_{n}]$. We shall also define
\[\mathrm{crit}[(X_{n})_{n\in\omega}]=\bigcup_{n\in\omega}\mathrm{crit}[X_{n}].\]
An inverse system $((X_{n})_{n\in\omega},(\phi_{n})_{n\in\omega})$ that satisfies $(1)$ in Definition \ref{4rh4c48bwrt8wgy} is vast precisely when
\begin{enumerate}
\item for all $x\in X_{n},\alpha\in\mathrm{crit}[X_{n}]$, there exists some $m>n$ and $y\in X_{m}$ with $\phi_{m,n}(y)=x$ but where $y^{\sharp}(\alpha)<\max(\mathrm{crit}[X_{n}])$, and

\item for all $x\in X_{n}$ there is some $m>n$ and some $y\in X_{m}$ with $\phi_{m,n}(y)=x$ and where
$y\not\in\mathrm{Li}(X_{m})$.
\end{enumerate}

\begin{exam}
The motivation behind the notion of vastness comes from the fact that the inverse system 
$(\mathcal{E}_{\lambda}/\equiv^{\kappa_{n}})_{n\in\omega}$ of quotient algebras of elementary embeddings is vast whenever $\kappa_{n}=j^{n}(\mathrm{crit}(j))$ for some non-trivial $j\in\mathcal{E}_{\lambda}$.
\end{exam}
\begin{prop}
\label{40q49yhp733y378qff70}
\begin{enumerate}
\item An inverse system of permutative LD-systems 
\[((X_{n})_{n\in\omega},(\phi_{n})_{n\in\omega})\]
is vast if and only if for each $r$-ary term $t$ in the language of LD-systems, if $x_{1},\ldots,x_{r}\in X_{m}$, then there is some $l\geq m$ and
$z_{1},\ldots,z_{r}\in X_{l}$ where
\[\phi_{l,m}(z_{1})=x_{1},\ldots,\phi_{l,m}(z_{r})=x_{r}\]
but where $t(z_{1},\ldots,z_{r})\not\in \mathrm{Li}(X_{l})$.

\item An inverse system of permutative LD-monoids $((X_{n})_{n\in\omega},(\phi_{n})_{n\in\omega})$ is vast if and only if for each $r$-ary term $t$ in the language of LD-monoids that does not include the constant symbol $1$, if $x_{1},\ldots,x_{r}\in X_{m}$, then there is some $l$ and
$z_{1},\ldots,z_{r}\in X_{l}$ where $\phi_{l,m}(z_{1})=x_{1},\ldots,\phi_{l,m}(z_{r})=x_{r}$ but where $t(z_{1},\ldots,z_{r})\not\in\mathrm{Li}(X_{l})$.
\end{enumerate}
\end{prop}
\begin{proof}
\begin{enumerate}

\item The direction $\leftarrow$ is trivial, so we only need to prove the direction $\rightarrow$. Suppose that
$(X_{n})_{n\in\omega}$ is vast.  We define the complexity $C(t)$ of a term $t$ by letting $C(x)=1$ for each variable $x$, and
$C(t_{1}*t_{2})=C(t_{1})+C(t_{2})$. We shall now prove $\rightarrow$ by induction on $C(t)$ in three cases.

\begin{enumerate}
\item Case 1: $t$ is a variable. Suppose that $x\in X_{m}$. Then there is some $l\geq m$ and $y_{1},y_{2}\in X_{l}$ with
$\phi_{l,m}(y_{1})=\phi_{l,m}(y_{2})=x,y_{1}*y_{2}\not\in\mathrm{Li}(X_{l})$. Therefore, $y_{2}\not\in\mathrm{Li}(X_{l})$ as well.

\item Case 2: $t=t'*v$ for some variable $v$. Suppose that
\[t(x_{1},\ldots,x_{r},x)=t'(x_{1},\ldots,x_{r})*x,x_{1},\ldots,x_{r},x\in X_{m},\]
and
\[t'(x_{1},\ldots,x_{r})=t_{1}(x_{1},\ldots,x_{r})*t_{2}(x_{1},\ldots,x_{r}).\]
Then, by the induction hypothesis,
there is some $n\geq m$ along with $y_{1},\ldots,y_{r},y\in X_{n}$ such that 
\[\phi_{n,m}(y_{1})=x_{1},\ldots,\phi_{n,m}(y_{r})=x_{r},\phi_{n,m}(y)=x\]
and $t_{2}(y_{1},\ldots,y_{r})*y\not\in\mathrm{Li}(X_{n})$.

Now let $c=t_{2}(y_{1},\ldots,y_{r})*y$. Then by applying the induction hypothesis again, we conclude that there is some $l\geq n$ along with
$d,z_{1},\ldots,z_{r}\in X_{l}$ with 
\[\phi_{l,n}(d)=c,\phi_{l,n}(z_{1})=y_{1},\ldots,\phi_{l,n}(z_{r})=y_{r}\]
and $t_{1}(z_{1},\ldots,z_{r})*d\not\in\mathrm{Li}(X_{l})$. Let $z\in X_{l}$ be an element with $\phi_{l,n}(z)=y$.
However, we have 
\[\mathrm{crit}(t(z_{1},\ldots,z_{r},z))=\mathrm{crit}(t_{1}(z_{1},\ldots,z_{r})*t_{2}(z_{1},\ldots,z_{r})*z)\]
\[\leq\mathrm{crit}((t_{1}(z_{1},\ldots,z_{r})*t_{2}(z_{1},\ldots,z_{r}))*(t_{1}(z_{1},\ldots,z_{r})*z))\]
\[=\mathrm{crit}(t_{1}(z_{1},\ldots,z_{r})*(t_{2}(z_{1},\ldots,z_{r})*z))\]
\[=\mathrm{crit}(t_{1}(z_{1},\ldots,z_{r})*d)<\max(\mathrm{crit}[X_{l}]).\]

\item Case 3: $t=t_{1}*t_{2}$ and $t_{2}$ is not a term. Suppose that $x_{1},\ldots,x_{r}\in X_{m}$. Then there is some
$n\geq m$ along with $y_{1},\ldots,y_{r}\in X_{n}$ with 
\[\phi_{n,m}(y_{1})=x_{1},\ldots,\phi_{n,m}(y_{r})=x_{r}\]
and where
$t_{2}(y_{1},\ldots,y_{r})\not\in\mathrm{Li}(X_{n})$. Let $y=t_{2}(y_{1},\ldots,y_{r})$. Then, by Case 2 of the induction hypothesis, there is some $l\geq n$ along with $z_{1},\ldots,z_{r},z\in X_{l}$ with 
\[\phi_{l,n}(z_{1})=y_{1},\ldots,\phi_{l,n}(z_{r})=y_{r},\phi_{l,n}(z)=y\]
and where $t_{1}(z_{1},\ldots,z_{r})*z\not\in\mathrm{Li}(X_{l})$. Therefore,
\[\mathrm{crit}(t(z_{1},\ldots,z_{r}))=\mathrm{crit}(t_{1}(z_{1},\ldots,z_{r})*t_{2}(z_{1},\ldots,z_{r}))\]
\[=\mathrm{crit}(t_{1}(z_{1},\ldots,z_{r})*z)<\max(\mathrm{crit}[X_{l}]).\]
\end{enumerate}

\item The direction $\leftarrow$ is trivial. For the direction $\rightarrow$, suppose that $t$ is a term in the language of LD-monoids and
$x_{1},\ldots,x_{r}\in X_{m}$. Then there exists terms $t_{1},\ldots,t_{k}$ in the language of LD-systems such that $t$ is equivalent
to $t_{1}\circ\ldots\circ t_{k}$. Therefore, there is some $l$ and $z_{1},\ldots,z_{r}\in X_{l}$ such that
\[\phi_{l,m}(z_{1})=x_{1},\ldots,\phi_{l,m}(z_{r})=x_{r}\]
and where $t_{1}(z_{1},\ldots,z_{r})\not\in \mathrm{Li}(X_{l})$. Therefore,
\[t(z_{1},\ldots,z_{r})=t_{1}(z_{1},\ldots,z_{r})\circ\ldots\circ t_{k}(z_{1},\ldots,z_{r})\not\in \mathrm{Li}(X_{l})\]
and
\[\phi_{l,m}(z_{1})=x_{1},\ldots,\phi_{l,m}(z_{r})=x_{r}.\]
\end{enumerate}
\end{proof}
\begin{prop}
An inverse system $((X_{n})_{n},(\phi_{n})_{n})$ of permutative LD-systems is vast if and only if for each $x\in X_{m}$ and natural number $r$ there exists some $n\geq m$ and $y\in X_{n}$ with $\phi_{n,m}(y)=x$ and where $y^{r}*y\not\in\mathrm{Li}(X_{n})$.
\end{prop}
\begin{proof}
$\rightarrow$. This direction follows from Proposition \ref{40q49yhp733y378qff70}.

$\leftarrow$. First observe that if $x\in X_{m}$, then there is some $n\geq m$ and $y\in X_{n}$ with $\phi_{n,m}(y)=x$ and
where $y=y^{0}*y\not\in\mathrm{Li}(X_{n})$.

Suppose that $x\in X_{m}$ and $\alpha=\max(\mathrm{crit}[X_{m}])$. Then let $r$ be a natural number with
$x^{r}*x\in\mathrm{Li}(X_{m})$. Then there is some $n>m$ and $y\in X_{n}$ with $\phi_{n,m}(y)=x$ and where
$y^{r+1}*y\not\in\mathrm{Li}(X_{n})$. Therefore,
\[\max(\mathrm{crit}[X_{n}])>\mathrm{crit}(y^{r+1}*y)=y^{\sharp}(\mathrm{crit}(y^{r}*y))\geq y^{\sharp}(\alpha).\]
\end{proof}
All results mentioned in this section about free LD-monoids and free LD-systems that do not directly refer to any sort of permutative LD-systems can be found in \cite{D}.
\begin{thm}
Suppose that $(X,*,\circ,1)$ is an LD-monoid freely generated by a set $A$. Then the set $A$ freely generates a sub-LD-system of
$(X,*)$.
\end{thm}
\begin{defn}
An LD-system $X$ is said to be \index{acyclic LD-system}\emph{acyclic} if there do not exist $x,x_{1},\ldots,x_{n}\in X$ with $x*x_{1}*\ldots*x_{n}=x$.
\end{defn}
\begin{thm}
\label{i5hg39c832id484395nv}
\begin{enumerate}
\item An LD-system $(X,*)$ generated by an element $x$ is freely generated by $x$ if and only if
$X$ is acyclic.

\item An LD-monoid $(X,*,\circ,1)$ generated by an element $x$ is freely generated by $x$ if and only if
$X\setminus\{1\}$ is closed under $*$ and $(X\setminus\{1\},*)$ is acyclic.
\end{enumerate}
\end{thm}
\begin{defn}
Suppose that $((X_{n})_{n\in\omega},(\phi_{n})_{n\in\omega})$ is vast. \index{$\mathrm{EE}$} Then let 
\[\mathrm{EE}((X_{n})_{n\in\omega})\subseteq\varprojlim_{n\in\omega}X_{n}\]
be the set of all threads
$(x_{n})_{n\in\omega}$ such that for all $\alpha\in \mathrm{crit}[\mathcal{X}]$ there is some
$m$ with $\alpha\in \mathrm{crit}[X_{m}]$ and where $x_{m}^{\sharp}(\alpha)<\max(\mathrm{crit}[X_{m}])$. Let
$\mathrm{EE}((X_{n})_{n\in\omega})^{+}$ be the set of all threads $(x_{n})_{n\in\omega}\in\mathrm{EE}((X_{n})_{n\in\omega})$ where
$x_{n}\not\in \mathrm{Li}(X_{n})$ for some $n\in\omega$.
\end{defn}
The letters $\mathrm{EE}$ stand for elementary embedding since we shall soon see that $\mathrm{EE}(\mathcal{X})$ algebraizes $(\mathcal{E}_{\lambda},*,\circ)$.

\begin{prop}
Suppose that $\mathcal{X}=((X_{n})_{n\in\omega},(\phi_{n})_{n\in\omega})$ is a vast system of permutative LD-systems.
\begin{enumerate}
\item $\mathrm{EE}(\mathcal{X})$ is a dense $G_{\delta}$ subset of $\varprojlim_{n\in\omega}X_{n}$, and
$\mathrm{EE}(\mathcal{X})^{+}$ is a dense open subset of $\mathrm{EE}(\mathcal{X})$.

\item $\mathrm{EE}(\mathcal{X})$ and $\mathrm{EE}(\mathcal{X})^{+}$ are closed under $*$. If each $X_{n}$ is reduced, then $\mathrm{EE}(\mathcal{X})$ and $\mathrm{EE}(\mathcal{X})^{+}$ are closed under $\circ$.

\item $(\mathrm{EE}(\mathcal{X})^{+},*)$ is an acyclic LD-system.

\item Every element in $\mathrm{EE}(\mathcal{X})^{+}$ freely generates a sub-LD-system of $(\mathrm{EE}(\mathcal{X})^{+},*)$. Furthermore, if $\mathcal{X}$ is a vast system of permutative LD-monoids, then every element in $\mathrm{EE}(\mathcal{X})\setminus\{1\}$ freely generates a sub-LD-monoid of
$(\mathrm{EE}(\mathcal{X}),*,\circ,1)$.
\end{enumerate}
\label{4tji20t24hu}
\end{prop}
\begin{proof}
\begin{enumerate}
\item We shall first show that $\mathrm{EE}(\mathcal{X})$ is a $G_{\delta}$-set. For each $\alpha\in\mathrm{crit}[\mathcal{X}]$,
let $U_{\alpha}$ be the set of all
$(x_{n})_{n\in\omega}\in\varprojlim X_{n}$ such that there is some $n$ with $x_{n}^{\sharp}(\alpha)<\max(\mathrm{crit}[X_{n}])$.
Let $A$ be a countable cofinal subset of $\mathrm{crit}[\mathcal{X}]$.
Then $\mathrm{EE}(\mathcal{X})=\bigcap_{\alpha\in A}U_{\alpha}$, but each $U_{\alpha}$ is open. Therefore,
$\mathrm{EE}(\mathcal{X})$ is a $G_{\delta}$-set.

To show that $\mathrm{EE}(\mathcal{X})$ is dense, it suffices to show that each $U_{\alpha}$ is dense and apply the Baire category theorem.
Suppose that $V$ is a non-empty open subset of $\varprojlim X_{n}$ and $\alpha\in\mathrm{crit}[\mathcal{X}]$. Then there is some $n$ with $\max(\mathrm{crit}[X_{n}])>\alpha$ and some $x\in X_{n}$ where $\{(x_{n})_{n\in\omega}\in\varprojlim X_{n}\mid x_{n}=x\}\subseteq V$. Therefore, by vastness, there exists some $l>n$ and $z\in X_{l}$ where $\phi_{l,n}(z)=x$ and where $z^{\sharp}(\alpha)<\max(\mathrm{crit}[X_{l}])$. Therefore, let $(x_{n})_{n\in\omega}\in\varprojlim X_{n}$ be a thread where $x_{l}=z,x_{n}=x$. Then $(x_{n})_{n\in\omega}\in V\cap U_{\alpha}$. Therefore, each $U_{\alpha}$ is dense, hence $\mathrm{EE}(\mathcal{X})$ is dense as well.

It is clear that the set $\mathrm{EE}^{+}(\mathcal{X})$ is open in $\mathrm{EE}(\mathcal{X})$. To show that
$\mathrm{EE}^{+}(\mathcal{X})$ is dense in $\mathrm{EE}(\mathcal{X})$, assume that $U$ is a non-empty open subset of
$\varprojlim_{n}X_{n}$. Then there is some $n$ along with some $x\in X_{n}$ where if $x_{n}=x$, then
$(x_{n})_{n\in\omega}\in U$. If $x\not\in\mathrm{Li}(X_{n})$, then
\[\emptyset\neq\{(x_{n})_{n\in\omega}\mid x_{n}=x\}\cap\mathrm{EE}(\mathcal{X})\subseteq\mathrm{EE}^{+}(\mathcal{X})\cap U.\]
If $x\in\mathrm{Li}(X_{n})$, then by vastness, there is some $l>n$ and some $y\in X_{l}\setminus\mathrm{Li}(X_{l})$ with $\phi_{l,n}(y)=x$.
Therefore, 
\[\emptyset\neq\{(x_{n})_{n\in\omega}\mid x_{l}=y\}\cap\mathrm{EE}(\mathcal{X})\subseteq\mathrm{EE}^{+}(\mathcal{X})\cap U.\]

\item Suppose that $(x_{n})_{n\in\omega},(y_{n})_{n\in\omega}\in \mathrm{EE}(\mathcal{X})$. Then for each $\alpha$, there is some
$n$ where $y_{n}^{\sharp}(\alpha)<\max(\mathrm{crit}[X_{n}])$. Now if $m\geq n$, then $y_{n}^{\sharp}(\alpha)=y_{m}^{\sharp}(\alpha)$. Therefore, there is some
$\beta\in \mathrm{crit}[\mathcal{X}]$ with $y_{m}^{\sharp}(\alpha)=\beta$ for all $m\geq n$. Now, there is some $m\geq n$ such that
$x_{m}^{\sharp}(\beta)<\max(\mathrm{crit}[X_{m}])$. Therefore, we have 
\[\max(\mathrm{crit}[X_{m}])>x_{m}^{\sharp}(\beta)=x_{m}^{\sharp}(y_{m}^{\sharp}(\alpha))\]
\[=(x_{m}*y_{m})^{\sharp}(x_{m}^{\sharp}(\alpha))\geq(x_{m}*y_{m})^{\sharp}(\alpha).\]
Therefore, $(x_{n}*y_{n})_{n\in\omega}\in \mathrm{EE}(\mathcal{X})$. If each $X_{n}$ is reduced, then 
\[(x_{m}\circ y_{m})^{\sharp}(\alpha)=x_{m}^{\sharp}(y_{m}^{\sharp}(\alpha))=x_{m}^{\sharp}(\beta)<\max(\mathrm{crit}[X_{m}]),\]
so $(x_{n})_{n\in\omega}\circ(y_{n})_{n\in\omega}\in \mathrm{EE}(\mathcal{X})$ as well. Thus, $\mathrm{EE}(\mathcal{X})$ is closed under $*$ and if each $X_{n}$ is an LD-monoid, then $\mathrm{EE}(\mathcal{X})$ is also closed under $\circ$.

I now claim that $\mathrm{EE}^{+}(\mathcal{X})$ is closed under $*$. Suppose that 
\[(x_{n})_{n\in\omega},(y_{n})_{n\in\omega}\in \mathrm{EE}^{+}(\mathcal{X}).\]
Then there is some $n$ where $y_{n}\not\in \mathrm{Li}(X_{n})$. In particular, there exists some $\alpha$ such that
for all $m\geq n$, we have $\mathrm{crit}(y_{m})=\mathrm{crit}(y_{n})=\alpha$.

Now, since $(x_{n})_{n\in\omega}\in \mathrm{EE}(\mathcal{X})$, there is some $m$ where $x_{m}^{\sharp}(\alpha)\not\in \mathrm{Li}(X_{m})$. Therefore, we have $\mathrm{crit}(x_{m}*y_{m})=x_{m}^{\sharp}(\alpha)<\max(\mathrm{crit}[X_{m}])$, so $x_{m}*y_{m}\not\in \mathrm{Li}(X_{m})$, hence $(x_{n})_{n\in\omega}*(y_{n})_{n\in\omega}\in \mathrm{EE}^{+}(\mathcal{X})$. We conclude that $\mathrm{EE}^{+}(\mathcal{X})$ is closed under $*$.

Now suppose that $(x_{n})_{n\in\omega},(y_{n})_{n\in\omega}\in \mathrm{EE}^{+}(\mathcal{X})$ and each $X_{n}$ is an LD-monoid. Then there is some $n$ where $x_{n}\neq 1$. Therefore, $x_{n}\circ y_{n}\neq 1$, so $(x_{n})_{n\in\omega}\circ(y_{n})_{n\in\omega}\in \mathrm{EE}^{+}(\mathcal{X})$. Therefore,
$\mathrm{EE}^{+}(\mathcal{X})$ is closed under $\circ$ whenever each $X_{n}$ is an LD-monoid.

\item Suppose to the contrary that $\mathrm{EE}^{+}(\mathcal{X})$ is not acyclic. Then there are
$(x_{n,1})_{n},\ldots,(x_{n,k})_{n}$ with $(x_{n,1})_{n}*\ldots*(x_{n,k})_{n}=(x_{n,1})_{n}$. Thus, there is some $n$ with
$x_{n,1}*\ldots*x_{n,j}\not\in \mathrm{Li}(X_{n})$ for $1\leq j\leq k$ and where $x_{n,1}*\ldots*x_{n,k}=x_{n,1}$ which is a contradiction. Therefore,
$\mathrm{EE}^{+}(\mathcal{X})$ is acyclic.

\item This follows from the fact that $\mathrm{EE}^{+}(\mathcal{X})$ is acyclic and from Theorem \ref{i5hg39c832id484395nv}.
\end{enumerate}
\end{proof}
We shall give every permutative LD-system the discrete topology and the inverse limit $\varprojlim_{n\in\omega}X_{n}$ the inverse limit topology. A subspace $Y$ of a complete metric space $X$ is a $G_{\delta}$-set if and only if $Y$ can be endowed with a complete metric that induces the subspace topology on $Y$. Therefore, since $\mathrm{EE}(\mathcal{X})$ is a $G_{\delta}$-subset of $\varprojlim X_{n}$, the space $\mathrm{EE}(\mathcal{X})$ can be endowed with a compatible complete metric.

Suppose that $\mathcal{E}_{\lambda}$ contains a non-trivial elementary embedding. Then give
$\mathcal{E}_{\lambda}$ the topology where we let a set $U\subseteq\mathcal{E}_{\lambda}$ be open if 
whenever $j\in U$ there is some $\gamma<\lambda$ such that if
$k\in\mathcal{E}_{\lambda}$ and $j|_{V_{\gamma}}=k|_{V_{\gamma}}$ then $k\in U$ as well. The mappings application $*$ and composition $\circ$ are both continuous operations on $\mathcal{E}_{\lambda}$.

The topological space $\mathcal{E}_{\lambda}$ is metrizable by the following metric $d$. Suppose that
$(\kappa_{n})_{n\in\omega\setminus\{0\}}$ is an increasing cofinal sequence of cardinals in $V_{\lambda}$. Then define
$d(j,k)=0$ if $j=k$ and otherwise define $d(j,k)=\frac{1}{n}$ where $n$ is the smallest natural number with $j|_{V_{\kappa_{n}}}\neq k|_{V_{\kappa_{n}}}$.

While the metric $d$ on $\mathcal{E}_{\lambda}$ depends on the choice of the sequence of cardinals $(\kappa_{n})_{n\in\omega\setminus\{0\}}$, the topology and also the uniformity generated by the metric space $(\mathcal{E}_{\lambda},d)$ does not depend on the choice of cardinals. Furthermore, since the notions of a Cauchy sequence and completeness only depend on the uniformity generated by $d$, the collection of all Cauchy sequences on $(\mathcal{E}_{\lambda},d)$ does not depend on the choice of the sequence of cardinals $(\kappa_{n})_{n\in\omega\setminus\{0\}}$. The metric $(\mathcal{E}_{\lambda},d)$ is complete and $d$ satisfies the strong triangle inequality $d(j,l)\leq \max(d(j,k),d(k,l))$

\begin{prop}
Suppose that $j,k\in\mathcal{E}_{\lambda}$ and $\gamma$ is a limit ordinal. 
\begin{enumerate}
\item If $j|_{V_{\gamma}}=k|_{V_{\gamma}}$, then $j\equiv^{\gamma}k$

\item If $j\equiv^{j(\gamma)}k$ then $j|_{V_{\gamma}}=k|_{V_{\gamma}}$
\end{enumerate}
\end{prop}
\begin{cor}
A subset $U\subseteq\mathcal{E}_{\lambda}$ is open if and only if whenever $j\in U$ there is some $\gamma<\lambda$ so that if
$k\in\mathcal{E}_{\lambda}$ and $j\equiv^{\gamma}k$ then $k\in U$ as well.
\end{cor}

Suppose that $(\kappa_{n})_{n\in\omega\setminus\{0\}}$ is a cofinal sequence of limit ordinals in $\lambda$. Consider the metric $\rho$ on $\mathcal{E}_{\lambda}$ defined by letting $\rho(j,k)=\frac{1}{n}$ where $j\neq k$ and $n$ is the least natural number with $j\not\equiv^{\kappa_{n}}k$ and $\rho(j,k)=0$ if $j=k$. Then, from the above corollary, we conclude that the metrics $d$ and $\rho$ generate the same topology on $\mathcal{E}_{\lambda}$. However,
the metrics $d$ and $\rho$ generate different uniformities on $\mathcal{E}_{\lambda}$, and the metric $\rho$ is inadequate since the metric space $(\mathcal{E}_{\lambda},\rho)$ is not complete.

\begin{thm}
Suppose that $(\kappa_{n})_{n\in\omega}$ is an increasing sequence of cardinals cofinal in $\lambda$ such that $\kappa_{n}=\mathrm{crit}(j_{n})$ for some $j_{n}\in\mathcal{E}_{\lambda}$ whenever $n\in\omega$.
Define a mapping $\Gamma:\mathcal{E}_{\lambda}\rightarrow \mathrm{EE}((\mathcal{E}_{\lambda}/\equiv^{\kappa_{n}})_{n\in\omega})$ by letting
$\Gamma(j)=([j]_{\kappa_{n}})_{n\in\omega}$. Then the mapping $\Gamma$ is a homeomorphism between topological spaces and
an isomorphism between LD-monoids.
\label{4th2u9it4}
\end{thm}
\begin{proof}
The mapping $\Gamma$ is clearly a homomorphism between LD-monoids. Furthermore, the mapping $\Gamma$ is injective since if
$j\neq k$ then $j(x)\neq k(x)$ for some $x$, hence $j(x)\cap V_{\kappa_{n}}\neq k(x)\cap V_{\kappa_{n}}$ for some $n$, so
$j\not\equiv^{\kappa_{n}}k$, thus $[j]_{n}\neq[k]_{n}$. Therefore, we conclude that $\Gamma(j)\neq\Gamma(k)$.

We shall now prove that the mapping $\Gamma$ is surjective. Suppose therefore that
\[([j_{n}]_{n})_{n\in\omega}\in \mathrm{EE}((\mathcal{E}_{\lambda}/\equiv^{\kappa_{n}})_{n\in\omega}).\]
I now claim that the sequence
$([j_{n}]_{n})_{n\in\omega}$ is Cauchy in the metric space $(\mathcal{E}_{\lambda},d)$.

Suppose that $\gamma<\lambda$. Then since $([j_{n}]_{n})_{n\in\omega}\in \mathrm{EE}((\mathcal{E}_{\lambda}/\equiv^{\kappa_{n}})_{n\in\omega})$, there is some
$N$ where $j_{N}(\gamma)<\kappa_{N}$. Therefore, we have $j_{n}(\gamma)<\kappa_{N}$ for all $n\geq N$. Thus, for
$m,n\geq N$ and $x\in V_{\gamma}$, we have
\[j_{m}(x)=j_{m}(x)\cap V_{\kappa_{N}}=j_{n}(x)\cap V_{\kappa_{N}}=j_{n}(x).\]
Thus, since
$j_{m}|_{V_{\gamma}}=j_{n}|_{V_{\gamma}}$ for all $m,n\geq N$, the sequence $(j_{n})_{n}$ is Cauchy. By the completeness of
$\mathcal{E}_{\lambda}$, there is some $j\in\mathcal{E}_{\lambda}$ with $j_{n}\rightarrow j$ in $\mathcal{E}_{\lambda}$.

Now, since $j_{n}\rightarrow j$, for all $n$ there is some $N\geq n$ where if $m\geq N$, then
$j_{m}|_{V_{\kappa_{n}}}=j|_{V_{\kappa_{n}}}$ which implies that
\[[j]_{\kappa_{n}}=[j_{m}]_{\kappa_{n}}=[j_{n}]_{\kappa_{n}}.\]
Therefore, $\Gamma(j)=([j_{n}])_{n\in\omega}$.

I now claim that $\Gamma$ is a homeomorphism. If $j_{r}\rightarrow j$, then $[j_{r}]_{n}\rightarrow[j]_{n}$ for all $n$, so
\[\Gamma(j_{r})=([j_{r}]_{n})_{n}\rightarrow([j]_{n})_{n}=\Gamma(j),\]
and thus $\Gamma$ is continuous.

Now suppose that $\Gamma(j_{r})\rightarrow\Gamma(j)$. Then let $\gamma<\lambda$ and let $\kappa_{n}>j(\gamma)$. Then
$[j_{r}]_{n}\rightarrow[j]_{n}$, so there is some $r$ where for all $s\geq r$, we have 
\[j_{s}(x)\cap V_{\kappa_{n}}=j(x)\cap V_{\kappa_{n}}.\]

However, since $j(\gamma)<\kappa_{n}$, if $x\in V_{\gamma}$, then 
\[j_{s}(x)=j_{s}(x\cap V_{\gamma})=j_{s}(x)\cap j_{s}(V_{\gamma})=
j_{s}(x)\cap V_{j_{s}(\gamma)}=j_{s}(x)\cap V_{j(\gamma)}\]
\[=j(x)\cap V_{j(\gamma)}=j(x)\cap j(V_{\gamma})=j(x\cap V_{\gamma})=j(x).\]

Therefore, $j_{r}\rightarrow j$, so $\Gamma^{-1}$ is continuous as well.
\end{proof}

By Theorem \ref{4th2u9it4}, we conclude that the algebras of the form $\mathrm{EE}(\mathcal{X})$ where $\mathcal{X}$ is a vast system of permutative LD-systems forms an algebraic model of the algebra of elementary embeddings $\mathcal{E}_{\lambda}$. 

The following development works to embed the free LD-systems and free LD-monoids into the algebras of the form
$\mathrm{EE}(\mathcal{X})$.
\begin{thm}
\begin{enumerate}
\item The free LD-monoid on one generator does not contain a free LD-system on two generators.

\item The free LD-system on two generators contains a free LD-subsystem on infinitely many generators.

\item The free LD-monoid on two generators contains a free LD-monoid on infinitely many generators.
\end{enumerate}
\end{thm}

The free LD-system on multiple generators computably embeds into the charged braid group and also the Larue group \cite{D9901},\cite{DL1994a}. Since the word problem for the Larue group is solvable, the word problem for the free LD-systems on multiple generators is also solvable. Likewise, the word problem for the free LD-monoid on an arbitrary number of generators is solvable. It is unknown whether the variety generated by the free LD-system on one generator is the variety of all LD-systems.

We have a characterization of the free LD-monoids on multiple generators.

\begin{thm}
Suppose that $X$ is an LD-monoid generated by a set $A$. Then $X$ is freely generated by $A$
if and only if $(X\setminus\{1\},*)$ is a cyclic sub-LD-system of $X$ and where
\[c*a*a_{1}*\ldots*a_{m}\neq c*b*b_{1}*\ldots*b_{n}\]
whenever $a,b\in A,a\neq b,c\in X$ and
\[a_{1},\ldots,a_{m},b_{1},\ldots,b_{n}\in X\setminus\{1\}.\]
\label{4thu2g4beqwf2}
\end{thm}
\begin{defn}
Let $\varepsilon$ denote the empty string. Suppose that $M$ is a multigenic Laver table over an alphabet $A$ with $L=M\setminus \mathrm{Li}(M)$. Then define a mapping $\phi:M/\mathrm{Li}(M)\rightarrow L\cup\{\varepsilon\}$ by letting $\phi(\mathrm{Li}(M))=\varepsilon$ and
$\phi([\mathbf{x}])=\mathbf{x}$ for each $\mathbf{x}\in L$. Then $\phi$ is a bijection.
Define an operation $\cdot$ on $L\cup\{\varepsilon\}$ by letting
\[\mathbf{x}\cdot\mathbf{y}=\phi(\phi^{-1}(\mathbf{x})*\phi^{-1}(\mathbf{y})).\]
In other words, $\cdot$ is the operation on $L\cup\{\varepsilon\}$ where 
\begin{enumerate}
\item $\mathbf{x}\cdot\mathbf{y}=\mathbf{x}*\mathbf{y}$ whenever $\mathbf{x},\mathbf{y},\mathbf{x}*\mathbf{y}\in L$

\item $\mathbf{x}\cdot\varepsilon=\varepsilon$ whenever $\mathbf{x}\in L\cup\{\varepsilon\}$

\item $\varepsilon\cdot\mathbf{x}=\mathbf{x}$ whenever $\mathbf{x}\in L\cup\{\varepsilon\}$

\item $\mathbf{x}\cdot\mathbf{y}=\varepsilon$ whenever $\mathbf{x}*\mathbf{y}\in\mathrm{Li}(M)$.
\end{enumerate}
The algebra $(L\cup\{\varepsilon\},\cdot)$ is a reduced locally Laver-like LD-system which we shall call a \index{reduced multigenic Laver table}\emph{reduced multigenic Laver table} over the alphabet $A$.
\end{defn}
\begin{defn}
Let $\lambda$ be a linearly ordered set of cofinality $\omega$. Let $A$ be a set. Let $\mathcal{X}$ denote an inverse system $((L_{\alpha})_{\alpha\in\lambda},(\phi_{\beta,\alpha})_{\alpha\leq\beta})$ of reduced multigenic Laver tables such that
\begin{enumerate}
\item whenever $\alpha,\beta\in\lambda,\alpha\leq\beta$ there is some $r\in\mathrm{crit}[X_{\beta}]$ where
$\equiv^{r}$ is equal to $\ker(\phi_{\beta,\alpha})$,

\item (vastness) for each $x,y\in X_{\alpha}$ there is some $\beta\geq\alpha$ and $x',y'\in X_{\beta}$ where
$x'*y'\not\in\mathrm{Li}(X_{\beta})$ and where $\phi_{\beta,\alpha}(x')=x,\phi_{\beta,\alpha}(y')=y$,

\item If $\beta\in\lambda$ and $r\in\mathrm{crit}[X_{\beta}]$ then there is some $\alpha\leq\beta$ where
$\equiv^{r}$ is equal to $\ker(\phi_{\beta,\alpha})$,

\item $\bigcup_{\alpha\in\lambda}L_{\alpha}=\bigcup_{n\in\omega}A^{n}$, and

\item if $a\in L_{\alpha}$, then $\phi_{\beta,\alpha}(a)=a$.

\end{enumerate}
Then we say that $\mathcal{X}$ is a \index{vast system of reduced multigenic Laver tables}\emph{vast system of reduced multigenic Laver tables}.
Define $\mathrm{EE}(\mathcal{X})$ to be the set of all $(x_{\alpha})_{\alpha\in\lambda}$ such that
$(x_{\alpha_{n}})_{n\in\omega}\in\mathrm{EE}((L_{\alpha_{n}})_{n\in\omega})$. $\mathrm{EE}^{+}(\mathcal{X})$ is defined analogously.

We say that
\[(\mathbf{x}_{\alpha})_{\alpha\in\lambda},(\mathbf{y}_{\alpha})_{\alpha\in\lambda}\in \mathrm{EE}(\mathcal{X})^{+}\]
are \index{eventually initially distinguishable}\emph{eventually initially distinguishable} if there is some $\alpha$ such that for all $\beta\geq\alpha$, the strings
$\mathbf{x}_{\beta},\mathbf{y}_{\beta}$ are incomparable by the prefix ordering $\preceq$.
\end{defn}
\begin{prop}
Suppose that $\mathcal{X}$ is a vast system of reduced multigenic Laver tables. If 
\[(\mathbf{x}_{\alpha})_{\alpha\in\lambda},(\mathbf{y}_{\alpha})_{\alpha\in\lambda}\in \mathrm{EE}(\mathcal{X})^{+}\]
are eventually initially distinguishable, then so are
\[(\mathbf{c}_{\alpha})_{\alpha\in\lambda}*(\mathbf{x}_{\alpha})_{\alpha\in\lambda}*(\mathbf{a}_{\alpha,1})_{\alpha\in\lambda}*\ldots*(\mathbf{a}_{\alpha,p})_{\alpha\in\lambda}\]
and 
\[(\mathbf{c}_{\alpha})_{\alpha\in\lambda}*(\mathbf{y}_{\alpha})_{\alpha\in\lambda}*(\mathbf{b}_{\alpha,1})_{\alpha\in\lambda}*\ldots*(\mathbf{b}_{\alpha,q})_{\alpha\in\lambda}\]
whenever $(\mathbf{c}_{\alpha})_{\alpha\in\lambda}\in\mathrm{EE}(\mathcal{X})$ and
\[(\mathbf{a}_{\alpha,1})_{\alpha\in\lambda},\ldots,(\mathbf{a}_{\alpha,p})_{\alpha\in\lambda},
(\mathbf{b}_{\alpha,1})_{\alpha\in\lambda},\ldots,(\mathbf{b}_{\alpha,q})_{\alpha\in\lambda}\in
\mathrm{EE}^{+}(\mathcal{X}).\]
\end{prop}
\begin{proof}
Suppose that 
\[(\mathbf{x}_{\alpha})_{\alpha\in\lambda},(\mathbf{y}_{\alpha})_{\alpha\in\lambda}\in \mathrm{EE}(\mathcal{X})^{+}\]
are eventually initially distinguishable.  Then I claim that
\[(\mathbf{c}_{\alpha})_{\alpha\in\lambda}*(\mathbf{x}_{\alpha})_{\alpha\in\lambda},(\mathbf{c}_{\alpha})_{\alpha\in\lambda}*(\mathbf{y}_{\alpha})_{\alpha\in\lambda}\]
are also eventually initially distinguishable whenever
\[(\mathbf{c}_{\alpha})_{\alpha\in\lambda}\in \mathrm{EE}(\mathcal{X}).\]

Since 
\[(\mathbf{x}_{\alpha})_{\alpha\in\lambda},(\mathbf{y}_{\alpha})_{\alpha\in\lambda}\in \mathrm{EE}(\mathcal{X})^{+}\]
are eventually initially distinguishable, there is some $\alpha$ where $\mathbf{x}_{\beta}$ and $\mathbf{y}_{\beta}$ are incomparable with respect to $\preceq$ for all $\beta\geq\alpha$. Now there is some $\gamma$ where $\mathbf{c}_{\gamma}^{\sharp}(\alpha)<\gamma$, and hence $\mathbf{c}_{\delta}^{\sharp}(\alpha)<\delta$ for each $\delta\geq\gamma$.

Suppose that $\delta\geq\gamma$. Then let $\beta$ be the least ordinal where
$\mathbf{c}_{\delta}^{\sharp}(\beta)\geq\delta$. Then 
\[\mathbf{c}_{\delta}*_{\delta}\mathbf{x}_{\delta}=\mathbf{c}_{\delta}*_{\delta}\mathbf{x}_{\beta}
=\mathbf{c}_{\delta}x_{1}*_{\delta}\ldots*_{\delta}\mathbf{c}_{\delta}x_{r}\]
where $\mathbf{x}_{\beta}=x_{1}\ldots x_{r}$. Similarly,
\[\mathbf{c}_{\delta}*_{\delta}\mathbf{y}_{\delta}=\mathbf{c}_{\delta}y_{1}*_{\delta}\ldots*_{\delta}\mathbf{c}_{\delta}y_{s}\]
where $\mathbf{y}_{\beta}=y_{1}\ldots y_{s}$. Now there is some $t\leq\min(r,s)$ where $x_{t}\neq y_{t}$ but where $x_{u}=y_{u}$ for each $u\in\{1,\ldots,t-1\}$.

Therefore, we have
\[(\mathbf{c}_{\delta}\circ_{\delta}x_{1}\ldots x_{t-1})x_{t}=\mathbf{c}_{\delta}*_{\delta}x_{1}\ldots x_{t}\preceq\mathbf{c}_{\delta}*_{\delta}\mathbf{x}_{\delta}\]
and
\[(\mathbf{c}_{\delta}\circ_{\delta}x_{1}\ldots x_{t-1})y_{t}=(\mathbf{c}_{\delta}\circ_{\delta}y_{1}\ldots y_{t-1})y_{t}\]
\[=\mathbf{c}_{\delta}*_{\delta}y_{1}\ldots y_{t}\preceq\mathbf{c}_{\delta}*_{\delta}\mathbf{y}_{\delta}.\]

Therefore, the elements $\mathbf{c}_{\delta}*_{\delta}\mathbf{x}_{\delta}$ and
$\mathbf{c}_{\delta}*_{\delta}\mathbf{y}_{\delta}$ are incomparable.

We therefore conclude that $(\mathbf{c}_{\alpha})_{\alpha\in\lambda}*(\mathbf{x}_{\alpha})_{\alpha\in\lambda},(\mathbf{c}_{\alpha})_{\alpha\in\lambda}*(\mathbf{y}_{\alpha})_{\alpha\in\lambda}$ are also
eventually initially distinguishable whenever $(\mathbf{c}_{\alpha})_{\alpha\in\lambda}\in \mathrm{EE}(\mathcal{X})$.

Now, since $(\mathbf{c}_{\alpha})_{\alpha\in\lambda}*(\mathbf{x}_{\alpha})_{\alpha\in\lambda},(\mathbf{c}_{\alpha})_{\alpha\in\lambda}*(\mathbf{y}_{\alpha})_{\alpha\in\lambda}$ are eventually initially distinguishable,
$\mathbf{c}_{\alpha}*\mathbf{x}_{\alpha}$ is a prefix of
$\mathbf{c}_{\alpha}*\mathbf{x}_{\alpha}*\mathbf{a}_{\alpha,1}*\ldots *\mathbf{a}_{\alpha,p}$ for large enough $\alpha$, and
$\mathbf{c}_{\alpha}*\mathbf{x}_{\alpha}$ is a prefix of
$\mathbf{c}_{\alpha}*\mathbf{x}_{\alpha}*\mathbf{b}_{\alpha,1}*\ldots *\mathbf{b}_{\alpha,q}$ for large enough $\alpha$, we conclude that

\[(\mathbf{c}_{\alpha})_{\alpha\in\lambda}*(\mathbf{x}_{\alpha})_{\alpha\in\lambda}*(\mathbf{a}_{\alpha,1})_{\alpha\in\lambda}*\ldots*(\mathbf{a}_{\alpha,p})_{\alpha\in\lambda}\]
and 
\[(\mathbf{c}_{\alpha})_{\alpha\in\lambda}*(\mathbf{y}_{\alpha})_{\alpha\in\lambda}*(\mathbf{b}_{\alpha,1})_{\alpha\in\lambda}*\ldots*(\mathbf{b}_{\alpha,q})_{\alpha\in\lambda}\]
are eventually initially distinguishable for sufficiently large $\alpha$.
\end{proof}
\begin{prop}
Suppose that $((\mathbf{x}_{i,\alpha})_{\alpha\in\lambda})_{i\in I}\in (\mathrm{EE}(\mathcal{X})^{+})^{I}$ is pairwise
eventually initially distinguishable. Then $((\mathbf{x}_{i,\alpha})_{\alpha\in\lambda})_{i\in I}$ freely generates a sub-LD-monoid of
$\mathrm{EE}(\mathcal{X})$.
\end{prop}
\begin{proof}
Suppose that $((\mathbf{x}_{i,\alpha})_{\alpha\in\lambda})_{i\in I}\in (\mathrm{EE}(\mathcal{X})^{+})^{I}$ is pairwise
eventually initially distinguishable. Then for each $i,j\in I$ with $i\neq j$, and each
$(\mathbf{c}_{\alpha})_{\alpha\in\lambda}\in\mathrm{EE}(\mathcal{X})$, and each
\[(\mathbf{a}_{\alpha,1})_{\alpha\in\lambda},\ldots,(\mathbf{a}_{\alpha,r})_{\alpha\in\lambda},
(\mathbf{b}_{\alpha,1})_{\alpha\in\lambda},\ldots,(\mathbf{b}_{\alpha,s})_{\alpha\in\lambda}\in
\mathrm{EE}^{+}(\mathcal{X}),\]
we have
\[(\mathbf{c}_{\alpha})_{\alpha\in\lambda}*(\mathbf{x}_{i,\alpha})_{\alpha\in\lambda}*
(\mathbf{a}_{\alpha,1})_{\alpha\in\lambda}*\ldots*(\mathbf{a}_{\alpha,r})_{\alpha\in\lambda}\]
\[\neq(\mathbf{c}_{\alpha})_{\alpha\in\lambda}*(\mathbf{x}_{j,\alpha})_{\alpha\in\lambda}*
(\mathbf{b}_{\alpha,1})_{\alpha\in\lambda}*\ldots*(\mathbf{b}_{\alpha,s})_{\alpha\in\lambda}\]
since
\[(\mathbf{c}_{\alpha})_{\alpha\in\lambda}*(\mathbf{x}_{i,\alpha})_{\alpha\in\lambda}*
(\mathbf{a}_{\alpha,1})_{\alpha\in\lambda}*\ldots*(\mathbf{a}_{\alpha,r})_{\alpha\in\lambda}\]
and
\[(\mathbf{c}_{\alpha})_{\alpha\in\lambda}*(\mathbf{x}_{j,\alpha})_{\alpha\in\lambda}*
(\mathbf{b}_{\alpha,1})_{\alpha\in\lambda}*\ldots*(\mathbf{b}_{\alpha,s})_{\alpha\in\lambda}\]
are eventually initially indistinguishable. We therefore conclude that $((\mathbf{x}_{i,\alpha})_{\alpha\in\lambda})_{i\in I}$
freely generates a subalgebra of $\mathrm{EE}(\mathcal{X})^{+}$ by Theorem \ref{4tji20t24hu} and
Theorem \ref{4thu2g4beqwf2}.
\end{proof}

\begin{thm}
Suppose that $(L_{\alpha})_{\alpha\in\lambda}$ is a vast system of reduced multigenic Laver tables over an alphabet $A$.
Let $r_{a,\beta}=a$ whenever $a\in L_{\beta}$ and $r_{a,\beta}=\varepsilon$ otherwise. Then $\{(r_{a,\beta})_{\beta\in\lambda}\mid a\in A\}$ freely generates a sub-LD-monoid $\varprojlim_{\alpha\in\lambda}L_{\alpha}$.
\end{thm}
\begin{thm}
Suppose that $(L_{\alpha})_{\alpha\in\lambda}$ is a vast system of reduced multigenic Laver tables over an alphabet $A$. Let $I$ be a finite or countable set.
\begin{enumerate}
\item Let $U$ be the set of all objects $(\mathbf{x}_{i})_{i\in I}\in(\varprojlim_{\alpha\in\lambda}L_{\alpha})^{I}$ such that
$(\mathbf{x}_{i})_{i\in I}$ freely generates a sub-LD-system of $\varprojlim_{\alpha\in\lambda}L_{\alpha}$. Then
$U$ is a dense $G_{\delta}$ set in $(\varprojlim_{\alpha\in\lambda}L_{\alpha})^{I}$.

\item Let $V$ be the set of all objects $(\mathbf{x}_{i})_{i\in I}\in(\varprojlim_{\alpha\in\lambda}L_{\alpha})^{I}$ such that
$(\mathbf{x}_{i})_{i\in I}$ freely generates a sub-LD-monoid of $\varprojlim_{\alpha\in\lambda}L_{\alpha}$. Then
$V$ is a dense $G_{\delta}$ set in $(\varprojlim_{\alpha\in\lambda}L_{\alpha})^{I}$.
\end{enumerate}
\end{thm}
\begin{proof}
We shall only prove 2 since the proof of 1 is similar. Let $S$ be the set of all pairs $(s,t)$ of terms in the language of
LD-monoids which are not LD-monoid equivalent to each other. Then $S$ is countable. However,
\[V=\bigcap_{(s,t)\in S}\{(\mathbf{x}_{i})_{i\in I}
\in(\varprojlim_{\alpha\in\lambda}L_{\alpha})^{I}\mid s(\mathbf{x}_{i})_{i\in I}\neq t(\mathbf{x}_{i})_{i\in I}\}\]
 which is the intersection of countably many open sets.
Therefore, $V$ is a $G_{\delta}$-set.

To prove that $V$ is dense, suppose that $O$ is a non-empty open set in $(\varprojlim_{\alpha\in\lambda}L_{\alpha})^{I}$.
Then there are $i_{1},\ldots,i_{q}\in I$ along with some $\beta\in\lambda$ and 
\[\mathbf{v}_{1},\ldots,\mathbf{v}_{q}\in L_{\beta}\]
such that if 
\[((\mathbf{x}_{\alpha,i})_{\alpha\in\lambda})_{i\in I}\in(\varprojlim_{\alpha\in\lambda}L_{\alpha})^{I}\]
and
\[\mathbf{x}_{\beta,i_{1}}=\mathbf{v}_{1},\ldots,\mathbf{x}_{\beta,i_{q}}=\mathbf{v}_{q},\]
then
\[((\mathbf{x}_{\alpha,i})_{\alpha\in\lambda})_{i\in I}\in O.\]

Suppose now that $\gamma>\beta$ and $\mathbf{w}_{1},\ldots,\mathbf{w}_{q}$ are pairwise $\preceq$-incomparable strings in $L_{\gamma}$ with
$\mathrm{crit}(\mathbf{w}_{1})\geq\beta,\ldots,\mathrm{crit}(\mathbf{w}_{q})\geq\beta$. Then $\mathbf{w}_{1}\mathbf{v}_{1},\ldots,\mathbf{w}_{q}\mathbf{v}_{q}\in L_{\gamma}$.
Let $((\mathbf{x}_{\alpha,i})_{\alpha\in\lambda})_{i\in I}$ be a system of threads in
$(\varprojlim_{\alpha\in\lambda}L_{\alpha})^{I}$ where
\[\mathbf{x}_{\delta,i_{1}}=\mathbf{w}_{1}\mathbf{v}_{1},\ldots,\mathbf{x}_{\delta,i_{q}}=\mathbf{w}_{q}\mathbf{v}_{q}\]
whenever $\delta\geq\gamma$. Then $((\mathbf{x}_{\alpha,i})_{\alpha\in\lambda})_{i\in I}\in O\cap V$.
\end{proof}
\begin{thm}
\label{4gui04hru984r2gh9ug24r}

Suppose that for all $n$ there exists an $n$-huge cardinal. Let $A$ be a set with $|A|>1$.
\begin{enumerate}
\item $\{(a)_{n\in\omega}\mid a\in A\}$ freely generates a sub-LD-monoid of the inverse limit of reduced multigenic Laver tables
$\varprojlim_{n\in\omega}A^{<2^{n}}$.

\item Suppose that $I$ is a countable or finite set. Then the set of all
$(\mathbf{x}_{i})_{i\in I}\in(\varprojlim_{n\in\omega}A^{<2^{n}})^{I}$ that freely generate a sub-LD-system of
$\varprojlim_{n\in\omega}A^{<2^{n}}$ is a dense $G_{\delta}$-set.

\item Suppose that $I$ is a countable or finite set. Then the set of all
$(\mathbf{x}_{i})_{i\in I}\in(\varprojlim_{n\in\omega}A^{<2^{n}})^{I}$ that freely generate a sub-LD-monoid of
$\varprojlim_{n\in\omega}A^{<2^{n}}$ is a dense $G_{\delta}$-set.
\end{enumerate}
\end{thm}
Theorem \ref{4gui04hru984r2gh9ug24r} is remarkable since the conclusion is a result of fundamental importance about finite algebras
but the result relies upon the very largest cardinals. It is unknown as to whether the large cardinal hypothesis can be removed from the above theorem, and this problem is likely intractable.

\begin{prop}
In ZFC, the following are equivalent.
\begin{enumerate}
\item The thread $(1)_{n}$ freely generates a subalgebra of the inverse limit $\varprojlim_{n}A_{n}$ of classical Laver tables.

\item $\lim_{n\rightarrow\infty}o_{n}(1)\rightarrow\infty$.

\item There exists a vast system of Laver-like LD-systems.

\item If $A$ is a set with $|A|>1$ and $I$ is a countable or finite set, then the set of all $(\mathbf{x}_{i})_{i\in I}\in\varprojlim_{n\in\omega}A^{<2^{n}}$ that freely generate a sub-LD-monoid of $\varprojlim_{n\in\omega}A^{<2^{n}}$ is a dense $G_{\delta}$-set.

\item For each set $A$, the set of threads $\{(a)_{n}\mid a\in A\}$ freely generates a sub-LD-monoid of
$\varprojlim_{n\in\omega}A^{<2^{n}}$.
\end{enumerate}
\end{prop}
\section{Extensions of Laver-like LD-systems}
In this section, we shall develop methods for producing new Laver-like LD-systems from old ones.
\begin{defn}
Let $M$ be a multigenic Laver table. Let $\alpha$ be the largest critical point in $\mathrm{crit}[M]\setminus\{\max(\mathrm{crit}[M])\}$.
Then we say that the multigenic Laver table \index{cover}$M$ covers $M\upharpoonright\alpha$.
\end{defn}
Recall that if $M$ is a multigenic Laver table that covers $N$, then there is some unique non-empty set $J\subseteq\mathrm{Li}(N)$ where
$M=N\cup\{\mathbf{xy}\mid\mathbf{x}\in J,\mathbf{y}\in N\}$, and clearly $J=\mathrm{Li}(N)\setminus\mathrm{Li}(M)$.

\begin{prop}
Suppose that $N$ is a multigenic Laver table over an alphabet $A$, $N'$ is a multigenic Laver table that covers $N$,
$J=\mathrm{Li}(N)\setminus \mathrm{Li}(N')$, $a_{0}\in A$, $J_{a_{0}}=\{\mathbf{x}\in J:\mathbf{x}[1]=a_{0}\}$, and
$M=N\cup\{\mathbf{xy}:\mathbf{x}\in J_{a_{0}},\mathbf{y}\in N\}$. Then $M$ is a multigenic Laver table that covers $N$ whenever
$J_{a_{0}}\neq\emptyset$.
\label{t0g42hnudl}
\end{prop}
\begin{proof}
Let $\mathbf{v}\in J,\gamma=\mathrm{crit}(\mathbf{v})$, $x_{a}=\mathbf{v}a$ for each $a\in A\setminus\{a_{0}\}$, and let $x_{a_{0}}=a_{0}$. Then I claim that $M=\mathbf{M}((x_{a})_{a\in A},N')$.

If $\mathbf{x}\in N'\setminus \mathrm{Li}(N')$, then
\[\mathbf{x}*^{N'}(\mathbf{v}a)=(\mathbf{x}*^{N'}\mathbf{v})*^{N'}\mathbf{x}a=\mathbf{x}a.\]
Therefore, $\mathbf{x}*^{N'}x_{a}=\mathbf{x}a$ whenever $\mathbf{x}\in N'\setminus \mathrm{Li}(N')$. 

Suppose now that $a_{1}\ldots a_{n}\in M$ and $1\leq m<n$. If
$a_{1}\ldots a_{m}\in N$, then
\[x_{a_{1}}*^{N'}\ldots*^{N'}x_{a_{m}}\equiv^{\gamma}a_{1}\ldots a_{m}.\]
Therefore, 
\[\mathrm{crit}^{N'}(x_{a_{1}}*^{N'}\ldots*^{N'}x_{a_{m}})=\mathrm{crit}^{N'}(a_{1}\ldots a_{m})<\gamma.\]
We conclude that
\[x_{a_{1}}*^{N'}\ldots*^{N'}x_{a_{m}}\not\in\mathrm{Li}(N').\]

Now assume that $1\leq r<m$ and $a_{1}\ldots a_{r}\in J_{a_{0}}$. Then
\[x_{a_{1}}*^{N'}\ldots*^{N'}x_{a_{m}}\equiv^{\gamma}a_{1}\ldots a_{r}*^{N'}x_{a_{r+1}}*^{N'}\ldots *^{N'}x_{a_{m}}\]
\[\equiv^{\gamma}x_{a_{r+1}}*^{N'}\ldots *^{N'}x_{a_{m}}\equiv^{\gamma}a_{r+1}\ldots a_{m}.\]
Therefore, since
\[\mathrm{crit}(x_{a_{1}}*^{N'}\ldots*^{N'}x_{a_{m}})
=\mathrm{crit}(a_{r+1}\ldots a_{m})<\gamma,\]
we conclude that
\[x_{a_{1}}*^{N'}\ldots*^{N'}x_{a_{m}}\not\in\mathrm{Li}(N').\]

Now assume that $a_{1}\ldots a_{m}\in J_{a_{0}}$. Then I claim that for $1\leq s\leq m$, we have $x_{a_{1}}*^{N'}\ldots *^{N'}x_{a_{s}}=a_{1}\ldots a_{s}$ and this result shall be proven by induction on $s$. The case where $s=1$ is trivial. Now assume that $s>1$. Then
$x_{a_{1}}*^{N'}\ldots *^{N'}x_{a_{s}}=a_{1}\ldots a_{s-1}*^{N'}x_{a_{s}}=a_{1}\ldots a_{s-1}a_{s}$ since
$x_{a_{s}}\equiv^{\gamma}a_{s}$ and $(a_{1}\ldots a_{s-1})^{\sharp}(\gamma)=\max(\mathrm{crit}[N'])$. Therefore
\[x_{a_{1}}*^{N'}\ldots *^{N'}x_{a_{m}}=a_{1}\ldots a_{m}\not\in\mathrm{Li}(N').\]

We conclude that if $a_{1}\ldots a_{n}\in M$ and $1\leq m<n$, then
$a_{1}*^{N'}\ldots*^{N'}a_{m}\not\in\mathrm{Li}(N')$. Therefore
$M\subseteq\mathbf{M}((x_{a})_{a\in A},N')$.

Suppose now that $a_{1}\ldots a_{n}\not\in M$. Then $a_{1}\ldots a_{n}\not\in N$, so there is some $m<n$ where
$a_{1}\ldots a_{m}\in \mathrm{Li}(N)$. Furthermore, $a_{1}\ldots a_{m}\not\in J_{a_{0}}$ or $a_{m+1}\ldots a_{n}\not\in N$.

Suppose that $a_{1}\ldots a_{m}\not\in J_{a_{0}}$ and $a_{1}\neq a_{0}$.
Then I claim that $x_{a_{1}}*^{N'}\ldots*^{N'}x_{a_{r}}=\mathbf{v}a_{1}\ldots a_{r}$ whenever $1\leq r\leq m$.

Clearly $x_{a_{1}}=\mathbf{v}a_{1}$. Now, if $1\leq r<m$, then
\[x_{a_{1}}*^{N'}\ldots*^{N'}x_{a_{r+1}}=\mathbf{v}a_{1}\ldots a_{r}*^{N'}x_{a_{r+1}}=\mathbf{v}a_{1}\ldots a_{r+1}\]
since $x_{a_{r+1}}\equiv^{\gamma}a_{r+1}$ and
$(\mathbf{v}a_{1}\ldots a_{r})^{\sharp}(\gamma)=\max(\mathrm{crit}[N']).$

Therefore, we have
$x_{a_{1}}*^{N'}\ldots*^{N'}x_{a_{m}}=\mathbf{v}a_{1}\ldots a_{m}\not\in\mathrm{Li}(N')$. Therefore, since $m<n$, we conclude that $a_{1}\ldots a_{n}\not\in\mathbf{M}((x_{a})_{a\in A},N')$.

Now suppose that $a_{1}\ldots a_{m}\not\in J,a_{1}=a_{0}$. Then I claim that
\[x_{a_{1}}*^{N'}\ldots*^{N'}x_{a_{r}}=a_{1}\ldots a_{r}\]
for $1\leq r\leq m$.
We have $x_{a_{1}}=a_{1}$ since $a_{1}=a_{0}$. Now assume that
\[x_{a_{1}}*^{N'}\ldots*^{N'}x_{a_{r}}=a_{1}\ldots a_{r}\]
and $1\leq r<m$. Then
\[x_{a_{1}}*^{N'}\ldots*^{N'}x_{a_{r}}*^{N'}x_{a_{r+1}}=a_{1}\ldots a_{r}*^{N'}x_{a_{r+1}}=a_{1}\ldots a_{r+1}\]
since $a_{1}\ldots a_{r}\not\in \mathrm{Li}(N')$ and $x_{a_{r+1}}\equiv^{\gamma}a_{r+1}$.
Therefore, we have
\[x_{a_{1}}*^{N'}\ldots*^{N'}x_{a_{m}}=a_{1}\ldots a_{m}\in \mathrm{Li}(N'),\]
so $a_{1}\ldots a_{n}\not\in\mathbf{M}((x_{a})_{a\in A},N').$

Now suppose that $a_{1}\ldots a_{m}\in J_{a_{0}}$ but $a_{m+1}\ldots a_{n}\not\in N$. Then there is some $r<n$ where $a_{m+1}\ldots a_{r}\in \mathrm{Li}(N)$.
Then using our usual induction, it follows that $x_{a_{1}}*^{N'}\ldots*^{N'}x_{a_{s}}=a_{1}\ldots a_{s}$ whenever $1\leq s\leq r$, and, in particular, $x_{a_{1}}*^{N'}\ldots*^{N'}x_{a_{r}}=a_{1}\ldots a_{r}\in \mathrm{Li}(N')$,
so since $r<n$, we have $a_{1}\ldots a_{n}\not\in\mathbf{M}((x_{a})_{a\in A},N')$.

We conclude that $\mathbf{M}((x_{a})_{a\in A},N')=M$, and hence $M$ is a multigenic Laver table.
\end{proof}
\begin{prop}
\label{25yut94anh}
Let $M$ be a multigenic Laver table over an alphabet $A$. Furthermore, suppose that for each $a\in A$,
$M_{a}=M$ or $M_{a}$ is a multigenic Laver table that covers $M$ and where if $\mathbf{x}\in \mathrm{Li}(M)\setminus \mathrm{Li}(M_{a})$, then $\mathbf{x}[1]=a$. Let $V=\bigcup_{a\in A}M_{a}$. Then $V$ is a multigenic Laver table that covers $M$. Let
$S:V\rightarrow M$ be the mapping where $S(\mathbf{x})=\mathbf{x}$ for $\mathbf{x}\in M$ and where $S(\mathbf{vx})=\mathbf{x}$ whenever
$\mathbf{v}\in \mathrm{Li}(M),\mathbf{vx}\in V\setminus M$. Then
$\mathbf{x}*^{V}\mathbf{y}=\mathbf{x}*^{M_{a}}S(\mathbf{y})$ whenever $\mathbf{x}\not\in \mathrm{Li}(V),\mathbf{x}[1]=a$ and 
$\mathbf{x}*^{V}\mathbf{y}=\mathbf{y}$ whenever $\mathbf{x}\in \mathrm{Li}(V)$.
\end{prop}
\begin{proof}
Define an operation $\bullet$ on $V$ by letting $\mathbf{x}\bullet\mathbf{y}=\mathbf{x}*^{M_{a}}S(\mathbf{y})$ whenever $\mathbf{x}\not\in \mathrm{Li}(V),\mathbf{x}[1]=a$ and where
$\mathbf{x}\bullet \mathbf{y}=\mathbf{y}$ whenever $\mathbf{x}\in \mathrm{Li}(V)$.

Suppose that $\mathbf{x}\not\in \mathrm{Li}(V),b\in A,\mathbf{x}[1]=a$. Then
$\mathbf{x}\bullet b=\mathbf{x}*^{M_{a}}S(b)=\mathbf{x}*^{M_{a}}b=\mathbf{x}b$.

We have
\[\mathbf{x}\bullet (\mathbf{y}\bullet \mathbf{z})=\mathbf{y}\bullet \mathbf{z}=(\mathbf{x}\bullet \mathbf{y})\bullet 
(\mathbf{x}\bullet \mathbf{z})\]
whenever
$\mathbf{x}\in \mathrm{Li}(V)$.

Now suppose that $\mathbf{x}\not\in \mathrm{Li}(V),\mathbf{y}\in \mathrm{Li}(V),
\mathbf{x}[1]=a$. Then $S(y)\in \mathrm{Li}(M)$, so
\[x\bullet y=x*^{M_{a}}S(y)\in\mathrm{Li}(M_{a}).\]
Since $x\bullet y\in\mathrm{Li}(M_{a})$ and $(x\bullet y)[1]=a$, we conclude that $x\bullet y\in\mathrm{Li}(M)$ as well. Therefore, we have $x\bullet (y\bullet z)=x\bullet z$ and
$(x\bullet y)\bullet (x\bullet z)=x\bullet z$.

Now assume that $\mathbf{x}\not\in \mathrm{Li}(V)$ and $\mathbf{y}\not\in \mathrm{Li}(V)$ and $\mathbf{x}[1]=a,\mathbf{y}[1]=b$. Then
\[x\bullet (y\bullet z)=x*^{M_{a}}S(y\bullet z)=x*^{M_{a}}S(y*^{M_{b}}z)\]
\[=x*^{M_{a}}(S(y)*^{M}S(z))=x*^{M_{a}}(S(y)*^{M_{a}}S(z))\]
\[=(x*^{M_{a}}S(y))*^{M_{a}}(x*^{M_{a}}S(z))\]
\[=(x\bullet y)*^{M_{a}}(x\bullet z)\]
\[=(x\bullet y)\bullet (x\bullet z).\]

Therefore, since $\bullet$ is self-distributive, and
$\mathbf{x}\bullet a=\mathbf{x}a$ for $\mathbf{x}\not\in\mathrm{Li}(V),a\in A$, the algebra $(V,\bullet)$ is a multigenic Laver table.
\end{proof}
\begin{lem}
Let $M$ be a multigenic Laver table that covers $N$, and suppose that $J=\mathrm{Li}(N)\setminus\mathrm{Li}(M)$.
\begin{enumerate}
\item If $\mathbf{x}a,\mathbf{x}b\in \mathrm{Li}(N),|\mathbf{x}|\neq 0$ and $\mathrm{crit}^{N}(a)=\mathrm{crit}^{N}(b)$, then $\mathbf{x}a\in J$ if and only if $\mathbf{x}b\in J$.

\item Let $\mathbf{c}\in\mathcal{I}_{N}$. If $\mathbf{x},\mathbf{y}$ are non-empty strings,
$\mathbf{cx},\mathbf{cy}\in \mathrm{Li}(N)$ and $\mathrm{crit}^{N}(\mathbf{c})=\mathrm{crit}^{N}(\mathbf{x})=\mathrm{crit}^{N}(\mathbf{y})$, then
$\mathbf{cx}\in J$ if and only if $\mathbf{cy}\in J$.

\item If $\mathbf{x}a,\mathbf{x}b\in \mathrm{Li}(N)$ and $\mathrm{crit}^{N}(a)<\mathrm{crit}^{N}(b)$, then $\mathbf{x}b\not\in J$.
\end{enumerate}
\end{lem}

\begin{proof}
\begin{enumerate}
\item Suppose 
\[\mathrm{crit}^{N}(a)=\mathrm{crit}^{N}(b)<\max(\mathrm{crit}[N]).\]
Then $\mathrm{crit}^{M}(a)=\mathrm{crit}^{M}(b)$, so
\[\mathrm{crit}^{M}(\mathbf{x}a)=\mathrm{crit}^{M}(\mathbf{x}*a)=\mathrm{crit}^{M}(\mathbf{x}*b)=\mathrm{crit}^{M}(\mathbf{x}b),\]
so
$\mathbf{x}a\in J$ iff $\mathbf{x}a\not\in\mathrm{Li}(M)$ iff $\mathbf{x}b\not\in \mathrm{Li}(M)$ iff $\mathbf{x}b\in J$.

Suppose now that $\mathrm{crit}^{N}(a)=\mathrm{crit}^{N}(b)=\max(\mathrm{crit}[N])$. Since $\mathbf{x}a\in N$, we have
$\mathrm{crit}^{N}(\mathbf{x})<\mathrm{crit}^{N}(a)$. If $a\in \mathrm{Li}(M)$, then $\mathbf{x}a\in \mathrm{Li}(M)$ as well. If $a\not\in \mathrm{Li}(M)$, then 
\[\mathrm{crit}^{M}(\mathbf{x})<\mathrm{crit}^{M}(a)<\mathrm{crit}^{M}(\mathbf{x}a),\] so $\mathbf{x}a\in \mathrm{Li}(M)$. Therefore, $\mathbf{x}a\in \mathrm{Li}(M)$, so $\mathbf{x}a\not\in J$. For a similar reason,
$\mathbf{x}b\not\in J$ as well.

\item I claim that if $\mathbf{z}$ is a proper prefix of $\mathbf{x}$, then
$\mathrm{crit}(\mathbf{z})<\mathrm{crit}(\mathbf{c})$. Suppose to the contrary that
$\mathbf{z}$ is the shortest proper prefix of $\mathbf{x}$ with
$\mathrm{crit}(\mathbf{z})\geq\mathrm{crit}(\mathbf{c})$. Then by Lemma \ref{32u8r9},
$\mathbf{cz}=\mathbf{c}*^{M}\mathbf{z}\in\mathrm{Li}(N)$ which is a contradiction. Therefore if $\mathbf{z}$ is a proper prefix of $\mathbf{x}$ or $\mathbf{y}$, then
$\mathrm{crit}(\mathbf{z})<\mathrm{crit}(\mathbf{c})$. Since
\[\mathrm{crit}^{N}(\mathbf{c})=\mathrm{crit}^{N}(\mathbf{x})=\mathrm{crit}^{N}(\mathbf{y}),\]
we have
\[\mathrm{crit}^{M}(\mathbf{c})=\mathrm{crit}^{M}(\mathbf{x})=\mathrm{crit}^{M}(\mathbf{y}),\] hence by Theorem \ref{t4j8i09t42huqa}, we have
\[\mathrm{crit}^{M}(\mathbf{cx})=\mathrm{crit}^{M}(\mathbf{c}*^{M}\mathbf{x})=\mathrm{crit}^{M}(\mathbf{c}*^{M}\mathbf{y})=\mathrm{crit}^{M}(\mathbf{cy}).\]

\item Suppose that $\mathbf{x}a,\mathbf{x}b\in \mathrm{Li}(N),\mathrm{crit}^{N}(a)<\mathrm{crit}^{N}(b)$. Then
\[\mathrm{crit}^{M}(a)<\mathrm{crit}^{M}(b),\]
so
\[\max(\mathrm{crit}[N])\leq\mathrm{crit}^{M}(\mathbf{x}a)<\mathrm{crit}^{M}(\mathbf{x}b).\]
Therefore, $\mathrm{crit}^{M}(\mathbf{x}b)=\max(\mathrm{crit}[M])$, so
$\mathbf{x}b\not\in J$.

\end{enumerate}
\end{proof}

Let $A$ be a set, and let $a\in A$. We say that a multigenic Laver table $M$ over an alphabet $A$ that covers $N$ is an irreducible cover of $N$ of type $a$ if $\mathbf{x}[1]=a$ whenever $\mathbf{x}\in M\setminus N$.

Suppose that $N$ is a multigenic Laver table over an alphabet $A$ and $\mathcal{A}_{a}$ is the set of all irreducible covers of
$N$ of type $a$ for each $a\in A$. Then by propositions \ref{t0g42hnudl} and
\ref{25yut94anh}, the set $\{\bigcup_{a\in A}M_{a}\mid M_{a}\in\mathcal{A}_{a}\cup\{N\}\,\textrm{for}\,a\in A\}\setminus\{N\}$ is precisely the set of all multigenic Laver tables that cover $N$.

\begin{algo}(\textbf{sketch})
\label{t42h890gt4bh} Suppose that $M$ is a finite multigenic Laver table over the alphabet $\{0,1\}$. Then the following construction
produces the set of all irreducible covers of $M$ of type $0$.

Let $L$ be the set of all $\mathbf{x}\in \mathrm{Li}(M)$ such that $\mathbf{x}[1]=0$ and if $\mathbf{x}=\mathbf{y}c$ and
$\mathbf{y}b,\mathbf{y}c\in\mathrm{Li}(M)$, then $\mathrm{crit}^{M}(b)\geq \mathrm{crit}^{M}(c)$. 

Let $\simeq$ be the smallest equivalence relation on $L$ where
\begin{enumerate}
\item $\mathbf{cx}\simeq\mathbf{cy}$ whenever $\mathbf{c}\in\mathcal{I}_{M}$ and $\mathbf{x},\mathbf{y}$ are non-empty strings with $\mathbf{cx},\mathbf{cy}\in L$ and $\mathrm{crit}^{M}(\mathbf{c})=\mathrm{crit}^{M}(\mathbf{x})=\mathrm{crit}^{M}(\mathbf{y})$,
and 

\item $\mathbf{x}b\simeq\mathbf{x}c$ whenever $\mathbf{x}b,\mathbf{x}c\in L,|\mathbf{x}|\neq 0$ and
$\mathrm{crit}^{M}(b)=\mathrm{crit}^{M}(c)$.
\end{enumerate}
For each subset $V\subseteq L/\simeq$, let
$J_{V}=\{\mathbf{x}\in L:[\mathbf{x}]_{\simeq}\in V\}$, and let $N_{V}=M\cup\{\mathbf{xy}:\mathbf{x}\in J_{V},\mathbf{y}\in M\}$.
Then let $\mathcal{A}$ be the set of all objects of the form $N_{V}$ where $N_{V}$ is a multigenic Laver table.
Then $\mathcal{A}\setminus\{N\}$ is the set of all multigenic Laver tables that cover $M$ irreducibly for $0$.
All the steps used to obtain $\mathcal{A}$ are effective. In this algorithm, there are $2^{|L/\simeq|}-1$ pre-multigenic Laver tables
$N_{V}$ to test for self-distributivity.
\end{algo}

We do not believe that Algorithm \ref{t42h890gt4bh} is optimal since we predict that there are unknown techniques
that could be used to improve upon Algorithm \ref{t42h890gt4bh}.

While Algorithm $\ref{t42h890gt4bh}$ still runs for multigenic Laver tables over alphabets $A$ with $|A|>2$, Algorithm $\ref{t42h890gt4bh}$ can be optimized when $|A|>2$ using the following observation.  Suppose that $N$ is a multigenic Laver table over $A$ and $a\in A$. Suppose also that $M$ covers $N$ irreducibly of type $a$. Then whenever $a\in B\subseteq A$, either $M\cap B^{+}=N\cap B^{+}$ or $M\cap B^{+}$ covers $N\cap B^{+}$ irreducibly of type $a$. Therefore, to find an irreducible cover of $N$ of type $a$, one first searches for multigenic Laver tables $N_{b}$ for $b\in A\setminus\{a\}$ where 
\begin{enumerate}
\item $N_{b}=N\cap(A\setminus\{b\})^{+}$ or $N_{b}$ is an irreducible cover $N\cap(A\setminus\{b\})^{+}$ of type $a$ and

\item for all $b,c\in A\setminus\{a\}$, we have $N_{b}\cap(A\setminus\{b,c\})^{+}=N_{c}\cap(A\setminus\{b,c\})^{+}$.
\end{enumerate}
One then uses the technique outlined in Algorithm $\ref{t42h890gt4bh}$ to search for some multigenic Laver table $M$ that covers $N$ irreducibly of type $a$ where $M\cap(A\setminus\{b\})^{+}=N_{b}$ for all $b\in A\setminus\{a\}$.

Observe that if $X$ is a finite permutative LD-system, then whenever $U$ is an upwards closed subset of
$(X,\preceq)$, the set $U$ is a subalgebra of $(X,*)$. Therefore, if $|X|=n$, there is a sequence of subalgebras
$X_{1}\subseteq X_{2}\subseteq\ldots \subseteq X_{n}=X$ where $|X_{i}|=i$ for $1\leq i\leq n$. Therefore, any finite permutative
LD-system can be built by starting with a one element algebra and then extending the algebra to a larger finite permutative LD-system one point at a time. We shall now give a method for obtaining one-point extensions of finite permutative LD-systems.

\begin{defn}
Let $X$ be a finite Laver-like algebra. Let $\gamma$ be the largest critical point in
$\mathrm{crit}[X]\setminus\{\max(\mathrm{crit}[X])\}$. Let $\mathbf{a}\in X/\equiv^{\gamma}$. Let $a'\in\mathbf{a}$.
Then we say that a homomorphism $\phi:X/\equiv^{\gamma}\rightarrow X$ is an \index{extension homomorphism}\emph{extension homomorphism} for $\mathbf{a}$ if $\phi([x]_{\gamma})\equiv^{\gamma}a'*x$ for each $x\in X$.
\end{defn}
\begin{thm}
Suppose that $\phi:X/\equiv^{\gamma}\rightarrow X$ is an extension homomorphism for $\mathbf{a}$. Let $a\not\in X$. Then extend
$(X,*)$ to an algebra $(X\cup\{a\},*)$ defined by
\begin{enumerate}
\item $x*a=x*a'$ for $x\not\in \mathrm{Li}(X)$,

\item $x*a=a$ for $x\in \mathrm{Li}(X)$,

\item $a*a=\phi([a']_{\gamma})$, and

\item $a*x=\phi([x]_{\gamma})$ for $x\in X$.
\end{enumerate}
Then $(X\cup\{a\},*)$ is a permutative LD-system.
\end{thm}
\begin{proof}
Suppose $r\in\{a,x\},s\in\{a,y\},t\in\{a,z\}$ and $x,y,z\in X$. Then we shall show that
$r*(s*t)=(r*s)*(r*t)$ in $8$ different cases depending on whether $r=a$ or $r=x$, $s=a$ or $s=x$, $t=a$ or $t=x$.
\begin{enumerate}
\item $x*(y*z)=(x*y)*(x*z)$.

This follows from the self-distributivity of $(X,*)$.

\item $x*(y*a)=(x*y)*(x*a)$.

If $y\in \mathrm{Li}(X)$, then $x*y\in \mathrm{Li}(X)$ and $x*(y*a)=x*a$ and
\[(x*y)*(x*a)=x*a.\]

If $x\in \mathrm{Li}(X)$, then
\[x*(y*a)=y*a=(x*y)*(x*a).\]

If $x\not\in \mathrm{Li}(X),y\not\in \mathrm{Li}(X)$, then 
\[x*(y*a)=x*(y*a')=(x*y)*(x*a')=(x*y)*(x*a).\]

\item $x*(a*z)=(x*a)*(x*z)$

If $x\in \mathrm{Li}(X)$, then $x*(a*z)=a*z=(x*a)*(x*z)$.

Now suppose that $x\not\in \mathrm{Li}(X)$. Then

\[x*(a*z)=x*\phi([z]_{\gamma})\]
\[=x*(a'*z)=(x*a')*(x*z)=(x*a)*(x*z).\]

\item $x*(a*a)=(x*a)*(x*a)$.

If $x\in \mathrm{Li}(X)$, then $x*(a*a)=a*a=(x*a)*(x*a)$.

If $x\not\in \mathrm{Li}(X)$, then

\[x*(a*a)=x*\phi([a']_{\gamma})\]
\[=x*(a'*a')=(x*a')*(x*a')=(x*a)*(x*a).\]

\item $a*(y*z)=(a*y)*(a*z)$

\[a*(y*z)=\phi([x*y]_{\gamma})\]
\[=\phi([x]_{\gamma})*\phi([y]_{\gamma})=(a*x)*(a*y).\]

\item $a*(y*a)=(a*y)*(a*a)$.

If $y\in \mathrm{Li}(X)$, then $a*y=\phi([y]_{\gamma})\in \mathrm{Li}(X)$. Therefore,
\[a*(y*a)=a*a=(a*y)*(a*a).\]

Now assume that $y\not\in \mathrm{Li}(X)$. Then

\[a*(y*a)=a*(y*a')=\phi([y*a']_{\gamma})\]
\[=\phi([y]_{\gamma})*\phi([a']_{\gamma})=(a*y)*(a*a).\]

\item $a*(a*z)=(a*a)*(a*z)$.

\[a*(a*z)=a*\phi([z]_{\gamma})=\phi([\phi([z]_{\gamma})]_{\gamma})=\phi([a'*z]_{\gamma})\]
\[=\phi([a']_{\gamma})*\phi([z]_{\gamma})=(a*a)*(a*z).\]

\item $a*(a*a)=(a*a)*(a*a)$.

\[a*(a*a)=a*\phi([a']_{\gamma})=\phi([\phi([a']_{\gamma})]_{\gamma})\]
\[=\phi([a'*a']_{\gamma})=\phi([a']_{\gamma})*\phi([a']_{\gamma})=(a*a)*(a*a).\]
\end{enumerate}

We therefore conclude that $(X\cup\{a\},*)$ is an LD-system.

It is easy to see that $\mathrm{Li}(X\cup\{a\})=\mathrm{Li}(X)$.
Furthermore, if $x\in\mathrm{Li}(X\cup\{a\})$, then $a*x=\phi([x]_{\gamma})\in\mathrm{Li}(X)=\mathrm{Li}(X\cup\{a\})$ as well since $\phi$ is a homomorphism between permutative LD-systems. Therefore, $\mathrm{Li}(X\cup\{a\})$ is a left-ideal in
$(X\cup\{a\},*)$. We shall now establish that $(X\cup\{a\},*)$ is permutative by two cases.

Claim 1: For all $x\in X$, there is some $n$ with $t_{n}(x,a)\in\mathrm{Li}(X\cup\{a\})$.

Proof: If $x\in\mathrm{Li}(X)$, then our claim is satisfied when $n=2$. Now assume that $x\not\in\mathrm{Li}(X)$. Then $x*a=x*a'\in X$. Therefore, since $X$ is permutative, there is some $N$ with
$t_{N+1}(x,a)=t_{N}(x*a,x)\in\mathrm{Li}(X)$.

Claim 2: For all $y\in X\cup\{a\}$, there is some $n$ with $t_{n}(a,y)\in\mathrm{Li}(X\cup\{a\})$.

Proof: We have $a*y\in X$. Therefore, by Claim 1, there is some $N$ where $t_{N+1}(a,y)=t_{N}(a*y,a)\in\mathrm{Li}(X\cup\{a\})$.

We therefore conclude that $(X\cup\{a\},*)$ is permutative.
\end{proof}

\begin{defn}
Suppose the following:
\begin{enumerate}
\item $X$ is a finite permutative LD-system. 

\item $\gamma$ is the largest critical point in $\text{crit}[X]\setminus\{\max(\textrm{crit}[X])\}$, 

\item $\mathbf{a}\in X/\equiv^{\gamma}$. 

\item $\delta$ is the least critical point such that $(a')^{\sharp}(\delta)\geq\gamma$ (the critical point $\delta$ only depends on
$\mathbf{a}$ and not the particular choice of representative $a'$).

\item $\epsilon$ is the least critical point greater than $\delta$.
\end{enumerate}

Then we say that a homomorphism $\phi:X/\equiv^{\epsilon}\rightarrow X$ is a \index{reduced extension homomorphism}\emph{reduced extension homomorphism} for $\mathbf{a}$ if $\phi([x]_{\epsilon})\equiv^{\gamma}a'*x$ for each $x\in X$ and where if $\epsilon=\max(\mathrm{crit}[X])$, then $\phi([x]_{\epsilon})=\phi([y]_{\epsilon})$ whenever $x\equiv^{\gamma}y$.
\end{defn}

\begin{prop}
Let $\pi:X/\equiv^{\gamma}\rightarrow X/\equiv^{\epsilon}$ be the projection mapping. Then
$\phi\mapsto\phi\pi$ is a one-to-one correspondence between the extension homomorphisms for $\mathbf{a}$ and the reduced
extension homomorphisms for $\mathbf{a}$.
\end{prop}


\begin{defn}
Suppose that $(X,*)$ is a finite reduced Laver-like LD-system. Define \index{$A_{X,n}$}
\[A_{X,n}=\{(x*x_{1},\ldots,x*x_{n})\mid x\neq 1,x_{1},\ldots,x_{n}\in X\},\]
and let \index{$B_{X,n}$}$B_{X,n}=X^{n}\setminus A_{X,n}$.
\end{defn}
\begin{defn}
Suppose that $(X,*)$ is a finite reduced Laver-like LD-system.
Let $\alpha$ be the largest critical point in $\mathrm{crit}[X]\setminus\{\max(\mathrm{crit}[X])\}$.
We say that a map $t:(X/\equiv^{\alpha})^{n}\rightarrow X/\equiv^{\alpha}$ is \index{consistent}\emph{consistent} for $X$ if whenever
$x,y\in X\setminus\{1\}$ and 
\[(x*x_{1},\ldots,x*x_{n})=(y*y_{1},\ldots,y*y_{n}),\]
then
\[x*t([x_{1}]_{\alpha},\ldots,[x_{n}]_{\alpha})=y*t([y_{1}]_{\alpha},\ldots,[y_{n}]_{\alpha}).\]
\end{defn}

Every consistent mapping $t:(X/\equiv^{\alpha})^{n}\rightarrow X/\equiv^{\alpha}$ satisfies the identity
\[[x]*t([x_{1}]_{\alpha},\ldots,[x_{n}]_{\alpha})=t([x]_{\alpha}*[x_{1}]_{\alpha},\ldots,[x]_{\alpha}*[x_{n}]_{\alpha}).\]
\begin{defn}
Define $\mathcal{T}^{n}_{X}$\index{$\mathcal{T}^{n}_{X}$} for all finite reduced Laver-like LD-systems by induction on the number of critical points of $X$. If $X$ has only one member, then let $\mathcal{T}^{n}_{X}=\{t\}$ where $t$ is the unique function from $X^{n}$ to $X.$
Now suppose that $X$ has multiple elements. Let $\alpha$ be the largest critical point in $\mathrm{crit}[X]\setminus\{\max(\mathrm{crit}[X])\}$. If $s\in\mathcal{T}^{n}_{X/\equiv^{\alpha}}$ is not consistent, then let $Q_{s}=\emptyset$. If $s\in\mathcal{T}^{n}_{X/\equiv^{\alpha}}$ is consistent, then let \index{$s^{+}$}$s^{+}:A_{X,n}\rightarrow X$ be the mapping where if $x\neq 1$, then 
\[s^{+}(x*x_{1},\ldots,x*x_{n})=x*s([x_{1}]_{\alpha},\ldots,[x_{n}]_{\alpha}).\]
Let $\mathcal{V}_{s}$ be the set of all functions $t:B_{X,n}\rightarrow X$ such that $t(x_{1},\ldots,x_{n})\in s([x_{1}],\ldots,[x_{n}])$. Let $Q_{s}=\{s^{+}\cup t\mid t\in\mathcal{V}_{s}\}$. Then define
\[\mathcal{T}^{n}_{X}=\bigcup_{s\in\mathcal{T}^{n}_{X/\equiv^{\alpha}}}Q_{s}.\]
\end{defn}
\begin{prop}
Let $X$ be a finite reduced Laver-like LD-system. Then $\mathcal{T}^{n}_{X}$ is the set of all functions $t:X^{n}\rightarrow X$ that satisfy the identity 
\[t(x*x_{1},\ldots,x*x_{n})=x*t(x_{1},\ldots,x_{n}).\]
\end{prop}

\section{Miscellaneous topics}
In this chapter, we shall cover several miscellaneous topics.

\subsection{Algebras of elementary embeddings}
\label{t4i80t4qwkao2}

We shall now establish results about finite Laver-like LD-systems using large cardinal hypotheses where it is unknown whether the large cardinal hypotheses can be removed or not.

\begin{defn}
If $j:V_{\lambda}\rightarrow V_{\lambda}$, then define \index{$j^{+}$}$j^{+}:V_{\lambda+1}\rightarrow V_{\lambda+1}$ by letting
\[j^{+}(A)=\bigcup_{\alpha<\lambda}j(A\cap V_{\alpha}).\]
\end{defn}
\begin{lem}
Let $j:V_{\lambda}\rightarrow V_{\lambda}$ be elementary. Then whenever $A\subseteq V_{\lambda}$, the mapping
$j$ is also an elementary embedding from $(V_{\lambda},\in,A)$ to $(V_{\lambda},\in,j^{+}(A))$.
\end{lem}

\begin{lem}
Let $j\in\mathcal{E}_{\lambda}^{+}$, and let $\mathrm{crit}(j)=\kappa$, and suppose that $A\subseteq V_{\kappa}$. Then
$j^{n}(A)\cap V_{j^{m}(\kappa)}=j^{m}(A)$ whenever $0\leq m\leq n$.
\end{lem}
\begin{proof}
It suffices to show that $j^{n+1}(A)\cap V_{j^{n}(\kappa)}=j^{n}(A)$ for all $n\in\omega$.

If $n=0$, then since $\mathrm{crit}(j)=\kappa$, we have $j(A)\cap V_{\kappa}=A\cap V_{\kappa}=A$.
Now suppose that $n>0$. Then
\[j^{n+1}(A)\cap V_{j^{n}(\kappa)}=((j^{n}*j)\circ j^{n})(A)\cap V_{j^{n}(\kappa)}\]
\[=(j^{n}*j)(j^{n}(A))\cap V_{j^{n}(\kappa)}=j^{n}(A)\cap V_{j^{n}(\kappa)}\]
\[=j^{n}(A\cap V_{\kappa})=j^{n}(A).\]
\end{proof}

\begin{lem}
Suppose that $j:V_{\lambda}\rightarrow V_{\lambda},\mathrm{crit}(j)=\kappa,$ and $A\subseteq V_{\kappa}$. 
\begin{enumerate}
\item There is a unique $B\subseteq V_{\lambda}$ where $A=B\cap V_{\kappa}$ and $j^{+}(B)=B$.

\item $B\cap V_{j^{n}(\kappa)}=j^{n}(A)$ for all $n\in\omega$.

\item $B=\bigcup_{n}j^{n}(A)$.

\item $j$ is an elementary embedding from $(V_{\lambda},B,\in)$ to $(V_{\lambda},B,\in)$.

\item $(V_{\kappa},A,\in)$ is an elementary substructure of $(V_{\lambda},B,\in)$.
\end{enumerate}
\end{lem}

\begin{defn}
If $\lambda$ is a cardinal and $A\subseteq V_{\lambda}$, then define \index{$\mathcal{E}_{\lambda}[A]$}$\mathcal{E}_{\lambda}[A]$ to be the set of all elementary embeddings
$j:V_{\lambda}\rightarrow V_{\lambda}$ such that $j^{+}(A)=A$ for each $\gamma<\lambda$.
\end{defn}
\begin{prop}
$\mathcal{E}_{\lambda}[A]$ is a closed subset of $\mathcal{E}_{\lambda}$ and $\mathcal{E}_{\lambda}[A]$ is closed under the operations $*,\circ$.
\end{prop}

The motivation behind $\mathcal{E}_{\lambda}[A]$ is that structure on $A$ induces structure on the algebra
$\mathcal{E}_{\lambda}[A]$ and even on the locally finite algebras $\mathcal{E}_{\lambda}[A]/\equiv^{\gamma}$.

The following lemma illustrates a method for producing new rank-into-rank embeddings from old ones.
\begin{lem}
(Square-Root Lemma) Suppose that $j\in\mathcal{E}_{\lambda}$ is an I1-embedding. Then there exists a $k\in\mathcal{E}_{\lambda}$ with
$k*k=j$
\end{lem}
\begin{proof}
Suppose that $j\in\mathcal{E}_{\lambda}$ is an I1-embedding. Then by elementarity,
\[V_{\lambda+1}\models\exists k\in\mathcal{E}_{\lambda}:k*k=j\]
if and only if
\[V_{\lambda+1}\models\exists k\in\mathcal{E}_{\lambda}:k*k=j*j\]
which is true when $k=j$. Therefore, there exists some $k\in\mathcal{E}_{\lambda}$ with $k*k=j$.
\end{proof}

We shall now extend the idea in the proof of the Square-Root Lemma to more general setting.
\begin{defn}
If $j_{1},\ldots,j_{n}\in\mathcal{E}_{\lambda}$, then let 
\[\mathrm{crit}_{r}(j_{1},\ldots,j_{n})\]
be the $r$-th element in the set 
\[\{\mathrm{crit}(j)\mid j\in\langle j_{1},\ldots,j_{n}\rangle\}\]
where we set
\[\mathrm{crit}_{0}(j_{1},\ldots,j_{n})=\min(\mathrm{crit}(j_{1}),\ldots,\mathrm{crit}(j_{n})).\]
\end{defn}
\begin{defn}
Suppose that $j\in\mathcal{E}_{\lambda}$ and $\gamma<\lambda$ is a limit ordinal. Then let
\index{$j\upharpoonright_{\gamma}$}
$j\upharpoonright_{\gamma}:V_{\gamma}\rightarrow V_{\gamma+1}$ be the mapping where
$j\upharpoonright_{\gamma}(x)=j(x)\cap V_{\gamma}$ for each $x\in V_{\gamma}$.
\end{defn}
Take note that $j\upharpoonright_{\gamma}=k\upharpoonright_{\gamma}$ if and only if $j\equiv^{\gamma}k$.
Furthermore, there is a formula $\phi$ such that
$(V_{\lambda},\in)\models\phi(j\upharpoonright_{\gamma},k\upharpoonright_{\gamma},\ell)$ if and only if
$\ell=(j*k)\upharpoonright_{\gamma}$.
\begin{defn}
Let $\mathcal{E}_{\lambda}\upharpoonright_{\gamma}=\{j\upharpoonright_{\gamma}:j\in\mathcal{E}_{\lambda}\}$ and let
$\mathcal{E}_{\lambda}[A]\upharpoonright_{\gamma}=\{j\upharpoonright_{\gamma}:j\in\mathcal{E}_{\lambda}[A]\}$. Then 
$\mathcal{E}_{\lambda}\upharpoonright_{\gamma}$ may be endowed with operations $*,\circ$ defined by
$j\upharpoonright_{\gamma}*k\upharpoonright_{\gamma}=(j*k)\upharpoonright_{\gamma}$ and
$j\upharpoonright_{\gamma}\circ k\upharpoonright_{\gamma}=(j\circ k)\upharpoonright_{\gamma}$.
\end{defn}

Let $T:\mathcal{E}_{\lambda}^{<\omega}\rightarrow V_{\omega\cdot 2}$ be a function where
\begin{enumerate}
\item $T$ is a definable function in $(V_{\lambda+1},\in)$, and

\item $T(j_{1},\ldots,j_{m})=T(k_{1},\ldots,k_{n})$ if and only if $m=n$ and for all $r\in\omega$ there is an isomorphism
\[\iota:\langle j_{1},\ldots,j_{m}\rangle\rightarrow\langle k_{1},\ldots,k_{n}\rangle\]
where $\iota(j_{v})=k_{v}$ for $1\leq v\leq m$.
\end{enumerate}

\begin{thm}
\label{t4h82ugbwerh}
Suppose the following:
\begin{enumerate}
\item $\lambda$ is a cardinal.
\item $R\subseteq V_{\lambda}$.
\item $\ell_{1},\ldots,\ell_{p},j_{1},\ldots,j_{m}\in\mathcal{E}_{\lambda}[R]$ and $(k_{r,s})_{1\leq r\leq n,1\leq s\leq p}\in\mathcal{E}_{\lambda}[R]^{n\cdot p}$.
\item $\ell_{1},\ldots,\ell_{p}$ are I1 embeddings.
\item $T(\ell_{s}*j_{1},\ldots,\ell_{s}*j_{m},k_{1,s},\ldots,k_{n,s})=x_{s}$ whenever $1\leq s\leq p$.
\item $v$ is a natural number.
\item there is some $\mu<\lambda$ where $\mu=\mathrm{crit}_{v}(k_{1,1},\ldots,k_{n,1})=\ldots=\mathrm{crit}_{v}(k_{1,p},\ldots,k_{n,p})$.
\item $k_{r,1}\equiv^{\mu}\ldots\equiv^{\mu}k_{r,p}$ for $1\leq r\leq n$.
\item $\ell_{1}\equiv^{\mu+\omega}\ldots\equiv^{\mu+\omega}\ell_{p}$.
\end{enumerate}
Then there are $(w_{r,s})_{1\leq r\leq n,1\leq s\leq p}$ in $\mathcal{E}_{\lambda}[R]$ where
\begin{enumerate}
\item $T(j_{1},\ldots,j_{m},w_{1,s},\ldots,w_{n,s})=x_{s}$ for $1\leq s\leq p$,

\item there is some $\alpha<\lambda$ where $\mathrm{crit}_{v}(w_{1,s},\ldots,w_{n,s})=\alpha$ for $1\leq s\leq p$, and

\item $w_{r,1}\equiv^{\alpha}\ldots\equiv^{\alpha}w_{r,p}$ for $1\leq r\leq n$.
\end{enumerate}
\end{thm}
\begin{proof}
For $1\leq s\leq p$, let 
\[A_{s}=\{(w_{1}\upharpoonright_{\mathrm{crit}_{v}(w_{1},\ldots,w_{n})},\ldots,w_{n}\upharpoonright_{\mathrm{crit}_{v}(w_{1},\ldots,w_{n})})\]
\[:w_{1},\ldots,w_{n}\in\mathcal{E}_{\lambda}[R],
T(j_{1},\ldots,j_{m},w_{1},\ldots,w_{n})=x_{s}\}.\]
Then 
\[\ell_{s}(A_{s})=\{(w_{1}\upharpoonright_{\mathrm{crit}_{v}(w_{1},\ldots,w_{n})},\ldots,w_{n}\upharpoonright_{\mathrm{crit}_{v}(w_{1},\ldots,w_{n})})\]
\[:w_{1},\ldots,w_{n}\in\mathcal{E}_{\lambda}[R],T(\ell_{s}*j_{1},\ldots,\ell_{s}*j_{m},w_{1},\ldots,w_{n})=x_{s}\}.\]
Therefore,
\[(k_{1,s}\upharpoonright_{\mu},\ldots,k_{n,s}\upharpoonright_{\mu})\in\ell_{s}(A_{s})\]
for $1\leq s\leq p$. Since
$k_{r,1}\upharpoonright_{\mu}=\ldots=k_{r,p}\upharpoonright_{\mu}$, we have
\[(k_{1,1}\upharpoonright_{\mu},\ldots,k_{n,1}\upharpoonright_{\mu})=\ldots=(k_{1,p}\upharpoonright_{\mu},\ldots,k_{n,p}\upharpoonright_{\mu}).\]
Thus, let 
\[(\mathfrak{k}_{1},\ldots,\mathfrak{k}_{n})=(k_{1,1}\upharpoonright_{\mu},\ldots,k_{n,1}\upharpoonright_{\mu}).\]
Then
\[(\mathfrak{k}_{1},\ldots,\mathfrak{k}_{n})\in\ell_{1}(A_{1})\cap\ldots\ell_{p}(A_{p})\cap V_{\mu+\omega}\]
\[=\ell_{1}(A_{1})\cap\ldots\cap\ell_{1}(A_{p})\cap V_{\mu+\omega}
\subseteq\ell_{1}(A_{1}\cap\ldots\cap A_{p}).\]
Therefore, $A_{1}\cap\ldots\cap A_{p}\neq\emptyset.$

Let $(\mathfrak{w}_{1},\ldots,\mathfrak{w}_{n})\in A_{1}\cap\ldots\cap A_{p}$. Then there are $(w_{r,s})_{1\leq r\leq n,1\leq s\leq p}$ in $\mathcal{E}_{\lambda}[R]$ where
\[(\mathfrak{w}_{1},\ldots,\mathfrak{w}_{n})=(w_{1,s}\upharpoonright_{\mathrm{crit}_{v}(w_{1,s},\ldots,w_{n,s})},\ldots,w_{n,s}\upharpoonright_{\mathrm{crit}_{v}(w_{1,s},\ldots,w_{n,s})})\]
and 
\[T(j_{1},\ldots,j_{m},w_{1,s},\ldots,w_{n,s})=x_{s}\]
for $1\leq s\leq p.$ Therefore, there is some $\alpha<\lambda$ with $\mathrm{crit}_{v}(w_{1,s},\ldots,w_{n,s})=\alpha$ for $1\leq s\leq p$ and $\mathfrak{w}_{r}=w_{r,1}\upharpoonright_{\alpha}=\ldots=w_{r,p}\upharpoonright_{\alpha}$ for $1\leq r\leq n.$ Thus, we have $w_{r,1}\equiv^{\alpha}\ldots\equiv^{\alpha}w_{r,p}$ for $1\leq r\leq n$ as well.
\end{proof}

\begin{defn}
Let $\leq$ be a linear ordering of an LD-system $(X,*)$. Then we say that $\leq$ is \index{compatible}\emph{compatible} with $(X,*)$ if $y\leq z$ implies that $x*y\leq x*z$.
\end{defn}
\begin{defn}
Suppose that $A\subseteq V_{\lambda}.$ Every linear ordering of $V_{\lambda}$ definable in $(V_{\lambda},\in,A)$ induces a compatible linear ordering of $\mathcal{E}_{\lambda}[A]$. Suppose that $<$ is a linear ordering of $V_{\lambda}$ definable in $(V_{\lambda},\in,A)$.
Then define a linear ordering \index{$\triangleleft_{<}$}$\triangleleft_{<}$ on $\mathcal{E}_{\lambda}[A]$ by letting $j\triangleleft_{<}k$ if there exists some limit ordinal $\alpha$ with $j\upharpoonright_{\alpha}<k\upharpoonright_{\alpha}$ but where $j\upharpoonright_{\beta}=k\upharpoonright_{\beta}$ for all limit ordinals $\beta<\alpha$. Then $\triangleleft_{<}$ is a linear ordering
on $\mathcal{E}_{\lambda}$. If the ordering $<$ is unambiguous, then we shall write $\triangleleft$ for $\triangleleft_{<}$.
If $\alpha$ is a limit ordinal, then define $j\upharpoonright_{\alpha}\triangleleft k\upharpoonright_{\alpha}$ if and only if
there is some limit ordinal $\beta$ with $\beta\leq\alpha$ and where $j\upharpoonright_{\beta}<k\upharpoonright_{\beta}$ but where
for all limit ordinals $\gamma$ with $\gamma<\beta$, we have $j\upharpoonright_{\gamma}=k\upharpoonright_{\gamma}$.
\end{defn}
By elementarity, the ordering $\triangleleft$ on $\mathcal{E}_{\lambda}[A]$ or on $\mathcal{E}_{\lambda}[A]\upharpoonright_{\gamma}$ is
compatible with $*$.

From Theorem \ref{t4h82ugbwerh}, we deduce the following purely algebraic result about finite permutative LD-systems.
\begin{thm}
\label{42gyu9bf2HU}
Suppose that there exists an I1 cardinal. Suppose furthermore that 
\begin{enumerate}
\item $U_{1},\ldots ,U_{p},V_{1},\ldots ,V_{m}$ and $(W_{r,s})_{1\leq r\leq n,1\leq s\leq p}$ are unary terms in the language with function symbols
$*,\circ$,

\item $L$ is an $np+1$-ary term in the language with function symbols $*,\circ$,

\item $v$ is a natural number,

\item There is some classical Laver table $A_{N}$ where in $A_{N}$, we have
\[\mu=\mathrm{crit}_{v}(W_{1,1}(1),\ldots ,W_{n,1}(1))=\ldots =\mathrm{crit}_{v}(W_{1,p}(1),\ldots ,W_{n,p}(1))<\mathrm{crit}(2^{N}).\]

\item $W_{r,1}(1)\equiv^{\mu}\ldots\equiv^{\mu}W_{r,p}(1)$ for $1\leq r\leq n$ in the classical Laver table $A_{N}$, and

\item $U_{1}(1)\equiv^{\mu^{+}}\ldots\equiv^{\mu^{+}}U_{p}(1)$ where
$\mu^{+}$ denotes the least critical point in $A_{N}$ greater than $\mu$.
\end{enumerate}
If $Y$ is a finite reduced permutative LD-system, then let $\approx$ be the relation on
$Y^{<\omega}$ where $(x_{1},\ldots ,x_{m})\approx(y_{1},\ldots ,y_{n})$ if and only if $m=n$ and whenever
$\langle x_{1},\ldots ,x_{m}\rangle$ and $\langle y_{1},\ldots ,y_{n}\rangle$ both have more than $v+1$ critical points, then
there is an isomorphism 
\[\iota:\langle x_{1},\ldots ,x_{m}\rangle/\equiv^{\mathrm{crit}_{v}(x_{1},\ldots ,x_{m})}\rightarrow
\langle y_{1},\ldots ,y_{n}\rangle/\equiv^{\mathrm{crit}_{v}(y_{1},\ldots ,y_{n})}\]
where $\iota([x_{i}])=[y_{i}]$ for $1\leq i\leq n$.

Then there is some finite permutative LD-monoid $X$ along with 
\[x,(y_{r,s})_{1\leq r\leq n,1\leq s\leq p}\in X\]
such that 
\begin{enumerate}
\item $X$ has a compatible linear ordering,

\item $(z*z)^{\sharp}(\alpha)\leq z^{\sharp}(\alpha)$ whenever $z\in X,\alpha\in\mathrm{crit}[X]$,

\item if $1\leq s\leq p$, then

\[(U_{s}(x)*V_{1}(x),\ldots ,U_{s}(x)*V_{m}(x),W_{1,s}(x),\ldots ,W_{n,s}(x))\]
\[\approx(V_{1}(x),\ldots ,V_{m}(x),y_{1,s},\ldots ,y_{n,s}),\]

\item there is some critical point $\alpha$ where $\mathrm{crit}_{v}(y_{1,s},\ldots ,y_{n,s})=\alpha$ for all $s$,

\item $y_{r,1}\equiv^{\alpha}\ldots\equiv^{\alpha}y_{r,p}$, and

\item $L(x,(y_{r,s})_{1\leq r\leq n,1\leq s\leq p})\neq 1$.
\end{enumerate}
\end{thm}
\begin{proof}
Suppose that there exists some I1-embedding $j:V_{\lambda}\rightarrow V_{\lambda}$. Let $\kappa=\mathrm{crit}(j)$ and let
$A$ be a linear ordering of $V_{\kappa}$. Then let $R$ be the unique subset of $V_{\lambda}$ such that
$R\cap V_{\kappa}=A$ and $j^{+}(R)=R$.

Then let
\[\ell_{1}=U_{1}(j),\ldots ,\ell_{p}=U_{p}(\ell_{p}),j_{1}=V_{1}(j),\ldots ,j_{m}=V_{m}(j)\]
and $k_{r,s}=W_{r,s}(j)$ for $1\leq r\leq n,1\leq s\leq p$. Apply Theorem \ref{t4h82ugbwerh} to obtain the elementary embeddings
$(w_{r,s})_{1\leq r\leq n,1\leq s\leq p}$ in $\mathcal{E}_{\lambda}[R]$. Then let $\gamma$ be a sufficiently large limit ordinal subject to the condition that
$\gamma<\lambda$. Let 
\[X^{\sharp}=\mathcal{E}_{\lambda}/\equiv^{\gamma},x=[j]_{\gamma},y_{r,s}=[w_{r,s}]_{\gamma}\]
for $1\leq r\leq n,1\leq s\leq p$ and let $X$ be the subalgebra of $X^{\sharp}$ generated by $x,(y_{r,s})_{1\leq r\leq n,1\leq s\leq p}$. Let
$\leq$ be the linear ordering of $X$ induced by $\triangleleft_{R}$.
Then 
\[(X,x,\leq,(y_{r,s})_{1\leq r\leq n,1\leq s\leq p})\]
satisfies all the requirements for this result.
\end{proof}
Theorems \ref{t4h82ugbwerh} and \ref{42gyu9bf2HU} can be generalized to stronger but more complex results by
using the large cardinal axioms $E_{n}(\lambda)$ which were studied by Laver in \cite{L97} and whose consistency strength lies between
I3 and I1.

\begin{exam}
Suppose that $s,t$ are unary terms in the language with function symbols $*,\circ$ and $u$ is a binary term in the language
with function symbols $*,\circ$. Suppose furthermore that suppose that $\mathrm{crit}(x)=\mathrm{crit}(s(x))=\mathrm{crit}(t(x))$ whenever $X$ is a permutative LD-system and $x\in X$. Then from Theorem \ref{42gyu9bf2HU},
there exists a finite permutative LD-monoid $(X,\circ,*)$ with a compatible linear ordering along with $x,y,z\in X$ such that
\begin{enumerate}
\item $s(x)*x=z$,

\item $t(y)*y=z$,

\item $u(y,z)\neq 1$,

\item $(r*r)^{\sharp}(\alpha)\leq r^{\sharp}(\alpha)$ for all $r\in X,\alpha\in\mathrm{crit}[X]$, and

\item $\mathrm{crit}(x)=\mathrm{crit}(y)$.
\end{enumerate}

\label{t42ihu9tgy2u}
\end{exam}
The conclusion of Example \ref{t42ihu9tgy2u} is a rather strong statement since from Example
\ref{t42ihu9tgy2u}, we may deduce the following results which are not known to be provable in ZFC:
\begin{enumerate}
\item $(1)_{n\in\omega}$ generates a free subalgebra of $\varprojlim_{n}A_{n}$.

\item $o_{n}(1)\leq o_{n}(2)$ for all $n\in\mathbb{N}$.
\end{enumerate}

\subsection{Algebras of elementary embeddings between extendible and rank-into-rank cardinals}
In this section, we shall briefly outline without proofs the basic theory of the algebras of elementary embeddings around $n$-hugeness.

Suppose that $\gamma$ is a limit ordinal and $f:V_{\gamma+1}\rightarrow V_{\gamma+1}$ is a mapping. Then
$f$ is said to be an \index{extendibility mapping}\emph{extendibility mapping} for $\gamma$ if for all $\alpha>\gamma$ there is an ordinal $\beta$ and elementary embedding
$j:V_{\alpha}\rightarrow V_{\beta}$ with $f=j\upharpoonright_{\gamma+1}$.

\begin{prop}
Suppose that $f,g$ are extendibility mappings for $\gamma$. Then $f\circ g$ is also an extendibility mapping for $\gamma$.
\end{prop}
\begin{prop}
Suppose that $f,g$ are extendibility mappings for $\gamma$. Suppose now that $\alpha_{1},\alpha_{2},\beta_{1},\beta_{2}>\gamma$ and
$j:V_{\alpha_{1}}\rightarrow V_{\beta_{1}},k:V_{\alpha_{2}}\rightarrow V_{\beta_{2}}$ are elementary embeddings with
$j\upharpoonright_{\gamma+1}=k\upharpoonright_{\gamma+1}$. Then $j(g)\upharpoonright_{\gamma+1}=k(g)\upharpoonright_{\gamma+1}$ and
$j(g)\upharpoonright_{\gamma+1}$ is an extendibility mapping for $\gamma$.
\end{prop}

If $\gamma$ is an inaccessible cardinal, then let $\mathbf{EE}_{\gamma}$ be the set of all extendibility mappings for
$\gamma$. Then $\mathbf{EE}_{\gamma}$ can be endowed with operations $*,\circ$ such that $\circ$ denotes composition and
$*$ is the mapping where $f*g=j(g)\upharpoonright_{\gamma+1}$ for some elementary embedding $j:V_{\alpha}\rightarrow V_{\beta}$ with
$j\upharpoonright_{\gamma+1}=f$.

\begin{thm}
If $\gamma$ is inaccessible, then $\mathbf{EE}_{\gamma}$ is a Laver-like LD-monoid.
\end{thm}
\subsection{Computing in endomorphic Laver tables}
The following example illustrates that the output of the fundamental operations on endomorphic Laver table operations is often too large for modern computers to handle.
\begin{exam}
Suppose the following.
\begin{enumerate}
\item $T:A_{n}^{2}\rightarrow A_{n}$ is a mapping where $x*T(y,z)=T(x*y,x*z)$ and $T(x,x)=x$ for all $x,y,z\in A_{n}$

\item $t:A_{n}^{3}\rightarrow A_{n}$ is the mapping where $t(x,y,z)=T(x,y)*z$ for all $x,y,z\in A_{n}$.

\item $L=\mathbf{L}((1)_{r\in\{e\}},A_{n},t)$. 

\item $\ell_{1}=e$ and $\ell_{s+1}=t(\ell_{s},\ell_{s})$ for $1\leq s<2^{n}$. 
\end{enumerate}
Then each term $\ell_{s}$ has $2^{s-1}$ instances of the varible $e$. However, $t^{\sharp}(\ell_{r},\ell_{r},\ell_{s})=\ell_{r*_{n}s}$ which has $2^{r*_{n}s}$ instances of the variable $e$. Therefore, in this case, it is infeasible to print the entire
term $t^{\sharp}(\ell_{r},\ell_{r},\ell_{s})$ except when $r*_{n}s$ is small.
\end{exam}

\begin{exam}
Let $t^{\bullet}:A_{2}\rightarrow A_{2}$ be the mapping where $t^{\bullet}(x,y,z)=y*z$.
Define $(L,t^{\sharp})=\mathbf{L}((1)_{r\in\{e\}},A_{2},t^{\bullet})$ and
define terms $\ell_{n}$ for all $n\in\omega$ by letting $\ell_{0}=e$ and $\ell_{i+1}=t(\ell_{i},e)$ for $i\in\omega$.
Then $t^{\sharp}(e,e,\ell_{n})$ has $6\cdot 2^{n}-4$ instances of the term $e$.
\end{exam}

We shall now develop a method for gathering information from the output of endomorphic Laver table operations even if the entire output of the endomorphic Laver table operations is very large. For notational simplicity, in the following development, we shall only work with endomorphic algebras with one fundamental operation.

\begin{defn}
Let $t$ be a function symbol. Then define 
\index{$\mathrm{sub}$}
\[\mathrm{sub}:\mathbf{T}_{\{t\}}[X]\times\{1,\ldots ,n\}^{*}\rightarrow\mathbf{T}_{\{t\}}[X]\]
inductively according to the following rules.
\begin{enumerate}
\item $\mathrm{sub}(\ell,\varepsilon)=\ell$.

\item $\mathrm{sub}(\ell,\mathbf{x}a)=\#$ whenever $\mathrm{sub}(\ell,\mathbf{x})=\#$ or $\mathrm{sub}(\ell,\mathbf{x})$
is a variable.

\item $\mathrm{sub}(\ell,\mathbf{x}i)=\ell_{i}$ whenever $\mathrm{sub}(\ell,\mathbf{x})=t(\ell_{1},\ldots ,\ell_{n})$.
\end{enumerate}
\end{defn}
\begin{defn}
Suppose the following.
\begin{enumerate}
\item $\mathcal{X}=(X,t^{\bullet})$ is an $n+1$-ary endomorphic Laver-like algebra.

\item $x'$ is a variable for each $x\in X$.

\item $X'=\{x'|x\in X\}$

\item $L=\mathbf{L}((x)_{x'\in X'},(X,t^{\bullet}))$.

\item $\phi:L\rightarrow(X,t^{\bullet})$ is the canonical homomorphism where $\phi(x')=x$ for $x\in X$.

\item $\#$ is a symbol.
\end{enumerate}

Let \index{$\mathbf{tree}^{\mathcal{X}}(\ell)$}
$\mathbf{tree}^{\mathcal{X}}(\ell):\{0,1\}^{*}\rightarrow X\cup\{\#\}$ be the mapping where $\mathbf{tree}^{\mathcal{X}}(\ell)(\mathbf{x})=\#$ whenever $\mathrm{sub}(\ell,\mathbf{x})=\#$ and $\mathbf{tree}^{\mathcal{X}}(\ell)(\mathbf{x})=\phi(\mathrm{sub}(\ell,\mathbf{x}))$ otherwise.

Let \index{$\Diamond(\mathcal{X})$}$\Diamond(\mathcal{X})=\{\mathbf{tree}^{\mathcal{X}}(\ell)\mid\ell\in L\}$.

Define an operation $t^{\sharp}$ on $\Diamond(\mathcal{X})$ by letting
\[t^{\sharp}(\mathbf{tree}^{\mathcal{X}}(\ell_{1}),\ldots ,\mathbf{tree}^{\mathcal{X}}(\ell_{n}),\mathbf{tree}^{\mathcal{X}}(\ell))\]
\[=\mathbf{tree}^{\mathcal{X}}(t^{\sharp}(\ell_{1},\ldots ,\ell_{n})).\]

Since the kernel of the function $\mathbf{tree}^{\mathcal{X}}$ is the equivalence relation $\simeq$ in Theorem
\ref{thanhz9auntsc2nj52}, we conclude that the operation $t^{\sharp}$ on $\Diamond(\mathcal{X})$ is well-defined.
\end{defn}
\begin{defn}
If $\mathfrak{l}\in\Diamond(\mathcal{X})$ and $\mathfrak{l}\neq\mathbf{tree}^{\mathcal{X}}(x')$ for all $x\in X$, then let
\index{$\mathfrak{l}\upharpoonright_{i}$}
$\mathfrak{l}\upharpoonright_{i}\in\Diamond(\mathcal{X})$ be defined by
$\mathfrak{l}\upharpoonright_{i}(\mathbf{x})=\mathfrak{l}(\mathbf{x}i)$.
\end{defn}
\begin{defn}
If $\mathfrak{l}_{1},\ldots ,\mathfrak{l}_{n}\in\Diamond(\mathcal{X})$ and
$(t,\mathfrak{l}_{1},\ldots ,\mathfrak{l}_{n})\not\in\mathrm{Li}(\Diamond(\mathcal{X}))$,
then let $t_{x}(\mathfrak{l}_{1},\ldots ,\mathfrak{l}_{n})\in\Diamond(\mathcal{X})$ be defined by
\[t_{x}(\mathfrak{l}_{1},\ldots ,\mathfrak{l}_{n})(\varepsilon)=t^{\bullet}(\mathfrak{l}_{1}(\varepsilon),\ldots ,\mathfrak{l}_{n}(\varepsilon),x)\]
and $t_{x}(\mathfrak{l}_{1},\ldots ,\mathfrak{l}_{n})(\mathbf{x}i)=\mathfrak{l}_{i}(\mathbf{x})$.
\index{$t_{x}(\mathfrak{l}_{1},\ldots ,\mathfrak{l}_{n})$}
\end{defn}
\begin{algo}
\label{y5rji8g0rwynuho}
$t^{\sharp}(\mathfrak{l}_{1},\ldots ,\mathfrak{l}_{n},\mathfrak{l})(\mathbf{x})$ can be calculated recursively using the following conditional statement
and dynamic programming.
\begin{enumerate}
\item If $\mathbf{x}=\varepsilon$, then return
$t^{\bullet}(\mathfrak{l}_{1}(\varepsilon),\ldots ,\mathfrak{l}_{n}(\varepsilon),\mathfrak{l}(\varepsilon))$,

\item else if $(t,\mathfrak{l}_{1}(\varepsilon),\ldots ,\mathfrak{l}_{n}(\varepsilon))\in\mathrm{Li}(\Gamma(\mathcal{X}))$, then
return $\mathfrak{l}(\mathbf{x})$,

\item else if $\mathfrak{l}(i)=\#$ for $1\leq i\leq n$, then set $\mathfrak{l}(\varepsilon)=x$ and return
$t_{x}(\mathfrak{l}_{1},\ldots ,\mathfrak{l}_{n})(\mathbf{x})$,

\item else if $\mathfrak{l}(\varepsilon)=x$, then return
\[t^{\sharp}(t^{\sharp}(\mathfrak{l}_{1},\ldots ,\mathfrak{l}_{n},\mathfrak{l}\upharpoonright_{1}),\ldots ,
t^{\sharp}(\mathfrak{l}_{1},\ldots ,\mathfrak{l}_{n},\mathfrak{l}\upharpoonright_{n}),
t_{x}(\mathfrak{l}_{1},\ldots ,\mathfrak{l}_{n})).\]
\end{enumerate}
\end{algo}

If $t^{\bullet}(x,y,z):A_{n}^{3}\rightarrow A_{n}$ is the mapping defined by $(x\wedge y)*z$ WHAT IS WEDGE, then
Algorithm \ref{y5rji8g0rwynuho} can usually calculate $t^{\sharp}(\mathfrak{l}_{1},\mathfrak{l}_{2},\mathfrak{l})(\mathbf{x})$ in a couple of seconds
even when $\mathbf{x}$ is a randomly generated string of bits with $|\mathbf{x}|=100$. We usually run out of memory when
attempting to compute $t^{\sharp}(\mathfrak{l}_{1},\mathfrak{l}_{2},\mathfrak{l})(\mathbf{x})$
if $t^{\sharp}(\mathfrak{l}_{1},\mathfrak{l}_{2},\mathfrak{l})(\mathbf{x})\neq\#$ and if $|\mathbf{x}|>1000$.

\subsection{Selectors}
In this section, we shall use the notion of a selector to easily produce various operations $t:X^{n}\rightarrow X$ such that $x*t(x_{1},\ldots,x_{n})=t(x*x_{1},\ldots,x*x_{n})$.

Suppose that $k$ is a natural number. A classical Laver table selector is a function $T$ with domain 
$\{(n,x_{1},\ldots,x_{k}):n\in\omega,x_{1},\ldots,x_{k}\in A_{n}\}$ and with range $\{0,1\}$ where 
$T(m,x_{1},\ldots,x_{k})=T(n,y_{1},\ldots,y_{k})$ whenever there is an isomorphism
\[\iota:\langle x_{1},\ldots,x_{k}\rangle^{A_{m}}\rightarrow \langle y_{1},\ldots,y_{k}\rangle^{A_{n}}\]
preserving the linear ordering $\leq^{\sharp}$ (there can only be one such isomorphism) and where $\iota(x_{i})=y_{i}$ for $1\leq i\leq k$.

Suppose that $T$ is a classical Laver table selector. Then the $n$-th selector function for $T$ is the function defined
inductively according to the following rules:
\begin{enumerate}
\item The 0-th selector function is the unique function $t_{n}:A_{0}^{k}\rightarrow A_{0}$.

\item Suppose that the $n-1$-th selector function $t_{n-1}:A_{n-1}^{k}\rightarrow A_{n-1}$ has already been defined.
Suppose that $x_{1},\ldots,x_{k}\in A_{n}$. Let $r=t_{n-1}((x_{1})_{2^{n-1}},\ldots,(x_{k})_{2^{n-1}})$. If
$|\{r,r+2^{n-1}\}\cap\langle x_{1},\ldots,x_{k}\rangle|=1$, then let $t_{n}(x_{1},\ldots,x_{k})$ be the unique element in
$\{r,r+2^{n-1}\}\cap\langle x_{1},\ldots,x_{k}\rangle$. Otherwise, define
\[t_{n}(x_{1},\ldots,x_{k})=t_{n-1}((x_{1})_{2^{n-1}},\ldots,(x_{k})_{2^{n-1}})+2^{n-1}\cdot T(n,x_{1},\ldots,x_{k}).\]
\end{enumerate}
\begin{prop}
If $t_{n}$ is the $n$-th selector function for $T$, then $t_{n}$ satisfies the identity
$x*t_{n}(x_{1},\ldots,x_{k})=t_{n}(x*x_{1},\ldots,x*x_{k})$.
\end{prop}

\subsection{Possible applications to cryptography}

In this section, we shall propose several public key cryptosystems that use endomorphic Laver tables as their platforms. If these cryptosystems are successful, then the endomorphic Laver tables will offer the first practical application of modern set theory!

The public key cryptosystems in practical use today withstand all attacks against classical computers. However, these same cryptosystems
can be broken by quantum computers. We therefore need to develop cryptosystems which can withstand the attacks from quantum computers but are also efficient enough to apply in practice \cite{PQC}. We do not believe that quantum computers would be any more effective at breaking endomorphic Laver table based cryptosystems than classical computers are since the current quantum algorithms generally
use only a limited number of techniques (such as the quantum Fourier transform) to solve specific classes of problems \cite{AM}.
It is unknown as to how well endomorphic Laver table based cryptosystems fare against attacks from even classical computers though.

We observe that the public key cryptosystems that use LD-systems as platforms are analogous to non-abelian group based cryptosystems. In this section, we shall modify these cryptosystems so that the endomorphic Laver tables may be used as a platform for these cryptosystems.

The classical Laver tables are currently an insecure platform for all cryptosystems mainly since
$A_{48}$ is the largest classical Laver table ever computed (which gives at most 48 bits of security) and because in practice $x*_{n}y$ is either close to $2^{n}$ or close to $x$. The multigenic Laver tables are not secure platforms for any known cryptosystem either since the factorization problem for multigenic Laver tables is easily solvable.

\begin{cryprob}(Factorization problem for magmas)

Input: A magma $(X,*)$, a subset $R\subseteq X^{2}$, and $x\in X$.

Objective: Find some $(r,s)\in R$ such that $r*s=x$ or determine that no such pair $(r,s)$ exists.

\label{t2hawo3tqwne}
\end{cryprob}

\begin{algo}
Suppose that $M$ is a multigenic Laver table over the alphabet $A$ and $\mathbf{z}\in M$.
Then the following algorithm can be used to find all factorizations of $\mathbf{z}$ into a product
$\mathbf{x}*\mathbf{y}$. We shall call a factorization $\mathbf{x}*\mathbf{y}$ of $\mathbf{z}$ a minimal factorization if
$(\mathbf{x},\mathbf{y})$ is a forward sequence.

Let $S:=\{(\mathbf{z'},a)\}$ where $\mathbf{z}=\mathbf{z'}a$.
\begin{enumerate}
\item For $1\leq i\leq n-2$, let $\mathbf{z}_{i}$ be the prefix of $\mathbf{z}$ of length $i$, and
\begin{enumerate}
\item for $j\geq 1$ do 
\begin{enumerate}
\item if $j=1$ then let $a_{j}=\mathbf{z}[i+1]$,

\item else if $(\mathbf{z}_{i}*a_{1}\ldots a_{j-1})\circ\mathbf{z}_{i}$ is a proper substring of $\mathbf{z}$ then let
$a_{j}$ be the letter so that $((\mathbf{z}_{i}*a_{1}\ldots a_{j-1})\circ\mathbf{z}_{i})a_{j}$ is a prefix of $\mathbf{z}$ and
if $\mathbf{z}=\mathbf{z}_{i}*a_{1}\ldots a_{j}$ then let $S:=S\cup\{(\mathbf{z}_{i},a_{1}\ldots a_{j})\}$ and terminate the loop for $j$,

\item else terminate the loop for $j$.
\end{enumerate}
\end{enumerate}
\end{enumerate}
Then $S$ is the set of all minimal factorizations of $\mathbf{z}$. The factorizations of $\mathbf{z}$ are precisely the factorizations
$\mathbf{x}*\mathbf{z}$ where $\mathbf{x}\in\mathrm{Li}(M)$, along with $\mathbf{x}*\mathbf{ry}$ where $(\mathbf{x},\mathbf{y})\in S$
and $\mathbf{r}=\varepsilon$ or $\mathbf{x}*\mathbf{r}\in\mathrm{Li}(M)$.
\end{algo}

Unlike the classical and multigenic Laver tables, the endomorphic Laver tables could be used as platforms for public key cryptosystems.

The following key-establishment cryptosystem is a simplified version of the cryptosystem in \cite{CLCHKP} for semigroups.
\begin{crypt}
\label{t432uhbto42h24ti4e}
A semigroup $(X,\circ)$ and an element $x\in X$ are known to the public.
\begin{enumerate}
\item Alice selects an $a\in X$ and publicly sends $r=a\circ x$ to Bob.

\item Bob selects an element $b\in X$ and then publicly sends $s=x\circ b$ to Alice.

The shared key is the element $K=a\circ x\circ b$.

\item Alice computes $K$. Alice is able to compute $K$ since $K=a\circ s$.

\item Bob computes $K$. Bob is able to compute $K$ since $K=r\circ b$.
\end{enumerate}
While Alice and Bob can compute $K$, an eavesdropper will, in principle, not be able to compute $K$ even when knowing $x$ and the communications $r,s$.
\end{crypt}

The security of cryptosystem \ref{t432uhbto42h24ti4e} depends on the difficulty of the following two problems.

\begin{cryprob}

(\textbf{finding left factor})

Input: A semigroup $(X,\circ)$ and $r,x\in X$ such that $r=a\circ x$ for some $a\in X$.

Problem: Find an $a'\in X$ such that $r=a'\circ x$.
\end{cryprob}
\begin{cryprob}

(\textbf{finding right factor})

Input: A semigroup $(X,\circ)$ and $s,x\in X$ such that $s=x\circ b$ for some $b\in X$.

Problem: Find an $b'\in X$ such that $s=x\circ b'$.
\end{cryprob}
The following key-establishment cryptosystem is a reformulation of cryptosystem \ref{t432uhbto42h24ti4e} for left-distributivity.
\begin{crypt}
An LD-system $X$ and some $x\in X$ are public.

\begin{enumerate}
\item Alice chooses some $a\in X$ and $n>1$ and sends $r_{1}=t_{n+1}(a,x),r_{2}=t_{n}(a,x)$ to Bob.

\item Bob chooses some $b\in X$ and sends $s=x*b$ to Alice.

The shared key is $K=a*(x*b)$.

\item Alice computes the shared key $K$. Alice is able to compute $K$ since
\[K=a*s.\]

\item Bob computes the shared key $K$. Bob is able to compute $K$ since
\[K=r_{1}*(r_{2}*b).\]
\end{enumerate}
\end{crypt}
\begin{crypt}
\label{42thumbknvc942dgio}

Suppose that $\mathcal{X}$ is a Laver-like endomorphic algebra with only one fundamental operation $t^{\bullet}$ of arity $n+1$. Then the public information consists of some algorithms for computing elements $\mathfrak{l}_{1},\ldots ,\mathfrak{l}_{n}\in\Diamond(\mathcal{X})$ where $(t,\mathfrak{l}_{1},\ldots ,\mathfrak{l}_{n})\not\in\mathrm{Li}(\Gamma(\Diamond(\mathcal{X})))$ along with
a string $\mathbf{x}\in\{1,\ldots ,n\}^{*}$ and an element $i\in\{1,\ldots ,n\}$.
\begin{enumerate}
\item Alice selects algorithms for computing some $\mathfrak{u}_{1},\ldots ,\mathfrak{u}_{n}\in\Diamond(\mathcal{X})$ and then sets
$(t,\mathfrak{j}_{1},\ldots ,\mathfrak{j}_{n})=(t,\mathfrak{u}_{1},\ldots ,\mathfrak{u}_{n})\circ(t,\mathfrak{l}_{1},\ldots ,\mathfrak{l}_{n})$.

\item Bob selects some algorithm for computing some $\mathfrak{v}_{1},\ldots ,\mathfrak{v}_{n}\in\Diamond(\mathcal{X})$ and then sets
$(t,\mathfrak{k}_{1},\ldots ,\mathfrak{k}_{n})=(t,\mathfrak{l}_{1},\ldots ,\mathfrak{l}_{n})\circ(t,\mathfrak{v}_{1},\ldots ,\mathfrak{v}_{n})$.

Define
\[(t,\mathfrak{w}_{1},\ldots ,\mathfrak{w}_{n})=(t,\mathfrak{u}_{1},\ldots ,\mathfrak{u}_{n})\circ(t,\mathfrak{l}_{1},\ldots ,\mathfrak{l}_{n})\circ
(t,\mathfrak{v}_{1},\ldots ,\mathfrak{v}_{n}).\]

Let $K=\mathfrak{w}_{i}(\mathbf{x})$ be the common key.

\item Alice computes the common key $K$ using her private algorithm for computing $\mathfrak{u}_{1},\ldots ,\mathfrak{u}_{n}$ and by inquiring Bob about specific values of the form $\mathfrak{k}_{j}(\mathbf{y})$.

\item Bob computes the common key $K$ using his private algorithm for computing $\mathfrak{v}_{1},\ldots ,\mathfrak{v}_{n}$ and by inquiring
Alice about specific values of the form $\mathfrak{j}_{j}(\mathbf{y})$.
\end{enumerate}
\end{crypt}

The following authentication scheme (Cryptosystem \ref{4thinuo0t42h08t42}) is a modification of the authentication system found in \cite{SCPD} so that any permutative partially endomorphic algebra could be used as a platform.

\begin{defn}
Let $(X,E,F)$ be a permutative partially endomorphic algebra and suppose $\alpha\in\mathrm{crit}(\Gamma(X,E))$. Let
$(t,x_{1},\ldots ,x_{n_{t}})\in\Gamma(X,E)$ be an involutive element with $\mathrm{crit}(t,x_{1},\ldots ,x_{n_{t}})=\alpha$. Then define
the congruence $\equiv^{\alpha}$ on $(X,E,F)$ by letting $x\equiv^{\alpha}y$ if and only if $t(x_{1},\ldots ,x_{n_{t}},x)=t(x_{1},\ldots ,x_{n_{t}},y)$. The congruence $\equiv^{\alpha}$ does not depend on the choice of $(t,x_{1},\ldots ,x_{n_{t}})$. We shall write
$[x]_{\alpha}$ for the equivalence class of $x$ with respect to $\equiv^{\alpha}$.
\end{defn}
\begin{crypt}
\label{4thinuo0t42h08t42}
Let $(X,E,F)$ be an efficiently computable permutative partially endomorphic algebra. Let $\alpha\in\mathrm{crit}(\Gamma(X,E))$. Alice's public key consists of elements $[a_{1}]_{\alpha},\ldots ,[a_{m}]_{\alpha}\in X/\equiv^{\alpha}$ along with $m+n$-ary terms $s_{1},\ldots ,s_{k},t_{1},\ldots ,t_{k}$. Alice's private key consists of $[x_{1}]_{\alpha},\ldots ,[x_{n}]_{\alpha}\in X/\equiv^{\alpha}$ where 
\[s_{i}([a_{1}]_{\alpha},\ldots ,[a_{m}]_{\alpha},[x_{1}]_{\alpha},\ldots ,[x_{n}]_{\alpha})=t_{i}([a_{1}]_{\alpha},\ldots ,[a_{m}]_{\alpha},[x_{1}]_{\alpha},\ldots ,[x_{n}]_{\alpha})\]
for $1\leq i\leq k$. In the following exchange, Alice proves to Bob that she knows elements
$x_{1}',\ldots ,x_{n}'$ so that 
\[s_{i}([a_{1}]_{\alpha},\ldots ,[a_{m}]_{\alpha},[x_{1}']_{\alpha},\ldots ,[x_{n}']_{\alpha})=t_{i}([a_{1}]_{\alpha},\ldots ,[a_{m}]_{\alpha},[x_{1}']_{\alpha},\ldots ,[x_{n}']_{\alpha})\]
for $1\leq i\leq k$ without revealing to Bob the elements $x_{1}',\ldots ,x_{n}'$ themselves. The following steps are repeated as many times as Bob desires.
\begin{enumerate}
\item Alice selects some inner endomorphism $L$ where
$(g,L(v_{1}),\ldots ,L(v_{n_{g}}))\in\mathrm{Li}(\Gamma(X,E))$ if and only if $\mathrm{crit}(g,v_{1},\ldots ,v_{n_{g}})\geq\alpha$. Alice then sends the information $y_{1}=L(x_{1}),\ldots ,y_{n}=L(x_{n})$ to Bob.

\item Bob randomly selects a bit $e\in\{0,1\}$ and sends $e$ to Alice.

\item If $e=0$, then Alice sends $L$ to Bob. Bob then verifies that 
$(g,L(v_{1}),\ldots ,L(v_{n_{g}}))\in\mathrm{Li}(\Gamma(X,E))$ if and only if $\mathrm{crit}(g,v_{1},\ldots ,v_{n_{g}})\geq\alpha$ and 
\[y_{1}=L(x_{1}),\ldots ,y_{n}=L(x_{n}).\]

\item If $e=1$, then Alice sends $b_{1}=L(a_{1}),\ldots ,b_{m}=L(a_{m})$ to Bob. Bob then verifies that
\[s_{i}(b_{1},\ldots ,b_{m},y_{1},\ldots ,y_{n})=t_{i}(b_{1},\ldots ,b_{m},y_{1},\ldots ,y_{n}).\]
\end{enumerate}
\end{crypt}

One can modify Cryptosystem \ref{4thinuo0t42h08t42} in the same way that we have modified Cryptosystem \ref{t432uhbto42h24ti4e} to 
obtain Cryptosystem \ref{42thumbknvc942dgio} to use the endomorphic Laver tables as a platform.

In \cite{KT}, the authors have constructed a key exchange protocol which could use any LD-system as a platform, and this key exchange could be considered as a distributive version of the Anshel-Anshel Goldfeld Key Exchange \cite{AAG}. We shall now present a slight generalization of the key exchange protocol in \cite{KT} so that one can use any partially endomorphic algebra as a platform instead of an LD-system.

\begin{crypt}
In this key exchange, a partially endomorphic algebra $(X,E,F)$ and subalgebras $A=\langle x_{1},\ldots ,x_{m}\rangle,B=\langle y_{1},\ldots ,y_{n}\rangle$ are public.

\begin{enumerate}
\label{4h2udnbfjkdsu0}

\item Alice's private key is an element $a\in A$ and a term $t$ with
$t(x_{1},\ldots,x_{m})=a$ along with some $f\in E,\mathbf{a}=(a_{1},\ldots,a_{n_{f}})\in X^{n_{f}}$.
Bob's private key is some $g\in E,\mathbf{b}=(b_{1},\ldots,b_{n_{g}})\in B^{n_{g}}$ along with terms
$t_{1},\ldots,t_{n_{g}}$ where $t_{i}(y_{1},\ldots,y_{n})=b_{i}$ for $1\leq i\leq n_{g}$.

\item Alice sends $u_{1}=f(\mathbf{a},y_{1}),\ldots,u_{n}=f(\mathbf{a},y_{n}),p_{0}=f(\mathbf{a},a)$ to Bob.
Bob sends the elements $v_{1}=g(\mathbf{b},x_{1}),\ldots,v_{m}=g(\mathbf{b},x_{m})$ to Alice.

Let $K=f(\mathbf{a},g(\mathbf{b},a))$.

\item Alice computes $K$. Alice can compute $K$ since $K=f(\mathbf{a},g(\mathbf{b},a))$, Alice knows $f,\mathbf{a}$ and
\[g(\mathbf{b},a)=g(\mathbf{b},t(x_{1},\ldots,x_{m}))\]
\[=t(g(\mathbf{b},x_{1}),\ldots,g(\mathbf{b},x_{m}))=t(v_{1},\ldots,v_{m})\]
and Alice knows $v_{1},\ldots,v_{m}$ and the function $t$.

\item Bob computes $K$. Bob can compute $K$ as well since
\[K=f(\mathbf{a},g(\mathbf{b},a))\]
\[=g(f(\mathbf{a},b_{1}),\ldots,f(\mathbf{a},b_{n_{g}}),f(\mathbf{a},a))\]
\[=g(f(\mathbf{a},t_{1}(y_{1},\ldots,y_{n})),\ldots,f(\mathbf{a},t_{n_{f}}(y_{1},\ldots,y_{n})),f(\mathbf{a},a))\]
\[=g(t_{1}(f(\mathbf{a},y_{1}),\ldots,f(\mathbf{a},y_{n})),\ldots,t_{n_{f}}(f(\mathbf{a},y_{1}),\ldots,f(\mathbf{a},y_{n})),f(\mathbf{a},a))\]
\[=g(t_{1}(u_{1},\ldots,u_{n}),\ldots,t_{n_{f}}(u_{1},\ldots,u_{n}),p_{0}),\]
and Bob knows $g,t_{1},\ldots,t_{n_{f}},u_{1},\ldots,u_{n},p_{0}$.
\end{enumerate}
No other party can compute $K$. Therefore $K$ is the common key between Bob and Alice.
\end{crypt}

The security of Cryptosystem \ref{4h2udnbfjkdsu0} depends on the difficulty of the following problem.
\begin{cryprob}(multiple endomorphism search problem)

Input: A partially endomorphic algebra $(X,E,F)$ along with $x_{1},\ldots ,x_{n},y_{1},\ldots ,y_{n}\in X$ such that there is an inner endomorphism
$L$ with \[L(x_{1})=y_{1},\ldots ,L(x_{n})=y_{n}.\]

Problem: Find some inner endomorphism $L'$ such that
\[L'(x_{1})=y_{1},\ldots,L'(x_{n})=y_{n}.\]
\end{cryprob}

As with cryptosystems \ref{4thinuo0t42h08t42} and \ref{t432uhbto42h24ti4e}, one can easily modify cryptosystem \ref{4h2udnbfjkdsu0} to obtain an endomorphic Laver table based cryptosystems.

The following proposition may be used to factorize elements in endomorphic Laver tables.
\begin{prop}
\label{4t2huo42th4q2e0}
Suppose that $\mathcal{X}$ is an $n+1$-ary Laver-like algebra. Suppose that
$\mathfrak{j}\in\Diamond(L)$ and $\mathfrak{j}=f^{\sharp}(\mathfrak{l}_{1},\ldots,\mathfrak{l}_{n},\mathfrak{l})$. If
$\mathfrak{j}(\mathbf{a})=\#$, then there is some suffix $\mathbf{b}$ of $\mathbf{a}$ such that
$\mathfrak{l}_{i}(\mathbf{c})=\mathfrak{j}(\mathbf{c}i\mathbf{b})$ for $1\leq i\leq n$ and for all strings $\mathbf{c}$.
\end{prop}

Even though Proposition \ref{4t2huo42th4q2e0} allows one to factorize in endomorphic Laver tables, Proposition \ref{4t2huo42th4q2e0} does not easily break endomorphic Laver table based cryptosystems since the enemy will only know a handful of specific values of elements $\mathfrak{j}\in\Diamond(\mathcal{X})$. 

We remark that the braid based cryptosystems are the closest cryptosystems to functional endomorphic Laver table based cryptosystems, but braid based cryptography is now considered to be insecure. We believe that if there is an insecurity with functional endomorphic Laver table based cryptosystems, then such an insecurity will come from a heuristic attack rather than a mathematically proven break simply due to the relentless complexity of functional endomorphic Laver tables.

The functional endomorphic Laver table based cryptosystems are highly interactive since Alice and Bob will probably have to exchange information several times before establishing a common key. This interactivity may be considered to be an efficiency weakness since such a key exchange will require many interactions among possibly very distant parties. An adversary could also use the time the information is inquired in order to attack the cryptosystems. However, one could modify the functional endomorphic based cryptosystems so that Alice and Bob make up their secret keys as they interact with each other. Since Alice and Bob could use their interactions to further develop their secret functions, Alice and Bob can use this information to increase the security and efficiency of these cryptosystems.

\section{Conclusion and further research}

\subsection{Philosophy on generalizations of Laver tables}
Before we have started investigating generalizations of Laver tables, most of the interesting results
about the classical Laver tables have been proven in the early to mid 1990's. In fact, no research on Laver tables has been published from around 1997 to around 2012 when Matthew Smedberg published \cite{MS} and when Patrick Dehornoy and Victoria Lebed calculated groups of cocycles in Laver tables in \cite{DL},\cite{VLC}. However, the calculations of the cocycle groups of the classical Laver tables follow from the fact that the classical Laver tables are retractive than from the more non-trivial properties of the classical Laver tables. On the other hand, the theory of the generalizations of Laver tables provides many more opportunities for investigation than we have simply by using the classical Laver tables. For instance, the multigenic Laver tables $(A^{\leq 2^{n}})^{+}$ have much more combinatorial complexity to study than the classical Laver tables have.

Although little has been published recently on Laver tables, the Laver tables and their generalizations have very noteworthy properties and they are unlike any other mathematical structure for both mathematical and philosophical reasons. On one hand, these algebraic structures are usually either finite or recursive and they are structures that mathematicians would naturally study for purely algebraic reasons. One could even argue that self-distributivity is fundamental to our understanding of algebraic structures and that the bulk of our understanding and interest in self-distributivity arises from knots and braids and from Laver-like LD-systems. On the other hand, some of these structures arise from rank-into-rank cardinals, and we have proven results about the generalizations of Laver tables which have no known proof that does not involve large cardinal axioms. These algebras are also the only objects which occur in modern set theory that have been studied through extensive computer calculations and which produce intricate fractal images. The classical Laver tables also produce functions which if total are known grow faster than any primitive recursive function and which possibly eventually dominate every function that ZFC proves recursive. Finally, the endomorphic Laver tables may be applicable to public key cryptography (though it is too early to judge the security of these cryptosystems) and give the first practical application of large cardinals.

It should not be too much of a surprise that large cardinals axioms can be used to prove elegant combinatorial theorems
that cannot be proven otherwise. If $P$ is a large cardinal axiom, then
$\mathrm{Con}(ZFC+P)$ can be thought of as a rather complicated finite combinatorial statement that cannot be proven in
$ZFC$ by Godel's second incompleteness theorem but which can be proven from stronger large cardinal axioms. Furthermore, the large cardinal axioms are natural statements that by their definitions encompass ideas in logic, measure theory, Ramsey theory, combinatorics, category theory, and other areas of mathematics. Therefore since the large cardinal axioms are elegant statements that already prove finitary combinatorial statements independent of $ZFC$ (such as $\mathrm{Con}(ZFC)$), it is plausible that large cardinal axioms would prove elegant finitary combinatorial statements (such as the fact that the finite self-distributive algebras do not satisfy any more identities than self-distributive algebras in general do).

Harvey Friedman has done much work on formulating combinatorial statements about
finite or countable objects that require large cardinal to prove including \cite{HF97},\cite{HF06},\cite{HF98},\cite{HF11}. 
These statements by Harvey Friedman currently do not have any relation with self-distributive algebras and these statements
about finite or countable objects use various levels of the large cardinal hierarchy. There is currently some debate as to whether these statements by Harvey Friedman about finite or countable objects should truly be considered natural statements of independent mathematical interest outside set theory. Harvey Friedman stated ``We believe that Concrete Mathematical Incompleteness -
where large cardinals are shown to be sufficient, and
weaker large cardinals are shown to be insufficient - will
ultimately become commonplace" in \cite{HF11}[p. 62] (here Friedman is referring to results about finite structures proven using large cardinals). It seems like in the near future mathematicians will prove using large cardinal hypotheses many statements about finite or recursive generalizations of Laver tables that cannot be proven in ZFC (though it is currently unclear how one could show that such statements about generalizations of Laver tables are independent of ZFC). We also believe that in the future mathematicians would prove results using the algebras of elementary embeddings in various areas of mathematics (such as algebraic topology) which cannot be established in ZFC. The mathematical community should therefore embrace the algebras of elementary embeddings as an irreplacable technique that goes beyond ZFC that can be used to prove results about finite or countable objects.

Set theorists generally believe that the large cardinal axioms below the Kunen inconsistency are consistent, and the author mostly  agrees with this position. Nevertheless, if there were an inconsistency in the large cardinal hierarchy, then
it seems like such an inconsistency is most likely to be first appear in the very large cardinals (beyond the cardinals in which one has canonical inner models and possibly beyond the huge cardinals). If someone doubts the existence or the consistency of large cardinals, then the author recommends for any such doubter to investigate these large cardinals and to investigate whether such a doubt of the consistency of large cardinals is warranted. We recommend algebras of elementary embeddings as a place to search for possible inconsistencies in the upper levels of the large cardinal hierarchy. The nightmare Laver tables seem like a good avenue to test the consistency of very large cardinals since if a nightmare Laver table appears in an algebra of elementary embeddings, then the large cardinal axiom that allows one to construct that nightmare Laver table in an algebra of elementary embeddings is inconsistent.

However, if after time the study of algebras of elementary embeddings continue to prove many results about finite or countable combinatorial objects that have no proof or much more difficult proofs otherwise, then these results suggest that these large cardinals are consistent and that they actually exist. A proof that such combinatorial statements have large cardinal consistency strength will further support the consistency and even the existence of these large cardinals.

\subsection{Open problems}
In this paper, we have only scratched the surface of the investigations on generalizations of Laver tables, so there are many questions not answered in this paper in which the author simply has not had the time to investigate. In this subsection, we shall list a collection of problems which we believe are tractible or important to the theory of generalizations of Laver tables.

\begin{prob}
$A_{48}$ is currently the largest classical Laver table ever computed in the sense that in $A_{48}$, $x*y$ can be computed
quickly from pre-computed data. $A_{48}$ was originally computed in the 1990's by Randall Dougherty \cite{RD95}. We believe that it is feasible to compute $A_{96}$ with today's technology. Once one is able to compute $A_{96}$, one should be able to compute entries in the final matrices $FM_{n}^{-}(x,y)$ for $n\leq 96$ and $x,y\in\{1,\ldots ,2^{n}\}$.

The pre-computed data for $A_{96}$ shall consist of a $2^{32}\times 96$ matrix whose entries are in
$\{1,\ldots ,2^{96}\}$. However, the $2^{32}\times 96$ matrix should be highly compressible. Therefore, after one computes
and compresses the $2^{32}\times 96$ matrix, we expect to for $A_{96}$ to be easily computable by making minor calculations and looking up information from the pre-computed data.
\end{prob}
\begin{prob}$\textbf{Hugeness calibration project:}$ 
We propose a sequence of new large cardinal axioms which shall refine the $n$-hugeness hierarchy. We shall say a cardinal $\kappa$ is $n$-tremendous if there is some $\gamma>\kappa$ and some $f\in\mathbf{EE}_{\gamma}$ with $\mathrm{crit}(f)=\kappa$ and where $\langle f\rangle\simeq A_{n}$. The goal of the Hugeness calibration project is to fit the $n$-tremendous cardinals and similar cardinals among the large cardinal hierarchy and to compare their sizes and consistency strengths.
\end{prob}
\begin{prob}
Assuming large cardinal hypotheses, which finite Laver-like LD-systems $X$ embed into some
$\mathcal{E}_{\lambda}^{V[G]}/\equiv^{\gamma}$ for some cardinal $\lambda$ and ordinal $\gamma<\lambda$ and forcing extension
$V[G]$? Which finite Laver-like LD-systems embed into some $\mathbf{EE}_{\gamma}^{V[G]}$?
\end{prob}
\begin{prob}
Our techniques for computing new multigenic Laver tables and new finite Laver-like LD-systems are still quite limited.
Suppose that $M$ is a finite multigenic Laver table. Then is there a much more efficient algorithm than Algorithm \ref{t42h890gt4bh} that enumerates all multigenic Laver tables $M'$ that cover $M$? We also need to refine our techniques for finding one-point extensions of finite permutative LD-systems since one-point extensions give us another method of producing new finite permutative LD-systems from old ones.
We would like to also obtain recursive constructions of new infinite families of finite Laver-like LD-systems with intricate combinatorial structure like the classical Laver tables.
\end{prob}
\begin{prob}
Suppose that $M$ is a multigenic Laver table. Then does there necessarily exist a multigenic Laver table $M'$ that covers $M$.
\end{prob}
\begin{prob}
Are the permutative LD-systems precisely the locally Laver-like LD-systems? We conjecture that there exists permutative LD-systems
which are not locally Laver-like.
\end{prob}
\begin{prob}
Very little is known about the free endomorphic algebras. Assume that for all $n$, there exists an $n$-huge cardinal. Do the free (partially, twistedly) endomorphic algebras embed into inverse limits of (partially, twistedly) endomorphic Laver tables? Are there any other natural examples of free (partially, twistedly) endomorphic algebras? Is the word problem for these free (partially, twistedly) endomorphic algebras solvable?

In the same way that the braid groups produce free left-distributive algebras, one can generalize the braid groups
to groups which we shall call the hyperbraid groups which produce endomorphic algebras.
Let $G$ be the group with presentation given by the generators $(\tau_{n})_{n\geq 1}$ and relations
$\tau_{n}\tau_{n+2}\tau_{n+2}\tau_{n}=\tau_{n+2}\tau_{n}\tau_{n+1}\tau_{n+2}$ and
$\tau_{i}\tau_{j}=\tau_{j}\tau_{i}$ whenever $|i-j|>2$. Let $sh:G\rightarrow G$ be the homomorphism where
$\mathrm{sh}(\tau_{i})=\tau_{i+1}$. Define a ternary operation $t$ on $G$ by letting
$t(x,y,z)=x\cdot\mathrm{sh}(y)\cdot\mathrm{sh}^{2}(z)\tau_{1}\cdot\mathrm{sh}(x\cdot\mathrm{sh}(y))^{-1}$. Then $t$ is a ternary self-distributive operation. This self-distributive algebra $(G,t)$ can be generalized in a straightforward way to other signatures of 
endomorphic algebras. We conjecture that each element in $(G,t)$ generates a free endomorphic algebra.
It is unclear how much of the theory of braid groups extends to the theory of hyperbraid groups.
\end{prob}

\subsection{Unexplained phenomena in the final matrices}
The Sierpinski triangle structure is the most prominent feature of the final matrices. However, we have not been able to prove that the set of all pairs $(x,y)$ where $FM_{n}^{-}(x,y)>0$ is a subset of the Sierpinski triangle. The following conjecture states that the set of all pairs $(x,y)$ where $FM_{n}^{-}(x,y)>0$ is truly a subset of the Sierpinski triangle. This conjecture is the most important conjecture regarding multigenic Laver tables since all of the combinatorial complexity of the multigenic Laver tables of the form $(A^{\leq 2^{n}})^{+}$ is contained in the corresponding final matrices.

\begin{defn}
Let $\mathrm{ST}_{0}=(1,1)$. Let $(x,y)\in \mathrm{ST}_{n+1}$ if and only if $((x)_{2^{n}},(y)_{2^{n}})\in \mathrm{ST}_{n}$ and either $x\leq 2^{n}$ or $y>2^{n}$.
\end{defn}
\begin{conj}
Suppose $n\geq 0$ and $\ell\leq x*y$ in $A_{n}$. Then
\begin{enumerate}
\item if $M_{n}(x,y,\ell)>0$, then $(x*M_{n}(x,y,\ell),\ell)\in \mathrm{ST}_{n}$, and

\item if $M_{n}(x,y,\ell)<0$, then $(-M_{n}(x,y,\ell),\ell)\in \mathrm{ST}_{n}$.
\end{enumerate}
\end{conj}

Let $A=\{0, 1, 2, 3, 4, 5, 7, 8, 9, 15, 16, 17, 31, 32, 47, 48\}$. Then $A$ is the set of all
natural numbers $k$ with $k\leq 48$ and such that $|FM_{k}^{-}(2^{x},2^{y})|$ is a power of $2$ whenever
$0\leq x\leq k$ and $0\leq y\leq k.$

\begin{conj}
Suppose that $k=2^{n}$. Then $|FM_{k}^{-}(2^{x},2^{y})|$ is a power of $2$ whenever
$0\leq x\leq k$ and $0\leq y\leq k.$
\end{conj}

Let $k=16$. The $(x,y)$-th entry in the following diagram is $\log_{2}(FM_{k}^{-}(2^{x},2^{y}))$ whenever $FM_{k}^{-}(2^{x},2^{y})$ is positive, and $\log_{2}(-FM_{k}^{-}(2^{x},2^{y}))$ whenever $FM_{k}^{-}(2^{x},2^{y})$ is negative. If $FM_{k}(2^{x},2^{y})=1$, then the $(x,y)$-th entry is $+0$ and if $FM_{k}(2^{x},2^{y})=-1$, then the $(x,y)$-th entry is $-0$.

$\begin{array}{r|rrrrrrrrrrrrrrrrr}
\small
*&0&1&2&3&4&5&6&7&8&9&10&11&12&13&14&15&16\\

\hline

0&-0&+0&-0&+0&-0&+0&-0&+0&-0&+0&-0&+0&-0&+0&-0&+0&4\\
1&-0&-1&-0&-1&-0&-1&-0&-1&-0&-1&-0&-1&-0&-1&-0&-1&4\\
2&-0&-1&-2&2&-0&-1&-2&2&-0&-1&-2&2&-0&-1&-2&2&4\\
3&-0&-1&-2&-3&-0&-1&-2&-3&-0&-1&-2&-3&-0&-1&-2&-3&4\\
4&-0&-1&-2&-3&-4&4&-4&4&-0&-1&-2&-3&-4&4&-4&4&7\\
5&-0&-1&-2&-3&-4&-5&-4&-5&-0&-1&-2&-3&-4&-5&-4&-5&8\\
6&-0&-1&-2&-3&-4&-5&-6&6&-0&-1&-2&-3&-4&-5&-6&6&8\\
7&-0&-1&-2&-3&-4&-5&-6&-7&-0&-1&-2&-3&-4&-5&-6&-7&8\\
8&-0&-1&-2&-3&-4&-5&-6&-7&-8&8&-8&8&-8&8&-8&8&11\\
9&-0&-1&-2&-3&-4&-5&-6&-7&-8&-9&-8&-9&-8&-9&-8&-9&12\\
10&-0&-1&-2&-3&-4&-5&-6&-7&-8&-9&-10&10&-8&-9&-10&10&12\\
11&-0&-1&-2&-3&-4&-5&-6&-7&-8&-9&-10&-11&-8&-9&-10&-11&12\\
12&-0&-1&-2&-3&-4&-5&-6&-7&-8&-9&-10&-11&-12&12&-12&12&14\\
13&-0&-1&-2&-3&-4&-5&-6&-7&-8&-9&-10&-11&-12&-13&-12&-13&14\\
14&-0&-1&-2&-3&-4&-5&-6&-7&-8&-9&-10&-11&-12&-13&-14&14&15\\
15&-0&-1&-2&-3&-4&-5&-6&-7&-8&-9&-10&-11&-12&-13&-14&-15&15\\
16&+0&1&2&3&4&5&6&7&8&9&10&11&12&13&14&15&16
\end{array}$

The following conjecture has been experimentally verified for $n\leq 5.$

\begin{conj}$(SE_{n})$

Suppose that $i,j\in\{0,\ldots,2^{n}-1\}$. Then 
\begin{enumerate}
\item $FM_{2^{n}}^{-}(2^{i},2^{j})>0$ if and only if $i$ is even, $j$ is odd, and $(i+1,j)\in \mathrm{ST}_{n}$.

\item If $FM_{2^{n}}^{-}(2^{i},2^{j})>0$, then $FM_{2^{n}}^{-}(2^{i},2^{j})=2^{i}.$

\item Suppose that $FM_{2^{n}}^{-}(2^{i},2^{j})<0$. Then $FM_{2^{n}}^{-}(2^{i},2^{j})=-2^{k}$ for some $k$.
Furthermore,
\begin{enumerate}
\item if $j$ is odd, then $0<k\leq i$, and
\[FM_{2^{n}}^{-}(2^{i'},2^{j})=-2^{k}\]
for $k\leq i'\leq i$, and
\[FM_{2^{n}}^{-}(2^{k-1},2^{j})=2^{k-1}.\]

\item if $j$ is even and $FM_{2^{n}}^{-}(2^{i},2^{j+1})>0$, then
\[FM_{2^{n}}^{-}(2^{i},2^{j})=-FM_{2^{n}}^{-}(2^{i},2^{j+1}).\]

\item if $j$ is even and $FM_{2^{n}}^{-}(2^{i},2^{j+1})<0$, then
\[2\cdot FM_{2^{n}}^{-}(2^{i},2^{j})=FM_{2^{n}}^{-}(2^{i},2^{j+1}).\]
\end{enumerate}
\end{enumerate}
\end{conj}

\begin{prop}
Suppose that there exists a rank-into-rank cardinal. Then for all $x,y\geq 1$ with $y>x+1$ there exists an $N$ such that
$FM_{n}^{-}(x,y)=-1$ for all $n\geq N$.

\label{t4hui34bndjfw2}
\end{prop}
\begin{proof}
We shall prove this result in the case that $x>1$. The case where $x=1$ is similar.
For extremely large $n$, we have
\[a_{1}\ldots a_{x}*b_{1}b_{2}=a_{1}\ldots a_{x}b_{1}*a_{1}\ldots a_{x}b_{2}\]
\[\succeq a_{1}\ldots a_{x}b_{1}*a_{1}a_{2}=a_{1}\ldots a_{x}b_{1}a_{1}*a_{1}\ldots a_{x}b_{1}a_{2}\]
\[\succeq a_{1}\ldots a_{x}b_{1}a_{1}*a_{1}a_{2}=a_{1}\ldots a_{x}b_{1}a_{1}a_{1}*a_{1}\ldots a_{x}b_{1}a_{1}a_{2}\]
\[\succeq a_{1}\ldots a_{x}b_{1}a_{1}a_{1}*a_{1}a_{2}\]
\[\succeq\ldots\succeq a_{1}\ldots a_{x}b_{1}a_{1}\ldots a_{1}.\]
Thus,
\[(a_{1}\ldots a_{x}*b_{1}b_{2})[y]=a_{1}\ldots a_{x}b_{1}a_{1}\ldots a_{1}[y]=a_{1},\]
so $FM_{n}^{-}(x,y)=-1$.
\end{proof}

\begin{prop}
Suppose that there exists a rank-into-rank cardinal. Then there exists a natural number $n$ where conjecture $SE_{n}$ is false.
\end{prop}
\begin{proof}
Proposition \ref{t4hui34bndjfw2} contradicts the statement $\forall n\in\omega,SE_{n}$.
\end{proof}
We expect that the first natural number $n$ where $SE_{n}$ fails to be an extremely large natural number.

\begin{conj}
$(EP_{n})$. Suppose that $r$ is a natural number and $x<2^{2^{r}}-1$. Then
$FM_{n}^{-}(x,2^{y})=FM_{n}^{-}(x,2^{z})$ whenever $y=z\mod 2^{r}$ and $\max(y,z)<j\cdot 2^{r}<n$ for some $j$.
\end{conj}

\begin{prop}
If there exists a rank-into-rank cardinal, then there is a natural number $n$ where the conjecture $EP_{n}$ is false.
\end{prop}
\begin{proof}
Proposition \ref{t4hui34bndjfw2} contradicts the statement $\forall n\in\omega,SE_{n}$.
\end{proof}

\section{Appendix: other computed data}
In this section, we shall show some computed data on multigenic Laver tables.
\footnote{At the author's website, one can find much more computed data on classical, multigenic, and endomorphic Laver tables along with the code and algorithms that allow one to compute such data}.

\begin{comp}
The $n$-th entry in the following 48 element list is the sample probability (with sample size 100000) that
$FM_{n}^{-}(x,y)>0$ given that $x,y\in\{1,\ldots ,2^{n}\}$ and $((x+1)_{2^{n}},y)\in \mathrm{ST}_{n}$. The list starts with the
1st element instead of the 0th element. For readability, we shall group the 48 elements in 6 sequences of length 8.

\hfill\break
(1.0000, 1.00000, 1.00000, 0.96394, 0.93958, 0.87997, 0.82056, 0.74388; 

0.68904, 0.62298, 0.56396, 0.50157, 0.45788, 0.40679, 0.36146, 0.31913;
	
0.28988, 0.26055, 0.23090, 0.20321, 0.18422, 0.16128, 0.14627, 0.12733; 
	
0.11542, 0.10156, 0.09137, 0.08083, 0.07270, 0.06393, 0.05771, 0.05028;
	
0.04613, 0.03936, 0.03665, 0.03203, 0.02917, 0.02506, 0.02264, 0.02026;

0.01854, 0.01604, 0.01355, 0.01242, 0.01097, 0.01023, 0.00843, 0.00790)
\end{comp}
We informally estimate for the fractal dimension of $\{(x,y):FM_{n}^{-}(x,y)>0\}\subseteq\{1,\ldots ,2^{n}\}^{2}$ to be 1.42. 

\begin{comp}
The following function $F$ is the function where $F(x)$ is the number of multigenic Laver tables
with alphabet $0,1$ of cardinality $x$. The domain of the function $F$ is the set of all natural numbers $x$ from
$1$ to $128$ such that there exists a multigenic Laver table with alphabet $0,1$ of cardinality $x$.
Extensions of the domain of this function beyond 128 have not been computed.

\hfill\break
F=\{ ( 2, 1 ), ( 4, 2 ), ( 6, 1 ), ( 8, 4 ), ( 12, 2 ), ( 16, 8 ), ( 18, 2 ), ( 24, 4 ), ( 30, 1 ), ( 32, 16 ), ( 36, 2 ), ( 48, 12 ), ( 54, 4 ), ( 64, 36 ), ( 72, 6 ), ( 90, 2 ), ( 96, 38 ), ( 108, 2 ), ( 120, 4 ), ( 128, 102 ) \}

\hfill\break
The function $F^{+}$ is the function where $\mathrm{Dom}(F^{+})=\mathrm{Dom}(F)$ and $F^{+}(x)$ is the number of multigenic Laver tables
with alphabet $\{0,1\}$ of cardinality at most $x$.

\hfill\break
$F^{+}$={ ( 2, 1 ), ( 4, 3 ), ( 6, 4 ), ( 8, 8 ), ( 12, 10 ), ( 16, 18 ), ( 18, 20 ), ( 24, 24 ), ( 30, 25 ), ( 32, 41 ),
  ( 36, 43 ), ( 48, 55 ), ( 54, 59 ), ( 64, 95 ), ( 72, 101 ), ( 90, 103 ), ( 96, 141 ), ( 108, 143 ), ( 120, 147 ),
  ( 128, 249 ) }

\hfill\break
24 out of these 249 multigenic Laver tables with alphabet $\{0,1\}$ are nightmare Laver tables.
There are 2 nightmare Laver tables over 0,1 of cardinality 64. There are 6 nightmare Laver tables over 0,1 of cardinality 96.
There are 16 nightmare Laver tables over 0,1 of cardinality 128.
\end{comp}
\begin{comp}
The function $G$ is the function where $G(x)$ is the number of multigenic Laver tables with a three element alphabet $0,1,2$
of cardinality $x$. The domain of the function $G$ is the set of all natural numbers $x\leq 108$ such that there exists a
multigenic Laver table on a three element alphabet of cardinality $x$. Extensions of the domain of this function past
108 have not been computed.

\hfill\break
G=\{ ( 3, 1 ), ( 6, 3 ), ( 9, 3 ), ( 12, 10 ), ( 18, 12 ), ( 24, 30 ), ( 27, 6 ), ( 36, 45 ), ( 45, 3 ), ( 48, 93 ), ( 54, 27 ), ( 72, 153 ), ( 81, 12 ), ( 84, 3 ), ( 90, 15 ), ( 96, 294 ), ( 108, 123 ) \}
\end{comp}
\begin{comp}
$$\{|M((i,j),A_{5})\mid :i,j\in\{1,\ldots ,32\}\}$$
$=\{$2, 4, 6, 8, 12, 16, 18, 24, 30, 32, 36, 48, 54, 64, 72, 90, 96, 108, 120, 144, 162, 192, 216, 240, 252, 270, 288,
  336, 360, 432, 450, 486, 504, 510, 576, 720, 756, 810, 864, 936, 1080, 1152, 1350, 1440, 1890, 1920, 2016, 2250,
  2400, 2430, 2520, 2688, 2970, 3360, 3528, 4050, 5712, 5850, 7650, 9576, 12096, 12600, 15120, 17010, 17550, 19710,
  20250, 36450, 41184, 45864, 65790, 86112, 131070, 157950, 313470, 8486910, 9003150, 9143550, 11372670, 11530350,
  22883310, 539017470, 547438590, 4295032830, 8589934590$\}$
\end{comp}
\begin{comp}
\[\{\frac{M((i,j),A_{6})}{M(((i)_{32},(j)_{32}),A_{5})}\mid i,j\in\{1,\ldots ,64\}\}=\]

\{1, 2, 3, 4, 5, 6, 7, 8, 9, 10, 11, 12, 13, 14, 15, 16, 17, 19, 20, 21, 23, 24, 25, 26, 27, 30, 31, 37, 39, 41, 44,
  45, 49, 51, 65, 73, 76, 79, 91, 92, 100, 108, 121, 129, 153, 172, 201, 225, 257, 289, 309, 324, 441, 513, 521, 529,
  577, 585, 617, 665, 785, 828, 969, 1161, 1289, 1473, 2761, 4161, 8193, 8321, 16385, 32769, 65537, 262145, 262153,
  266241, 270401, 295425, 296009, 300105, 300681, 532481, 594433, 595017, 1189449, 16777217, 16777345, 17039361,
  17571841, 1073741825, 1073774593, 2147483649, 4294967297.\}
\end{comp}
\begin{comp}
The following set is the set of all multigenic Laver tables over the alphabet $\{0,1\}$ of cardinality at most 16.

$\{$ $\{$ 0, 1 $\}$, $\{$ 0, 1, 00, 01 $\}$, $\{$ 1, 0, 10, 11 $\}$, $\{$ 0, 1, 00, 01, 10, 11 $\}$,

$\{$ 0, 00, 000, 1, 01, 001, 0000, 0001 $\}$, $\{$ 0, 1, 10, 00, 01, 11, 100, 101 $\}$,

$\{$ 1, 11, 111, 0, 10, 110, 1110, 1111 $\}$, $\{$ 1, 0, 01, 10, 11, 00, 010, 011 $\}$,
	
$\{$ 0, 1, 00, 10, 000, 01, 11, 001, 100, 101, 0000, 0001 $\}$, $\{$ 1, 0, 11, 01, 111, 10, 00, 010, 011, 110, 1110, 1111 $\}$,
			
$\{$ 0, 1, 00, 10, 000, 100, 1000, 01, 11, 001, 101, 0000, 0001, 1001, 10000, 10001 $\}$, 
			
$\{$ 0, 1, 10, 00, 000, 001, 0010, 01, 11, 100, 101, 0000, 0001, 0011, 00100, 00101 $\}$,

$\{$ 0, 1, 10, 11, 110, 111, 1110, 00, 01, 100, 101, 1100, 1101, 1111, 11100, 11101 $\}$,

$\{$ 1, 0, 01, 00, 001, 000, 0001, 10, 11, 010, 011, 0010, 0011, 0000, 00010, 00011 $\}$,

$\{$ 1, 0, 01, 11, 111, 110, 1101, 00, 10, 010, 011, 1110, 1111, 1100, 11010, 11011 $\}$,

$\{$ 1, 0, 11, 01, 111, 011, 0111, 10, 00, 110, 010, 1110, 1111, 0110, 01110, 01111 $\}$,

$\{$ 0, 00, 000, 0000, 00000, 000000, 0000000, 1, 01, 001, 0001, 00001, 000001, 0000001, 00000000, 00000001 $\}$,

$\{$ 1, 11, 111, 1111, 11111, 111111, 1111111, 0, 10, 110, 1110, 11110, 111110, 1111110, 11111110, 11111111 $\}$ $\}$
\end{comp}
\begin{comp}
The $n$-th position in the following sequence denotes the number of multigenic Laver tables $M$ whose alphabet
is of the form $\{1,\ldots,r\}$ and such that $|M|=n$. To improve readability, blocks of 10 elements are grouped together by semicolons.
Only the first 80 terms of this sequence are given. As expected, due to Lagrange's theorem for multigenic Laver tables, the count of multigenic Laver tables of highly composite cardinalities is much higher than the count of multigenic Laver tables whose cardinalities have few prime factors. 

\hfill\break
( 1, 2, 1, 4, 1, 5, 1, 10, 4, 6; 1, 25, 1, 8, 11, 38, 1, 39, 1, 47; 22, 12, 1, 167, 6, 14, 43, 99, 1, 158; 1, 238, 56, 18,
  36, 471, 1, 20, 79, 510; 1, 387, 1, 309, 285, 24, 1, 1650, 8, 266; 137, 482, 1, 1034, 331, 1346, 172, 30, 1, 3171; 1, 32, 547, 2486, 716, 1827, 1, 1004, 254, 1597; 1, 7912, 1, 38, 1736,
  1369, 463, 3524, 1, 8474 )
\end{comp}
\begin{comp}
For $n\in\{0,\ldots ,13\}$, the function $f_{n}$ is the function with domain
$\{1,\ldots ,2^{n}\}$ such that $f_{n}(x)=|\{(r,s)\in\{1,\ldots ,2^{n}\}^{2}\mid FM_{n}^{-}(r,s)=x\}|$ for $x\in\{1,\ldots ,2^{n}\}$.
Here $\chi_{R}$ denotes the characteristic function for $R$. More specifically, $\chi_{R}$ has domain $R$ and $\chi_{R}(x)=1$ for all $x\in R$.
\hfill\break

$f_{0}=\bigcup\{\chi_{\{1\}}\cdot 1\}$\newline

$f_{1}=\bigcup\{\chi_{\{2\}}\cdot 1,\chi_{\{1\}}\cdot 2\}$\newline

$f_{2}=\bigcup\{\chi_{\{3,4\}}\cdot 1,\chi_{\{2\}}\cdot 3,\chi_{\{1\}}\cdot 4\}$\newline

$f_{3}=\bigcup\{\chi_{\{5,\ldots ,8\}}\cdot 1,\chi_{\{3,4\}}\cdot 4,\chi_{\{2\}}\cdot 7,\chi_{\{1\}}\cdot 8\}$\newline

$f_{4}=\bigcup
\{\chi_{\{9,\ldots ,16\}}\cdot 1,\chi_{\{5,\ldots ,8\}}\cdot 3,\chi_{\{3,4\}}\cdot 11,\chi_{\{2\}}\cdot 15,\chi_{\{1\}}\cdot 21\}$\newline

$f_{5}=\bigcup\{\chi_{\{17,\ldots,32\}}\cdot 1,\chi_{\{9,\ldots,16\}}\cdot 4,\chi_{\{6,\ldots,8\}}\cdot 10,\chi_{\{5\}}\cdot 15,\chi_{\{3,4\}}\cdot 26,\chi_{\{2\}}\cdot 31,\chi_{\{1\}}\cdot 52\}$\newline

$f_{6}=\bigcup\{\chi_{\{33,\ldots,64\}}\cdot 1,\chi_{\{17,\ldots,32\}}\cdot 3,\chi_{\{9,\ldots,16\}}\cdot 12,\chi_{\{6,\ldots,8\}}\cdot 24,\chi_{\{5\}}\cdot 39,\chi_{\{4\}}\cdot 62,\chi_{\{3\}}\cdot 67,\chi_{\{2\}}\cdot 72,\chi_{\{1\}}\cdot 153\}$\newline

$f_{7}=\bigcup\{\chi_{\{65,\ldots,128\}}\cdot 1,\chi_{\{33,\ldots,64\}}\cdot 4,\chi_{\{17,\ldots,32\}}\cdot 7,\chi_{\{12,14,15,16\}}\cdot 31,\chi_{\{10,11,13\}}\cdot 33,\chi_{\{9\}}\cdot 39,\chi_{\{6,\ldots,8\}}\cdot 53,\chi_{\{5\}}\cdot 88,
\chi_{\{4\}}\cdot 149,\chi_{\{3\}}\cdot 182,\chi_{\{2\}}\cdot 185,\chi_{\{1\}}\cdot 463\}$\newline

$f_{8}=\bigcup\{\chi_{\{129,\ldots,256\}}\cdot 1,\chi_{\{65,\ldots,128\}}\cdot 3,\chi_{\{33,\ldots,64\}}\cdot 11,\chi_{\{17,\ldots,32\}}\cdot 15,\chi_{\{12,14,15,16\}}\cdot 77,\chi_{\{10,11,13\}}\cdot 83,\chi_{\{9\}}\cdot 101,\chi_{\{8\}}\cdot 116, 
\chi_{\{6,7\}}\cdot 118,\chi_{\{5\}}\cdot 197,\chi_{\{4\}}\cdot 373,\chi_{\{2\}}\cdot 503,\chi_{\{3\}}\cdot 505,\chi_{\{1\}}\cdot 1381\}$\newline

$f_{9}=\bigcup\{\chi_{\{257,\ldots,512\}}\cdot 1,\chi_{\{129,\ldots,256\}}\cdot 4,\chi_{\{81,\ldots,128\}}\cdot 10,\chi_{\{65,\ldots,80\}}\cdot 15,\chi_{\{33,\ldots,64\}}\cdot 26,\chi_{\{17,\ldots,32\}}\cdot 32,\chi_{\{12,14,15,16\}}\cdot 181,\chi_{\{13\}}\cdot 204,
\chi_{\{10,11\}}\cdot 221,\chi_{\{8\}}\cdot 251,\chi_{\{6,7\}}\cdot 259,\chi_{\{9\}}\cdot 339,\chi_{\{5\}}\cdot 449,\chi_{\{4\}}\cdot 943,\chi_{\{3\}}\cdot 1388,\chi_{\{2\}}\cdot 1409,\chi_{\{1\}}\cdot 4064\}$\newline

$f_{10}=\bigcup\{\chi_{\{513,\ldots,1024\}}\cdot 1,\chi_{\{257,\ldots,512\}}\cdot 3,\chi_{\{129,\ldots,256\}}\cdot 12,\chi_{\{81,\ldots,128\}}\cdot 24,\chi_{\{65,\ldots,80\}}\cdot 39,\chi_{\{49,\ldots,64\}}\cdot 63,\chi_{\{33,\ldots,48\}}\cdot 68,\chi_{\{17,\ldots,32\}}\cdot 75,
\chi_{\{12,14,15,16\}}\cdot 432,\chi_{\{13\}}\cdot 489,\chi_{\{10,11\}}\cdot 540,\chi_{\{8\}}\cdot 560,\chi_{\{6\}}\cdot 591,\chi_{\{7\}}\cdot 592,\chi_{\{9\}}\cdot 858,\chi_{\{5\}}\cdot 1027,\chi_{\{4\}}\cdot 2417,\chi_{\{3\}}\cdot 3824, 
\chi_{\{2\}}\cdot 3958,\chi_{\{1\}}\cdot 11746\}$\newline

$f_{11}=\bigcup\{\chi_{\{1025,\ldots,2048\}}\cdot 1,\chi_{\{513,\ldots,1024\}}\cdot 4,\chi_{\{257,\ldots,512\}}\cdot 7, 
\chi_{\{177,\ldots,192,209,\ldots,256\}}\cdot 31,
\chi_{\{145,\ldots,176,193,\ldots,208\}}\cdot 33,\chi_{\{129,\ldots,144\}}\cdot 39,\chi_{\{81,\ldots,128\}}\cdot 54,\chi_{\{65,\ldots,80\}}\cdot 89,\chi_{\{49,\ldots,64\}}\cdot 153, 
\chi_{\{33,\ldots,48\}}\cdot 186,\chi_{\{17,\ldots,32\}}\cdot 192,\chi_{\{12,14,15,16\}}\cdot 1048,\chi_{\{13\}}\cdot 1177,\chi_{\{8\}}\cdot 1287,\chi_{\{10,11\}}\cdot 1296,\chi_{\{6\}}\cdot 1391,\chi_{\{7\}}\cdot 1394,\chi_{\{9\}}\cdot 2026, 
\chi_{\{5\}}\cdot 2377,\chi_{\{4\}}\cdot 6312,\chi_{\{3\}}\cdot 10615,\chi_{\{2\}}\cdot 11180,\chi_{\{1\}}\cdot 33750\}$\newline

$f_{12}=\bigcup\{\chi_{\{2049,\ldots,4096\}}\cdot 1,\chi_{\{1025,\ldots,2048\}}\cdot 3,\chi_{\{513,\ldots,1024\}}\cdot 11,\chi_{\{257,\ldots,512\}}\cdot 15, 
\chi_{\{177,\ldots,192,209,\ldots,256\}}\cdot 78, 
\chi_{\{145,\ldots,176,193,\ldots,208\}}\cdot 84,\chi_{\{129,\ldots,144\}}\cdot 102,\chi_{\{113,\ldots,128\}}\cdot 119,\chi_{\{81,\ldots,112\}}\cdot 121,\chi_{\{65,\ldots,80\}}\cdot 200, 
\chi_{\{49,\ldots,64\}}\cdot 384,\chi_{\{33,\ldots,48\}}\cdot 516,\chi_{\{17,\ldots,32\}}\cdot 518,\chi_{\{12,14,15,16\}}\cdot 2588,\chi_{\{13\}}\cdot 2861,\chi_{\{8\}}\cdot 3044,\chi_{\{10,11\}}\cdot 3116,\chi_{\{6\}}\cdot 3361, 
\chi_{\{7\}}\cdot 3369,\chi_{\{9\}}\cdot 4670,\chi_{\{5\}}\cdot 5599,\chi_{\{4\}}\cdot 16433,\chi_{\{3\}}\cdot 28882,\chi_{\{2\}}\cdot 30821,\chi_{\{1\}}\cdot 93933\}$\newline

$f_{13}=\bigcup\{\chi_{\{4097,\ldots,8192\}}\cdot 1,\chi_{\{2049,\ldots,4096\}}\cdot 4,\chi_{\{1281,\ldots,2048\}}\cdot 10,\chi_{\{1025,\ldots,1280\}}\cdot 15,\chi_{\{513,\ldots,1024\}}\cdot 26,\chi_{\{257,\ldots,512\}}\cdot 31,
\chi_{\{177,\ldots,192,209,\ldots,256\}}\cdot 184, 
\chi_{\{193,\ldots,208\}}\cdot 204,\chi_{\{145,\ldots,176\}}\cdot 209,\chi_{\{113,\ldots,128\}}\cdot 260,\chi_{\{81,\ldots,112\}}\cdot 268,\chi_{\{129,\ldots,144\}}\cdot 279,\chi_{\{65,\ldots,80\}}\cdot 449,\chi_{\{49,\ldots,64\}}\cdot 966, 
\chi_{\{17,\ldots,32\}}\cdot 1450,\chi_{\{33,\ldots,48\}}\cdot 1459,\chi_{\{16\}}\cdot 6580,\chi_{\{12,14\}}\cdot 6612,\chi_{\{15\}}\cdot 6613,\chi_{\{13\}}\cdot 7199,\chi_{\{8\}}\cdot 7493,\chi_{\{10,11\}}\cdot 7736,\chi_{\{6\}}\cdot 8414, 
\chi_{\{7\}}\cdot 8430,\chi_{\{9\}}\cdot 11204,\chi_{\{5\}}\cdot 13562,\chi_{\{4\}}\cdot 43958,\chi_{\{3\}}\cdot 80121,\chi_{\{2\}}\cdot 86449,\chi_{\{1\}}\cdot 264699\}$
\end{comp}
\begin{comp}
The following functions $g_{0},\ldots,g_{13}$ are the functions where for each $n\in\{0,\ldots,13\}$, the domain of
$g_{n}$ is $\{1,\ldots,2^{n}\}$ and where $g_{n}(x)=|\{(r,s)\in\{1,\ldots,2^{n}\}^{2}\mid FM^{+}_{n}(r,s)=x\}|$.
\hfill\break

$g_{0}=\bigcup\{\chi_{\{1\}}\cdot 1\}$\newline
 
$g_{1}=
\bigcup\{\chi_{\{1\}}\cdot 1,\chi_{\{2\}}\cdot 2\}$\newline
 
$g_{2}=
\bigcup\{\chi_{\{1,2\}}\cdot 1,\chi_{\{3\}}\cdot 3,\chi_{\{4\}}\cdot 4\}$\newline
 
$g_{3}=
\bigcup\{\chi_{\{1,\ldots,4\}}\cdot 1,\chi_{\{5,6\}}\cdot 4,\chi_{\{7\}}\cdot 7,\chi_{\{8\}}\cdot 8\}$\newline
 
$g_{4}=
\bigcup\{\chi_{\{1,\ldots,8\}}\cdot 1,\chi_{\{9,\ldots,12\}}\cdot 3,\chi_{\{14\}}\cdot 11,\chi_{\{15\}}\cdot 15,\chi_{\{13,16\}}\cdot 16\}$\newline
 
$g_{5}=
\bigcup\{\chi_{\{1,\ldots,16\}}\cdot 1,\chi_{\{17,\ldots,24\}}\cdot 4,\chi_{\{26,\ldots,28\}}\cdot 10,\chi_{\{25\}}\cdot 15,\chi_{\{30\}}\cdot 26,\chi_{\{31\}}\cdot 31,\chi_{\{32\}}\cdot 32,\chi_{\{29\}}\cdot 46\}$\newline
 
$g_{6}=
\bigcup\{\chi_{\{1,\ldots,32\}}\cdot 1,\chi_{\{33,\ldots,48\}}\cdot 3,\chi_{\{53,\ldots,56\}}\cdot 12,\chi_{\{49,\ldots,52\}}\cdot 17,\chi_{\{60\}}\cdot 24,\chi_{\{58\}}\cdot 28,\chi_{\{59\}}\cdot 29,\chi_{\{57\}}\cdot 53,\chi_{\{62\}}\cdot 57,\chi_{\{63\}}\cdot 63,\chi_{\{64\}}\cdot 64,\chi_{\{61\}}\cdot 127\}$\newline
 
$g_{7}=
\bigcup\{\chi_{\{1,\ldots,64\}}\cdot 1,\chi_{\{65,\ldots,96\}}\cdot 4,\chi_{\{97,\ldots,112\}}\cdot 7,\chi_{\{118,\ldots,120\}}\cdot 31,\chi_{\{117\}}\cdot 33,\chi_{\{124\}}\cdot 53,\chi_{\{116\}}\cdot 60,\chi_{\{114\}}\cdot 70,\chi_{\{115\}}\cdot 71, 
\chi_{\{122\}}\cdot 74,\chi_{\{113\}}\cdot 76,\chi_{\{123\}}\cdot 77,\chi_{\{126\}}\cdot 120,\chi_{\{127\}}\cdot 127,\chi_{\{128\}}\cdot 128,\chi_{\{121\}}\cdot 161,\chi_{\{125\}}\cdot 345\}$\newline
 
$g_{8}=
\bigcup\{\chi_{\{1,\ldots,128\}}\cdot 1,\chi_{\{129,\ldots,192\}}\cdot 3,\chi_{\{209,\ldots,224\}}\cdot 11,\chi_{\{225,\ldots,240\}}\cdot 15,\chi_{\{193,\ldots,208\}}\cdot 16,\chi_{\{248\}}\cdot 72,\chi_{\{246,247\}}\cdot 74,\chi_{\{245\}}\cdot 84, 
\chi_{\{252\}}\cdot 111,\chi_{\{250\}}\cdot 188,\chi_{\{244\}}\cdot 193,\chi_{\{251\}}\cdot 196,\chi_{\{242\}}\cdot 244,\chi_{\{243\}}\cdot 246,\chi_{\{254\}}\cdot 247,\chi_{\{255\}}\cdot 255,\chi_{\{256\}}\cdot 256,\chi_{\{241\}}\cdot 274, 
\chi_{\{249\}}\cdot 461,\chi_{\{253\}}\cdot 914\}$\newline
 
$g_{9}=
\bigcup\{\chi_{\{1,\ldots,256\}}\cdot 1,\chi_{\{257,\ldots,384\}}\cdot 4,\chi_{\{401,\ldots,448\}}\cdot 10,\chi_{\{385,\ldots,400\}}\cdot 15,\chi_{\{465,\ldots,480\}}\cdot 26,\chi_{\{481,\ldots,496\}}\cdot 32,\chi_{\{449,\ldots,464\}}\cdot 46,\chi_{\{504\}}\cdot 161, 
\chi_{\{502,503\}}\cdot 169,\chi_{\{501\}}\cdot 208,\chi_{\{508\}}\cdot 231,\chi_{\{506\}}\cdot 462,\chi_{\{507\}}\cdot 478,\chi_{\{510\}}\cdot 502,\chi_{\{511\}}\cdot 511,\chi_{\{512\}}\cdot 512,\chi_{\{500\}}\cdot 582,\chi_{\{499\}}\cdot 820, 
\chi_{\{498\}}\cdot 848,\chi_{\{497\}}\cdot 1067,\chi_{\{505\}}\cdot 1239,\chi_{\{509\}}\cdot 2452\}$\newline
 
$g_{10}=
\bigcup\{\chi_{\{1,\ldots,512\}}\cdot 1,\chi_{\{513,\ldots,768\}}\cdot 3,\chi_{\{833,\ldots,896\}}\cdot 12,\chi_{\{769,\ldots,832\}}\cdot 17,\chi_{\{945,\ldots,960\}}\cdot 24,\chi_{\{913,\ldots,928\}}\cdot 28,\chi_{\{929,\ldots,944\}}\cdot 29,\chi_{\{897,\ldots,912\}}\cdot 53, 
\chi_{\{977,\ldots,992\}}\cdot 58,\chi_{\{993,\ldots,1008\}}\cdot 66,\chi_{\{961,\ldots,976\}}\cdot 128,\chi_{\{1016\}}\cdot 343,\chi_{\{1014\}}\cdot 374,\chi_{\{1015\}}\cdot 375,\chi_{\{1020\}}\cdot 471,\chi_{\{1013\}}\cdot 495, 
\chi_{\{1022\}}\cdot 1013,\chi_{\{1023\}}\cdot 1023,\chi_{\{1024\}}\cdot 1024,\chi_{\{1018\}}\cdot 1119,\chi_{\{1019\}}\cdot 1153,\chi_{\{1012\}}\cdot 1658,\chi_{\{1011\}}\cdot 2491,\chi_{\{1010\}}\cdot 2649,\chi_{\{1017\}}\cdot 3244,
\chi_{\{1009\}}\cdot 3555,\chi_{\{1021\}}\cdot 6459\}$\newline
 
$g_{11}=
\bigcup\{\chi_{\{1,\ldots,1024\}}\cdot 1,\chi_{\{1025,\ldots,1536\}}\cdot 4,\chi_{\{1537,\ldots,1792\}}\cdot 7,\chi_{\{1873,\ldots,1920\}}\cdot 31,\chi_{\{1857,\ldots,1872\}}\cdot 33,\chi_{\{1969,\ldots,1984\}}\cdot 54,\chi_{\{1841,\ldots,1856\}}\cdot 60, 
\chi_{\{1809,\ldots,1824\}}\cdot 70,\chi_{\{1825,\ldots,1840\}}\cdot 71,\chi_{\{1937,\ldots,1952\}}\cdot 75,\chi_{\{1793,\ldots,1808\}}\cdot 76,\chi_{\{1953,\ldots,1968\}}\cdot 78,\chi_{\{2001,\ldots,2016\}}\cdot 124,\chi_{\{2017,\ldots,2032\}}\cdot 134, 
\chi_{\{1921,\ldots,1936\}}\cdot 162,\chi_{\{1985,\ldots,2000\}}\cdot 349,\chi_{\{2040\}}\cdot 713,\chi_{\{2038\}}\cdot 817,\chi_{\{2039\}}\cdot 820,\chi_{\{2044\}}\cdot 952,\chi_{\{2037\}}\cdot 1164,\chi_{\{2046\}}\cdot 2036, 
\chi_{\{2047\}}\cdot 2047,\chi_{\{2048\}}\cdot 2048,\chi_{\{2042\}}\cdot 2736,\chi_{\{2043\}}\cdot 2806,\chi_{\{2036\}}\cdot 4654,\chi_{\{2035\}}\cdot 7351,\chi_{\{2034\}}\cdot 7975,\chi_{\{2041\}}\cdot 8506,\chi_{\{2033\}}\cdot 11183, 
\chi_{\{2045\}}\cdot 17125\}$\newline
 
$g_{12}=
\bigcup\{\chi_{\{1,\ldots,2048\}}\cdot 1,\chi_{\{2049,\ldots,3072\}}\cdot 3,\chi_{\{3329,\ldots,3584\}}\cdot 11,\chi_{\{3585,\ldots,3840\}}\cdot 15,\chi_{\{3073,\ldots,3328\}}\cdot 16,\chi_{\{3953,\ldots,3968\}}\cdot 73,\chi_{\{3921,\ldots,3952\}}\cdot 75, 
\chi_{\{3905,\ldots,3920\}}\cdot 85,\chi_{\{4017,\ldots,4032\}}\cdot 114,\chi_{\{3985,\ldots,4000\}}\cdot 191,\chi_{\{3889,\ldots,3904\}}\cdot 194,\chi_{\{4001,\ldots,4016\}}\cdot 199,\chi_{\{3857,\ldots,3872\}}\cdot 245,\chi_{\{3873,\ldots,3888\}}\cdot 247, 
\chi_{\{4049,\ldots,4064\}}\cdot 258,\chi_{\{4065,\ldots,4080\}}\cdot 270,\chi_{\{3841,\ldots,3856\}}\cdot 275,\chi_{\{3969,\ldots,3984\}}\cdot 464,\chi_{\{4033,\ldots,4048\}}\cdot 930,\chi_{\{4088\}}\cdot 1458,\chi_{\{4086\}}\cdot 1775, 
\chi_{\{4087\}}\cdot 1783,\chi_{\{4092\}}\cdot 1914,\chi_{\{4085\}}\cdot 2726,\chi_{\{4094\}}\cdot 4083,\chi_{\{4095\}}\cdot 4095,\chi_{\{4096\}}\cdot 4096,\chi_{\{4090\}}\cdot 6688,\chi_{\{4091\}}\cdot 6832,\chi_{\{4084\}}\cdot 12678, 
\chi_{\{4083\}}\cdot 20737,\chi_{\{4089\}}\cdot 22027,\chi_{\{4082\}}\cdot 22808,\chi_{\{4081\}}\cdot 32930,\chi_{\{4093\}}\cdot 44847\}$\newline
 
$g_{13}=
\bigcup\{\chi_{\{1,\ldots,4096\}}\cdot 1,\chi_{\{4097,\ldots,6144\}}\cdot 4,\chi_{\{6401,\ldots,7168\}}\cdot 10,\chi_{\{6145,\ldots,6400\}}\cdot 15,\chi_{\{7425,\ldots,7680\}}\cdot 26,\chi_{\{7681,\ldots,7936\}}\cdot 31,\chi_{\{7169,\ldots,7424\}}\cdot 46, 
\chi_{\{8049,\ldots,8064\}}\cdot 164,\chi_{\{8017,\ldots,8048\}}\cdot 172,\chi_{\{8001,\ldots,8016\}}\cdot 208,\chi_{\{8113,\ldots,8128\}}\cdot 240,\chi_{\{8081,\ldots,8096\}}\cdot 487,\chi_{\{8097,\ldots,8112\}}\cdot 503, 
\chi_{\{8145,\ldots,8160\}}\cdot 528,\chi_{\{8161,\ldots,8176\}}\cdot 543,\chi_{\{7985,\ldots,8000\}}\cdot 582,\chi_{\{7953,\ldots,7968\}}\cdot 829,\chi_{\{7969,\ldots,7984\}}\cdot 837,\chi_{\{7952\}}\cdot 984,\chi_{\{7937,\ldots,7950\}}\cdot 1016, 
\chi_{\{7951\}}\cdot 1017,\chi_{\{8065,\ldots,8080\}}\cdot 1303,\chi_{\{8129,\ldots,8144\}}\cdot 2505,\chi_{\{8184\}}\cdot 2960,\chi_{\{8188\}}\cdot 3841,\chi_{\{8182\}}\cdot 3881,\chi_{\{8183\}}\cdot 3897,\chi_{\{8181\}}\cdot 6467, 
\chi_{\{8190\}}\cdot 8178,\chi_{\{8191\}}\cdot 8191,\chi_{\{8192\}}\cdot 8192,\chi_{\{8186\}}\cdot 16809,\chi_{\{8187\}}\cdot 17097,\chi_{\{8180\}}\cdot 35088,\chi_{\{8185\}}\cdot 58160,\chi_{\{8179\}}\cdot 59119,\chi_{\{8178\}}\cdot 65722, 
\chi_{\{8177\}}\cdot 97065,\chi_{\{8189\}}\cdot 120350\}$\newline
\end{comp}

\begin{comp}
The following table lists all classical Laver table periods for $A_{5}$.

[ [ 2, 12, 14, 16, 18, 28, 30, 32 ], [ 3, 12, 15, 16, 19, 28, 31, 32 ], [ 4, 8, 12, 16, 20, 24, 28, 32 ], [ 5, 6, 7, 8, 13, 14, 15, 16, 21, 22, 23, 24, 29, 30, 31, 32 ],
      [ 6, 24, 30, 32 ], [ 7, 24, 31, 32 ], [ 8, 16, 24, 32 ], [ 9, 10, 11, 12, 13, 14, 15, 16, 25, 26, 27, 28, 29, 30, 31, 32 ], [ 10, 28, 30, 32 ], [ 11, 28, 31, 32 ],
      [ 12, 16, 28, 32 ], [ 13, 14, 15, 16, 29, 30, 31, 32 ], [ 14, 16, 30, 32 ], [ 15, 16, 31, 32 ], [ 16, 32 ], [ 17, 18, 19, 20, 21, 22, 23, 24, 25, 26, 27, 28, 29, 30, 31, 32 ],
      [ 18, 28, 30, 32 ], [ 19, 28, 31, 32 ], [ 20, 24, 28, 32 ], [ 21, 22, 23, 24, 29, 30, 31, 32 ], [ 22, 24, 30, 32 ], [ 23, 24, 31, 32 ], [ 24, 32 ],
      [ 25, 26, 27, 28, 29, 30, 31, 32 ], [ 26, 28, 30, 32 ], [ 27, 28, 31, 32 ], [ 28, 32 ], [ 29, 30, 31, 32 ], [ 30, 32 ], [ 31, 32 ], [ 32 ],[ 1, 2, 3, 4, 5, 6, 7, 8, 9, 10, 11, 12, 13, 14, 15, 16, 17, 18, 19, 20, 21, 22, 23, 24, 25, 26, 27, 28, 29, 30, 31, 32 ] ]
\end{comp}
\begin{comp}
The following table lists all classical Laver table periods for $A_{6}$.
			
  [ [ 2, 12, 14, 48, 50, 60, 62, 64 ], [ 3, 12, 15, 48, 51, 60, 63, 64 ], [ 4, 8, 12, 16, 20, 24, 28, 32, 36, 40, 44, 48, 52, 56, 60, 64 ],
      [ 5, 6, 7, 8, 45, 46, 47, 48, 53, 54, 55, 56, 61, 62, 63, 64 ], [ 6, 56, 62, 64 ], [ 7, 56, 63, 64 ], [ 8, 48, 56, 64 ],
      [ 9, 10, 11, 12, 45, 46, 47, 48, 57, 58, 59, 60, 61, 62, 63, 64 ], [ 10, 60, 62, 64 ], [ 11, 60, 63, 64 ], [ 12, 48, 60, 64 ],
      [ 13, 14, 15, 16, 29, 30, 31, 32, 45, 46, 47, 48, 61, 62, 63, 64 ], [ 14, 48, 62, 64 ], [ 15, 48, 63, 64 ], [ 16, 32, 48, 64 ],
      [ 17, 18, 19, 20, 21, 22, 23, 24, 25, 26, 27, 28, 29, 30, 31, 32, 49, 50, 51, 52, 53, 54, 55, 56, 57, 58, 59, 60, 61, 62, 63, 64 ], [ 18, 28, 30, 32, 50, 60, 62, 64 ],
      [ 19, 28, 31, 32, 51, 60, 63, 64 ], [ 20, 24, 28, 32, 52, 56, 60, 64 ], [ 21, 22, 23, 24, 29, 30, 31, 32, 53, 54, 55, 56, 61, 62, 63, 64 ], [ 22, 56, 62, 64 ],
      [ 23, 56, 63, 64 ], [ 24, 32, 56, 64 ], [ 25, 26, 27, 28, 29, 30, 31, 32, 57, 58, 59, 60, 61, 62, 63, 64 ], [ 26, 60, 62, 64 ], [ 27, 60, 63, 64 ], [ 28, 32, 60, 64 ],
      [ 29, 30, 31, 32, 61, 62, 63, 64 ], [ 30, 32, 62, 64 ], [ 31, 32, 63, 64 ], [ 32, 64 ], [ 33, 34, 35, 36, 37, 38, 39, 40, 41, 42, 43, 44, 45, 46, 47, 48,
          49, 50, 51, 52, 53, 54, 55, 56, 57, 58, 59, 60, 61, 62, 63, 64 ], [ 34, 44, 46, 48, 50, 60, 62, 64 ], [ 35, 44, 47, 48, 51, 60, 63, 64 ], [ 36, 40, 44, 48, 52, 56, 60, 64 ],
      [ 37, 38, 39, 40, 45, 46, 47, 48, 53, 54, 55, 56, 61, 62, 63, 64 ], [ 38, 56, 62, 64 ], [ 39, 56, 63, 64 ], [ 40, 48, 56, 64 ],
      [ 41, 42, 43, 44, 45, 46, 47, 48, 57, 58, 59, 60, 61, 62, 63, 64 ], [ 42, 60, 62, 64 ], [ 43, 60, 63, 64 ], [ 44, 48, 60, 64 ],
      [ 45, 46, 47, 48, 61, 62, 63, 64 ], [ 46, 48, 62, 64 ], [ 47, 48, 63, 64 ], [ 48, 64 ], [ 49, 50, 51, 52, 53, 54, 55, 56, 57, 58, 59, 60, 61, 62, 63, 64 ], [ 50, 60, 62, 64 ],
      [ 51, 60, 63, 64 ], [ 52, 56, 60, 64 ], [ 53, 54, 55, 56, 61, 62, 63, 64 ], [ 54, 56, 62, 64 ], [ 55, 56, 63, 64 ], [ 56, 64 ], [ 57, 58, 59, 60, 61, 62, 63, 64 ],
      [ 58, 60, 62, 64 ], [ 59, 60, 63, 64 ], [ 60, 64 ], [ 61, 62, 63, 64 ], [ 62, 64 ], [ 63, 64 ], [ 64 ],
      [ 1, 2, 3, 4, 5, 6, 7, 8, 9, 10, 11, 12, 13, 14, 15, 16, 17, 18, 19, 20, 21, 22, 23, 24, 25, 26, 27, 28, 29, 30, 31,
          32, 33, 34, 35, 36, 37, 38, 39, 40, 41, 42, 43, 44, 45, 46, 47, 48, 49, 50, 51, 52, 53, 54, 55, 56, 57, 58, 59, 60, 61, 62, 63, 64 ] ]
\end{comp}
\begin{comp}
The following table gives the period data for $(\{1,\ldots ,2^{5}\},\#_{5}).$

[ [ [ 1 ], 17, 18, 23, 24, 27, 28, 31, 32 ], [ [ 2 ], 18, 24, 28, 32 ], [ [ 3 ], 19, 24, 28, 32 ], [ [ 4 ], 20, 24, 28, 32 ], [ [ 5 ], 21, 24, 30, 32 ],
  [ [ 6 ], 22, 24, 30, 32 ], [ [ 7 ], 23, 24, 31, 32 ], [ [ 8 ], 24, 32 ], [ [ 9 ], 25, 26, 27, 28, 29, 30, 31, 32 ], [ [ 10 ], 26, 28, 30, 32 ], [ [ 11 ], 27, 28, 31, 32 ],
  [ [ 12 ], 28, 32 ], [ [ 13 ], 29, 30, 31, 32 ], [ [ 14 ], 30, 32 ], [ [ 15 ], 31, 32 ], [ [ 16 ], 32 ], [ [ 17 ], 9, 10, 15, 16, 27, 28, 31, 32 ],
  [ [ 18 ], 10, 16, 28, 32 ], [ [ 19 ], 11, 16, 28, 32 ], [ [ 20 ], 12, 16, 28, 32 ], [ [ 21 ], 13, 16, 30, 32 ], [ [ 22 ], 14, 16, 30, 32 ], [ [ 23 ], 15, 16, 31, 32 ], [ [ 24 ], 16, 32 ],
  [ [ 25 ], 5, 6, 7, 8, 13, 14, 15, 16, 21, 22, 23, 24, 29, 30, 31, 32 ], [ [ 26 ], 6, 8, 14, 16, 22, 24, 30, 32 ], [ [ 27 ], 7, 8, 15, 16, 23, 24, 31, 32 ],
  [ [ 28 ], 8, 16, 24, 32 ], [ [ 29 ], 3, 4, 7, 8, 11, 12, 15, 16, 19, 20, 23, 24, 27, 28, 31, 32 ], [ [ 30 ], 4, 8, 12, 16, 20, 24, 28, 32 ],
  [ [ 31 ], 2, 4, 6, 8, 10, 12, 14, 16, 18, 20, 22, 24, 26, 28, 30, 32 ], [ [ 32 ], 1, 2, 3, 4,
      5, 6, 7, 8, 9, 10, 11, 12, 13, 14, 15, 16, 17, 18, 19, 20, 21, 22, 23, 24, 25, 26, 27, 28, 29, 30, 31, 32 ] ]

\end{comp}
\begin{comp}
The following table gives the period data for $(\{1,\ldots ,2^{6}\},\#_{6}).$

[ [ [ 1 ], 33, 36, 45, 48, 53, 56, 62, 64 ], [ [ 2 ], 34, 36, 46, 48, 54, 56, 62, 64 ], [ [ 3 ], 35, 36, 47, 48, 55, 56, 63, 64 ], [ [ 4 ], 36, 48, 56, 64 ],
  [ [ 5 ], 37, 48, 56, 64 ], [ [ 6 ], 38, 48, 56, 64 ], [ [ 7 ], 39, 48, 56, 64 ], [ [ 8 ], 40, 48, 56, 64 ], [ [ 9 ], 41, 48, 60, 64 ], [ [ 10 ], 42, 48, 60, 64 ],
  [ [ 11 ], 43, 48, 60, 64 ], [ [ 12 ], 44, 48, 60, 64 ], [ [ 13 ], 45, 48, 62, 64 ], [ [ 14 ], 46, 48, 62, 64 ], [ [ 15 ], 47, 48, 63, 64 ], [ [ 16 ], 48, 64 ],
  [ [ 17 ], 49, 50, 51, 52, 53, 54, 55, 56, 57, 58, 59, 60, 61, 62, 63, 64 ], [ [ 18 ], 50, 52, 54, 56, 58, 60, 62, 64 ], [ [ 19 ], 51, 52, 55, 56, 59, 60, 63, 64 ],
  [ [ 20 ], 52, 56, 60, 64 ], [ [ 21 ], 53, 56, 62, 64 ], [ [ 22 ], 54, 56, 62, 64 ], [ [ 23 ], 55, 56, 63, 64 ], [ [ 24 ], 56, 64 ], [ [ 25 ], 57, 60, 62, 64 ],
  [ [ 26 ], 58, 60, 62, 64 ], [ [ 27 ], 59, 60, 63, 64 ], [ [ 28 ], 60, 64 ], [ [ 29 ], 61, 62, 63, 64 ], [ [ 30 ], 62, 64 ], [ [ 31 ], 63, 64 ], [ [ 32 ], 64 ],
  [ [ 33 ], 17, 20, 29, 32, 53, 56, 62, 64 ], [ [ 34 ], 18, 20, 30, 32, 54, 56, 62, 64 ], [ [ 35 ], 19, 20, 31, 32, 55, 56, 63, 64 ], [ [ 36 ], 20, 32, 56, 64 ],
  [ [ 37 ], 21, 32, 56, 64 ], [ [ 38 ], 22, 32, 56, 64 ], [ [ 39 ], 23, 32, 56, 64 ], [ [ 40 ], 24, 32, 56, 64 ], [ [ 41 ], 25, 32, 60, 64 ], [ [ 42 ], 26, 32, 60, 64 ],
  [ [ 43 ], 27, 32, 60, 64 ], [ [ 44 ], 28, 32, 60, 64 ], [ [ 45 ], 29, 32, 62, 64 ], [ [ 46 ], 30, 32, 62, 64 ], [ [ 47 ], 31, 32, 63, 64 ], [ [ 48 ], 32, 64 ],
  [ [ 49 ], 9, 12, 14, 16, 25, 28, 30, 32, 41, 44, 46, 48, 57, 60, 62, 64 ], [ [ 50 ], 10, 12, 14, 16, 26, 28, 30, 32, 42, 44, 46, 48, 58, 60, 62, 64 ],
  [ [ 51 ], 11, 12, 15, 16, 27, 28, 31, 32, 43, 44, 47, 48, 59, 60, 63, 64 ], [ [ 52 ], 12, 16, 28, 32, 44, 48, 60, 64 ],
  [ [ 53 ], 13, 14, 15, 16, 29, 30, 31, 32, 45, 46, 47, 48, 61, 62, 63, 64 ], [ [ 54 ], 14, 16, 30, 32, 46, 48, 62, 64 ], [ [ 55 ], 15, 16, 31, 32, 47, 48, 63, 64 ],
  [ [ 56 ], 16, 32, 48, 64 ], [ [ 57 ], 5, 8, 14, 16, 21, 24, 30, 32, 37, 40, 46, 48, 53, 56, 62, 64 ], [ [ 58 ], 6, 8, 14, 16, 22, 24, 30, 32, 38, 40, 46, 48, 54, 56, 62, 64 ],
  [ [ 59 ], 7, 8, 15, 16, 23, 24, 31, 32, 39, 40, 47, 48, 55, 56, 63, 64 ], [ [ 60 ], 8, 16, 24, 32, 40, 48, 56, 64 ],
  [ [ 61 ], 3, 4, 7, 8, 11, 12, 15, 16, 19, 20, 23, 24, 27, 28, 31, 32, 35, 36, 39, 40, 43, 44, 47, 48, 51, 52, 55, 56, 59, 60, 63, 64 ],
  [ [ 62 ], 4, 8, 12, 16, 20, 24, 28, 32, 36, 40, 44, 48, 52, 56, 60, 64 ], [ [ 63 ], 2, 4, 6, 8, 10, 12, 14, 16, 18, 20, 22, 24, 26, 28, 30, 32, 34, 36, 38, 40, 42, 44, 46, 48, 50,
      52, 54, 56, 58, 60, 62, 64 ], [ [ 64 ], 1, 2, 3, 4, 5, 6, 7, 8, 9, 10, 11, 12, 13, 14, 15, 16, 17, 18, 19, 20, 21, 22, 23, 24, 25,
      26, 27, 28, 29, 30, 31, 32, 33, 34, 35, 36, 37, 38, 39, 40, 41, 42, 43, 44, 45, 46, 47, 48, 49, 50, 51, 52, 53, 54, 55, 56, 57, 58, 59, 60, 61, 62, 63, 64 ] ]
\end{comp}

\printindex

\section*{Acknowledgement}

I am grateful for Mohamed Elhamdadi for initially motivating my research on the generalizations of Laver tables. Mohamed originally suggested to work on the algebras of elementary embeddings and self-distributive algebras. Furthermore, I and Mohamed have had several informal seminars about the algebras of elementary embeddings which have prepared me for my original observations about the generalizations of Laver tables.

\bibliographystyle{amsplain}

\begin{thebibliography}{9}
\bibitem{AAG}
Iris Anshel, Michael Anshel, and Dorian Goldfeld. An algebraic method for public--key cryptography. Mathematical Research Letters 6, 1–-5 (1999)


\bibitem{B83} Hans-J. Bandelt and Jarmila Hedlíková. Median algebras.
Discrete Mathematics. Volume 45, Issue 1, Pages 1--30. 1983

\bibitem{PQC}
Daniel J. Bernstein, Johannes Buchmann, Erik Dahmen. Post-Quantum Cryptography. Springer 2009.

\bibitem{FIB}
Jean Berstel. Fibonacci words— A Survey. In: The Book of L. Springer, Berlin, Heidelberg.


\bibitem{B} Guy Roger Biyogmam. Lie n-racks. Comptes Rendus Mathematique. Volume 349, Issues 17--18, Pages 957--960. September 2011.


\bibitem{PK} Paul Corazza. The gap between I3 and the wholeness axiom. 2003. Fundamenta Mathematicae. 179. 1. 43--60 .

\bibitem{VHOD} Paul Corazza. Consistency of V = HOD With the Wholeness Axiom. Archive for Mathematical Logic. April 2000, Volume 39, Issue 3, pp 219–-226.

\bibitem{SCPD}Patrick Dehornoy. Using shifted conjugacy in braid-based cryptography.  L. Gerritzen et al. (ed.), Algebraic Methods in Cryptography (Mainz, 2005; Bochum, 2005), Contemporary Mathematics 418, AMS, Providence (2006) 65–-73.

\bibitem{D} Patrick Dehornoy. Braids and self-distributivity.Birkh\"auser Verlag, Basel. Progress in Mathematics. vol 192. 2000.


\bibitem{D94}
P. Dehornoy. Braid groups and left distributive operations. Trans. Amer. Math. Soc. 345--1, 115--151. (1994)

\bibitem{LRLDT}
Laver's results and low-dimensional topology. Patrick Dehornoy. Arch. Math. Logic (2016) 55: 49.

\bibitem{D89}
Patrick Dehornoy. Sur la structure des gerbes libres. Comptes Rendus
de l’Acad´emie des Sciences de Paris, 309:143--148, 1989.

\bibitem{D9901} Patrick Dehornoy. Construction of Self-Distributive Operations and Charged Braids. International Journal of Algebra and ComputationVol. 10, No. 01, pp. 173--190 (2000) 


\bibitem{D97} Patrick Dehornoy. Multiple left distributive systems. Comment.Math.Univ.Carolin. 38,4 (1997) 615–-625.

\bibitem{DL}
Patrick Dehornoy and Victoria Lebed. Two- and three-cocycles for Laver tables J. Knot Theory Ramifications Vol. 23, No. 04, 1450017 (2014)

\bibitem{RD92} Randall Dougherty. Critical points in an algebra of elementary embeddings I. Annals of Pure and Applied Logic 65 (3):211--241 (1993) 

\bibitem{RD95} Randall Dougherty. Critical points in an algebra of elementary embeddings (II). In Wilfrid Hodges, editor, Logic: From Foundations to Applications, pages 103-–136. Academic Press, San Diego, 1996.

\bibitem{DJ} Finite Left-Distributive Algebras and Embedding Algebras. Randall Dougherty, Thomas Jech. Advances in Mathematics. Volume 130, Issue 2, 25 September 1997, Pages 201--241.

\bibitem{AD93}Homomorphisms of primitive left distributive groupoids. Aleš Drápal. Pages 2579--2592 Received 01 Mar 1993, Published online: 27 Jun 2007


\bibitem{AD97} Drapal. 1997 Finite left distributive algebras with one generator, J. Pure Appl. Algebra, 121, 233--25l. 1

\bibitem{AD97+} Drapal. 1997 Finite left distributive groupoids with one generator, Int. J. Algebra 8 Computation, 7--6, 723--748. 

\bibitem{EGM} Ternary Distributive Structures and Quandles. Mohamed Elhamdadi. Matthew Green. Abdenacer Makhlouf. Preprint.

\bibitem{FK} Handbook of set theory. Vols. 1, 2, 3. Foreman, M.
Kanamori, A. 2010. Springer, Dordrecht.

\bibitem{HF97} Applications of Large Cardinals to Graph Theory, October 23, 1997, 36 pages, Harvey Friedman

\bibitem{HF06} $\Pi_{0}^{1}$-incompleteness: finite graph theory 1, Harvey Friedman

\bibitem{HF98}Finite functions and the necessary use of large cardinals.  Harvey M. Friedman. Annals of Mathematics
Second Series, Vol. 148, No. 3 (Nov., 1998), pp. 803--893.

\bibitem{HF11} Boolean relation theory and incompleteness. Harvey M. Friedman. Book.

\bibitem{J} A classifying invariant of knots, the knot quandle. Joyce, D. J. Pure Appl. Algebra. 1982.



\bibitem{K} Kanamori, A. The higher infinite. Springer Monographs in Mathematics. 2009.


\bibitem{TK} Tomáš Kepka, Notes on left distributive groupoids. Acta Universitatis Carolinae. Mathematica et Physica (1981). Volume: 022, Issue: 2, page 23-37. ISSN: 0001--7140


\bibitem{CLCHKP}
K. Ko, S. Lee, J. Cheon, J. Han, J. kang C. Park. New public key cryptosystem using braid groups. Crypto’2000, LNCS 1880, pp. 166--183, Springer 2000.

\bibitem{DL1994a}
David Larue. Left-distributive and left-distributive idempotent algebras, Ph.D. Thesis. University of Colorado, Boulder. 1994

\bibitem{L95} Richard Laver. On the algebra of elementary embeddings of a rank into itself, Advances in Math. Volume 110, Issue 2, February 1995, Pages 334--346

\bibitem{L92}
Richard Laver. The left distributive law and the freeness of an algebra of elementary embeddings. Advances in Mathematics. Volume 91, Issue 2, February 1992, Pages 209--231

\bibitem{L97} Implications between strong large cardinal axioms. Richard Laver. Annals of Pure and Applied Logic. Volume 90, Issues 1–3, 15 December 1997, Pages 79--90

\bibitem{VLC}
Cohomology of finite monogenic self-distributive structures. Victoria Lebed. Journal of Pure and Applied Algebra. Volume 220, Issue 2, February 2016, Pages 711--734.

\bibitem{AM}
Quantum algorithms: an overview. Ashley Montanaro. npj Quantum Information volume 2, Article number: 15023 (2016)

\bibitem{SR}A. B. Romanowska, and Jonathan D. H. Smith. Modes. River Edge, NJ: World Scientific, 2002.

\bibitem{KS}
Kentaro Sato. Double helix in large large cardinals and iteration of elementary embeddings. Annals of Pure and Applied Logic. Volume 146, Issues 2–3, May 2007, Pages 199--236

\bibitem{MS}
Matthew Smedberg. A dense family of well-behaved finite monogenerated left-distributive groupoids. Archive for Mathematical Logic, Volume 52, Issue 3, pp 377--402. May 2013


\bibitem{S} J. R. Steel. The Well-Foundedness of the Mitchell Order. The Journal of Symbolic Logic. Vol. 58, No. 3 (Sep. 1993), pp. 931--940


\bibitem{KT} Arkadius Kalka and Mina Teicher. Non-associative key establishment for left distributive systems. Groups Complexity Cryptology. Volume 5. Issue 2. Published online 2013.


\end{thebibliography}

\end{document}